\documentclass[oneside,a4paper,english, 11pt]{amsart}
\usepackage{ulem} %

\usepackage{amssymb,amsmath,amsthm}
\usepackage[mathscr]{eucal}%
\usepackage[all]{xy}
\usepackage{lmodern}
\usepackage[T1]{fontenc}
\usepackage[utf8]{inputenc}
\usepackage{tikz-cd}
\usepackage[shortlabels]{enumitem}
\usepackage{relsize}
\usepackage{geometry}
\usepackage{mathtools}
\usepackage{adjustbox}
\usepackage{lscape}
\usepackage{diagbox}
\usepackage{slashbox}
\usepackage{bm}
\usepackage{caption}
\usepackage{float}

\usepackage{stmaryrd}
\usepackage{mathrsfs}  
\usepackage{multirow}
\usepackage[hypertexnames=false]{hyperref}

\theoremstyle{plain}
\newtheorem{prop}{Proposition}[subsection]

\newtheorem{lem}[prop]{Lemma}
\newtheorem{thm}[prop]{Theorem}
\newtheorem{cor}[prop]{Corollary}
\theoremstyle{definition}
\newtheorem{definit}[prop]{Definition}

\newtheorem{rem}[prop]{Remark}

\theoremstyle{plain}

\theoremstyle{definition}

\theoremstyle{plain}
\newtheorem{propIntro}{Proposition}[section]

\newtheorem{thmIntro}[propIntro]{Theorem}

\theoremstyle{definition}

\DeclareMathAlphabet{\mathpzc}{OT1}{pzc}{m}{it}

\DeclarePairedDelimiter{\norm}{||}{||}

\DeclarePairedDelimiter{\set}{\{}{\}}

\DeclareMathOperator{\End}{End}
\DeclareMathOperator{\Hom}{Hom}
\DeclareMathOperator{\Ind}{Ind}

\DeclareMathOperator{\Res}{Res}

\DeclareMathOperator{\GL}{GL}

\DeclareMathOperator{\SL}{SL}

\DeclareMathOperator{\Gal}{Gal}
\DeclareMathOperator{\Image}{Im}

\DeclareMathOperator{\tr}{tr}

\DeclareMathOperator{\Proj}{Proj}
\DeclareMathOperator{\Inj}{Inj}

\DeclareMathOperator{\Id}{Id}

\DeclareMathOperator{\soc}{soc}
\DeclareMathOperator{\cosoc}{cosoc}

\DeclareMathOperator{\rad}{rad}

\DeclareMathOperator{\Ext}{Ext}
\DeclareMathOperator{\nr}{nr}
\DeclareMathOperator{\gr}{gr}

\DeclareMathOperator{\ad}{ad}

\DeclareMathOperator{\pr}{pr}
\DeclareMathOperator{\Ann}{Ann}

\newcommand{\m}{\mathfrak{m}}

\newcommand{\NN}{\mathbb{N}}
\newcommand{\ZZ}{\mathbb{Z}}

\newcommand{\cC}{\mathcal{C}}

\newcommand{\QK}[1]{I_{1}}

\newcommand{\cW}{\mathcal{W}}

\newcommand{\Qp}{\mathbb{Q}_{p}}

\newcommand{\Zp}{\mathbb{Z}_{p}}
\newcommand{\F}{\mathbb{F}}
\newcommand{\Fp}{\mathbb{F}_{p}}

\newcommand{\Qpbar}{\overline{\mathbb{Q}}_p}

\newcommand{\Fpbar}{\overline{\mathbb{F}}_p}

\newcommand{\sgn}{\mathrm{sgn}}

\newcommand{\QQ}{\mathbb{Q}}

\newcommand{\RR}{\mathbb{R}}
\newcommand{\bV}{\mathbb{V}}
\newcommand{\wt}[1]{\widetilde{#1}}
\newcommand{\ang}[1]{\langle #1 \rangle}

\newcommand{\onto}{\twoheadrightarrow}
\newcommand{\into}{\hookrightarrow}
\newcommand{\congto}{\xrightarrow{\,\sim\,}}
\newcommand{\eps}{\varepsilon}
\newcommand{\cE}{\mathcal{E}}
\newcommand{\cI}{\mathcal{I}}
\newcommand{\cJ}{{\mathcal{J}}}
\newcommand{\cL}{\mathcal{L}}
\newcommand{\cM}{\mathcal{M}}
\newcommand{\cO}{\mathcal{O}}

\newcommand{\rbar}{\overline{r}}
\newcommand{\rhobar}{\overline{\rho}}
\newcommand{\taubar}{\overline{\tau}}

\newcommand{\brho}{\overline{\rho}}%
\newcommand{\smatr}[4]{\bigl(\begin{smallmatrix} {#1}& {#2}\\ {#3}&{#4}\end{smallmatrix}\bigl)}%

\newcommand{\defeq}{\stackrel{\textrm{\tiny{\upshape{def}}}}{=}}
\newcommand{\teich}[1]{\widetilde{#1}}

\newcommand{\ovl}[1]{\overline{#1}}

\newcommand{\un}[1]{\underline{#1}}
\renewcommand{\bf}[1]{\mathbf{#1}}
\newcommand{\tld}[1]{\tilde{#1}}
\newcommand{\wtld}[1]{\widetilde{#1}}

\DeclareMathOperator{\JH}{\mathrm{JH}}
\newcommand{\rig}{\mathrm{rig}}

\newcommand{\orient}{\mathrm{or}}
\newcommand{\phz}{\varphi}
\newcommand{\ra}{\rightarrow}

\newcommand{\lra}{\longrightarrow}

\newcommand{\ppar}[1]{(\mkern-3mu(#1)\mkern-3mu)}
\newcommand{\bbra}[1]{\llbracket #1\rrbracket}

\newcommand{\fM}{\mathfrak{M}}

\newcommand{\fp}{\mathfrak{p}}
\newcommand{\fQ}{\mathfrak{q}}

\newcommand{\fS}{\mathfrak{S}}

\newcommand{\fW}{\mathfrak{w}}

\newcommand{\fm}{\mathfrak{m}}

\renewcommand{\t}{\mathfrak{t}}

\DeclareMathOperator{\Mat}{Mat}
\DeclareMathOperator{\Ad}{Ad}
\DeclareMathOperator{\Adm}{Adm}
\DeclareMathOperator{\id}{id}

\DeclareMathOperator{\reg}{reg}

\DeclareMathOperator{\Spec}{Spec}

\DeclareMathOperator{\Supp}{Supp}

\DeclareMathOperator{\Lie}{Lie}
\DeclareMathOperator{\M}{M}

\newcommand{\xto}[1][]{\xrightarrow{#1}}
\newcommand{\simto}{\xto[\sim]} %

\renewcommand{\subset}{\subseteq}
\renewcommand{\simeq}{\cong}

\newcommand{\oS}{{\ovl S}}
\newcommand{\fq}{\fQ} 
\definecolor{olive}{rgb}{0.5, 0.5, 0.0}

\title{Gelfand--Kirillov dimension and mod $p$ cohomology for $\GL_2$}

\author{Christophe Breuil}
\address{CNRS, Universit\'e Paris--Saclay, Laboratoire de math\'ematiques d'Orsay, 91405, Orsay, France}
\email{christophe.breuil@universite-paris-saclay.fr}

\author{Florian Herzig}
\address{Department of Mathematics, University of Toronto, 40 St. George Street, Toronto, ON M5S 2E4, Canada}
\email{herzig@math.toronto.edu}
 
\author{Yongquan Hu}
\address{Morningside Center of Mathematics, Academy of Mathematics and Systems Science, Chinese Academy of Sciences, Beijing 100190, China; University of the Chinese Academy of Sciences, Beijing 100049, China}
\email{yhu@amss.ac.cn}

\author{Stefano Morra}
\address{Universit\'e Paris 8, Laboratoire d'Analyse, G\'eom\'etrie et Applications, LAGA, Universit\'e Sorbonne Paris Nord, CNRS, UMR 7539, F-93430, Villetaneuse, France}
\email{morra@math.univ-paris13.fr}

\author{Benjamin Schraen}
\address{Universit\'e Paris--Saclay, CNRS, Laboratoire de math\'ematiques d'Orsay, 91405, Orsay, France}
\email{benjamin.schraen@universite-paris-saclay.fr}

\begin{document}
 
 \begin{abstract}
Let $p$ be a prime number, $F$ a totally real number field unramified at places above $p$ and $D$ a quaternion algebra of center $F$ split at places above $p$ and at no more than one infinite place. Let $v$ be a fixed place of $F$ above $p$ and $\rbar : {\rm Gal}(\overline F/F)\rightarrow \GL_2(\Fpbar)$ an irreducible modular continuous Galois representation which, at the place $v$, is semisimple and sufficiently generic (and satisfies some weak genericity conditions at a few other finite places). We prove that many of the admissible smooth representations of $\GL_2(F_v)$ over $\Fpbar$ associated to $\rbar$ in the corresponding Hecke-eigenspaces of the mod $p$ cohomology have Gelfand--Kirillov dimension $[F_v:\Qp]$, as well as several related results. 
\end{abstract}

\maketitle

\setlength{\parskip}{0mm}
\tableofcontents
\setlength{\parskip}{3mm}

\section{Introduction}

\normalem %

\subsection{{Torsion in cohomology and Gelfand--Kirillov dimension}}

Fix a prime number $p$, a totally real number field $F$ which is unramified at places above $p$, and a quaternion algebra $D$ of center $F$ which is split at places above $p$ and at exactly one infinite place. For $V$ a compact open subgroup of $(D\otimes_F{\mathbb A}_F^\infty)^\times$ denote by $X_V$ the associated smooth projective Shimura curve over $F$. Let $v$ be a fixed place of $F$ above $p$ and $\F$ a finite extension of $\Fp$ (``sufficiently large'', as usual). This paper is concerned with admissible smooth representations of $\GL_2(F_v)$ over $\F$ of the form
\begin{equation}\label{goal}
\pi\defeq \varinjlim_{V_v}\Hom_{{\rm Gal}(\overline F/F)}\!\big(\rbar,H^1_{{\rm \acute et}}(X_{V^vV_v} \times_F \overline F, \F)\big),
\end{equation}
where $V^v$ is a fixed compact open subgroup of $(D\otimes_F{\mathbb A}_F^{\infty,v})^\times$, the inductive limit running over compact open subgroups $V_v$ of $(D\otimes_FF_v)^\times\cong \GL_2(F_v)$ and $\rbar : {\rm Gal}(\overline F/F)\rightarrow \GL_2(\F)$ is a continuous absolutely irreducible Galois representation such that $\pi\ne 0$. Understanding such representations $\pi$ of $\GL_2(F_v)$ attached to Galois representations is important, as it is hoped that they realize a mod $p$ Langlands correspondence. For instance, when $F=\QQ$ (and $X_V$ is the compactified modular curve), {under weak assumptions on $\rbar\vert_{{\rm Gal}(\Qpbar/\Qp)}$  the representation $\pi$ of $\GL_2(\Qp)$ is well understood (see \cite{emerton-local-global})}.

{This is far from being the case when $F_v\ne \Qp$, despite a great amount of effort during the past 20 years and we only have few guidelines from modularity lifting expectations.
In particular, the work of \cite{GN}, which follows the heuristic of \cite[\S 3.1.1]{EmertonICM}, shows how relevant geometric properties of the ``big'' Hecke algebra are consequences of the \emph{Gelfand--Kirillov dimension} of $\pi$ (a measure of the growth of the dimension of invariant subspaces under principal congruence subgroups).}
{
For $F_v=\Qp$ this dimension is known by \cite{morra-inv}, thanks to the explicit description of the \emph{supersingular} representations of $\GL_2(\Qp)$ \cite{breuilI}, but if  $F_v\ne \Qp$ the (over-)abundance of  supersingular representations (\cite{BP}, \cite{yongquan-algebra}) makes it more difficult to obtain information, even for the invariants under the first congruence subgroup (\cite{LMS}, \cite{HuWang}, \cite{DanWild}, which are based on the patching construction of \cite{EGS}).
}

{The aim of this work is to lift a corner of the veil surrounding the smooth representations $\pi$ coming from cohomology, by establishing their Gelfand--Kirillov dimension.
{Besides applications to the flatness of completed homology over a big Hecke algebra (Theorem~\ref{thm:platitude-intro} below) and on the candidate of \cite{CEGGPS} for the $p$-adic Langlands correspondence (Theorem~\ref{thmGN:application} below), our methods also lead us to an abelian subcategory of the category of smooth representations of $\GL_2(F_v)$ that has desirable finiteness property, with further applications to a functor towards Galois representations; cf.\ our subsequent work (\cite{BHHMS2,BHHMS3}).}
}

{
We now describe in more detail our results.
}

\subsection{{The main theorem and its consequences}}

In order to state our main theorem, we first give the precise definition of $\dim_{\GL_2(F_v)}(\pi)$, the Gelfand--Kirillov dimension of $\pi$ in the context of smooth $\GL_2(F_v)$-representations over mod $p$ vector spaces.\footnote{Strictly speaking, this is not quite the Gelfand--Kirillov dimension of $\pi$, see Remark \ref{GKdimisnotGKdim} in the text, but this is the only dimension we will consider.} We let $f\defeq [F_v:\Qp]$, $K\defeq \GL_2({\mathcal O}_{F_v})$, $K_n\defeq 1+p^nM_2({\mathcal O}_{F_v})\subset K$ for $n\geq 1$, $Z_1$ the center of $K_1$, and we assume $p>2$. For $\pi$ a nonzero admissible smooth representation of $\GL_2(F_v)$ over $\F$ with a central character, we set (see \S\ref{sec:kirillov})
$$\dim_{\GL_2(F_v)}(\pi) \defeq 3f - \min\set{d\geq 0 : \Ext^d_{\F\bbra{K_1/Z_1}}(\pi^\vee,\F\bbra{K_1/Z_1})\neq0},$$
where $\F\bbra{K_1/Z_1}$ is the Iwasawa algebra of $K_1/Z_1$ and $\pi^\vee$ is the algebraic dual of $\pi$, considered as a module over $\F\bbra{K_1/Z_1}$ (note that $Z_1$ acts trivially on $\pi$ and that $3f=\dim(\GL_2(F_v)/Z_1)$). Another equivalent and maybe more intuitive definition of $\dim_{\GL_2(F_v)}(\pi)$ is the following: it is the unique integer such that there exist $a\leq b$ in ${\mathbb R}_{>0}$ satisfying
$$a\leq \frac{\dim_{\F}(\pi^{K_n})}{p^{n\dim_{\GL_2(F_v)}(\pi)}} \leq b$$
for all $n\geq 1$ (see Remark \ref{GKdimisnotGKdim}). 
(As alluded above, the dimension $\dim_{\GL_2(F_v)}(\pi)$ measures the growth of $\pi^{K_n}$ when $n$ grows: for instance it is $0$ if and only if $\dim_{\F}(\pi)$ is finite and nonzero.)

\begin{thmIntro}[Corollary \ref{mainmain}]\label{mainintro}
Keep all the above assumptions on $F$, $D$, and assume that $\rbar$ is generic and %
that $\rbar |_{G_{F(\!\sqrt[p]{1})}}$ is absolutely irreducible. 
Let $V^v=\prod_{w\ne v}V_w$ with $V_w=\GL_2({\mathcal O}_{F_w})$ if neither $D$ nor $\rbar$ ramifies at $w$, and $V_w\subset 1+pM_2({\mathcal O}_{F_w})$ if $w\!\mid\!p$ \emph{(}$w\ne v$\emph{)}. 

Then for $\pi$ as in \eqref{goal} we have $\dim_{\GL_2(F_v)}(\pi)=f$.
\end{thmIntro}

We also prove the same statement for the analog of $\pi$ when $D$ is totally definite. Although we did not check it carefully, the same method should also work in other global settings in which the group is $\GL_2(F_v)$ at the place $v$, like for instance unitary groups which are forms of $\GL_2$. Moreover, from exchanges with Kozio\l{}, we believe the same result applies when, in the global setup, the unitary group is a nonsplit unramified unitary group at $v$. In a companion paper (and the same global setup), Hu and Wang prove an analog of Theorem~\ref{prop1intro} below and apply our Theorem~\ref{prop2intro} to deduce $\dim_{\GL_2(F_v)}(\pi)=[F_v:\Qp]$ when $\rbar\vert_{{\rm Gal}(\overline F_v/F_v)}$ is {\it not} semisimple and sufficiently generic (\cite{HuWang2}). 
The method of \emph{loc.~cit.} uses at several places that $\rbar\vert_{{\rm Gal}(\overline F_v/F_v)}$ is not semisimple, but in fact the method of the present work extends more or less directly to the non-semisimple case, see \cite{YWang}.
Finally, a variant of the strategy used in this paper was used by Hu and Wang in \cite{HuWang3} to prove an analog of Theorem~\ref{mainintro} in the case of quaternion algebras over $\Qp$.

{
By work of Gee--Newton (see \cite{GN}), Theorem \ref{mainintro} can be applied to obtain ``big $R$ equals big $T$'' results and flatness for the completed homology of towers of Shimura curves, when considered as a module over the ``big Hecke algebra''.
More precisely let $\psi$ be
the Teichm\"uller lift of the product of $\det(\rbar)$ and the
mod $p$ cyclotomic character, let
\[
  \widehat{H}^1(V^v)_{\rbar}^{\psi^{-1}}\defeq\varprojlim_{n}\varinjlim_{V_v}H^1(X_{V^vV_v}\times_F\overline{F},W(\F)/p^n)^{\psi^{-1}}_{\rbar}\]
be the $\psi^{-1}$-isotypic subspace of the completed cohomology
``localized at $\rbar$'', let
$\widehat{\mathbb{T}}(V^v)_{\rbar}^{\psi^{-1}}$ be the ``big Hecke algebra''
acting on it, and let $R_{\rbar,S}^{\psi}$ be the universal deformation ring of $\rbar$ parametrizing deformations $r$ of $\rbar$ which are unramified outside of $S$ and such that $\eps\det(r)=\psi$ (see \S\ref{sec:platitude_Hecke} for precise
definitions). Assume moreover that $p$ is inert in $F$ and that
$V_{w_1}$ is sufficiently small at a conveniently chosen place $w_1$ of $F$.}

\begin{thmIntro}[Corollary \ref{cor:platitude_Hecke}]\label{thm:platitude-intro}
{  There is an isomorphism $R_{\rbar,S}^{\psi}\simto\widehat{\mathbb{T}}(V^v)_{\rbar}^{\psi^{-1}}$, the $\widehat{\mathbb{T}}(V^v)_{\rbar}^{\psi^{-1}}$-module
  $\Hom_{W(\F)}(\widehat{H}^1(V^v)_{\rbar}^{\psi^{-1}},W(\F))$ is faithfully
  flat, and $\widehat{\mathbb{T}}(V^v)_{\rbar}^{\psi^{-1}}$ is a complete
  intersection.}
\end{thmIntro}

{We also prove the analogous result in the case of definite quaternion
algebras. Note that the isomorphism $R_{\rbar,S}^{\psi}\simto\widehat{\mathbb{T}}(V^v)_{\rbar}^{\psi^{-1}}$ is related to a theorem of Allen (\cite[Thm.~6.3.6]{Allen}) building on previous results of Gouv\^ea--Mazur and Chenevier (but without the determinant condition); however, flatness is new. This flatness was known in the case of modular
curves using the full strength of the $p$-adic Langlands
correspondence for $\GL_2(\Qp)$ and the local-global compatibility
result of \cite{emerton-local-global}.}

As mentioned above, Theorem \ref{mainintro} {also} has important consequences for the existence of admissible unitary Banach representations of $\GL_2(F_v)$ lifting the eigenspace of $\rbar$.
From now on we let
\begin{equation}\label{tobepatched}
\pi\defeq \varinjlim_{V_v}\Hom_{\prod_{\overset{w|p}{w\ne v}}\GL_2({\mathcal O}_{F_w})}\!\bigg(\bigotimes_{\overset{w|p}{w\ne v}} \sigma_w, \Hom_{G_F}\!\big(\rbar, H^1_{{\rm \acute et}}(X_{V^vV_v} \times_F \overline F, \F)\big)\bigg),
\end{equation}
where, for $w\!\mid\!p$, $w\ne v$, $\sigma_w$ is any Serre weight in the set $W(\rbar_w^\vee)$ of \cite[\S 3]{BDJ}, $V_w\subset 1+pM_2({\mathcal O}_{F_w})$ is normal in $\GL_2({\mathcal O}_{F_w})$, and $V_{w_1}$ is sufficiently small at a nice place $w_1$ where nothing ramifies.
(Note that, by d\'evissage, we can always replace $\pi$ as in (\ref{goal}) by (\ref{tobepatched}).)
The representation $\pi$ of $\GL_2(F_v)$ in (\ref{tobepatched}) can be ``patched'' as in \cite{CEGGPS} or \cite[\S 6]{DoLe} giving rise to a ``big'' profinite $R_\infty$-module ${\mathbb M}_\infty$ endowed with an $R_\infty$-linear continuous action of $\GL_2(F_v)$ such that ${\mathbb M}_\infty/{\mathfrak m}_\infty\cong \pi^\vee$.

\begin{thmIntro}[Corollary \ref{padiclanglands}]
\label{thmGN:application}
Keep the assumptions of Theorem \ref{mainintro} and let $x:R_\infty\rightarrow {\mathcal O}'$ be any homomorphism of local $W(\F)$-algebras, where ${\mathcal O}'$ is the ring of integers of a finite extension $E'$ of $W(\F)[1/p]$. Then
$$\Hom_{{\mathcal O}'}^{\rm cont}\big({\mathbb M}_\infty\otimes_{R_\infty,x}{\mathcal O}',E'\big)$$
is a \emph{(}nonzero\emph{)} admissible unitary Banach representation of $\GL_2(F_v)$ over $E'$ with a $\GL_2(F_v)$-invariant unit ball lifting $\pi\otimes_{\F}{\F'}$, where $\F'$ is the residue field of $\mathcal O'$.
\end{thmIntro}

{Note that $x:R_\infty\rightarrow {\mathcal O}'$ gives rise to a Galois representation $\rho_x :
\Gal(\overline F_v/F_v) \to \GL_2(E')$ and that $\Hom_{{\mathcal O}'}^{\rm cont}({\mathbb M}_\infty\otimes_{R_\infty,x}{\mathcal
  O}',E')$ is the natural candidate of \cite{CEGGPS} for the Banach space representation of $\GL_2(F_v)$ associated to
$\rho_x$ by the hypothetical $p$-adic Langlands correspondence. So far it was not known that this representation is nonzero in
this generality.}

{To deduce this from Theorem~\ref{mainintro}}, by Schikhof duality (see \cite[\S 1]{schneider-teitelbaum-IL}), it is enough to prove that ${\mathbb M}_\infty\otimes_{R_\infty,x}{\mathcal O}'$ is flat over ${\mathcal O}'$. But an argument due to Gee and Newton in \cite[Cor.\ A.30]{GN} (and usually called ``Miracle Flatness'') shows that, when $\dim_{\GL_2(F_v)}(\pi)= f$, the $R_\infty$-module ${\mathbb M}_\infty$ is indeed flat over $R_\infty$, whence the result by base change.

We also prove several variants and generalizations of Theorem \ref{mainintro}. For instance, without the assumption $V_w\subset 1+pM_2({\mathcal O}_{F_w})$ for $w\!\mid\!p$, we still have $\dim_{\GL_2(F_v)}(\pi)\leq f$, see Remark \ref{abitfurther}. We can take $V_w=\GL_2({\mathcal O}_{F_w})$ for $w$ outside any finite set $S$ containing the ramification places of $D$ and $\rbar$ provided $R_{\rbar_w}$ is formally smooth for all $w\in S$ prime to $p$ (see \emph{loc.~cit.}). It is likely that other variants of Theorem \ref{mainintro} can be proven, e.g.\ by fixing types at some places $w$ prime to $p$ instead of assuming $R_{\rbar_w}$ formally smooth. For instance, we have $\dim_{\GL_2(F_v)}(\pi_{D,v}(\rbar))=f$, where $\pi_{D,v}(\rbar)$ is the ``local factor'' $\pi_{D,v}(\rbar)$ of \cite[(3.3)]{BD} and \cite[\S 6.5]{EGS} (see Remark \ref{localfactor}).

{The notion of genericity for $\rbar$ appearing in Theorem \ref{mainintro} is mainly dictated by the current technology for studying potentially crystalline deformation rings (cf.~\cite{MLM}). It is made explicit as follows.} %
For a finite place $w$ of $F$, let $I_{F_w}$ be the inertia subgroup at $w$ and $\omega_{f'}$, $f'\in \{f,2f\}$ be Serre's fundamental character of level $f'$.
Then:
\begin{enumerate}
\item\label{wramifies}for $w\!\nmid\! p$ such that either $D$ or $\rbar$ ramifies, the framed deformation ring $R_{\rbar_w}$ of $\rbar_w\defeq \rbar\vert_{{\rm Gal}(\overline F_w/F_w)}$ over the Witt vectors $W(\F)$ is formally smooth;
\item for $w\!\mid\!p$, $w\ne v$, $\rbar\vert_{I_{F_w}}$ is generic in the sense of \cite[Def.\ 11.7]{BP};
\item\label{wisv}$\rbar\vert_{I_{F_v}}$ is semisimple of one of the following forms up to twist:
\begin{enumerate}
\item$\begin{pmatrix}\omega_f^{(r_0+1)+\cdots+p^{f-1}(r_{f-1}+1)}&0\\0&1\end{pmatrix}$\ \ {$12\leq r_i\leq p-15$},
\item$\begin{pmatrix}\omega_{2f}^{(r_0+1)+\cdots+p^{f-1}(r_{f-1}+1)}&0\\0&\omega_{2f}^{p^f({\rm same})}\end{pmatrix}$\ {$13\leq r_0\leq p-14$, $12\leq r_i\leq p-15$ for $i>0$.}
\end{enumerate}
\end{enumerate}

Note that \ref{wisv} {implies $p>23$} and that \ref{wramifies} can be made explicit (\cite{Shotton}).

\subsection{The proof}
We now sketch the proof of Theorem \ref{mainintro}. 

\subsubsection{Smooth representations}
\label{sub:sub:smoothrep}

{A key step in our method is to show that the representations $\pi$ appearing in Theorem \ref{mainintro} satisfy a ``minimal multiplicity'' condition, namely condition (\ref{eq:fund:intro}) of Proposition \ref{prop1intro} below.
{It is this condition that plays a key role in our subsequent work \cite{BHHMS2,BHHMS3}.}
}

{We describe these results in more detail.}
We let $k(\cong {\mathbb F}_{p^f})$ be the residue field of $F_v$, and for each Serre weight $\sigma\in W(\rbar_v^\vee)$, we define $D_{0,\sigma}$ as the largest subrepresentation of the injective envelope $\Inj_{\GL_2(k)}\sigma$ such that $\sigma$ only appears in the socle of $D_{0,\sigma}$ and no other Serre weight of $W(\rbar_v^\vee)$ is a constituent of $D_{0,\sigma}$. We set $D_0(\rbar_v^\vee)\defeq \bigoplus_{\sigma\in W(\rbar_v^\vee)}D_{0,\sigma}$ as in \cite[\S13]{BP}. We also denote by $\mathfrak{m}_{K_1/Z_1}$ the maximal ideal of $\F\bbra{K_1/Z_1}$. In order to get the above upper bound on $\dim_{\GL_2(F_v)}(\pi)$, we will apply the following theorem to $\pi$ in (\ref{tobepatched}).

\begin{thmIntro}[Theorem \ref{thm:GKdim-criterion}]\label{thmintro1}
Let $\pi$ be an admissible smooth representation of $\GL_2(F_v)$ over $\F$ with a central character. Assume that
\begin{enumerate}
\item\label{k1intro}we have an isomorphism $\pi^{K_1}=\pi[\mathfrak{m}_{K_1/Z_1}]\cong D_0(\rbar_v^\vee)^{\oplus r}$ of representations of $\GL_2(k)$ for some $r\geq 1$;
\item\label{multintro}we have $[\pi[\mathfrak{m}_{K_1/Z_1}^2] : \sigma]=[\pi[\mathfrak{m}_{K_1/Z_1}] : \sigma]$ for all $\sigma\in W(\rbar_v^\vee)$.
\end{enumerate}
Then $\dim_{\GL_2(F_v)}(\pi)\leq f$.
\end{thmIntro}

(In fact we prove in Theorem \ref{thm:GKdim-criterion} a slightly stronger statement.) Condition \ref{k1intro} in Theorem \ref{thmintro1} is already familiar, for instance it is satisfied with $r=1$ by the representation $\pi_{D,v}(\rbar)$ mentioned above (see \cite{HuWang} and \cite{LMS}, which build upon \cite{BP} and \cite{EGS}). Thus it is rather condition \ref{multintro} which is important. Though it is purely local, the proof of Theorem \ref{thmintro1} is not at all trivial, and it took us a long time before finding a proof (or even convincing ourselves that the statement was true!). The key idea is to look at the action on $\pi$ of the {\it Iwahori subgroup} $I$ of $K$ instead of $K$ itself. The proof of Theorem~\ref{thmintro1} is divided into two steps. The first step is the following result, where $I_1\subset I$ is the pro-$p$-Iwahori subgroup and $\mathfrak{m}_{I_1/Z_1}$ is the maximal ideal of the Iwasawa algebra $\F\bbra{I_1/Z_1}$.

\begin{thmIntro}[Proposition \ref{prop-W3topi=dim1}]\label{prop1intro}
Let $\pi$ be an admissible smooth representation of $\GL_2(F_v)$ over $\F$ with a central character and assume $\pi$ satisfies \ref{k1intro} and \ref{multintro} of Theorem \ref{thmintro1}. Then for all continuous characters $\chi:I\rightarrow \F^\times$ such that $[\pi[\mathfrak{m}_{I_1/Z_1}] : \chi]\neq 0$ we have:
\begin{equation}
\label{eq:fund:intro}
[\pi[\mathfrak{m}_{I_1/Z_1}^3] : \chi]=[\pi[\mathfrak{m}_{I_1/Z_1}] : \chi].
\end{equation}
\end{thmIntro}

Note that ${\rm socle}(\pi|_I)=\pi[\mathfrak{m}_{I_1/Z_1}]=\pi^{I_1}$ since $p>2$. The proof of Theorem \ref{prop1intro} is given in \S\ref{sec:smooth:rep}. It is a bit long and technical, but is rather standard (to apply Proposition \ref{prop-W3topi=dim1} to $\pi$ as in Theorem \ref{thmintro1} one actually needs Corollary \ref{cor:J-fil} and Lemma \ref{lem:connected}, see \S\ref{sec:multiplicityoneprop}).

The second step is the following key result which gives the sought-after upper bound on the Gelfand--Kirillov dimension.

\begin{thmIntro}[Corollary \ref{cor:GKdim}]\label{prop2intro}
Let $\pi$ be an admissible smooth representation of $\GL_2(F_v)$ over $\F$ with a central character and assume $[\pi[\mathfrak{m}_{I_1/Z_1}^3] : \chi]=[\pi[\mathfrak{m}_{I_1/Z_1}] : \chi]$ for all $\chi:I\rightarrow \F^\times$ such that $[\pi[\mathfrak{m}_{I_1/Z_1}] : \chi]\neq 0$. Then $\dim_{\GL_2(F_v)}(\pi)\leq f$.
\end{thmIntro}

Let us sketch the proof of Theorem \ref{prop2intro}. We view the algebraic dual $\pi^\vee$ as a (finitely generated) module over $\F\bbra{I_1/Z_1}$ and denote by $\gr_\mathfrak{m}\pi^\vee$ the associated graded module over $\gr_\mathfrak{m}\F\bbra{I_1/Z_1}$ for the $\mathfrak{m}_{I_1/Z_1}$-adic filtration. The graded ring $\gr_\mathfrak{m}\F\bbra{I_1/Z_1}$ is not commutative, as the pro-$p$ group $I_1/Z_1$ is not uniform (see \cite{Clozel} and \S\ref{sec:propIwahoriGL2}). But the assumption $[\pi[\mathfrak{m}_{I_1/Z_1}^3] : \chi]=[\pi[\mathfrak{m}_{I_1/Z_1}] : \chi]$ implies that the action of $\gr_\mathfrak{m}\F\bbra{I_1/Z_1}$ on $\pi^\vee$ factors through a {\it commutative} quotient $(\gr_\mathfrak{m}\F\bbra{I_1/Z_1})/I_{I_1/Z_1}$, where $I_{I_1/Z_1}$ is an explicit $2$-sided ideal of $\gr_\mathfrak{m}\F\bbra{I_1/Z_1}$ generated by certain degree $2$ elements (see Theorem \ref{quotientalg}). More precisely one has
\begin{equation}\label{grquotient}
\big(\gr_\mathfrak{m}\F\bbra{I_1/Z_1}\big)/I_{I_1/Z_1}\cong \F[e_i,f_i;\,0\leq i\leq f-1]/(e_if_i ;\, 0\leq j\leq f-1),
\end{equation}
where the (commutative) polynomial algebra $\F[e_i,f_i;\,0\leq i\leq f-1]$ is itself the quotient of $\gr_\mathfrak{m}\F\bbra{I_1/Z_1}$ by a regular sequence $(h_0,\dots,h_{f-1})$ of central elements. By a general lemma (Lemma \ref{prop:grsemiabelian}), $\dim_{\GL_2(F_v)}(\pi)$ is equal to the dimension of the support of $\gr_\mathfrak{m}\pi^\vee$ in the polynomial algebra
$$\big(\gr_\mathfrak{m}\F\bbra{I_1/Z_1}\big)/(h_0,\dots,h_{f-1})\cong \F[e_i,f_i;\,0\leq i\leq f-1],$$
which by (\ref{grquotient}) is smaller or equal than $\dim(\gr_\mathfrak{m}\F\bbra{I_1/Z_1}/I_{I_1/Z_1})=2f-f=f$. So we see that the fact that $\gr_\mathfrak{m}\pi^\vee$ (for an admissible smooth representation of $\GL_2(F_v)$ over $\F$) is a module over $(\gr_\mathfrak{m}\F\bbra{I_1/Z_1})/I_{I_1/Z_1}$, and not just over $\gr_\mathfrak{m}\F\bbra{I_1/Z_1}$, turns out to be an important condition.

\subsubsection{{Patching: the setup}}
\label{subsub:patch}

We now apply Theorem \ref{thmintro1} to $\pi$ in (\ref{tobepatched}). For this, we need to prove that $\pi$ satisfies conditions \ref{k1intro} and \ref{multintro} of Theorem \ref{thmintro1}. We first sketch the proof of (ii), which is the harder and more important one. We fix an arbitrary Serre weight $\sigma$ in $W(\rbar_v^\vee)$. We need to prove
\begin{eqnarray}\label{mult1intro}
\Hom_{K}(\sigma, \pi)\buildrel\sim\over\longrightarrow\Hom_{K}\big((\Proj_{K/Z_1}\sigma)/\mathfrak{m}_{K_1/Z_1}^2, \pi\big),
\end{eqnarray}
where $\Proj_{K/Z_1}\sigma$ is the algebraic dual of the injective envelope $\Inj_{K/Z_1}\sigma^\vee$ of $\sigma^\vee$ in the category of smooth representations of $K/Z_1$ over $\F$. 

We do not know any other way to prove (\ref{mult1intro}) than to ``patch'' (the dual of) both sides using the patching functors of \cite{EGS}. This strategy is not new: it is initially due to Emerton, Gee, Savitt in \cite{EGS} (generalizing work of Diamond, of Fujiwara, and using of course the work of Taylor, Wiles and of Kisin) and has been generalized by Le, Morra, Schraen, by Hu, Wang, and by Le in \cite{LMS}, \cite{HuWang}, \cite{Le} who proved (under various hypotheses) a result analogous to (\ref{mult1intro}) but with $\mathfrak{m}_{K_1/Z_1}$ instead of $\mathfrak{m}_{K_1/Z_1}^2$. Recall that a patching functor is an exact (covariant) functor $M_\infty$ from the category of continuous representations of $K$ on finite type $W(\F)$-modules to the category of finite type $R_\infty$-modules satisfying several ``Cohen--Macaulay'' properties, see \cite[\S 6]{EGS}. Here $R_\infty$ is the relevant patched deformation ring, a power series ring over $R^{\rm loc}$ (using standard notation), see \S \ref{patching}. Note that one also has to be careful about determinants and central characters, but we ignore this minor issue in the introduction. 

Thus proving (\ref{mult1intro}) is equivalent to proving
\begin{eqnarray}\label{mult1infiniintro}
M_\infty\big((\Proj_{K/Z_1}\sigma)/\mathfrak{m}_{K_1/Z_1}^2\big)/{\mathfrak m}_\infty\buildrel\sim\over\longrightarrow M_\infty(\sigma)/{\mathfrak m}_\infty,
\end{eqnarray}
where ${\mathfrak m}_\infty$ is the maximal ideal of $R_\infty$. The strategy in the above references to prove (a ``multiplicity one'' variant of) (\ref{mult1infiniintro}) with $\mathfrak{m}_{K_1/Z_1}^2$ replaced by $\mathfrak{m}_{K_1/Z_1}$ is to use the isomorphism
\begin{equation*}\label{tildeintro}
M_\infty(\widetilde \Proj_{\GL_2(k)}\sigma)/(p)\cong M_\infty(\Proj_{\GL_2(k)}\sigma)=M_\infty\big((\Proj_{K/Z_1}\sigma)/\mathfrak{m}_{K_1/Z_1}\big),
\end{equation*}
where $\widetilde \Proj_{\GL_2(k)}\sigma$ is the unique projective $W(\F)[\GL_2(k)]$-module lifting $\Proj_{\GL_2(k)}\!\sigma\cong \Inj_{\GL_2(k)}\!\sigma$, and to determine the support of $M_\infty(\widetilde \Proj_{\GL_2(k)}\sigma)$ in $R_\infty$.

\subsubsection{{Lattices in locally algebraic representations}}
\label{subsub:lattices:patch}

 We apply a similar strategy in our case, which means we first have to lift $(\Proj_{K/Z_1}\sigma)/\mathfrak{m}_{K_1/Z_1}^2$ to a $W(\F)[K]$-module. This is significantly more complicated than to lift $(\Proj_{K/Z_1}\sigma)/\mathfrak{m}_{K_1/Z_1}$. It is easy to check that the $K$-representation $(\Proj_{K/Z_1}\sigma)/\mathfrak{m}_{K_1/Z_1}^2$ is a nonsplit extension
\begin{equation*}
0\longrightarrow (\mathfrak{m}_{K_1/Z_1}/\mathfrak{m}_{K_1/Z_1}^2)\otimes_{\F} \Proj_{\GL_2(k)}\sigma \longrightarrow (\Proj_{K/Z_1}\sigma)/\mathfrak{m}_{K_1/Z_1}^2 \longrightarrow \Proj_{\GL_2(k)}\sigma \longrightarrow 0.
\end{equation*}
For convenience, let us fix an embedding $\sigma_0:k\cong{\mathbb F}_{p^f}\hookrightarrow \F$ and write all others as $\sigma_0\circ \varphi^{j}$, $j\in \{0,\dots,f-1\}$, where $\varphi$ is the Frobenius $x\mapsto x^p$ on $k$. Then we have
$$\mathfrak{m}_{K_1/Z_1}/\mathfrak{m}_{K_1/Z_1}^2\cong \bigoplus_{j=0}^{f-1} \big({\rm Sym}^2({\F}^2)\otimes_{{\F}}{\rm det}^{-1}\big)^{(j)},$$
where $(j)$ means that $\GL_2(k)$ acts via $\sigma_0\circ \varphi^{j}$. Moreover, for each $j$, we fix a (non-canonical) $\GL_2(k)$-equivariant embedding
$$\iota_j:\Proj_{\GL_2(k)}\sigma\hookrightarrow \big({\rm Sym}^2({\mathbb F}^2)\otimes_{{\mathbb F}}{\rm det}^{-1}\big)^{(j)}\otimes_{\F}\Proj_{\GL_2(k)}\sigma.$$
We set $L_{-1}\defeq \widetilde \Proj_{\GL_2(k)}\sigma$ and
$$R_{2,j}\defeq \big({\rm Sym}^2(W({\F})^2)\otimes_{W({\F})}{\rm det}^{-1}\big)^{(j)}\otimes_{W(\F)}L_{-1}\ \ \ \ j\in \{0,\dots,f-1\},$$
and we define a $K$-invariant lattice $L_j$ in the locally algebraic representation
$$L_{-1}[1/p]\oplus \Big(\bigoplus_{j'=0}^{j} R_{2,j'}[1/p]\Big)$$
as follows
\begin{multline*}
L_j \defeq \{(x,(x_{j'})_{0\leq j'\leq j})\in L_{-1} \oplus \big(\bigoplus_{j'=0}^{j} R_{2,j'}\big) : (x_{j'}\bmod pR_{2,j'})=(x\bmod pL_{-1})\\
{\rm via}\ \iota_{j'}:L_{-1}/pL_{-1}\hookrightarrow R_{2,j'}/pR_{2,j'}\ \forall\ j'\in \{0,\dots,j\}\}.
\end{multline*}
Equivalently, we have for $j\in \{0,\dots,f-1\}$ that
\begin{equation}\label{produitfibreintro}
L_j \defeq L_{j-1}\times_{\Proj_{\GL_2(k)}\sigma} R'_{2,j},
\end{equation}
where $R'_{2,j}\defeq \{x\in R_{2,j} : (x\bmod pR_{2,j})\in\ \iota_j(L_{-1}/pL_{-1})\}$ (another $K$-invariant lattice in $R_{2,j}[1/p]$). By explicit computations carried out in \S \ref{sec:lattices}, we first prove that the lattice $L_{f-1}$ lifts $(\Proj_{K/Z_1}\sigma)/\mathfrak{m}_{K_1/Z_1}^2$.

\begin{thmIntro}[Corollary \ref{rpr}]\label{reseauintro}
We have a $K$-equivariant isomorphism
$$L_{f-1}/pL_{f-1}\cong (\Proj_{K/Z_1}\sigma)/\mathfrak{m}_{K_1/Z_1}^2.$$
\end{thmIntro}

We then prove the following theorem.

\begin{thmIntro}[Corollary \ref{HT102-1}]\label{freeintro}
For $j\in \{-1,\dots,f-1\}$ the $R_\infty$-module $M_\infty(L_j)$ is free of finite rank over $R_\infty/{\rm Ann}_{R_\infty}(M_\infty(L_j))$. Moreover this rank depends neither on $j$ nor on the fixed Serre weight $\sigma$ in $W(\rbar_v^\vee)$.
\end{thmIntro}

Denote by $r\geq 1$ the rank in Theorem \ref{freeintro}. Applying Theorem \ref{freeintro} to both $j=-1$ and $j=f-1$, and using Theorem \ref{reseauintro} when $j=f-1$, we see that the two $\F$-vector spaces in (\ref{mult1infiniintro}) both have dimension $r$. Since the natural map from left to right in (\ref{mult1infiniintro}) is surjective by exactness of $M_\infty$, we obtain that (\ref{mult1infiniintro}) is an isomorphism, and hence that $\pi$ satisfies condition \ref{multintro} of Theorem \ref{thmintro1}.

We now sketch the proof of Theorem \ref{freeintro}, which is by induction on $j$. We first prove the following two statements for $j\in \{0,\dots,f-1\}$:
\begin{enumerate}
\item\label{infini-1intro}$M_\infty(L_{-1})$ is free of rank $r$ over $R_\infty/{\rm Ann}_{R_\infty}(M_\infty(L_{-1}))$;
\item\label{infini-iintro}$M_\infty(R'_{2,j})$ is free of rank $r$ over $R_\infty/{\rm Ann}_{R_\infty}(M_\infty(R'_{2,j}))$.
\end{enumerate}

Statement \ref{infini-1intro} is proven in \S \ref{tobefree1} (see Proposition \ref{HT10}) by a refinement of the techniques in \cite[\S 10]{EGS} and \cite[\S 4]{LMS} together with some commutative algebra. Statement \ref{infini-iintro} is proven in Theorem \ref{HT2-1} using standard d\'evissage techniques and ``elementary'' properties of the functor $M_\infty$ (in particular \cite[Lemma~4.5]{DanWild} instead of \cite[Lemma 10.1.13]{EGS}) and some results of \S \ref{tobefree1}.

By exactness of $M_\infty$, (\ref{produitfibreintro}) implies
\[M_\infty(L_j)\cong M_\infty(L_{j-1})\times_{M_\infty(\Proj_{\GL_2(k)}\sigma)}M_\infty(R'_{2,j}).\]
We know that $M_\infty(R'_{2,j})$ is free of rank $r$ by \ref{infini-iintro} above and  that $M_\infty(L_{j-1})$ is free of rank $r$ over $R_\infty/{\rm Ann}_{R_\infty}(M_\infty(L_{j-1}))$ by our induction hypothesis (which holds for $j=0$ by \ref{infini-1intro}). Hence, to deduce the same statement for $M_\infty(L_{j})$, it is enough (in fact equivalent using Lemma \ref{lem:hard-glueing}) to prove
\begin{equation}\label{sum-intro}
{\rm Ann}_{R_\infty}(M_\infty(\Proj_{\GL_2(k)}\sigma))\subseteq {\rm Ann}_{R_\infty}(M_\infty(L_{j-1})) + {\rm Ann}_{R_\infty}(M_\infty(R'_{2,j})).
\end{equation}

\subsubsection{{Deformation rings, and conclusion}}
\label{subsub:multiHT}

Statement (\ref{sum-intro}) is the most subtle and the most technical part of the paper and is ultimately proven in Theorem \ref{HT102-1}, though in a somewhat indirect way as we explain now.

Recall that $R_{\rbar_v^\vee}$ is the local $W(\F)$-algebra parametrizing framed deformations of $\rbar_v^\vee$. We let $R_{\rbar_v^\vee}^{(1,0),\tau}$, resp.\ $R_{\rbar_v^\vee}^{(2,-1)_j,{\tau}}$ for $j\in \{0,\dots,f-1\}$, be the reduced $p$-torsion free quotient of $R_{\rbar_v^\vee}$ parametrizing those deformations which have inertial type $\tau$ and parallel Hodge--Tate weights $(1,0)$, resp.\ Hodge--Tate weights $(2,-1)$ in the embedding $F_v\hookrightarrow W(\F)[1/p]$ induced by $\sigma_0\circ \varphi^{j}$ and $(1,0)$ elsewhere. An explicit computation that builds on the recent advances of Le--Le Hung--Levin--Morra \cite{LLLM}, \cite{LLL} (see Proposition \ref{prop:def:ring}) shows that these rings are all domains. It follows (see Proposition \ref{HT10}) that $R_\infty/{\rm Ann}_{R_\infty}(M_\infty(L_{-1}))$ is a power series ring over $R_{\rbar_v^\vee}/\cap_{\tau}{\mathfrak p}_\tau^{(1,0)}$, where ${\mathfrak p}_\tau^{(1,0)}$ is the prime ideal $\ker(R_{\rbar_v^\vee}\twoheadrightarrow R_{\rbar_v^\vee}^{(1,0),{\tau}})$ and $\tau$ runs over the tame inertial types such that $\sigma$ is a Jordan--H\"older factor in the mod $p$ semisimplification of $\sigma(\tau)$ (here $\sigma(\tau)$ is the usual irreducible smooth representation of $K$ associated by Henniart to $\tau$ in the appendix to \cite{BM}). Likewise, $R_\infty/{\rm Ann}_{R_\infty}(M_\infty(R'_{2,j}))$ is a power series ring over $R_{\rbar_v^\vee}/\cap_{\tau}{\mathfrak p}_\tau^{(2,-1)_j}$, where ${\mathfrak p}_\tau^{(2,-1)_j}=\ker(R_{\rbar_v^\vee}\twoheadrightarrow R_{\rbar_v^\vee}^{(2,-1)_j,{\tau}})$ and $\tau$ runs over the same tame types (see Theorem \ref{HT2-1}).

In the first version of our work, we tried to prove (\ref{sum-intro}) directly. For that one has to deal with ${\rm Ann}_{R_\infty}(M_\infty(R'_{2,j}))$ which is essentially (forgetting formal variables) $\cap_{\tau}{\mathfrak p}_\tau^{(2,-1)_j}$. However, computing elements in this intersection over the $2^f$ types $\tau$ turns out to be very hard because the ideals ${\mathfrak p}_\tau^{(2,-1)_j}$ do not have simple generators (this is mainly due to the technical monodromy condition which appears as we have Hodge--Tate weights $(2,-1)$) and there was a gap in our proof. To avoid this intersection, we use the following detour, which is inspired by the proof of \cite[Prop.~4.18]{HuWang2}.

Choose a tame inertial type $\tau_0$ such that the set of irreducible constituents of $\sigma(\tau_0)/p\sigma(\tau_0)$ coincides with the set $W(\rbar_v^\vee)$ (such a type exists) and define for $j\in \{0,\dots,f-1\}$
\begin{eqnarray*}
T_{2,j}&\defeq &\big({\rm Sym}^2(W({\F})^2)\otimes_{W({\F})}{\rm det}^{-1}\big)^{(j)}\otimes_{W(\F)}\sigma(\tau_0)^0,\\
T'_{2,j}&\defeq &\textrm{image of the composition } R'_{2,j}\hookrightarrow R_{2,j}\twoheadrightarrow T_{2,j},
\end{eqnarray*}
where $\sigma(\tau_0)^0$ is the image of $L_{-1}$ in $\sigma(\tau_0)$ (equivalently the unique $K$-invariant lattice in $\sigma(\tau_0)$ with cosocle $\sigma$).
Then the surjection $R'_{2,j}\twoheadrightarrow T'_{2,j}$ induces a surjection
\[L_j\twoheadrightarrow N_j \defeq L_{j-1}\times_{Y_j} T'_{2,j}\]
where $Y_j$ is an explicit quotient of $\Proj_{\GL_2(k)}\sigma$ such that $M_\infty(Y_{j})=M_\infty(T'_{2,j}/pT'_{2,j})$ (Lemma \ref{lem:WinYj}). We first prove that $M_\infty(L_{j})$ is free of rank $r$ (over its schematic support) if and only if $M_\infty(N_{j})$ is free of rank $r$ (see Proposition \ref{prop:Lj-Nj-equivalence-HW} and the last paragraph of the proof of Theorem \ref{HT102-1}). To prove the latter, as for (\ref{sum-intro}) we have to prove for $j\in \{0,\dots,f-1\}$
\[{\rm Ann}_{R_\infty}(M_\infty(Y_j))\subseteq {\rm Ann}_{R_\infty}(M_\infty(L_{j-1})) + {\rm Ann}_{R_\infty}(M_\infty(T'_{2,j}))\]
or equivalently since ${\rm Ann}_{R_\infty}(M_\infty(Y_j))=(p) + {\rm Ann}_{R_\infty}(M_\infty(T'_{2,j}))$ and since $T'_{2,j}$ is a lattice in $T_{2,j}[1/p]$,
\begin{equation}\label{sumjintro}
p\in {\rm Ann}_{R_\infty}(M_\infty(L_{j-1})) + {\rm Ann}_{R_\infty}(M_\infty(T'_{2,j}))={\rm Ann}_{R_\infty}(M_\infty(L_{j-1})) +  {\mathfrak p}_{\tau_0}^{(2,-1)_j}.
\end{equation}
Note that we have replaced the intersection $\cap_{\tau}{\mathfrak p}_\tau^{(2,-1)_j}$ by just ${\mathfrak p}_{\tau_0}^{(2,-1)_j}$! It is then possible to check (\ref{sumjintro}) by an explicit computation, which can be done entirely ``by hand'', see Proposition \ref{prop:p:in:inter} and the proof of Theorem \ref{HT102-1}. We have compiled in Tables 1 to 5 all the explicit computations of deformation rings that we use in the proofs (everything was checked ``by hand'').

To apply Theorem \ref{thmintro1} to $\pi$ in (\ref{tobepatched}), it remains to show that $\pi$ satisfies condition \ref{k1intro} of Theorem \ref{thmintro1}. But using (\ref{mult1infiniintro}) together with standard injectivity properties of localizations of Hecke modules at non-Eisenstein maximal ideals and (a lot of) representation theory of $K$ (see Corollary \ref{cor:J-fil}), we actually obtain the complete structure of $\pi[\mathfrak{m}_{K_1/Z_1}^2]$ as a representation of $K$.

\begin{thmIntro}[Theorem \ref{largest}]\label{largestintro}
Let $\pi$ be as in \eqref{tobepatched}, we have
\begin{equation}\label{eq:13}
\pi[\mathfrak{m}_{K_1/Z_1}^2]\cong \Big(\bigoplus_{\sigma\in W(\rbar_v^\vee)}\widetilde D_{\sigma}\Big)^{\oplus r},
\end{equation}
where $r$ is the rank in Theorem \ref{freeintro} and $\widetilde D_{\sigma}$ is the largest subrepresentation of $(\Inj_{K/Z_1}\sigma)[\mathfrak{m}_{K_1/Z_1}^2]$ containing $\sigma$ with multiplicity $1$ \emph{(}= its socle\emph{)} and no other Serre weights of $W(\rbar_v^\vee)$. Moreover, each irreducible constituent of $\pi[\mathfrak{m}_{K_1/Z_1}^2]$ has multiplicity $r$.
\end{thmIntro}

Condition \ref{k1intro} of Theorem \ref{thmintro1} then immediately follows from the isomorphism~\eqref{eq:13} in Theorem \ref{largestintro} by taking $K_1$-invariants on both sides. In particular we finally obtain:

\begin{thmIntro}[Theorem \ref{mainpatching2}]\label{mainbisintro}
Let $\pi$ be as in \eqref{tobepatched}. Then $\dim_{\GL_2(F_v)}(\pi)= f$.
\end{thmIntro}

\subsection{Notation}
\label{sec:notation}

We only give some very general notation here, more specific notation will be given in each section. We fix an algebraic closure $\ovl{\mathbb{Q}}_p$ of $\Qp$. All finite extensions of $\Qp$ will be considered as subfields of $\ovl{\mathbb{Q}}_p$. We let $v_p$ denote the valuation of $\ovl{\mathbb{Q}}_p$ such that $v_p(p) = 1$.

We let $E$ be a finite extension of $\Qp$, with ring of integers $\cO$, uniformizer $\varpi$ and residue field $\F$, and will always assume that $E$ is \emph{sufficiently large}. We let $k$ be a finite extension of $\Fp$ of degree $f\defeq [k:\Fp]$. We fix an embedding $\sigma_0 : k \into \F$ and let $\sigma_j \defeq \sigma_0\circ\phz^{j}$, where $\phz:x\mapsto x^p$ is the arithmetic Frobenius on $k$. Then the set $\cJ\defeq\Hom(k,\F)$ is identified with $\{0,\dots,f-1\}$.

We let $\varepsilon$ (resp.\ $\omega$) denote the $p$-adic (resp.\ mod $p$) cyclotomic character of the absolute Galois group $G_F$, where $F$ is any finite extension of $\QQ$ or $\QQ_p$. We normalize Hodge--Tate weights so that $\varepsilon$ has Hodge--Tate weight 1 at every embedding.

Given a profinite group $G$, we write $\F\bbra{G}$ for its completed group algebra with $\F$-coefficients, with augmentation ideal denoted by $\fm_G$. We recall that Pontryagin duality $M\mapsto M^\vee$ induces an exact anti-equivalence between the category of smooth $G$-representations over $\F$, and the category of pseudocompact $\F\bbra{G}$-modules. Recall that given a pseudocompact $\F\bbra{G}$-module $M$, we have the radical $\rad_G M\defeq \fm_GM$. Dually, given a smooth $G$-representation $M$ we write $\soc_G M$ for its socle.

If $G$ is a group and $V$ a representation of $G$ on a finite-dimensional $E$-vector space we denote by $\ovl{V}$ the semisimplification of a $G$-stable $\cO$-lattice in $V$. If $V$ is a representation of $G$ on a finite-dimensional vector space, we let $\JH(V)$ denote the set of Jordan--H\"older factors of $V$. Also, if $\sigma$ is an irreducible representation of $G$, we let $[V:\sigma]$ be the multiplicity of $\sigma$ in the semisimplification of $V$.

\subsection{Acknowledgements}
\label{sec:merci}

The initial impetus for this work was a SQuaRE meeting at the American Institute of Mathematics at San Jose in August 2019 (though it was not quite clear at the time where we were really heading!). We heartily thank AIM for hosting and supporting us and for outstanding working conditions. We are also very grateful to Sug Woo Shin and Karol Kozio\l{} for participating in this meeting and for sharing their thoughts with us. 
{We are particularly grateful to Karol Kozio\l{} for pointing out a mistake in an earlier version of this work.}
Finally, we heartily thank an anonymous referee for his or her report, especially for pointing out
an embarrassing mistake in our previous use of multi-type deformation rings.

C.\;B.\ thanks X.\ Caruso for discussions in an early attempt to approach the Gelfand--Kirillov dimension via computational techniques, and Ahmed Abbes and all the organizers of the S\'eminaire de G\'eom\'etrie Arithm\'etique Paris--P\'ekin--Tokyo for their invitation to give the very last talk of this seminar on this work in June 2020. Y.\;H.\ thanks Ahmed Abbes for inviting him to I.H.\'E.S.\ for the period of November--December 2019 and I.H.\'E.S.\ for its hospitality.

C.\;B., F.\;H., S.\;M., B.\;S. thank Y.\;H. and Haoran Wang for sharing a preliminary version of \cite{HuWang2}, which inspired us in the early stages of our project.

F.\;H.\ is partially supported by an NSERC grant.
Y.\;H.\ is partially supported by  National Key R$\&$D Program of China 2020YFA0712600, National Natural Science Foundation of China Grants 12288201 and 11971028; National Center for Mathematics and Interdisciplinary Sciences and Hua Loo-Keng Key Laboratory of Mathematics, Chinese Academy of Sciences. 
S.\;M.\ and B.\;S.\ are partially supported by Institut Universitaire de France.
C.\;B., S.\;M. and B.\;S. are members of the A.N.R.\ project CLap-CLap ANR-18-CE40-0026.

\section{Preliminaries}
\label{sec:preliminaries}

Throughout this section $K$ denotes the unramified extension of $\Qp$ of degree $f$ with ring of integers $\cO_K$ and residue field $k$.
Recall from \S \ref{sec:notation} that we have fixed an embedding $\sigma_0:k\into\F$, hence an embedding $K\into E$ which we still denote by the same symbol $\sigma_0$.
In particular we have compatible identifications of $\cJ=\Hom(k,\F)$ with $\Hom_{\Qp}(K,E)$ and with $\{0,\dots,f-1\}$.

\subsection{Group theoretic preliminaries}
\label{sec:GT:prel}

We consider the group scheme $\GL_n$ defined over $\ZZ$, let $T\subseteq \GL_n$ be the diagonal maximal torus and $Z$ its center.
We write $R$ for the set of roots of $(\GL_n,T)$, $W$ for its Weyl group, with longest element $\fW$ and let
$B\subset \GL_n$ denote the Borel of upper-triangular matrices. 
In particular, $B$ determines the subsets $R^+$ of positive roots.
We identify the set of characters $X^*(T) = \Hom(T,\mathbb G_m)$ with $\ZZ^n$ in the standard way.
If $n = 2$, let $\alpha\in R^+$ correspond to $(1,-1)\in \ZZ^2$ so that ${R}^+ = \{\alpha\}$.
If $A$ is any ring, we write ${\GL_n}_{/A}$ to denote the base change of $\GL_n$ to $A$.

Let $\un{G}_0$ be the algebraic group $ \Res_{\cO_K/\Zp} {\GL_n}_{/\cO_K}$ with $\un{T}_0$ the diagonal maximal torus and center $\un{Z}_0$. 
Let $\un{G}$ be the base change $\un{G}_0 \times_{\Zp} \cO$, and similarly define $\un{T}$ and $\un{Z}$.

There is a natural isomorphism $\un{G} \cong \prod_{\cJ} {\GL_n}_{/\cO}$ induced by the ring homomorphism $\cO_K\otimes_{\Zp}\cO\cong \cO^\cJ$ defined by $x\otimes1\mapsto (\sigma_j(x))_{j\in\cJ}$. 
One has similar isomorphisms for $\un{T}$, $\un{Z}$, $X^*(\un{T})$, $\un{R}$, $\un{R}^\vee$, where $\un{R}$ (resp.~$\un{R}^\vee$) denotes the set of roots (resp.~coroots) of $(\un{G},\un{T})$.
If $\mu \in X^*(\un{T})$, then we correspondingly write $\mu = (\mu_j)_{j\in \cJ}$.
We define an automorphism $\pi$ on $X^*(\un{T})$ by $\pi(\mu)_j\defeq \mu_{j-1}$ (it is the automorphism coming from the descent data on $\un{T}$ induced by $\un{T}_0$ and corresponding to the arithmetic Frobenius on $\cO_K$).

We identify $X^*(\un{T}) = \bigoplus_{\cJ} X^*(T)$ with $(\ZZ^n)^\cJ$ as above.
Moreover, if $(a_1,\dots,a_n)\in\ZZ^n$ we write $(\un{a_1},\dots,\un{a_n})$ to denote the element of $X^*(\un{T})$ whose corresponding tuple equals $(a_1,\dots,a_n)$ at each embedding $j\in\cJ$.
We let $\eta_{j}$ be $(n-1,\dots,1,0)$ in the $j$-th coordinate and $0$ otherwise.
We let $\eta\defeq \sum_j\eta_j=(\un{n-1},\dots,\un{1},\un{0})$.

Given $\lambda\in X^*(T)$ (resp.~$\lambda\in X^*(\un{T})$), we let $V(\lambda)_{/\cO}$ denote the algebraic Weyl module of ${\GL_n}_{/\cO}$ (resp.~$\un{G}$) with highest weight $\lambda$ as defined in \cite[II.8.3]{RAGS}.
If $A$ is an $\cO$-algebra, we write $V_A(\lambda)$ to denote the restriction of $V(\lambda)_{/\cO}(A)$ to $\GL_n(\cO_K)$ via the map $\GL_n(\cO_K)\ra \GL_n(A)$ induced by the ring homomorphism $\sigma_0$. 
If $j\in\cJ$ and $\lambda\in X^*(T)$, we write $V(\lambda)^{(j)}_{/\cO}$ to denote the algebraic representation of $\un{G}$ obtained, by inflation from the $j$-th projection $\un{G}\cong \prod_{\cJ}{\GL_n}_{/\cO}\stackrel{\pi_j}{\onto} {\GL_n}_{/\cO}$, from the algebraic Weyl module $V(\lambda)_{/\cO}$ of ${\GL_n}_{/\cO}$.

Let $\un{R}^+\subseteq \un{R}$ (resp.~$\un{R}^{\vee,+}\subseteq \un{R}^\vee$) be the subset of positive roots (resp.~coroots) of $\un{G}$ with respect to the upper-triangular Borel in each embedding.
If $n = 2$, let $\alpha_j\in\un{R}$ be $(1,-1)$ in the $j$-th coordinate and $0$ otherwise, so that $\un{R}^+ = \{\alpha_j : j=0,\dots,f-1\}$.

Let $X^*_+(\un{T})$ be the set of dominant weights, i.e.\ the set of weights $\lambda\in X^*(\un{T})$ satisfying $0\leq \langle\lambda,\alpha^\vee\rangle$ for all $\alpha\in \un{R}^{+}$.
We denote by $X_1(\un{T}) \subseteq X^*_+(\un{T})$ the subset of \emph{$p$-restricted} weights $\lambda\in X^*_+(\un{T})$ satisfying $0\leq \langle \lambda,\alpha^\vee\rangle\leq p-1$ for all simple roots $\alpha\in \un{R}^+$.
Let $X_{\reg}(\un{T}) \subseteq X^*_+(\un{T})$ be the subset of weights $\lambda\in X^*_+(\un{T})$ satisfying $0\leq \langle \lambda,\alpha^\vee\rangle < p-1$ for all simple roots $\alpha\in \un{R}^+$.
Finally, we let $X^0(\un{T}) \subseteq X^*_+(\un{T})$ be the subset of weights $\lambda\in X^*(\un{T})$ satisfying $\langle \lambda,\alpha^\vee\rangle=0$ for all simple roots $\alpha\in \un{R}^+$.

The lowest alcove is defined as
\begin{equation*}
\un{C}_0\defeq \{\lambda\in X^*(\un{T}) \otimes \RR : 0< \langle \lambda+\eta,\alpha^\vee\rangle<p\ \forall\, \alpha\in \un{R}^{+}\}.
\end{equation*}
Given $N\geq 0$ and $\mu\in \un{C}_0$ we say that $\mu$ is \emph{N-deep in $\un{C}_0$} if
$N<\ang{\mu+\eta,\alpha^\vee}<p-N$ for all $\alpha\in \un{R}^{+}$.
(Thus the existence of an $N$-deep weight in $\un C_0$ implies $p \ge 2N+2$.)

In particular, when $n = 2$, via the identifications above
\begin{align*}
X_1(\un{T})&=\{\lambda\in (\ZZ^2)^f : 0\leq \lambda_{j,1}-\lambda_{j,2}\le p-1\ \forall\, j=0,\dots,f-1\},\\
X_{\reg}(\un{T})&=\{\lambda\in (\ZZ^2)^f : 0\leq \lambda_{j,1}-\lambda_{j,2}< p-1\ \forall\, j=0,\dots,f-1\},
\end{align*}
and $\un{C}_0 \cap X^*(\un T) = X_{\reg}(\un T)$.

Let $\un{W}$ be the Weyl group of $(\un{G},\un T)$, with longest element $w_0$.
It acts on $X^*(\un{T})$ and we have a compatible identification of $\un{W}$ with $\prod_{j\in\cJ} W$.
Given $w\in \un{W}$, we write $w_{j}$ to denote its $j$-th component via the identification above.

Let $\un{W}_a$ and $\widetilde{\un{W}}$ be the affine Weyl group and extended affine Weyl group, respectively, of $\un{G}$.
Concretely, $\un{W}_a\cong \Lambda_R\rtimes \un{W}$ and $\wtld{\un{W}}\cong X^*(\un{T})\rtimes \un{W}$, where $\Lambda_R \subseteq X^*(\un{T})$ is the root lattice of $\un{G}$.
The image of $\lambda\in X^*(\un{T})$ in $\wtld{\un{W}}$ is denoted by $t_\lambda$.
Note that $\wtld{\un{W}} \cong (\ZZ^n \rtimes S_n)^f$ and we will also write $t_{\un a}$ for the image of $\un a \in \ZZ^n$ in $\ZZ^n \rtimes S_n$.
We have the \emph{$p$-dot action} of $\wtld{\un{W}}$ on $X^*(\un{T})$, defined as follows: if $\tld{w}=wt_\nu\in \wtld{\un{W}}$ and $\mu\in X^*(\un{T})$ then $\tld{w}\cdot \mu\defeq w(\mu+\eta+p\nu)-\eta$.

Let $\Omega$ be the stabilizer of the lowest alcove $\un{C}_0$ in $\wtld{\un{W}}$.
One checks that $\wtld{\un{W}}= \un{W}_a\rtimes \Omega$.
Concretely, when $n=2$, it is the subgroup of $\wtld{\un{W}}$ generated by $X^0(\un{T})$ and $\big\{1, \fW t_{-(1,0)} \big\}^{\cJ}$. 

Recall that the choice of $\un{C}_0$ endows $\un{W}_a$ with the structure of a Coxeter group generated by the reflections with respect to the walls of $\un{C}_0$ (cf.\ \cite[II.6.3]{RAGS}), and thus with a Bruhat order, which is denoted by $\leq$.
This induces a partial order $\leq$ on $\wtld{\un{W}}$, namely 
$\tld{w}_{a}\omega \le \tld{w}'_{a}\omega'$ in $\un{W}_a\rtimes \Omega=\wtld{\un{W}}$ if and only if
$\tld{w}_{a}\leq \tld{w}'_{a}$ in $\un{W}_a$ and $\omega = \omega'$ in $\Omega$.
We denote $\wtld{\un{W}}^\vee$ the group $\wtld{\un{W}}$, endowed with the Bruhat order induced by the choice of the \emph{antidominant} base alcove, i.e.\
\begin{equation*}
\un{C}_0^\vee\defeq \{\lambda\in X^*(\un{T}) \otimes \RR : -p< \langle \lambda+\eta,\alpha^\vee\rangle<0\ \forall\, \alpha\in \un{R}^{+}\}.
\end{equation*}
We have an anti-isomorphism
\begin{align*}
\wtld{\un{W}}^\vee&\congto \wtld{\un{W}}\\
\tld{w}&\mapsto \tld{w}^*
\end{align*}
defined by $((st_\mu)^*)_j=t_{\mu_{f-1-j}}s_{f-1-j}^{-1}$ such that $\tld{w}_1 \le \tld{w}_2$ if and only if $\tld{w}_2^* \le \tld{w}_1^*$
\cite[Lemma 2.1.3]{LLL}.
Given $\lambda\in X^*(\un{T})$ we define $\Adm^\vee(t_\lambda)\defeq \big\{\tld{w}\in \wtld{\un{W}}^\vee : \tld{w}\leq t_{w(\lambda)}{\rm{\ for\ some\ }}w\in \un{W}\big\}$. 
It is the $\lambda$-admissible set in the sense of \cite{KoRa} relative to the Bruhat order defined above on $\wtld{\un{W}}^\vee$.

Let $R$ be a commutative ring.
If $(x_1,\dots,x_n)\in R^n$ we write $\mathrm{Diag}(x_1,\dots,x_n)$ for the diagonal matrix of $\M_n(R)$ whose $i$-th diagonal entry is $x_i$.
If $\mu\in\ZZ^n$ and $x\in R$ then we write $x^\mu$ for the diagonal matrix $\mathrm{Diag}(x^{\mu_1},\dots,x^{\mu_n})\in \M_n(R)$.
Sometimes it will be convenient to consider $\wtld{\un{W}}^\vee$ as subgroup of $\GL_n(\F\ppar{v})^f$
by the injective homomorphism
sending $s t_\mu$ to $(\dot s_j v^{\mu_j})_j$, where $\dot s_j$ is the permutation matrix associated to $s_j \in S_n$.

If $w \in S_n$ we let $\sgn(w)\in\{\pm 1\}$ denotes its sign.

\subsection{The inertial local Langlands correspondence and Serre weights}
\label{sec:inert-local-langl}

An \emph{inertial type} is a representation $\tau:I_K\rightarrow \GL_n(\ovl{\QQ}_p)$ with open kernel which can be extended to $W_K$ (or equivalently to $G_K$).
When $n=2$ a result of Henniart (see the appendix to \cite{BM}) shows that given an inertial type $\tau$, there is an irreducible smooth $\GL_2(\cO_K)$-representation $\sigma(\tau)$ over $\ovl{\QQ}_p$
associated to it. We normalize it as in \cite[\S 2.1.1]{BM} when $\tau$ is non-scalar, and when $\tau = \chi \oplus \chi$ is scalar we let $\sigma(\tau) \defeq \chi \circ \det$ (via local class field theory).
(This is often referred to as the \emph{inertial local Langlands correspondence}; the representation $\sigma(\tau)$ above is the same as the representation $\sigma(\tau)$ appearing in \cite[Thm.~3.7]{CEGGPS} where, in the notation of \emph{loc.~cit.},~$G=\GL_2(K)$.)
We remark that for any inertial type $\tau$, the representation $\sigma(\tau)$ can be realized over $E$, up to enlarging $E$ if necessary. %

A \emph{Serre weight} of $\un{G}_0\times_{\Zp}\Fp$ is an isomorphism class of an (absolutely) irreducible representation of $\un{G}_0(\Fp) = \GL_n(k)$ over $\F$. %
If $\lambda\in X_1(\un{T})$, we write $L(\lambda)_{/\F}$ (or sometimes just $L(\lambda)$) for the irreducible algebraic representation of $\un{G}\times_{\cO}\F$ of highest weight $\lambda$, and $F(\lambda)$ for the restriction of $L(\lambda)_{/\F}$ to the group ${\un{G}_0(\Fp)}$. %
The map $\lambda\mapsto F(\lambda)$ induces a bijection between $X_1(\un{T})/(p-\pi)X^0(\un{T})$ and the set of Serre weights of $\un{G}_0\times_{\Zp}\Fp$ (cf.~\cite[Lemma 9.2.4]{GHS}).
A Serre weight $\sigma$ is \emph{regular} if $\sigma\cong F(\lambda)$ with $\lambda\in X_{\reg}(\un T)$, cf.~\cite[Def.\ 6.1]{herzig-duke}.

If $n = 2$ and $\rhobar:G_K\rightarrow \GL_2(\F)$ is a tame Galois representation then we have a set $W(\rhobar)$ of Serre weights, defined by Buzzard--Diamond--Jarvis in \cite[\S 3]{BDJ}. %
We recall that $W(\rhobar)$ depends only on $\rhobar|_{I_K}$.

\subsection{Tame inertial types}
\label{sec:tame-inertial-types}

Fix a pair $(s, \mu) \in \un{W} \times X^*(\un{T})$, which we will use to define a tame inertial type.

Writing $s = (s_0, \ldots, s_{f-1}) \in \un{W}$ we set $s_{\tau} \defeq s_0 s_{f-1} s_{f-2} \cdots s_1 \in S_n$ and let $r$ denote the order of $s_\tau$.
Let $f' \defeq rf$, $e'\defeq p^{f'}-1$.
Let $K'/K$ be the unramified extension of $K$ of degree $r$ with residue field $k'$. %
We fix an embedding $\sigma_0':k'\into \F$ extending $\sigma_0$, so we can identify $\cJ'\defeq \Hom(k',\F)$ with the set $\{0,\dots,f'-1\}$ via $\sigma'_{j'}\defeq \sigma_0'\circ\phz^{j'}\mapsto j'$.
We define the tame fundamental character $\omega_{f'} : I_K \to \F^\times$ as the composition $I_K = I_{K'} \onto \cO_{K'}^\times \onto k'^\times \to \F^\times$,
where the first map is the local Artin map, normalized so that uniformizers correspond to geometric Frobenius elements, and the last map
is given by $\sigma_0'$. We also let $\teich{\omega}_{f'} : I_K \to \cO^\times$ denote the Teichm\"uller lift of $\omega_{f'}$.

Define $\bm{\alpha}'_{(s,\mu)} \in  (\ZZ^n)^{\Hom(k', \F)} \cong X^*(\un{T})^{r}$ %
by 
\[
\bm{\alpha}'_{(s,\mu), j'} \defeq s_1^{-1} s_2^{-1} \cdots s_{j'}^{-1}(\mu_{j'}+\eta_{j'}) \in \ZZ^n,\ \ j' \in \ZZ,
\]
where the indices on the right-hand side are considered modulo $f$. In particular, $\bm{\alpha}'_{(s,\mu), j'+kf} = s_{\tau}^{-k} \bm{\alpha}'_{(s,\mu), j'}$,
showing that $\bm{\alpha}'_{(s,\mu), j'}$ only depends on $j'$ modulo $f'$. Also define
\begin{equation}\label{eq:bold-a}
  \bf{a}^{\prime\,(j')}_{(s,\mu)}\defeq \sum_{i'=0}^{f'-1}\bm{\alpha}'_{(s,\mu), -j' + i'}p^{i'} \in \ZZ^n, \ \ j' \in \ZZ.
\end{equation}

\begin{definit} \label{def:tau} Given $(s, \mu) \in \un{W} \times X^*(\un{T})$ define
\begin{equation*}
\tau(s,\mu+\eta) \defeq \bigoplus_{1 \leq i \leq n} \teich{\omega}_{f'}^{\bf{a}^{\prime\,(0)}_{(s,\mu), i}} : I_K \ra \GL_{n}(\cO).
\end{equation*}
Setting $\bf{a}^{(0)} \defeq \sum_{j =0}^{f-1} \bm{\alpha}'_{(s, \mu), j} p^j$ we can also write it as
\begin{equation} \label{eq:def:type}
\tau(s,\mu+\eta) = \bigoplus_{1 \leq i \leq n} \teich{\omega}_{f'}^{\sum_{0 \leq k \leq r-1} \bf{a}^{(0)}_{s_{\tau}^{k}(i)} p^{fk}}.
\end{equation}
\end{definit} 

From~\eqref{eq:def:type} we see that $\tau(s,\mu+\eta)$ is a tame inertial type, i.e.\ can be extended to $G_K$.
Given a tame inertial type $\tau(s,\mu+\eta)$, we write $\taubar(s,\mu+\eta)$ for its reduction mod $\varpi$.

\begin{rem}
\label{rmk:LLL:error}
Due to our choice of labeling of the embeddings of $k$ in $\F$, namely $\sigma_j=\sigma_0\circ\varphi^j$, our definition of $\tau(s,\mu+\eta)$ is not compatible with \cite[Def.~2.2.1]{LLL}. This choice is motivated by the fact that we do not think that the definition in \emph{loc.~cit.} is compatible with \cite{herzig-duke} and \cite{GHS}. However we checked that it does not affect our further references to \cite{LLL}.
\end{rem}

\begin{definit}
\label{defi:gen}
Let $\tau$ be a tame inertial type and $N\in \ZZ_{\geq 0}$. 
\begin{enumerate}
\item 
\label{defi:gen:type}
We say that $\tau$ is \emph{$N$-generic} if there is an isomorphism $\tau\cong \tau(s,\lambda + \eta)$ for some $s \in \un{W}$ and $\lambda \in X^*(\un{T})$ which is $N$-deep in alcove $\un{C}_0$.

\item 
\label{def:LApres}
A \emph{lowest alcove presentation} of $\tau$ is a pair $(s, \mu) \in \un{W} \times \un{C}_0$ such that $\tau \cong \tau(s, \mu + \eta)$ (which by definition exists exactly when $\tau$ is 0-generic).
\end{enumerate}
\end{definit}

We also recall the following definition.
\begin{definit}
\label{def:rhobar:gen}
Let $\rhobar: G_K\ra \GL_n(\F)$ be a Galois representation and let $N\in\NN$.
Let $\rhobar^{\mathrm{ss}}|_{I_{K}}$ denote the restriction to $I_K$ of the semisimplification of $\rhobar$.
We say that $\rhobar$ is \emph{$N$-generic} if $\rhobar^{\mathrm{ss}}|_{I_{K}}\cong \taubar(s,\mu)$ for some $s\in\un{W}$ and $\mu-\eta\in X^*(\un{T})$ which is $N$-deep in alcove $\un{C}_0$.
(We denote here $\mu$ what was previously denoted $\mu + \eta$ because we will rather use this notation in the sequel.)
\end{definit}

\begin{rem}
\label{rem:emb}
Note that if a type $\tau$ is $N$-generic and $(s,\lambda)$ is a lowest alcove presentation of $\tau$, the weight $\lambda$ is not necessarily $N$-deep in $\un{C}_0$. However by \cite[Prop.~2.2.16]{LLL}, we know that $\lambda$ is $(N-1)$-deep in $\un{C}_0$. Similar comments apply to an $N$-generic $\rhobar$.
\end{rem}

Below we will need the ``orientation'' $s'_{\mathrm{or}} \in (S_n)^{\Hom(k', \F)} \cong \un{W}^{r}$ of $\bm{\alpha}'_{(s, \mu)}$, which is defined by
\begin{equation*}
  s'_{\mathrm{or}, j'} \defeq s_1^{-1} s_2^{-1} \cdots s_{f'-1-j'}^{-1}\in S_n,\ j'\in \ZZ,
\end{equation*}
where the indices on the right-hand side are considered modulo $f$. Hence $s'_{\mathrm{or}, j' + kf} = s_{\tau}^{k} s'_{\mathrm{or}, j'}$,
showing that $s'_{\mathrm{or}, j'}$ only depends on $j'$ modulo $f'$.

\begin{rem}
\label{rmk:orient}
We remark that if $\mu\in X^*(\un{T})$ is $0$-deep in $\un{C}_0$ then $s'_{\mathrm{or}, j'}$ is the unique element of $W$ such that $(s'_{\mathrm{or}, j'})^{-1}(\bf{a}^{\prime\,(j')}_{(s,\mu)})\in X^*(T)$ is dominant.
Observe also that 
\begin{equation}\label{eq:orientation}
  (s'_{\mathrm{or}, -j'})^{-1}(\bm{\alpha}'_{(s,\mu), j'}) = s_{j'}^{-1}(\mu_{j'}+\eta_{j'}).
\end{equation}
\end{rem}

\subsection{Combinatorics of types and Serre weights}
\label{sec:ext:graph}

Let $n = 2$. We collect results on Serre weights for mod $p$ Galois representations and Jordan--H\"older constituents of reductions of generic 
Deligne--Lusztig representations, expressed in terms of the extension graph of \cite[\S 2]{LMS}.
We caution the reader that we modify slightly the definition of the extension graph and translation map appearing in \emph{loc.~cit.}

Let $\Lambda_W \defeq X^*(\un{T})/X^0(\un{T})$ denote the weight lattice of $\big(\Res_{\cO_K/\Zp}{\mathrm{SL}_2}_{/\cO_K}\big)\times_{\Zp}\cO$.
We identify $\Lambda_W$ with $\ZZ^\cJ$ in the usual way.
For $\mu\in X^*(\un T)$ %
we define 
\[
\Lambda_W^{\mu}\defeq \{\omega\in \Lambda_W\, :\, 0\leq \langle\ovl{\mu}+\omega,\alpha^\vee\rangle<p-1\ \forall\, \alpha\in \un{R}^{+}\},
\]
where $\ovl{\mu}$ denotes the image of $\mu$ in $\Lambda_W$.
The set $\Lambda_W^{\mu}$ is called the \emph{extension graph associated to $\mu$}.

We define below an injective map
\[
\t_\mu:\Lambda^\mu_W\rightarrow X_{\reg}(\un T)/(p-\pi)X^0(\un{T})
\]
whose image consists of %
the weights $\lambda \in X_{\reg}(\un T)$ such that $\lambda|_{\un{Z}}=\mu|_{\un{Z}}$ modulo $(p-\pi)X^*(\un{Z})$.
(In other words, the map $\omega\mapsto F(\t_\mu(\omega))$ defines a bijection between $\Lambda_W^{\mu}$ and regular Serre weights with central character $\mu|_{\un{Z}_0(\Fp)}$.)

The map $\t_\mu$ is constructed as follows. 
Given $\omega' \in X^*(\un T)$ there is a unique $\tld w' \in \Omega \cap t_{-\pi^{-1}(\omega')} \un W_a$. 
Setting 
\[
\t'_\mu(\omega')\defeq \tld w' \cdot (\mu + \omega') \mod (p-\pi)X^0(\un{T})
\] 
we thus obtain a map $\t'_\mu: X^*(\un{T}) \rightarrow X^*(\un{T})/(p-\pi)X^0(\un{T})$,
which further factors through $X^*(\un{T}) \onto X^*(\un{T})/X^0(\un{T})=\Lambda_W$, by the definition of $\tld w'$ and since $\cdot$ is the $p$-dot action.
We write $\t_\mu$ for the restriction of such a map to $\Lambda^\mu_W$, and note that $\t_\mu$ has image in $X_{\reg}(\un T)/(p-\pi)X^0(\un{T})$ by definition of $\Lambda^\mu_W$.

\begin{rem}
In \cite[\S 2.2]{LMS} the set $\Lambda_W^{\mu}$ above is denoted by $\Lambda_W^{\mu+\eta}$, and the map $\t_\mu$ above by $\t_{\mu+\eta}$.
\end{rem}

In terms of the identification $\Lambda_W\cong \ZZ^{\cJ}$ the map $\t_\mu$ is described as follows: if $\mu=(a_j,b_j)_j\in X^*(\un T)$ and $\omega=(2n_j+\delta_j)_j\in \Lambda_W^{\mu}$ with $n_j\in\ZZ$, $\delta_j\in\{0,1\}$, then 
a representative of $\mathfrak{t}_\mu(\omega)$ is given by
\begin{align}
\label{eq:expl:tmu}
(\mathfrak{t}_\mu(\omega))_j&=
\begin{cases}
(a_j+n_j+\delta_j,b_j-n_j) &\text{if $\delta_{j+1}=0$,} \\
(b_j-1-n_j,a_j+n_j+\delta_j-p+1) &\text{if $\delta_{j+1}=1$.}
\end{cases}
\end{align}

We now recall and slightly improve on a few results about $\mathfrak{t}_{\mu}$ which will be important in \S \ref{sec:galo-deform-rings} (for the combinatorics of tame inertial types and Serre weights) and in \S \ref{sec:some-repr-K} (for the structure of certain $\GL_2(\cO_K)$-representations with $\F$-coefficients).

Given $J\subset\cJ$ we define $\eta_J\defeq \sum_{j\in J}\eta_j\in X^*(\un{T})$ and write $\ovl{\eta}_J$ for the image of $\eta_J$ in $\Lambda_W=X^*(\un{T})/X^0(\un{T})$.
Define $\Sigma\subseteq \Lambda_W$ to be the set $\{\ovl{\eta}_J\ :\  J\subseteq \cJ\}$. %

\begin{prop}%
\label{prop:SW:graph}
Suppose that $\rhobar : G_K\rightarrow \GL_2(\F)$ is a tame Galois representation such that
$\rhobar|_{I_K} \cong \taubar(s,\mu)$ for some $(s,\mu) \in \un W \times X^*(\un T)$ with $\mu-\eta$ lying $1$-deep in alcove $\un{C}_0$.
Then 
\begin{equation}
\label{eq:SW?}
W(\rhobar)=\left\{ F(\mathfrak{t}_{\mu-\eta}(s\omega))\ :\ \omega\in\Sigma\right\}.
\end{equation}
\end{prop}
\begin{proof}
From the proof of \cite[Prop.\ 2.11]{LMS} we see that the right-hand side of (\ref{eq:SW?}) is $W_{\mathrm{obv}}(\rhobar)$, which is the set of weights defined in \cite[Def.\ 7.1.3]{GHS}.
By \cite[Ex.\ 7.1.7]{GHS} we have $W_{\mathrm{obv}}(\rhobar)=W(\rhobar)$.
\end{proof}

\begin{prop}
\label{prop:JH:graph}
Suppose $\tau \defeq \tau(sw^{-1},\mu-sw^{-1}(\nu))$ for some $(s,\mu)$, $(w,\nu) \in \un W \times X^*(\un T)$ such that $\mu-sw^{-1}(\nu)-\eta$ is $1$-deep in alcove $\un{C}_0$.
If $\nu \in \eta + \Lambda_R$, then 
\[
\JH\left(\ovl{\sigma(\tau)}\right)=\left\{ F(\mathfrak{t}_{\mu-\eta}(sw^{-1}(\omega-\ovl\nu)))\ :\ \omega\in\Sigma\right\}.
\]
\end{prop}
\begin{proof}
Recall that, in the notation of \cite{DoLe,LLL}, we have $\sigma(\tau) \cong R_{sw^{-1}}(\mu-sw^{-1}(\nu))$ by \cite[Cor.\ 2.3.5]{LLL} (the deepness assumption on $\mu-sw^{-1}(\nu)-\eta$ ensures that $\tau$ is $1$-generic in the terminology of \emph{loc.~cit.}, hence regular, see \cite[Def.\ 2.2.9]{LLL} and the comment after it; thus \cite[Cor.\ 2.3.5]{LLL} applies).
Moreover, the deepness assumption on $\mu-sw^{-1}(\nu)-\eta$ reads
$1<\ang{\mu-sw^{-1}(\nu),\alpha^\vee}<p-1$ for $\alpha\in \un{R}^{+}$ and since $\ang{sw^{-1}(\Sigma),\alpha^\vee}\in\{-1,0,1\}$ we conclude that $0<\ang{\ovl\mu+sw^{-1}(\Sigma-\ovl\nu),\alpha^\vee}<p$ for $\alpha\in \un{R}^{+}$.
This is exactly the condition that $sw^{-1}(\Sigma-\ovl\nu)\subseteq \Lambda_W^{\mu-\eta}$ and the statement is thus immediate from \cite[Prop.\ 2.15]{DoLe} (keeping in mind that the translation map in \emph{loc.~cit.}~is an $\eta$-shift of ours).
\end{proof}

We recall the following ``change of origin'' formula for the map $\t_\lambda$, obtained from \cite[Prop.\ 2.5]{LMS}. For $\omega\in\Lambda^\mu_W$ let $\omega' \in X^*(\un T)$ denote a lift of $\omega$ and 
define $w_\omega$ as the image of the unique element $\tilde w' \in \Omega \cap t_{-\pi^{-1}(\omega')}\un{W}_a$ (as above)
in $\un W$.
By definition, $w_\omega$ does not depend on the choice of lift $\omega'$ of $\omega$ and in fact only depends on the image of $\omega$ in $\Lambda_W/\Lambda_R \cong (\ZZ/2\ZZ)^{\cJ}$.

\begin{lem}\label{lm:change-origin}%
  Let $\omega\in\Lambda^\mu_W$ and let $\lambda\in X^*(\un{T})$ be such that $\t_\mu(\omega) \equiv\lambda \mod (p-\pi)X^0(\un{T})$.
Then $w_{\omega}^{-1}(\omega')+\omega \in \Lambda^\mu_W$ and 
$\t_\lambda(\omega') = \t_\mu(w_{\omega}^{-1}(\omega')+\omega)$ for all $\omega'\in\Lambda^\lambda_W$. 
Equivalently $\t_\mu(\omega')=\t_\lambda(w_\omega(\omega'-\omega))$ for $\omega' \in \Lambda^\mu_W$.
\end{lem}

\begin{rem}\label{rk:t_lambda}
Recall from \S \ref{sec:GT:prel} that we have a natural inclusion $\Lambda_R\into \Lambda_W$, which identifies $\Lambda_R$ with $(2\ZZ)^{\cJ}$ via the isomorphism $\Lambda_W\cong \ZZ^{\cJ}$.
We remark the following facts:
\begin{enumerate}
\item
\label{it:t_lambda:1}
Given $J\subset\cJ$ we let $w_{0,J}\defeq\prod_{j+1\in J}\fW_{j}$ where $\fW_{j}\in\un{W}$ is nontrivial exactly at the embedding $j$. Recall moreover the element $\eta_J = \sum_{j\in \cJ} \eta_j\in X^*(\un{T})$ associated to $J$. %
Then $w_\omega=w_{0,J}$ if $\omega\equiv\eta_J\mod \Lambda_R$.
\item
\label{it:t_lambda:3}
 If $\nu\in\Lambda_R$, we have $w_\nu=1$ and Lemma \ref{lm:change-origin} implies that $\t_{\mu+\nu}(\omega) = \t_\mu(\omega+\nu)$. (Note that $\t_\mu(\nu) \equiv  \mu+\nu \mod (p-\pi)X^0(\un{T})$.)
\item 
\label{it:t_lambda:2}
From the definition, $\t_\mu(\omega) \in \un{C}_0$ if and only if $\mu+\omega' \in \un{C}_0$, where $\omega'\in X^*(\un{T})$ denotes a lift of $\omega$. In particular %
  \begin{equation*}
    \t_\mu(\sum a_i \ovl\eta_i) \in \un{C}_0 \iff 0 \le \ang{\mu,\alpha_i^\vee}+a_i \le p-2 \quad \forall\ i.%
  \end{equation*}
\item
\label{it:t_lambda:4}
Likewise, $\t_\mu(\omega)$ is $n$-deep in $\un{C}_0$ if and only if $\mu+\omega'$ is $n$-deep in $\un{C}_0$.
\end{enumerate}
\end{rem}

We use the terminology of \cite[Def.\ 2.8]{LMS}: two elements $\omega$, $\omega'$ of $\Lambda_W^\mu$ are \emph{adjacent} if $\omega-\omega'\equiv \pm \eta_j\mod X^0(\un{T})$ for some $j\in\cJ$. This gives $\Lambda_W^\mu$ the structure of a graph.
We have the following slight improvement of \cite[Prop.\ 2.9]{LMS}.%

\begin{lem}\label{lm:ext1}
  Let $\omega, \omega'$ be elements of $\Lambda_W^\mu$.  
Then
  \begin{equation*}
    \dim_{\F}\bigg(\Ext^1_{\GL_2(k)}(F(\t_\mu(\omega)), F(\t_\mu(\omega')))\bigg) =
    \begin{cases}
      1 & \text{if $\omega$, $\omega'$ are adjacent,}\\
      0 & \text{otherwise.}
    \end{cases}
  \end{equation*}
\end{lem}

\begin{proof}
Let $\lambda\defeq\t_\mu(\omega)$. By Lemma~\ref{lm:change-origin} we have $\t_\mu(\omega')=\t_\lambda(\omega'')$ with $\omega''\defeq w_\omega(\omega'-\omega)\in \Lambda^\lambda_W$. As $\omega''$ and $0$ are adjacent if and only if $\omega$ and $\omega'$ are adjacent, we may replace $\mu$ by $\lambda$ and assume that $\omega = 0$. By letting $\ovl{\eta}_i$ be $\eta_i\mod X^0(\un{T})$ we compute
\begin{alignat*}{3}
\t_\mu(\ovl{\eta}_i)&\equiv \fW_{i-1}t_{-\eta_{i-1}} \cdot (\mu + \eta_i) &&\mod (p-\pi)X^0(\un{T}), \\
\t_\mu(-\ovl{\eta}_i)&\equiv t_{\eta_{i-1}}\fW_{i-1} \cdot (\mu - \eta_i) &&\mod (p-\pi)X^0(\un{T}).
\end{alignat*}
These are precisely the Serre weights that have a nonsplit extension with $F(\mu)$ by \cite[Cor.\ 5.6]{BP}. (Note that by assumption all Serre weights in this lemma are regular.)
\end{proof}

\begin{rem}\label{rk:graph-auto}
The ``change of origin'' map $\Lambda_W^\lambda \congto \Lambda_W^\mu$ sending $\omega'$ to $w_{\omega}^{-1}(\omega')+\omega$
(see Lemma~\ref{lm:change-origin}) clearly preserves adjacency, i.e.\ is a graph automorphism. Under the identification
$\Lambda_W\cong \ZZ^\cJ$ it is of the form $(a_0,\dots,a_{f-1}) \mapsto (\eps_0 a_0+n_0,\dots,\eps_{f-1} a_{f-1}+n_{f-1})$
for some $\eps_i \in \{\pm 1\}$ and $n_i \in \ZZ$.
\end{rem}

\section{Galois deformations: background and lemmas}
\label{sec:galois-deformations-background}

\subsection{Kisin modules with descent data and the monodromy condition}
\label{sec:some-background}

We keep the setup of \S \ref{sec:preliminaries}, in particular $K$ denotes the unramified extension of $\Qp$ of degree $f$, with residue field $k$.
For this section we will recall and slightly extend some relevant background and notation from \cite{LLLM}, \cite{LLLM2}, and \cite{LLL}.

\subsubsection{Kisin modules}

From now on we fix a tame inertial type $\tau$ together with a lowest alcove presentation $(s,\mu)$ for $\tau$, and we assume throughout this section that $\mu$ is $1$-deep in alcove $\un{C}_0$. (The lowest alcove presentation fixes an ordering of the characters in $\tau$. This will be important in defining many of the concepts below, see Remark~\ref{rk:depends-on-LAP}.)
Recall that $s_{\tau} = s_0 s_{f-1} s_{f-2} \cdots s_1 \in S_n$ and that $r$ denotes the order of $s_\tau$.

As in \S\ref{sec:tame-inertial-types} we let $K'/K$ be the unramified extension of $K$ of degree $r$ with residue field $k'$.
Fix an $e'$-th root $(-p)^{1/{e'}}$ of $-p$, let $E(u')\defeq (u')^{e'}+p \defeq v+p$ denote the minimal polynomial of $(-p)^{1/{e'}}$ over $K'$, and let $L'\defeq K'((-p)^{1/{e'}})$.

Let $\Delta'\defeq \Gal(L'/K')\subseteq \Delta\defeq \Gal(L'/K)$.
If $R$ is a complete noetherian local $\cO$-algebra with finite residue field define $\fS_{L', R} \defeq (W(k') \otimes_{\Zp} R)\bbra{u'}$. The isomorphism of $R$-algebras $W(k') \otimes_{\Zp} R\congto \prod_{j'\in\cJ'}R$, $x\otimes r\mapsto (\sigma_{-j'}(x)r)_{j'\in\cJ'}$ induces an $R$-linear isomorphism $\fS_{L', R} \congto \bigoplus_{j'\in\cJ'}R\bbra{u'}$. Given a $\fS_{L', R}$-module $\fM$ we thus have an $R$-linear isomorphism $\fM \congto \bigoplus_{j'\in\cJ'}\fM^{(j')}$. (We warn the reader that, due to our choice of normalization $\sigma'_{j'}\defeq \sigma'_0 \circ \phz^{j'}$, we need to use $\sigma'_{-j'}$ in the definition $\fM^{(j')}$ in order to be compatible with the convention of \cite{LLL} on Kisin modules, see Remark \ref{rmk:LLL:error} above.)

Recalling from \cite[\S 3.1]{LLLM2} that $\fS_{L', R}$ is endowed with an action of $\Delta$ 
 and that $v = (u')^{e'}$ we have
\[
(\fS_{L', R})^{\Delta = 1} = (W(k) \otimes_{\Zp} R)\bbra{v}.
\]

Let $h\ge 0$ be an integer. We define the groupoid of Kisin modules over $R$ of $E(u')$-height $\le h$ and descent data of type $\tau$ as in \cite[Def.\ 3.1.3]{LLLM2} (with the caveat that we consider modules of rank $n$ as opposed to $3$ in \emph{loc.~cit.}), and denote it by $Y^{[0, h], \tau}(R)$. 
(By \cite[Rk.\ 5.1.4(2)]{MLM}, $Y^{[0, h], \tau}(R)$ is given by the $R$-points of a $p$-adic formal algebraic stack $Y^{[0, h], \tau}$ in the sense of \cite[Def.\ A.2]{CEGS}.)
Given an object $(\fM,\phi_{\fM})$ (or, for short, just $\fM$) of $Y^{[0, h], \tau}(R)$ we have the notion of \emph{eigenbasis} $\beta = (\beta^{(j')})$ for $\fM$, as defined in \cite[Def.\ 3.1.6]{LLLM2}, \cite[Def.\ 3.2.8]{LLL}.

In particular, given %
a Kisin module $\fM\in Y^{[0,h],\tau}(R)$ and an eigenbasis $\beta$ of $\fM$ we can consider the matrix of the Frobenius morphism $\phi_{\fM}$. 
In the definition below we let $\phz$ be the $R$-linear endomorphism of $R\bbra{u'}$ which sends $u'$ to $(u')^p$.

\begin{definit}\label{def:C-and-A}
We let $C_{\fM,\beta}^{(j')} \in \M_n(R\bbra{u'})$ denote the matrix of $\phz^*(\fM^{(j')})\ra\fM^{(j'+1)}$ with respect to the bases $\phz^*(\beta^{(j')})$ and $\beta^{(j'+1)}$, i.e.\ $\beta^{(j'+1)}C_{\fM,\beta}^{(j')}=\phi^{(j')}_\fM(\phz^*(\beta^{(j')}))$.
We denote by $A_{\fM,\beta}^{(j')} \in \M_n(R\bbra{v})$ the matrix
\[
A_{\fM,\beta}^{(j')}\defeq\Ad\Big ((\dot s'_{\orient,j'+1})^{-1}(u')^{-\bf{a}^{\prime\,(j'+1)}_{(s,\mu)}}\Big) (C^{(j')}_{\fM,\beta})
\]
(see also \cite[equation (5.4)]{MLM}, where $C^{(j')}_{\fM,\beta}$ in \emph{loc.~cit.} denotes the matrix of $\phz^*(\fM^{(j'-1)})\ra\fM^{(j')}$).
\end{definit}

\begin{rem}\label{rk:u-to-the-mu}
  We caution that $\Ad(\dot s (u')^{\mu})$ denotes $\Ad(\dot s)\Ad((u')^{\mu})$ and \emph{not}
 $\Ad((u')^{s(\mu)})$, and we remind the reader that $\dot s$ is the permutation matrix representing $s$ and that
 we have $(u')^{\mu} = \mathrm{Diag}((u')^{\mu_1},\dots,(u')^{\mu_n})$ for $\mu \in \ZZ^n$.
\end{rem}

\begin{rem}\label{rk:depends-on-LAP}
 We stress that the notion of eigenbasis and the definition of $A_{\fM,\beta}^{(j')}$ depend on the choice of lowest
 alcove presentation $(s,\mu)$ for $\tau$. Moreover, as $\mu$ is assumed to be $1$-deep in alcove $\un{C}_0$, the matrix $A_{\fM,\beta}^{(j')}$ only depends on $j'$ modulo $f$ and is upper-triangular modulo $v$ (see the discussion after \cite[Rk.~ 5.1.7]{MLM}).
\end{rem}
If $\lambda=(\lambda_{j,1},\dots,\lambda_{j,n})_j\in X^*(\un{T})$ is a dominant character such that $\lambda_{j,i}\in\{0,\dots,h\}$ for all $j,i$, we have a closed $p$-adic formal substack $Y^{\leq \lambda,\tau}$ of $Y^{[0,h],\tau}$ defined in \cite[Thm.\ 5.3]{CL}, which is flat over $\cO$ and has reduced versal rings.
It is characterized by the property that for any flat $p$-adically complete noetherian local $\cO$-algebra $R$, a Kisin module $\fM\in Y^{[0,h],\tau}(R)$ belongs to $Y^{\leq \lambda,\tau}(R)$ if and only if all $i$ by $i$ minors of $A_{\fM,\beta}^{(j)}$ are divisible by $(v+p)^{\sum_{k=1}^i\lambda_{j,n+1-k}}$, for $i\in\{1,2,\dots,n\}$ and $j\in \ZZ$ (cf.~\cite{MLM} the discussion after Warning 5.3.2, see also \cite[Prop.\ 4.18]{LLLM}).
This definition does not depend on the choice of the eigenbasis for $\fM$.

\begin{definit}\label{def:shape}
Let $\ovl{\fM}\in Y^{[0,h],\,\tau}(\F)$.
Write $\cI(\F)$ for the Iwahori subgroup of $\GL_n(\F\bbra{v})$ consisting of matrices which are upper-triangular modulo $v$.
We say that $\ovl{\fM}$ has \emph{shape} $\tld{w}\in \wtld{\un{W}}^\vee$ with respect to $\tau$ if for any choice of eigenbasis $\ovl{\beta}$ of $\ovl{\fM}$ the equality 
\[
\cI(\F)A^{(j)}_{\ovl{\fM},\ovl{\beta}}\cI(\F)=\cI(\F)\tld{w}_j\cI(\F)
\]
holds in $\GL_n(\F\ppar{v})$ for all $j=0,\dots,f-1$. 
This notion is independent of $\ovl{\beta}$ by \cite[Prop.\ 2.15, 2.16]{LLLM}, but again depends on the choice of lowest alcove presentation of $\tau$.
\end{definit}

Fix $\ovl{\fM}\in Y^{[0,h],\tau}(\F)$ we recall that an eigenbasis $\ovl{\beta}$ is a \emph{gauge basis} if $A^{(j)}_{\ovl{\fM},\ovl{\beta}}$ has a particularly
simple form \cite[Def.\ 3.2.23]{LLL}. A gauge basis always exists and is unique up to scaling by $\{(t_j)_{j\in\cJ'}\in T(\F)^{f'} : t_j=t_k\text{ for $j\equiv k\mod f$}\}$
(this is \cite[Prop.\ 3.2.22]{LLL} in the particular case $h=n-1$, and the general case follows from \cite[Prop.\ 5.1.8, Lemma\ 5.2.2]{MLM}).

We now fix $\ovl{\fM}\in Y^{[0,h],\tau}(\F)$ together with a gauge basis $\ovl{\beta}$ for it.
Write $\tld{w}=(w_jt_{\nu_j})_j\in\wtld{\un{W}}^\vee$ for its shape with respect to $\tau$.

The following result, generalizing \cite[Thm.\ 4.1, Thm.\ 4.16]{LLLM}, \cite[Prop.\ 3.4.3]{LLL}, is a particular case of \cite[Prop.\ 5.2.7]{MLM}.

\begin{prop}
\label{prop:GB}
Let $R$ be a complete noetherian local $\cO$-algebra with residue field $\F$, and let $\tau$ be an $(h+1)$-generic tame inertial type \emph{(}see Definition \ref{defi:gen}\emph{)}.
Let $\fM\in Y^{[0,h], \tau}(R)$ together with an isomorphism $\fM\otimes_R\F\cong \ovl{\fM}$.

Then there exists an eigenbasis $\beta$ for $\fM$ lifting $\ovl{\beta}$ such that for all $1\leq i,k\leq n$ and all $j=0,\dots,f-1$ we have
\begin{enumerate}
\item \label{it:GB:1}
$A^{(j)}_{ik}\in v^{\delta_{i>k}}R[v+p]$,
\item \label{it:GB:2}
 $\deg_v(A^{(j)}_{ik})\leq \nu_{j,k}-\delta_{i<w_j(k)}$ with equality if $(i,k)=(w_j(k),k)$,
 \end{enumerate}
where $A^{(j)}\defeq A^{(j)}_{\fM,\beta}$ and $\delta_{i>k}\in\{0,1\}$ equals $1$ if and only if $i>k$ (resp.~$\delta_{i<w_j(k)}\in\{0,1\}$ equals $1$ if and only if $i<w_j(k)$).
Furthermore, such a $\beta$ is uniquely determined up to scaling by the group $\{(t_j)_{j\in\cJ'}\in \big(\ker(T(R)\ra T(\F))\big)^{f'} : t_j=t_k\text{ for $j\equiv k\mod f$}\}.$ 
\end{prop}

\begin{definit}
\label{def:gauge:basis}
Let $R$ be a complete noetherian local $\cO$-algebra with residue field $\F$, and let $\fM\in Y^{[0,h], \tau}(R)$ together with an isomorphism $\fM\otimes_R\F\cong \ovl{\fM}$.
A \emph{gauge basis} of $\fM$ is an eigenbasis $\beta$ lifting $\ovl{\beta}$ that satisfies conditions \ref{it:GB:1} and \ref{it:GB:2} of Proposition \ref{prop:GB}.
\end{definit}

\subsubsection{Monodromy condition}
\label{subsec:MC}

Let $R$ be a $p$-adically complete flat $\cO$-algebra that is topologically of finite type. 
Define $\cO^{\rig}_R$ as the inverse limit over $n\geq 1$ of $R\bbra{u',\frac{{u'}^{n}}{p}}[1/p]$, the transition maps being the natural inclusions, so $\cO^{\rig}_R$ is a subring of $R[1/p]\bbra{u'}$.
The Frobenius $\phz:u'\mapsto (u')^p$ on $R\bbra{u'}$ extends naturally to $\cO^{\rig}_R$.
By letting
$$
\lambda \defeq \prod_{n=0}^{\infty} \phz^n \left(\frac{E(u')}{p} \right) = \prod_{n=0}^{\infty} \bigg(1+\frac{v^{p^n}}p\bigg) \in \cO^{\rig}_{\cO} \subseteq \cO^{\rig}_R
$$ 
we have the derivation $N_{\nabla} \defeq - u' \lambda \frac{d}{d (u')}$ of $\cO^{\rig}_R$.

Let $\fM \in Y^{[0, h], \tau}(R)$ and write $\fM^{\rig}$ for the base change $\fM \otimes_{R\bbra{u'}} \cO^{\rig}_R$, which decomposes
as $\fM^{\rig} = \bigoplus_{j'} \fM^{\rig,(j')} = \bigoplus_{j'} \fM^{(j')}\otimes_{R\bbra{u'}} \cO^{\rig}_R$.

The following result builds on \cite[Cor.\ 1.3.15]{KisinFcrys} and is stated in \cite[Prop.\ 7.1.3]{MLM}.
\begin{prop} \label{prop:Kisin-connection} 
Let $\fM \in Y^{[0, h], \tau}(R)$ for $R$ a $p$-adically complete flat $\cO$-algebra that is topologically of finite type. Then $\fM^{\rig}[1/\lambda]$ is equipped with a unique derivation $N_{\fM^{\rig}}$ over $N_\nabla$ such that 
\begin{equation*} \label{commrel}
N_{\fM^{\rig}} \phi_{\fM^{\rig}} = E(u') \phi_{\fM^{\rig}} N_{\fM^{\rig}}
\end{equation*}
and $N_{\fM^{\rig}}\mod u' = 0$.
\end{prop} 
We have a decomposition of $N_{\fM^{\rig}}$ into $N^{(j')}_{\fM^{\rig}}:\fM^{\rig,(j')}[1/\lambda]\ra\fM^{\rig,(j')}[1/\lambda]$ and we write $N_{\fM^{\rig},\beta}^{(j')}\in \M_n(\cO^{\rig}_R[1/\lambda])$ to denote the matrix of the endomorphism $N^{(j')}_{\fM^{\rig}}$ with respect to the basis $\beta^{(j')}$ (short for $\beta^{(j')}\otimes 1$) of $\fM^{\rig,(j')}[1/\lambda]$, i.e.\ $\beta^{(j')}N_{\fM^{\rig},\beta}^{(j')}=N^{(j')}_{\fM^{\rig}}(\beta^{(j')})$.

\begin{definit}
\label{def:Mcond}
Let $\fM \in Y^{[0,h],\tau}(R)$ with eigenbasis $\beta$.
The \emph{monodromy condition} is the condition that $\lambda^{h-1} N_{\fM^{\rig},\beta}^{(j')}$ %
vanishes to order $h-1$ at $u'=(-p)^{1/{e'}}$ for all $j'$.
\end{definit}
We see as in \cite[Prop.\ 5.3]{LLLM} that the condition above is equivalent to $N_{\fM^{\rig}}(\fM^{\rig})\subseteq \fM^\rig$. 
As in the proof of \cite[Thm.\ 6.14]{LLLM}, the monodromy condition only depends on $j'$ modulo $f$.

As in \cite[Thm.\ 5.6]{LLLM}, \cite[Prop.\ 3.4.12]{LLL}, given $\fM \in Y^{[0,h],\tau}(R)$ with eigenbasis $\beta$, the matrix $N_{\fM^{\rig},\beta}^{(j')}$ can be expressed as
\[
N_{\fM^{\rig},\beta}^{(j')}=
N_1^{(j')}+\sum_{i=1}^\infty\Bigg(
\prod_{k=0}^{i-1}\phz^k(C_{\fM,\beta}^{(j'-k-1)})\Bigg)\phz^{i}(N_1^{(j'-i)})\Bigg(
\prod_{k=i-1}^{0}\phz^k\big(E(u')(C_{\fM,\beta}^{(j'-k-1)})^{-1}\big)\Bigg),
\]
where $N_1^{(j')}$ satisfies
\begin{align*}
&\Ad\Big((\dot s'_{\mathrm{or}, j'})^{-1} (u')^{-\mathbf{a}_{(s, \mu)}^{\prime \, (j')}}\Big) (\lambda^{h-1}N_1^{(j')})=\\
&\qquad\qquad=-\bigg(\frac{\phz(\lambda)}{p}\bigg)^{h}\left(- e'v \frac{d}{dv} A_{\fM,\beta}^{(j'-1)} - \left[\mathrm{Diag}((s'_{\mathrm{or}, j'})^{-1}(\mathbf{a}_{(s, \mu)}^{\prime \, (j')})), A_{\fM,\beta}^{(j'-1)}\right] \right)  (v+p)^{h}(A_{\fM,\beta}^{(j'-1)})^{-1}
\end{align*}
and $[M,N]\defeq MN-NM$. 

In what follows, define the \emph{leading term of the monodromy condition}
\begin{align}
\label{eq:non-expl_mon}
P_N(A_{\fM,\beta}^{(j-1)})\defeq \left(- e'v \frac{d}{dv} A_{\fM,\beta}^{(j-1)} - \left[\mathrm{Diag}((s'_{\mathrm{or}, j})^{-1}(\mathbf{a}_{(s, \mu)}^{\prime \, (j)})), A_{\fM,\beta}^{(j-1)}\right] \right)  (v+p)^{h}(A_{\fM,\beta}^{(j-1)})^{-1},
\end{align}
which again only depends on $j$ modulo $f$.

\begin{prop}[\cite{LLLM}]
\label{prop:monodromy} 
Let $\fM \in Y^{[0,h],\tau}(R)$ with eigenbasis $\beta$.
The monodromy condition is equivalent to the condition that
\begin{align}
\label{eq:expl_mon}
\Big(\frac{d}{du'}\Big)^{\!t}\big|_{u'=(-p)^{1/{e'}}} \left[\Ad\Big((\dot s'_{\mathrm{or}, j'})^{-1} (u')^{-\mathbf{a}_{(s, \mu)}^{\prime \, (j')}}\Big) (\lambda^{h-1}N_{\fM^{\rig},\beta}^{(j')})\right]=0
\end{align}
for all $t=0,\dots, h-2$, $j' = 0,\dots,f'-1$ and only depends on $j'$ modulo $f$.

Assume that $\tau$ is $N$-generic, where $N \ge 2h-3$ and $(N-1)(p-1) \ge h$.
Then the monodromy condition has the form
{\[
\Big(\frac{d}{dv}\Big)^t\big|_{v=-p} \left(P_N(A^{(j-1)}_{\fM,\beta})\right)+O(p^{N-(h-1)-t})=0
\]
for all $j=0,\dots,f-1$ and all $t=0,\dots,h-2$, where the terms $O(p^{N-(h-1)-t})$ denote specific but inexplicit elements of $p^{N-(h-1)-t}\M_2(R)$.}
\end{prop}
\begin{proof}
The proof is a slight generalization of the argument appearing in the proof of \cite[Prop.\ 3.4.12]{LLL} (which is the particular case where $h=n-1$ and $N=2n-1$).

As in the proof of \cite[Prop.\ 3.4.12]{LLL} the monodromy condition of Definition \ref{def:Mcond} is equivalent to condition (\ref{eq:expl_mon}) for all $t=0,\dots, h-2$ and all $j'$, as $u'$ is invertible in $\big(R[u']/(E(u'))\big)[1/p]$.

Defining $Z_i^{(j')}$, $M^{(j')}\in \M_n(R[1/p][[v]])$ in analogy to $Z_i^{(j)}$, $M^{(j)}$ in \emph{loc.~cit}.~(replacing $n-1$  and $j$ in \emph{loc.~cit}.~by $h$ and $j'$ respectively) we see as in \cite[Prop.\ 3.4.12]{LLL} that $Z_i^{(j')}\in \displaystyle\frac{v^{(N-1)p^{i-1}}}{p^{i(h-1)}}\M_n(R\bbra{v})$ for $i>1$ and $Z_1^{(j')}\in \frac{v^{N}}{p^{h-1}}\M_n(R\bbra{v})$ (as $\tau$ is $N$-generic), hence that 
\begin{equation}
\label{eq:M_in_p}\Big(\frac{d}{dv}\Big)^{\!t}\big|_{v=-p} M^{(j')} \in p^{N-(h-1)-t}\M_n(R)\qquad \text{for $t=0,\dots, h-2$ and all $j'$.}
\end{equation}
(Note that 
\[
\Big(\frac{d}{dv}\Big)^{\!t}\big|_{v=-p} (\phz^{i+1}(\lambda)/\phz(\lambda))^{h}Z_i^{(j')}
\] 
is contained
in $\sum_{t' = 0}^t \Zp\Big(\frac{d}{dv}\Big)^{\!t'}\big|_{v=-p} Z_i^{(j')}$.
Here we use that $(N-1)(p-1) \ge h$ to deal with the terms for $i \ge 2$.)
From the definition of $Z_i^{(j')}$ and $M^{(j')}$ we deduce from (\ref{eq:expl_mon})
that the monodromy condition is equivalent to \[ \Big(\frac{d}{dv}\Big)^{\!t}\big|_{v=-p} \left[-P_N(A^{(j'-1)}_{\fM,\beta})+M^{(j')}\right]=0 \] for all $j'$ and all $t=0,\dots, h-2$ (note that $(\phz(\lambda)/p)^h$ does not vanish at $u'=(-p)^{1/{e'}}$), which gives the second part of the statement thanks to \eqref{eq:M_in_p}.
\end{proof}

\subsection{Lemmas on mod \texorpdfstring{$p$}{p} Galois representations}
\label{sec:lemmas-mod-p}

Given $(s, \mu) \in \un{W} \times X^*(\un{T})$, consider the reduction $\taubar(s,\mu) : I_K \to \GL_n(\F)$ of the tame inertial type $\tau(s,\mu)$ of Definition \ref{def:tau} (with $\mu$ instead of $\mu + \eta$).
Typically, the length of $\taubar(s,\mu)$ as representation of $I_K$ equals the number of orbits of $s_\tau = s_f s_{f-1} \cdots s_1 \in S_n$.
The following definition gives the precise condition for this to be true.

\begin{definit}
\label{def:good}
We say that $(s,\mu)\in \un{W}\times X^*(\un{T})$ is \emph{good} if
\[
\sum_{j=0}^{f d(i)-1}p^j(s_1^{-1}\cdots s_j^{-1}(\mu_j))_{i}\not\equiv 0\pmod {\frac{q^{d(i)}-1}{q^d-1}}\qquad
\forall\, 1 \le i \le n\ \forall\, d\mid d(i), 1\leq d < d(i),
\]
where $d(i)\geq 1$ is minimal such that $s_1^{-1}s_2^{-1}\cdots s_{f d(i)}^{-1}(i)=i$ (and where the indices are considered modulo $f$).
\end{definit}
\begin{rem}
\label{rmk:good}
Definition \ref{def:good} generalizes \cite[Def.\ 6.19]{herzig-duke}.
\label{rem:irr:rep}
We see that $\taubar(s,\mu)$ is the restriction to $I_K$ of an irreducible $n$-dimensional representation of $G_K$ if and only if %
$s_\tau$ has order $n$ and $(s,\mu)$ is good.
Just note from Definition~\ref{def:tau} that
\[
\taubar(s,\mu) \cong \bigoplus_{i=1}^{n}\omega_{fd(i)}^{\sum_{j=0}^{f d(i)-1}p^j(s_1^{-1}\cdots s_j^{-1}\mu_j)_{i}}.
\]
In this case, any extension of $\taubar(s,\mu)$ to a $G_K$-representation is irreducible.
\end{rem}

\begin{lem}\label{lem:good-if-C0}
If $\mu-\eta\in\un{C}_0$, then $(s,\mu)$ is good for any $s \in \un W$.
\end{lem}

\begin{proof}
  Fix $i \in \{1,\dots,n\}$.
  Let $\nu \defeq \sum_{j=0}^{f-1} p^j s_1^{-1}\cdots s_j^{-1}(\mu_j) \in \ZZ^n$ and let $c_k \defeq (s_\tau^{-k} \nu)_{i}$. By assumption, $0 < \ang{\mu_j,\alpha_j^\vee} < p$ for
  all $i$, which implies that $0 < |c_k - c_\ell| < q$ for all $k \not\equiv \ell \pmod {d(i)}$. It suffices to show that
  $\sum_{k = 0}^{d(i)-1} q^k c_k \not\equiv 0\pmod {\frac{q^{d(i)}-1}{q^d-1}}$ for all $d\mid d(i), 1\leq d < d(i)$.
  This follows exactly as in the proof of \cite[Lemma 6.24]{herzig-duke}.
  (Alternatively one can check that Definition~\ref{def:good} is equivalent to the definition given in \cite[\S 2.2]{LLL}
  and invoke \cite[Lemma 2.2.3]{LLL}.)
\end{proof}

\begin{definit}(\cite[Def.\ 3.1.1]{LLL})
\label{def:ss-phi-mod}
For $\tld{w}\in\widetilde{\un{W}}^\vee$ and $D\in\un{T}(\F)$, let $\cM(\tld{w},D)$ denote the \'etale $\phz$-module which is free of rank $n$ over $k\ppar{v}\otimes_{\Fp}\F$ and such that $\Mat(\phz^{(j)})=D_j\tld{w}_j\in \GL_n(\F\ppar{v})$ with respect to the standard basis (where $\widetilde{\un{W}}^\vee$ embeds into $\GL_n(\F\ppar{v})^f$ as at the end of \S \ref{sec:GT:prel}).
\end{definit}

Recall that $\cO_{\mathcal{E}}$ denotes the $p$-adic completion of $W(k)\bbra{v}[1/v]$. 
For any complete noetherian local $\Zp$-algebra $R$ with maximal ideal $\fm_R$ and finite residue field we let $\cO_{\mathcal{E},R}\defeq \cO_{\mathcal{E}}\widehat{\otimes}_{\Zp}R$ where the completed tensor product is with respect to the $p$-adic topology on $\cO_{\mathcal{E}}$ and the $\fm_R$-adic topology on $R$.
Let $K_\infty\defeq\bigcup_{n\in\NN}K(p_n)$ where $(p_n)_{n\in\NN}\in(\ovl{\mathbb{Q}}_p)^{\NN}$ satisfies $p_0=-p$ and $p_{n}^p=p_{n-1}$ if $n\geq 1$.
By \cite[\S 3.1]{LLL} (building on the classical result of Fontaine \cite{fontaine-fest}), there is a rank preserving exact contravariant functor $\bV^*_K$ from the category of finite rank projective \'etale $\phz$-modules over $\cO_{\mathcal{E},R}$ to the category of continuous representations of $G_{K_\infty}$ over finite rank projective $R$-modules.

\begin{definit}\label{def:ss-phi-mod2}
For $\tld{w}\in\widetilde{\un{W}}^\vee$ and $D\in\un{T}(\F)$, let $V(\tld{w},D)$ be the unique tame representation of $G_K$ over $\F$ of dimension $n$ such that 
\[
V(\tld{w},D)|_{G_{K_\infty}}\cong \bV^*_K(\cM(\tld{w},D)).
\]
Its existence and uniqueness is guaranteed by \cite[Prop.\ 3.1.2]{LLL} and the equivalence for tame representations in \cite[\S 3.1]{LLL}.
\end{definit}

\begin{lem}
\label{lem:un:tw}
For $\lambda\in(\F^\times)^f$ we have
\[
V(\tld{w},\lambda D)\cong V(\tld{w},D)\otimes_{\F}\nr\Big(\prod_{j=0}^{f-1}\lambda_j\Big),
\]
where $\nr(\alpha)$ denotes the unramified character of $G_K$ sending an arithmetic Frobenius to $\alpha\in\F^\times$.
\end{lem}
\begin{proof}
As $\cM(\tld{w},\lambda D)$ is the tensor product of $\cM(\tld{w},D)$ and $\cM(1,\lambda )$ over $k\ppar{v}\otimes_{\Fp}\F$ and $\bV^*_K$ is a tensor functor, it suffices to show that 
\[
V(1,\lambda) \cong \nr(\prod_{j=0}^{f-1}\lambda_j).
\]
Note that $\cM(1,\lambda )$ is isomorphic to the rank one \'etale $\phz$-module with 
\[
\phz^{(j)}=
\begin{cases}
  1&\text{if $0 \le j<f-1$,}\\
  \prod_{j'=0}^{f-1}\lambda_{j'}&\text{if $j=f-1$}  
\end{cases}
\]
in the standard basis.
By the proof of \cite[Lemma 6.3]{GLS}, $\bV^*_K(\cM(1,\lambda))\cong\nr(\prod_{j=0}^{f-1}\lambda_j)|_{G_{K_\infty}}$.
\end{proof}

\begin{prop}
\label{prop:un:twist}
Let $\tld{w}\in\widetilde{\un{W}}^\vee$, and suppose that $\tld{w}^*=t_{\mu'}s'\in \widetilde{\un{W}}$ is such that $(s',\mu')\in \un{W}\times X^*(\un{T})$ is good.
Then
\begin{equation}
\left\{ 
\rhobar: G_K\ra \GL_n(\F)\ \text{semisimple} :\ \rhobar|_{I_K}\cong \taubar(s',\mu')
\right\}_{/\cong}
=
\left\{
V(\tld{w},D)\ :\ D\in \un{T}(\F)
\right\}_{/\cong}.\label{eq:semisimple-rhobar}
\end{equation}
\end{prop}
\begin{proof}
By \cite[Prop.\ 3.1.2]{LLL} we know that $V(\tld{w},D)|_{I_K} \cong \taubar(s',\mu')$.
It remains to verify that $V(\tld{w},D)$ is semisimple and that the left-hand side of~\eqref{eq:semisimple-rhobar} is contained in the right-hand side of~\eqref{eq:semisimple-rhobar}.
As in line 1 of the proof of \cite[Prop.\ 3.1.2]{LLL} we may assume that $(\tld{w}^*)_j=1$ for all $0\leq j<f-1$.
Decompose $\{1,\dots,n\}$ and $X^*(\un{T})$ according to the orbits of $(s^{\prime \ast})_{f-1}= (s'_0)^{-1}\in S_n$ (i.e.\ find a minimal Levi subgroup containing $(s',\mu')$ and decompose it into a product of smaller general linear groups).
Correspondingly, $s' = \prod_{i=1}^t s^{\prime (i)}$ and $\mu' = \sum_{i=1}^t \mu^{\prime (i)}$ with each $(s^{\prime (i)},\mu^{\prime (i)})$ good such that $\taubar(s',\mu') \cong \bigoplus_{i=1}^t \taubar(s^{\prime (i)},\mu^{\prime (i)})$.
If $\rhobar$ is semisimple with $\rhobar|_{I_K}\cong \taubar(s',\mu')$, then, since each $(s^{\prime (i)},\mu^{\prime (i)})$ is good, we deduce that there exists a decomposition $\rhobar \cong \bigoplus_{i=1}^t \rhobar_i$ such that $\rhobar_i$ irreducible and $\rhobar_i|_{I_K} \cong \taubar(s^{\prime (i)},\mu^{\prime (i)})$ for all $i$.
Likewise, from the definitions, there exists a decomposition $V(\tld{w},D) \cong \bigoplus_{i=1}^t V(\tld{w}^{(i)},D^{(i)})$ with $\tld{w}^{(i)*} = t_{\mu^{\prime (i)}}s^{\prime (i)}$ and $D = \prod_{i=1}^t D^{(i)}$.
In this way we are reduced to the case where $\rhobar$ is irreducible or equivalently $s'_0$ has only one orbit.
Then $V(\tld{w},D)$ is irreducible for each $D \in \un{T}(\F)$ (cf.\ Remark~\ref{rem:irr:rep}).
On the other hand, if $\rhobar$ is as on the left-hand side of~\eqref{eq:semisimple-rhobar}, we know that
$V(\tld{w},D)|_{I_K} \cong \taubar(s',\mu') \cong \rhobar|_{I_K}$ for all $D \in \un{T}(\F)$.
Since $\rhobar$ is irreducible, Lemma \ref{lem:un:tw} implies that $\rhobar$ is contained in the right-hand side of~\eqref{eq:semisimple-rhobar}.
\end{proof}

Recall that $\rhobar : G_K \to \GL_n(\F)$ is \emph{cyclotomic free} if $\rhobar$
becomes upper-triangular over an unramified extension $K'/K$ of degree prime to $p$ such that $H^0(G_{K'},(\rhobar|_{G_{K'}})^{\mathrm{ss}} \otimes_\F \omega^{-1}) = 0$
\cite[Def.\ 3.8]{LLLM}.

\begin{lem}
\label{lem:ff}
If $\rhobar_1, \rhobar_2$ are finite-dimensional representations of $G_K$ over $\F$ such that $\rhobar_1^\vee\otimes_\F \rhobar_2$ is cyclotomic free, then the natural map
\[
\Hom_{G_K}(\rhobar_1,\rhobar_2)\rightarrow \Hom_{G_{K_\infty}}({\rhobar_1}|_{G_{K_\infty}},{\rhobar_2}|_{G_{K_\infty}})
\]
is an isomorphism.
\end{lem}
\begin{proof}
This follows from (the proof of) \cite[Lemma 7.2.10(3)]{MLM}.
\end{proof}

\begin{cor}
\label{cor:ff}
If $\rhobar_1, \rhobar_2$ are finite-dimensional representations of $G_K$ over $\F$ such that $\rhobar_1$ is $2$-generic, then the natural injective map
\[
\mathrm{Isom}_{G_K}(\rhobar_1,\rhobar_2)\rightarrow \mathrm{Isom}_{G_{K_\infty}}({\rhobar_1}|_{G_{K_\infty}},{\rhobar_2}|_{G_{K_\infty}})
\]
is a bijection.
\end{cor}
\begin{proof}
We first claim that $\rhobar^{\mathrm{ss}}|_{G_{K_\infty}} \cong (\rhobar|_{G_{K_\infty}})^{\mathrm{ss}}$ for any finite-dimensional representation $\rhobar$ of $G_K$ over $\F$,
i.e.\ that $\rhobar^{\mathrm{ss}}|_{G_{K_\infty}}$ is already semisimple. This follows as in \cite[\S 3.1]{LLL}: $\rhobar^{\mathrm{ss}}$ is a representation of
$G_K/I_K^w$, where $I_K^w$ is the wild inertia group and $G_{K_\infty}/(G_{K_\infty} \cap I_K^w) \cong G_K/I_K^w$, as $K_\infty/K$ is a totally ramified $p$-extension.

Assume $\mathrm{Isom}_{G_{K_\infty}}({\rhobar_1}|_{G_{K_\infty}},{\rhobar_2}|_{G_{K_\infty}})\neq 0$.
By the previous paragraph and again by the beginning of \cite[\S 3.1]{LLL} we thus have $\rhobar_1^{\mathrm{ss}}\cong\rhobar_2^{\mathrm{ss}}$, hence
$(\rhobar_1^\vee\otimes_\F \rhobar_2)^{\mathrm{ss}}\cong \ad(\rhobar_1)^{\mathrm{ss}}$.
As $\ad(\rhobar_1)$ is cyclotomic free by the analog of~\cite[Prop.\ 3.9]{LLLM} (and noting that $2$-generic in our context implies $2$-generic in the sense of \cite[Def.\ 3.7]{LLLM}, see \cite[Rk.\ 2.2.3]{LLLM2}), we obtain that 
$\rhobar_1^\vee \otimes_\F \rhobar_2$ is cyclotomic free, and we can then conclude by Lemma \ref{lem:ff}.
\end{proof}

\subsection{A commutative algebra lemma}
\label{sec:few-lemmas}

If $A$ is a local ring we denote by $\mathfrak{m}_A$ its maximal ideal.

\begin{lem}
\label{lem:irreducible}
Let $A\defeq \cO\bbra{x_1,\dots,x_n}$, where $\cO$ is a complete DVR with uniformizer $\varpi$ and $n \ge 2$.
If $f\in A^\times$ and $d>0$, then $x_1x_2+\varpi^d f$ is irreducible in $A$.
Moreover the ideals $(x_1x_2+\varpi^d f)$ and $(x_1)$ are distinct, and 
the ideals $(x_1x_2+\varpi^d f_1)$ and $(x_1x_2+\varpi^d f_2)$ are distinct if $f_1\not\equiv f_2\mod \mathfrak{m}_A$.
\end{lem}
\begin{proof}
By the $\cO$-automorphism of $A$ sending $x_2$ to $x_1+x_2$ and fixing $x_i$ ($i \ne 2$), we may instead consider $x_1^2+x_1x_2+\varpi^d g$ ($g\in A^\times$), which is distinguished in the variable $x_1$.
By the Weierstrass preparation theorem, if $x_1^2+x_1x_2+\varpi^d g$ is reducible then it has a factor of the form $x_1-b$ for some $b\in\fm_{\cO\bbra{x_2,\dots,x_n}}$.
Evaluating at $x_1 = b$ we see that $b^2+bx_2+\varpi^d g(b,x_2,\dots,x_n) = 0$, so $\varpi^d\mid b(b+x_2)$.
Then one easily checks $\varpi^d\mid b$ or $\varpi^d\mid (b+x_2)$. In the first case, $b=\varpi^d c$ and $\varpi^d c^2+c x_2+g(\varpi^dc,x_2,\dots,x_n)=0$, which implies $g\in \fm_A$, contradiction.
The second case is similar.

For the last part, the first two ideals are distinct by the Weierstrass preparation theorem.
Suppose that $x_1^2+x_1x_2+\varpi^d g_1 = u(x_1^2+x_1x_2+\varpi^d g_2)$ for some $u \in A^\times$.
By working modulo $(\varpi,x_2)$ we deduce that $u(0) \equiv 1 \pmod \varpi$.
On the other hand, $g_1(0) = u(0)g_2(0)$ in $\cO$, so $g_1 \equiv g_2 \mod \mathfrak{m}_A$, as required.
\end{proof}

\section{Galois deformation rings}
\label{sec:galo-deform-rings}

\subsection{Setup}
\label{sec:setup}

From now on we consider the situation where $n=2$.

Throughout this section we fix a semisimple Galois representation $\rhobar: G_K\ra \GL_2(\F)$ and $(s,\mu)\in \un{W}\times X^*(\un{T})$ such that $\rhobar|_{I_K}\cong \taubar(s,\mu)$, where
\begin{enumerate}
\item 
\label{it:sf:1}
$s_j\neq1$ (hence, $s_j=\fW$) precisely when $j=0$ and $\rhobar$ is irreducible;
\item 
\label{it:sf:2}
$\mu-\eta$ is $N$-deep in $\un{C}_0$ with {$N \ge 12$}.
\end{enumerate}
(The pair $(s,\mu)$ is not uniquely determined by $\rhobar|_{I_K}$ and depends on the choice of the embedding $\sigma_0$; however when $\rhobar$ is {13}-generic the conditions \ref{it:sf:1}--\ref{it:sf:2} above can always
be arranged by an appropriate choice of $s$, see Remark \ref{rem:emb} and \cite[Proposition 2.2.15]{LLL}.)
Up to a twist by a power of $\omega_f$ we can furthermore assume that $\mu_j = (r_j+2,1)_j\in \ZZ^2$ with $N< r_j+1<p-N$ for all $j$, and hence
\begin{equation*}
\rhobar|_{I_K} \cong
\begin{cases}
\bigg(\omega_{f}^{\sum_{j=0}^{f-1}(r_j+1)p^j}\oplus 1\bigg)\otimes \omega&\text{ if $\rhobar$ is reducible,}
\\
\bigg(\omega_{2f}^{\sum_{j=0}^{f-1} (r_{j}+1)p^j}\oplus \omega_{2f}^{\sum_{j=0}^{f-1}(r_j+1)p^{j+f}}\bigg)\otimes\omega%
&\text{ if $\rhobar$ is irreducible.}
\end{cases}
\end{equation*}

In this section we will study various framed Galois deformation rings of $\rhobar$, for which $3^f$ tame inertial types play a role, and we now introduce them.
(These are precisely the tame inertial types $\tau$ such that $\JH\left(\ovl{\sigma(\tau)}\otimes_\F (N_{k/\Fp}\circ \det)\right) \cap W(\rhobar) \ne \emptyset$, cf.\ Lemma~\ref{lem:inter} below.)
Given $$\tld{w}\in\Adm^\vee(t_{(\un 2,\un 1)})=\big\{t_{(2,1)},\, \fW t_{(2,1)},\, t_{(1,2)}\big\}^f$$ arbitrary, write $\tld{w}^*=t_\nu w$ for $(w,\nu)\in \un{W}\times X^*(\un{T})$.
Define the type $$\tau_{\tld{w}} \defeq \tau(sw^{-1},\mu-sw^{-1}(\nu))$$ 
{(or just $\tau$ when there is no ambiguity on $\tld{w}$),}
which we always consider together with its lowest alcove presentation $(s(\tau),\mu(\tau))\defeq(sw^{-1},\mu-sw^{-1}(\nu)-\eta)$. %

Concretely, $s(\tau)_j=w_{j}^{-1}$ except when $j=0$ and $\rhobar$ is irreducible, in which case we have $s(\tau)_0=\fW w_{0}^{-1}$, and
\begin{equation*}
\mu(\tau)_j + \eta_j =
\begin{cases}
  (r_j,0) & \text{if $(t_{\nu_j}w_j,s_j)\in\{(t_{(2,1)},1),\, (t_{(2,1)}\fW,\fW), (t_{(1,2)},\fW)\}$,} \\
  (r_j+1,-1) & \text{if $(t_{\nu_j}w_j,s_j)\in\{(t_{(2,1)},\fW),\, (t_{(2,1)}\fW ,1), (t_{(1,2)},1)\}$.}
\end{cases}
\end{equation*}
Then
\begin{equation*}
\tau_{\tld{w}}\cong
\begin{cases}
\widetilde\omega_{f}^{\bf{a}^{(0)}_1}\oplus \widetilde\omega_{f}^{\bf{a}^{(0)}_2}&\text{if $\prod_{j=0}^{f-1} s(\tau)_j = 1$,}
\\
\widetilde\omega_{2f}^{\bf{a}^{(0)}_1+p^f\bf{a}^{(0)}_2}\oplus
\widetilde\omega_{2f}^{\bf{a}^{(0)}_2+p^f\bf{a}^{(0)}_1}
&\text{otherwise,}
\end{cases}
\end{equation*}
where $\bf{a}^{(0)}=(\bf{a}^{(0)}_1,\bf{a}^{(0)}_2)\in \ZZ^2$ is defined to be $\bf{a}^{(0)}\defeq \sum_{j=0}^{f-1}p^j(\prod_{i=1}^{j}w_{j})(\mu(\tau)_j+\eta_j).$

Recall from \cite[\S 3.2]{LLL} that we have a functor $T^*_{dd}$ from $Y^{\leq(3,0), \tau_{\tld{w}}}(\F)$ to $2$-dimensional continuous representations of $G_{K_\infty}$ over $\F$.
\begin{lem}
\label{lem:ss:Kisin}
Up \ to \ isomorphism \ there \ exists \ a \ unique \ \emph{(}semisimple\emph{)} \ Kisin \ module \ $\ovl{\fM}$ \ in $Y^{\leq(3,0), \tau_{\tld{w}}}(\F)$ of shape $\tld{w}$ 
such that $T^*_{dd}(\ovl{\fM})\cong \rhobar|_{G_{K_\infty}}$.
\end{lem}
\begin{proof}
Define a Kisin module $\ovl{\fM}$ of type $\tau_{\tld{w}}$ by $A^{(j)}=D_j\tld{w}_j$ (keeping the notation of Definition~\ref{def:C-and-A} and setting $A^{(j)}\defeq A^{(j)}_{\ovl{\fM},\ovl{\beta}}$ for some choice of eigenbasis $\ovl{\beta}$ on $\ovl{\fM}$) for some $D = (D_j)\in \un{T}(\F)$.
By definition it has shape $\tld{w}$.
As $\tld{w}\in\Adm^\vee(t_{(\un 2,\un 1)}) \subseteq \Adm^\vee(t_{(\un 3,\un 0)})$ we know that $\ovl{\fM}\in Y^{\leq(3,0), \tau_{\tld{w}}}(\F)$ (\cite[\S 3.2]{LLL}).
By \cite[Prop.\ 3.2.1]{LLLM2} the associated \'etale $\phz$-module is given by
\[
\Mat(\phz^{(j)})=\big(D\tld{w}(sw^{-1})^*t_{(\mu-sw^{-1}(\nu))^*}\big)_j=(Ds^*t_{\mu^*})_j
\]
in some suitable basis.
As $\mu-\eta\in \un{C}_0$ we know by Lemma~\ref{lem:good-if-C0} that $(s,\mu)$ is good, hence by Proposition \ref{prop:un:twist} we can choose $D\in\un{T}(\F)$ such that $T^*_{dd}(\ovl{\fM})\cong \rhobar|_{G_{K_\infty}}$.
The uniqueness of $\ovl{\fM}$ follows as in \cite[Thm.\ 3.2]{LLLM}, \cite[Prop.\ 3.2.18]{LLL}
(this uses that $3 < \ang{\mu(\tau)_j+\eta_j,\alpha_j^\vee} < p-4$ for all $j$).
\end{proof}

\begin{lem}
\label{lem:inter}
There is a unique bijection $\theta: W(\rhobar)\rightarrow\big\{t_{(2,1)},\, t_{(1,2)}\big\}^f$ such that for $\sigma\in W(\rhobar)$ and $\tld{w}\in \Adm^\vee(t_{(\un 2,\un 1)})$ we have
\begin{equation}
  \sigma\in \JH\left(\ovl{\sigma(\tau_{\tld{w}})}\otimes_\F (N_{k/\Fp}\circ \det)\right)\Leftrightarrow \left(\tld{w}_j\neq \theta(\sigma)_j\ \forall\, j\right).\label{eq:theta}
\end{equation}

Moreover, if $\tau$ is any tame inertial type, then $\JH\left(\ovl{\sigma(\tau)}\otimes_\F (N_{k/\Fp}\circ \det)\right) \cap W(\rhobar) \ne \emptyset$ if and only if $\tau = \tau_{\tld{w}}$ for some $\tld{w}\in \Adm^\vee(t_{(\un 2,\un 1)})$.
\end{lem}
\begin{proof}
We note that
$\ovl{\sigma(\tau_{\tld{w}})}\otimes_\F (N_{k/\Fp}\circ \det) \cong \ovl{\sigma(\tau(sw^{-1},\mu-sw^{-1}(\nu)+(\un 1,\un 1)))}$, and
as $\tld{w}\in \Adm^\vee(t_{(\un 2,\un 1)})$ we see that $\nu-(\un 1,\un 1) \in \eta+\Lambda_R$.

Recall from \S\ref{sec:ext:graph} that the map $\omega \mapsto F(\mathfrak{t}_{\mu-\eta}(\omega))$ induces a bijection between $\Lambda^{\mu-\eta}_W \subseteq \Lambda_W$
and the set of regular Serre weights with central character $(\mu-\eta)|_{\un{Z}_0(\Fp)}$.
By Proposition~\ref{prop:SW:graph}, this map induces a bijection between $s\Sigma\subseteq \Lambda_W^{\mu-\eta}$ and the set $W(\rhobar)$ (see just before \emph{loc.~cit}.~for $\Sigma$), and
by Proposition~\ref{prop:JH:graph} this map induces a bijection between $sw^{-1}(\Sigma-\ovl\nu)\subseteq \Lambda_W^{\mu-\eta}$ and the set $\JH\left(\ovl{\sigma(\tau_{\tld{w}})}\otimes_\F N_{k/\Fp}\circ \det\right)$ (note that $\ovl{\nu-(\un 1,\un 1)}=\ovl\nu$ in $\Lambda_W$). 
(Note that Propositions~\ref{prop:SW:graph}, \ref{prop:JH:graph} apply as soon as $\mu-\eta$ is $2$-deep in alcove $\un{C}_0$, and we have $N \ge 2$.)

We conclude that the statement of the proposition is equivalent to: there is a unique bijection $\theta^{\Sigma}: \Sigma\rightarrow\big\{t_{(2,1)},\, t_{(1,2)}\big\}^f$ such that for $\omega\in \Sigma$ and $\tld{w}\in \Adm^\vee(t_{(\un 2,\un 1)})$ we have
\begin{equation}
\label{eq:char:prop}
\omega\in w^{-1}(\Sigma-\ovl\nu)\cap \Sigma
\ \Leftrightarrow\ \left(\theta^\Sigma(\omega)_j\neq\tld{w}_j\ \forall\, j\right).
\end{equation}
We first consider the case $f = 1$.
In that case, $$\tld w \in \Adm^\vee(t_{(2,1)})=\{t_{(2,1)}, \fW t_{(2,1)}, t_{(1,2)}\}$$ and note that correspondingly
$$(w,\ovl\nu) \in \{(1,\ovl{\eta}),(\fW,\ovl{\eta}),(1,-\ovl{\eta})\}.$$
\begin{figure}[t]
\caption{Extension graph}
\label{fig:ext-graph}
\includegraphics[scale=0.35]{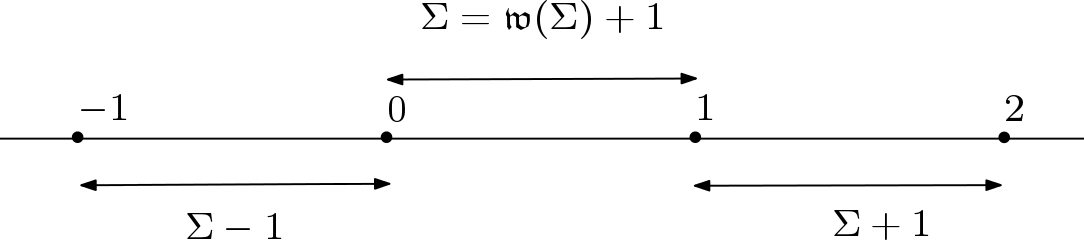}
\end{figure}%
As $\fW$ acts by $-1$ on $\Lambda_W$, we see from Figure~\ref{fig:ext-graph} and~\eqref{eq:char:prop} that
$\theta^{\Sigma}(0) = t_{(1,2)}$ and $\theta^{\Sigma}(\ovl{\eta}) = t_{(2,1)}$ is the desired unique bijection.

For general $f$, the existence follows from the case $f = 1$ by taking $\theta^{\Sigma}(\omega)_j = t_{(1,2)}$ if $\omega_{f-1-j}=0$
and $\theta^{\Sigma}(\omega)_j = t_{(2,1)}$ if $\omega_{f-1-j}=\ovl{\eta}_{f-1-j}$.
For uniqueness, fix $j$ and let $\tld{w}_{j'} \defeq \fW t_{(2,1)}$ for all $j'\ne j$.
Then \eqref{eq:char:prop} simplifies to 
$$\omega_{f-1-j}\in w_{f-1-j}^{-1}(\Sigma_{f-1-j}-\ovl\nu_{f-1-j}) \cap \Sigma_{f-1-j} \ \Leftrightarrow\ \theta^\Sigma(\omega)_j \neq \tld{w}_j,$$
where $\Sigma_{f-1-j} \defeq \{0,\ovl\eta_{f-1-j}\}$, which gives uniqueness by the case $f = 1$.

We now justify the final assertion.
We already know the ``if'' direction by~\eqref{eq:theta}.
Conversely, suppose that $\sigma \in \JH\left(\ovl{\sigma(\tau)}\otimes_\F (N_{k/\Fp}\circ \det)\right) \cap W(\rhobar)$.
By~\eqref{eq:theta} there exist $2^f$ elements $\tld{w}\in \Adm^\vee(t_{(\un 2,\un 1)})$ such that $\sigma \in \JH\left(\ovl{\sigma(\tau_{\tld{w}})}\otimes_\F (N_{k/\Fp}\circ \det)\right)$.
As $\sigma \in \JH\left(\ovl{\sigma(\tau)}\otimes_\F (N_{k/\Fp}\circ \det)\right)$, the representation $\sigma(\tau) \otimes_E \chi$ is an irreducible constituent of $\wt P_\sigma[1/p]$, where $\chi$ denotes the Teichm\"uller lift of $N_{k/\Fp}\circ \det$ and $\wt P_\sigma$ denotes the projective cover of $\sigma$ as $\cO[\GL_2(k)]$-module.
It thus suffices to show that $\wt P_\sigma[1/p]$ has length $2^f$. 

We prove that $\wt P_\sigma[1/p]$ has length $2^f$  for any Serre weight $\sigma=F(\lambda)$ such that $\lambda\in X^*(\underline{T})$ satisfies $0\leq \langle \lambda,\alpha_i^\vee\rangle \leq p-2$ for all $i\in\cJ$. We first treat the case $\dim_{\F}(\sigma)\geq 2$. Then we have $\dim_E (\wt P_\sigma[1/p]) =\dim_{\F}(P_{\sigma})= (2p)^f$ by the structure of $P_{\sigma}\defeq\wt{P}_{\sigma}\otimes_{\cO}\F$ described in \cite[\S3]{BP} (as $P_{\sigma}$ is isomorphic to  the injective envelope  of $\sigma$ as $\F[\GL_2(k)]$-module), and $\dim_E(V)\geq p^f-1$ for any irreducible constituent $V$ of $\wt{P}_{\sigma}[1/p]$ (as $[\overline{V}:\sigma]\neq0$). Hence, if $\wt{P}_{\sigma}[1/p]$ had length $\geq 2^f+1$ we would obtain a contradiction  as  $(2^f+1)(p^f-1) > (2p)^f$ (as $p > 3$). On the other hand, since any irreducible $E[\GL_2(k)]$-module has dimension $\leq p^f+1$, a similar computation shows that $\wt{P}_{\sigma}[1/p]$ has length $\geq 2^f$, which proves the claim.  The case $\dim_{\F}(\sigma)=1$ can be treated in a similar way, noting that $\dim_{E}(\wt{P}_{\sigma}[1/p])=(2^f-1)p^f$ and that $\wt{P}_{\sigma}[1/p]$ contains a unique irreducible constituent of dimension $1$. 
Alternatively, using \cite[Lemmas 3.4, 3.5 $\&$ 3.8]{BP} one checks that $[P_{\sigma}:\sigma]=2^f$ and concludes by noting that  irreducible  $E[\GL_2(k)]$-modules are residually multiplicity-free (\cite[Prop.\ 1.1, Prop.\ 1.3]{diamond-durham}. This moreover shows that $\wt{P}_{\sigma}[1/p]$ is multiplicity-free.
\end{proof}

\subsection{Deformation rings I: single type}
\label{sec:deformation-rings1}

We now compute some Galois deformation rings of $\rhobar$ for a single type $\tau$ and Hodge--Tate weights $\le (3,0)$, meaning Hodge--Tate weights $(3,0)$ or $(2,1)$.

We suppose that $\rhobar$ is as in \S\ref{sec:setup}.
Fix now $\tld{w}\in \Adm^\vee(t_{(\un 2,\un 1)})$ and $\ovl{\fM}\in Y^{\leq (3,0),\tau_{\tld{w}}}(\F)$ semisimple of shape $\tld{w}$ such that $T^*_{dd}(\ovl{\fM})\cong\rhobar|_{G_{K_\infty}}$ (Lemma \ref{lem:ss:Kisin}). By the proof of Lemma \ref{lem:ss:Kisin}, $\ovl{\fM}$ is such that the associated matrix $\ovl{A}^{(j)}\defeq {A}^{(j)}_{\ovl{\fM},\ovl{\beta}}$ equals $D_j\tld{w}_j$ for some $D_j\in T(\F)$ and some choice of eigenbasis $\ovl{\beta}$ for $\ovl{\fM}$.

We use the notation 
\begin{equation*}
D_{f-1-j}=
\begin{cases}
\begin{pmatrix}
\ovl{e_{11}^{\ast (j)}}&0\\
0& \ovl{d_{22}^{\ast (j)}}
\end{pmatrix}& \text{if $\tld{w}_{f-1-j}=t_{(2,1)}$,}\\
\begin{pmatrix}
\ovl{d_{12}^{\ast (j)}}&0\\
0&\ovl{d_{21}^{\ast (j)}}
\end{pmatrix}& \text{if $\tld{w}_{f-1-j}=\fW t_{(2,1)}$,}\\
\begin{pmatrix}
\ovl{d_{11}^{\ast (j)}}&0\\
0& \ovl{e_{22}^{\ast (j)}}
\end{pmatrix}& \text{if $\tld{w}_{f-1-j}=t_{(1,2)}$.}
\end{cases}
\label{eq:5}
\end{equation*}
(See Tables \ref{Table1FV}--\ref{Table3FV}, where the superscript $(j)$ is omitted for readability.)

Let $R^{\leq(3,0),\tau_{\tld{w}}}_{\rhobar}$ denote the maximal reduced, $\cO$-flat quotient of $R^{\square}_{\rhobar}$
that parametrizes lifts of $\rhobar$ of Hodge--Tate weights $\le (3,0)$ in each embedding and tame inertial type $\tau_{\tld{w}}$.
For each dominant character $\lambda\in X^*(\un{T})$ let $R^{\lambda,\tau_{\tld{w}}}_{\rhobar}$ denote the maximal reduced, $\cO$-flat quotient of $R^{\square}_{\rhobar}$
that parametrizes lifts of $\rhobar$ of Hodge--Tate weights $\lambda_j$ in the $j$-th embedding $\sigma_j$ for all $j$ and tame inertial type $\tau_{\tld{w}}$.

\begin{prop}
\label{prop:def:ring}
We have an isomorphism
\[
R^{\leq(3,0), \tau_{\tld{w}}}_{\rhobar}\bbra{X_1,\dots,X_{2f}}\cong\Big(
R/\sum_j I^{(j)}
\Big)\bbra{Y_1,\dots,Y_4},
\]
where $R \defeq \widehat{\bigotimes}_{\cO, 0\leq j\leq f-1}R^{(j)}$ and the $\cO$-algebras $R^{(j)}$ and the ideals $I^{(j)}$ of $R$ are found in Tables \ref{Table1FV}--\ref{Table3FV}. The irreducible
components of $\Spec R^{\leq(3,0), \tau_{\tld{w}}}_{\rhobar}$ are given by the $\Spec R^{\lambda, \tau_{\tld{w}}}_{\rhobar}$,
where $\lambda = (\lambda_j) \in \{(3,0), (2,1)\}^f$.

More precisely, via the above isomorphism, for any choice of $\lambda = (\lambda_j) \in \{(3,0), (2,1)\}^f$ the kernel of the natural surjection
$R^{\leq(3,0), \tau_{\tld{w}}}_{\rhobar}\bbra{X_1,\dots,X_{2f}} \onto R^{\lambda, \tau_{\tld{w}}}_{\rhobar}\bbra{X_1,\dots,X_{2f}}$ is
generated by the prime ideal $\sum_{j=0}^{f-1} \fp^{(j),\lambda_{f-1-j}}$ of $R/\sum_j I^{(j)}$, where the ideals $\fp^{(j),\lambda_{f-1-j}}$ of $R/\sum_j I^{(j)}$ are found in  Tables \ref{Table1FV}--\ref{Table3FV}.
\end{prop}
\begin{rem}
To obtain Proposition \ref{prop:def:ring} we cannot use directly the results of \cite{MLM}, namely Theorem 7.3.2(2) there.
In fact, on the one hand we need the precise equations for the ideals $I^{(j)}$ to perform the computations in Proposition \ref{prop:p:in:inter} (where we check that $p$ is contained in suitably chosen ideals in multi Hodge-type deformation rings).
On the other hand we need to perform Elkik's approximation theorem (used in the proof of \cite[Thm.\ 7.3.2(2)]{MLM}) in an effective way to have ``explicit'' generators of the minimal primes of the multi-type deformation rings.
As a byproduct, we have less stringent conditions on the tame inertial types appearing in Proposition \ref{prop:def:ring} above, in that the genericity of $\tau_{\tld{w}}$ is the explicit requirement that $\mu(\tau_{\tld{w}})$ is {$11$-deep} in $\un{C}_0$, rather than a condition on an inexplicit polynomial $P_{\tau_{\tld{w}}}\in \ZZ[X_1,X_2]$ such that $P_{\tau_{\tld{w}}}(\mu(\tau_{\tld{w}})_j)\not\equiv 0\pmod{p}$ for all $j\in\cJ$ (cf.~the geometric genericity condition described in \cite[\S 1.3]{MLM}).
\end{rem}

\begin{proof}
We let $\tau\defeq \tau_{\tld{w}}$ for short.

As $\ovl{A}^{(j)}=D_j\tld{w}_j$, the standard basis $\ovl{\beta}$ is a gauge basis of $\ovl{\fM}$ in the sense of \cite[Def.\ 3.2.23]{LLL}.
(There, $\ovl{\fM}\in Y^{\eta,\tau}(\F)$ but $\eta$ plays no role.)
For $R'$ a complete noetherian local $\cO$-algebra with residue field $\F$ define $D^{\leq (3,0),\tau}_{\ovl{\fM},\ovl{\beta}}(R')$ to be the groupoid of triples $(\fM,\beta,\jmath)$, where $\fM\in Y^{\leq(3,0),\tau}(R')$, $\beta$ is a gauge basis of $\fM$ (Definition \ref{def:gauge:basis}) and $\jmath: \fM\otimes_{R'}\F\stackrel{\sim}{\ra}\ovl{\fM}$ sending $\beta$ to $\ovl{\beta}$.
From the definition of a gauge basis, for any lift $(\fM,\beta,\jmath)\in D^{\leq (3,0),\tau}_{\ovl{\fM},\ovl{\beta}}(R)$ the corresponding matrices $A^{(j)}$ are given in row 1 of Tables \ref{Table1FV}--\ref{Table3FV}, where the entries $c_{11}^{(j)}, c_{12}^{(j)}$, \dots\ are in $R$, subject to $A^{(f-1-j)}$ reducing to our fixed $\ovl{A}^{(f-1-j)}$ modulo $\fm_{R}$.
By the analog of \cite[Prop.~3.2.1]{LLLM2}, the \'etale $\varphi$-module over $\cO_{\cE,R}$ corresponding to $T^*_{dd}(\fM)$ is given, at embedding $(f-1-j)$, by row 2 of Tables \ref{Table1FV}--\ref{Table3FV}.

By the analog of \cite[Prop.\ 4.18]{LLLM} the finite height conditions are given by
$$\det A^{(f-1-j)}\in R^\times (v+p)^3\ \forall\, j,$$
giving rise to the generators of the ideal $I^{(j),\leq (3,0)}$ in row 4 of Tables \ref{Table1FV}--\ref{Table3FV}.
As in \cite[Thm.\ 4.17]{LLLM}, $D_{\ovl{\fM},\ovl{\beta}}^{\leq (3,0),\tau}$ is represented by the maximal reduced $p$-flat quotient of $\widehat{\bigotimes}_{\cO, 0\leq j\leq f-1}R^{(j)}/I^{(j),\leq (3,0)}$, which we also denote by $R_{\ovl{\fM},\ovl{\beta}}^{\leq (3,0),\tau}$.

By the second statement of Proposition \ref{prop:monodromy} (applied with $h = 3$ {and noting that $\tau$ is $(N-1)$-generic}) the monodromy conditions are given by
\[
\left(\frac{d}{dv}\right)^{\!t}\Big|_{v=-p}\left[P_N(A^{(f-1-j)})\right]+O(p^{N-3-t})=0%
\]
for all $0\leq t\leq 1,\, 0\leq j\leq f-1$.
{(Recall that the $O(p^{N-3-t})$ denote specific but inexplicit elements of $p^{N-3-t}\M_2(R)$.)}
Note that 
\begin{align*}
P_N(A^{(f-1-j)})&\equiv
\left[
-e'v \frac{d}{dv}A^{(f-1-j)}+ A^{(f-1-j)}\begin{pmatrix}b^{(j)}&0\\0&c^{(j)}\end{pmatrix}
\right](v+p)^3 (A^{(f-1-j)})^{-1}\\
&\equiv
-e'\left[
v \frac{d}{dv}A^{(f-1-j)}- A^{(f-1-j)}\begin{pmatrix}a^{(j)}&0\\0&0\end{pmatrix}
\right](v+p)^3 (A^{(f-1-j)})^{-1}
\end{align*}
modulo $(v+p)^3 \M_2(R\bbra{v})$, where $(b^{(j)},c^{(j)})\defeq (s'_{\orient,f-j})^{-1}(\bf{a}^{\prime\, (f-j)}_{(s(\tau),\mu(\tau))})\in \ZZ^2$ (see~\eqref{eq:bold-a} for $\bf{a}^{\prime\, (f-j)}_{(s(\tau),\mu(\tau))}$) and $a^{(j)}\defeq \frac{b^{(j)}-c^{(j)}}{e'}\in \ZZ_{(p)}$.
(Note that the ``other'' term $\text{\tiny{$\begin{pmatrix}b^{(j)}&0\\0&c^{(j)}\end{pmatrix}$}}A^{(f-1-j)}
(v+p)^3 (A^{(f-1-j)})^{-1}$ from the Lie bracket in equation (\ref{eq:non-expl_mon}) is in $(v+p)^3 \M_2(R\bbra{v})$.)
We emphasize that the constants $a^{(j)}, b^{(j)}$ and $c^{(j)}$ \emph{depend on the whole $f$-tuple} $\tld{w}\in \Adm^\vee(t_{\un{2},\un{1}})$.

Combining this, the monodromy condition is 
\[
\left(\frac{d}{dv}\right)^{\!t}\Big|_{v=-p}\left\{\left[
v \frac{d}{dv}A^{(f-1-j)}- A^{(f-1-j)}\begin{pmatrix}a^{(j)}&0\\0&0\end{pmatrix}
\right](v+p)^3 (A^{(f-1-j)})^{-1}\right\}+O(p^{N-3-t})=0
\]
for all $0\leq t\leq 1,\, 0\leq j\leq f-1$.
The entries of the left-hand side give rise to the eight generators in row 5 of Tables \ref{Table1FV}--\ref{Table3FV}, where we denote $a^{(j)}$ by $a_1$, $a_2$, $a_3$ respectively ($j$ being omitted in the tables).

By equation~\eqref{eq:orientation} in Remark \ref{rmk:orient} we have
\[
(b^{(j)},c^{(j)})\equiv (s'_{\orient,f-j})^{-1}(\bm{\alpha}'_{(s(\tau),\mu(\tau)),j-f})\equiv s(\tau)_j^{-1}(\mu(\tau)+\eta)_j\equiv (ws^{-1}(\mu)-\nu)_j\pmod p,
\]
recalling that $(s(\tau),\mu(\tau))=(sw^{-1},\mu-sw^{-1}(\nu)-\eta)$. As $e'=p^{f'}-1$, we get $a^{(j)} \equiv -\langle(ws^{-1}(\mu)-\nu)_j,\alpha_j^\vee\rangle\pmod p$.
As $\mu_j = (r_j+2,1)$, this gives us the explicit formulas for $a^{(j)} \pmod p$ listed below Tables~\ref{Table1FV}--\ref{Table3FV}.

Let $R^{\leq(3,0),\tau,\nabla}_{\ovl{\fM},\ovl{\beta}}$ be the maximal reduced and $\cO$-flat quotient of
$R/ \sum_j(I^{(j),\leq (3,0)}+I^{(j),\nabla})$.
As in \cite[\S 5]{LLLM}, using that $\ad(\rhobar)$ is cyclotomic free we get
\begin{equation}
\label{eq:def:ring}
R^{\leq(3,0), \tau}_{\rhobar}\bbra{X_1,\dots,X_{2f}}\cong R^{\leq(3,0),\tau,\nabla}_{\ovl{\fM},\ovl{\beta}}\bbra{Y_1,\dots,Y_4}.
\end{equation}
(See in particular Thm.\ 5.12, Cor.\ 5.13, and Diagram (5.9) in \cite{LLLM}, noting that for us $n = 2$, so the addition of the gauge basis requires $2f$ instead
of $3f$ variables and the framing of the Galois deformation requires $2^2 = 4$ instead of $3^2 = 9$ variables. Note also that $T_8$ should be $T_9$ in \cite[Cor.\ 5.13]{LLLM},
cf.\ the errata in \cite[\S6]{LLLM2}.
Finally note that we allow deformations with \emph{any} Hodge--Tate weights $\le (3,0)$, so we do not have a restriction on the shape as in \cite[Cor.\ 5.13]{LLLM}.)

We now compute ``explicit'' generators of 
\[
I_\infty\defeq \ker\left(R\onto R^{\leq(3,0),\tau,\nabla}_{\ovl{\fM},\ovl{\beta}}\right)
\]
and show that $I_\infty=\sum_{j}I^{(j)}$, where the ideals $I^{(j)}$ of $R$ are given in row 6 of Tables \ref{Table1FV}--\ref{Table3FV}.
(Note that the $O(p^{N-8})$ tails in Tables \ref{Table1FV}--\ref{Table3FV} involve variables
of \emph{all} embeddings. In particular, the tails depend on $\tld w$ and not just on $\tld w_{f-1-j}$, and $I^{(j)}$ is not an ideal of $R^{(j)}$ in general!)

We first define a dense polynomial sub-$\cO$-algebra $R^{(j)}_{\mathrm{poly}}$ of $R^{(j)}$ for each $0\leq j\leq f-1$ by
\begin{align*}
  R^{(j)}_{\mathrm{poly}}\defeq \cO[c_{11},d_{11},x_{11}^*,\frac{c_{12}}{e_{11}^*},c_{21},\frac{d_{21}}{d_{22}^*},c_{22},x_{22}^*]&\qquad\text{if $\tilde w_{f-1-j} = t_{(2,1)}$}&\text{ (see Table \ref{Table1FV})},\\
  R^{(j)}_{\mathrm{poly}}\defeq \cO[c_{11},\frac{d_{11}}{d^*_{12}},c_{12},x^*_{12},c_{21},x_{21}^*,c_{22},\frac{d_{22}}{d_{21}^*}]&\qquad\text{if $\tilde w_{f-1-j} = \fW t_{(2,1)}$}&\text{ (see Table \ref{Table2FV})},\\
  R^{(j)}_{\mathrm{poly}}\defeq \cO[c_{11},x^*_{11},c_{12},\frac{d_{12}}{d_{11}^*},\frac{c_{21}}{e_{22}^*},c_{22},d_{22},x_{22}^*]&\qquad\text{if $\tilde w_{f-1-j} = t_{(1,2)}$}&\text{  (see Table \ref{Table3FV})}.
\end{align*}
Note in fact that the subspace topology on $R^{(j)}_{\mathrm{poly}}$ is the $\fm$-adic topology, where $\fm$ is the maximal ideal generated by all the polynomial variables above as well as $\varpi$. (Note that the polynomial variables above are power series generators of $R^{(j)}$.)
Let $R_{\mathrm{poly}}\defeq \bigotimes_{\cO,j}R_{\mathrm{poly}}^{(j)}$ and $I_{\mathrm{poly}}\defeq \sum_jI^{(j)}_{\mathrm{poly}}$, where $I^{(j)}_{\mathrm{poly}}$ is the ideal of $R_{\mathrm{poly}}^{(j)}$ generated by the elements in row 6 of Tables \ref{Table1FV}--\ref{Table3FV} \emph{without their $O(p^{N-8})$ tails}.

We first show that $I_{\mathrm{poly}}\subseteq (I_{\infty},\ p^{N-5})$.

In the following, we will focus on Table \ref{Table2FV} (the other cases being similar).
Let us label the elements on the right side of row 4 by $(H_i)$ ($1\leq i\leq 3$), of row 5 by $(M_i)$ ($1\leq i\leq 8$), and of row 6 \emph{without their $O(p^{N-8})$ tails} by $(G_i)$ ($1\leq i\leq 5$).
Then, omitting superscripts $(j)$ for simplicity,
\begin{align}
\label{eq:first:incl}
\frac{1}{p}\left[-(M_7)+\frac{1}{p}(M_8)\right]&=
d_{12}^*c_{21}+(a_2-2)(c_{12}d_{21}^*+d_{12}^*c_{21})+(d_{11}d_{22}+pd_{12}^\ast d_{21}^*)+O(p^{N-5})\\
&=-c_{12}d_{21}^*+(a_2-1)(c_{12}d_{21}^*+d_{12}^*c_{21})+(d_{11}d_{22}+pd_{12}^\ast d_{21}^*)+O(p^{N-5}),\nonumber
\end{align}
so replacing $c_{12}d_{21}^*+d_{12}^*c_{21}$ by $d_{11}d_{22}+pd_{12}^\ast d_{21}^*$ using $(H_1)$ we see that $(G_1),\ (G_2)\in (I_{\infty}, p^{N-5})$ (noting that the left-hand side of equation (\ref{eq:first:incl}) is in the $p$-saturation of the ideal $I^{(j),\leq (3,0)}+I^{(j),\nabla}$, so is in particular an element of $I_\infty$).
From $(M_3)$ and $(G_2)$ we get $(G_3)\in  (I_{\infty}, p^{N-5})$, as $a_2\not\equiv -1\pmod p$.

From $\frac{1}{p}[-(M_5)+\frac{1}{p}(M_6)]$ and $(G_1)$ we get $(G_4)\in  (I_{\infty}, p^{N-5})$, as $a_2\not\equiv 2\pmod p$.
Replacing $c_{12}$, $c_{21}$, $c_{11}$ in $\frac{1}{p}(M_8)$ by using the elements $(G_1), (G_2), (G_3)$ and as $a_2\not\equiv 0, -1\pmod p$ we get
\[
(d_{11}d_{22}+pd_{12}^*d_{21}^*)\left(d_{11}d_{22}+p\frac{(a_2-2)(a_2+1)}{a_2(a_2-1)}d_{12}^*d_{21}^*\right)+O(p^{N-5})\in I_\infty,
\]
hence $(G_5)\in (I_\infty,p^{N-5})$.
Thus we have $I_{\mathrm{poly}}\subseteq (I_{\infty},\ p^{N-5})$.

For any $0 \le j \le f-1$ we can then consider the commutative diagram of $\cO$-algebras
\begin{equation}\label{eq:elkik-diag}
  \begin{gathered}
    \xymatrix{
      R/(I_\infty,p^{N-5})&\ar_-{\;\phi^{(j)}}[l]R^{(j)}_{\mathrm{poly}}/I^{(j)}_{\mathrm{poly}}\\
      R/I_\infty\ar@{->>}[u]&\cO\ar[u]\ar[l] }
  \end{gathered}
\end{equation}
where $\phi^{(j)}$ is induced by the inclusions $R^{(j)}_{\mathrm{poly}}\into R^{(j)}\into R$.
Let $H^{(j)}$ be the ideal of the polynomial ring $R^{(j)}_{\mathrm{poly}}$ defined in \cite[\S 0.2]{elkik} (and denoted there by $\sum_{(\alpha),p} K_{(\alpha)} \Delta_{(\alpha)}$) for the finitely presented algebra $\cO\rightarrow R^{(j)}_{\mathrm{poly}}/I^{(j)}_{\mathrm{poly}}$.

\begin{lem}
\label{lem:Elkik:1}
We have $p^{3}\in H^{(j)}+I^{(j)}_{\mathrm{poly}}$.
\end{lem}
\begin{proof}
We give detail for the case $\tld{w}_{f-1-j}=\fW t_{(2,1)}$ (Table \ref{Table2FV}), the others being simpler.
To ease notation, we set $x\defeq \frac{d_{11}}{d^*_{12}}$, $y\defeq \frac{d_{22}}{d_{21}^*}$ and $a\defeq \frac{(a_2-2)(a_2+1)}{a_2(a_2-1)} \in \cO^\times$ so that $(G_5)=a^{-1}(xy+p)(xy+ap)$.
It follows directly from the definitions that $H^{(j)}$ contains the $5 \times 5$ minors of the Jacobian matrix of $I^{(j)}_{\mathrm{poly}} = ((G_1),\dots,(G_5))$ (i.e.~the ideal $M_{(\alpha)}$ with $\alpha=(1,\dots,q=5)$ in the notation of \cite[\S 0.2]{elkik}), in particular, by direct inspection, contains the element $\frac{\partial}{\partial x}(G_5)=a^{-1}(2xy^2+p(a+1)y)$.
Thus, the ideal $H^{(j)}+I^{(j)}_{\mathrm{poly}}$ contains 
\begin{align*}
\Big(2(a+1)xy+p(a-1)^2\Big)(G_5)-\Big((a+1)xy+p(a^2+1)\Big)x\frac{\partial}{\partial x}(G_5)&=p^3(a-1)^2.
\end{align*}
As $a-1=-\frac{2}{a_2(a_2-1)}$ and $a_2-1 \equiv \pm(r_j+1) \pmod p$, we conclude that $(a-1)^2 \in \cO^\times$ and hence $p^3\in H^{(j)}+I^{(j)}_{\mathrm{poly}}$.
The cases where $\tilde w_{f-1-j}\in\{ t_{(1,2)},t_{(2,1)}\}$ are similar, giving actually $p^2\in H^{(j)}+I^{(j)}_{\mathrm{poly}}$.
\end{proof}

We apply Elkik's lemma analogously to \cite[Prop.~3.3.9]{MLM}. Let $A \defeq R/I_\infty\cong R^{\leq(3,0),\tau,\nabla}_{\ovl{\fM},\ovl{\beta}}$ which by definition is $p$-torsion free and $p$-adically complete. Write $R^{(j)}_{\mathrm{poly}} = \cO[X_1,\dots,X_8]$ (relabeling the generators above).
Let $H^{(j)}_B\subseteq A[X_1,\dots,X_8]$ denote the Elkik ideal for the finitely presented algebra $A \to B \defeq A \otimes_{\cO} R^{(j)}_{\mathrm{poly}}/I^{(j)}_{\mathrm{poly}}\cong A[X_1,\dots,X_8]/I^{(j)}_{\mathrm{poly}}$, so that $H^{(j)}_B$ contains the image of $H^{(j)}$ via $R^{(j)}_{\mathrm{poly}} \ra A[X_1,\dots,X_8]$.  Diagram~\eqref{eq:elkik-diag} gives a surjection $B\twoheadrightarrow A/(p^{N-5})$ of $A$-algebras and taking lifts in $A$ of the images of the $X_i$ gives $\un a = (a_1,\dots,a_8)\in A^8$ such that $I^{(j)}_{\mathrm{poly}}(\un a) \subset p^{N-5}A$. By
Lemma~\ref{lem:Elkik:1} we get $p^3 \in H_B(\un a)+I^{(j)}_{\mathrm{poly}}(\un a) \subset H_B(\un a)+p^{N-5}A$, so $p^3 \in H_B(\un
a)$. As $N-5 > 2\times 3$ we may apply \cite[Lemme 1]{elkik} (with $I =A$, $t = p$, $k = 0$, $n=N-5$ and $h=3$ in the notation of the reference) to find $\wt{\un a} \in A^8$ that is congruent to $\un a$ modulo $p^{N-8}$ and such that $I^{(j)}_{\mathrm{poly}}(\wt{\un a})=0$. In other words, we
deduce the existence of an $\cO$-algebra homomorphism $\wt{\phi}^{(j)}:R^{(j)}_{\mathrm{poly}}/I^{(j)}_{\mathrm{poly}}\rightarrow R/I_\infty$ such that $\wt{\phi}^{(j)}$ agrees with $\phi^{(j)}$ (i.e.\ the natural map) modulo $p^{N-8}$.
By taking a tensor product of the $\wt{\phi}^{(j)}$ for $0 \le j \le f-1$ we get an $\cO$-algebra homomorphism $\wt{\phi}:R_{\mathrm{poly}}/I_{\mathrm{poly}}\rightarrow R/I_\infty$ such that $\wt{\phi}$ agrees with the natural map modulo $p^{N-8}$.
Since $N>8$, $\wt{\phi}$ is continuous and hence induces $\wt{\phi}:R/I_{\mathrm{poly}}\rightarrow R/I_\infty$ that agrees with the natural map modulo $p^{N-8}$. 
As $N\geq 10$, the map $\wt{\phi}:R/I_{\mathrm{poly}}\rightarrow R/I_\infty$ has to be surjective.

By Lemma \ref{lem:irreducible} and the explicit description of $I_{\mathrm{poly}}^{(j)}=(G_1,\dots,G_5)$ (cf.\ row 6 of Tables \ref{Table1FV}--\ref{Table3FV}), one easily gets that $R^{(j)}/I_{\mathrm{poly}}^{(j)}$ is reduced, $\cO$-flat, with two irreducible components that are geometrically integral and of relative dimension 3 over $\cO$.
By \cite[Lemma 2.6]{calegariCrelle} and \cite[Lemma 3.3]{BLGHT}, $R/I_{\mathrm{poly}}=\widehat{\bigotimes}_{\cO,j}R^{(j)}/I^{(j)}_{\mathrm{poly}}$ is reduced, $\cO$-flat with $2^f$ irreducible components, each of relative dimension $3f$ over $\cO$.
Hence the surjection
\begin{equation}
\wt{\phi}:R/I_{\mathrm{poly}}\onto R/I_\infty \cong R^{\leq(3,0),\tau,\nabla}_{\ovl{\fM},\ovl{\beta}}\label{eq:7}
\end{equation}
is an isomorphism, provided that $R^{\leq(3,0),\tau,\nabla}_{\ovl{\fM},\ovl{\beta}}$, or equivalently $R^{\leq(3,0),\tau}_{\rhobar}$ by~(\ref{eq:def:ring}), has at least $2^f$ irreducible components.
To see this, it suffices to show that for any choice of $\lambda\in\{(3,0), (2,1)\}^f$, $\rhobar$ admits a potentially crystalline lift $\rho$ of type $\tau$ with $\mathrm{HT}_j(\rho)=\lambda_j$ for all $j$.
This in turn follows from \cite[Thm.\ D]{GHLS}, provided 
\begin{equation}
\label{eq:int}
\JH(\ovl{\sigma(\tau)\otimes_E \bigotimes_{E,j} V_E(\lambda_j-(1,0))^{(j)}})\cap W(\rhobar)\neq 0.
\end{equation}

The left-hand side contains $\JH(\ovl{\sigma(\tau)\otimes_E\bigotimes_{E,j} V_E((1,1))^{(j)}})\cap W(\rhobar)$ using $L(a,b)\otimes_{\F} L(2,0)\cong L(a+2,b)\oplus L(a+1,b+1)\oplus L(a,b+2)$ if $2\leq a-b\leq p-3$ when $\lambda_j=(3,0)$.
(Note that the highest weights of the elements of $\JH(\ovl{\sigma(\tau)})$ are $7$-deep, as follows from Proposition~\ref{prop:JH:graph} and Remark~\ref{rk:t_lambda}\ref{it:t_lambda:4}.)
Hence \eqref{eq:int} follows from Lemma \ref{lem:inter}.

As~(\ref{eq:7}) is an isomorphism and induces the natural map modulo $p^{N-8}$, we conclude that $(I_{\mathrm{poly}},p^{N-8})=(I_\infty,p^{N-8})$.
\begin{lem}\label{lem:automorphism}
There exists an automorphism of local $\cO$-algebras $\psi:R\congto R$ such that
\[
\xymatrix{
R\ar^-{\psi}_-{\sim}[r]\ar@{->>}[d]&R\ar@{->>}[d]\\
R/I_{\mathrm{poly}}\ar^-{\wt{\phi}}_-{\sim}[r]&R/I_\infty
}
\]
commutes and such that $\psi$ induces the identity modulo $p^{N-8}$.
\end{lem}
\begin{proof}
Let us write $R=\cO\bbra{X_1,\dots,X_k}$.
As $\wt{\phi}$ induces the identity modulo $p^{N-8}$ we see that for each $x\in R$ there exists $\eps(x)\in R$ such that $\wt{\phi}(x+I_{\mathrm{poly}})=x+p^{N-8}\eps(x)+I_\infty$.
Define $\psi$ by demanding that $\psi(X_i)=X_i+p^{N-8}\eps(X_i)$ for all $1\leq i\leq k$.
As $N\geq 10$ (in fact, even $N\ge 9$ suffices) it follows that $\psi$ is an automorphism of complete noetherian local $\cO$-algebras, and the lemma follows.
\end{proof}
In particular, $\psi$ identifies $I_{\mathrm{poly}}$ with $I_\infty$.
Thus $I_\infty=\sum_j I^{(j)}$, where $I^{(j)}$ is the ideal of $R$ given by the explicit generators in Tables \ref{Table1FV}--\ref{Table3FV} (by applying $\psi$ to the generators of $I_{\mathrm{poly}}$).
Moreover it follows that the ideals $\fp^{\lambda}\defeq \sum_j\fp^{(j),\lambda_{f-1-j}}$ of $R$ for $\lambda\in\{(2,1),\,(3,0)\}^f$, where the $\fp^{(j),\lambda_{f-1-j}}$ are defined in Tables \ref{Table1FV}--\ref{Table3FV}, are the distinct minimal primes containing $I_\infty$.

By the above argument that~(\ref{eq:7}) is an isomorphism, we know that the irreducible components of
$\Spec R^{\leq(3,0), \tau_{\tld{w}}}_{\rhobar}$ are in bijection with the set $\{(3,0), (2,1)\}^f$, explicitly given by
sending a component $\cC$ to the labeled Hodge--Tate weights of the framed deformation corresponding to any closed point of
the generic fiber of $\cC$. So the components are indeed given by the $\Spec R^{\lambda, \tau_{\tld{w}}}_{\rhobar}$, where
$\lambda = (\lambda_j) \in \{(3,0), (2,1)\}^f$.

It remains to establish the final claim identifying irreducible components. 
For any $\lambda = (\lambda_j) \in \{(3,0), (2,1)\}^f$ consider the kernel of the composition
\[
\phi_\lambda: R\onto R/I_\infty \cong R^{\leq(3,0),\tau,\nabla}_{\ovl{\fM},\ovl{\beta}}\onto R^{\leq\lambda,\tau,\nabla}_{\ovl{\fM},\ovl{\beta}}.
\]
By above we know that $\ker(\phi_\lambda)$ is of the form $\bigcap_{\lambda' \in X} \fp^{\lambda'}$ for some subset $X$ of $\{(3,0), (2,1)\}^f$ of cardinality $2^k$, 
where $k \defeq \#\{j : \lambda_j = (3,0)\}$. For the identification of components it suffices, by induction on $\lambda$, to show that
$\lambda_j = (2,1)$ implies that $\lambda'_j = (2,1)$ for all $\lambda' \in X$.
If this is false, then there exists $0 \le j \le f-1$ and $\lambda' \in X$ such that $\lambda_{f-1-j} = (2,1)$ and $\lambda'_{f-1-j} = (3,0)$.
By the same argument as above for row 4 of Tables \ref{Table1FV}--\ref{Table3FV} (as $\lambda_{f-1-j}=(2,1)$, the finite height conditions imply that each entry of $A^{(f-1-j)}$ is divisible by $v+p$), we deduce that $c^{(j)}_{ik}\in\ker(\phi_\lambda) \subset \fp^{\lambda'}$  for all $1\leq i,k\leq 2$ and moreover $d_{11}^{(j)}d_{22}^{(j)}+pd_{12}^{\ast(j)}d_{21}^{\ast(j)}\in\ker(\phi_\lambda)\subset \fp^{\lambda'}$ in case of Table \ref{Table2FV} (using row 4).
From the additional assumption that $\lambda'_{f-1-j} = (3,0)$ it is now easy to see, using row 8 of Tables \ref{Table1FV}--\ref{Table3FV}, that $p\in \fp^{\lambda'}$, which is a contradiction.
(In the notation of Remark~\ref{rem:lowest-hodge-type} we have $p \in \fq^{(j),(2,1)} + \fp^{(j),(3,0)} \subset \fp^{\lambda'}$, where $\fq^{(j),(2,1)}$
denotes the ideal defined there.)
\end{proof}

\begin{rem}\label{rem:lowest-hodge-type}
  Suppose that $\lambda \in \{(3,0), (2,1)\}^f$ is such that $\lambda_{f-1-j} = (2,1)$ and let $\fp^\lambda \defeq \sum_{j'} \fp^{(j'),\lambda_{f-1-j'}}$ (an ideal of $R$).
  As observed at the end of the proof of Proposition~\ref{prop:def:ring}, we see that $c_{ik}^{(j)} \in \fp^\lambda$.
  Using row 4 of Tables~\ref{Table1FV}--\ref{Table3FV} we can even say that
  \begin{align*}
    (c_{11},c_{12},c_{21},c_{22},d_{11}) \subset \fp^\lambda &\qquad\text{if $\tilde w_{f-1-j} = t_{(2,1)}$},\\
    \bigg(c_{11},c_{12},c_{21},c_{22},\frac{d_{11}d_{22}}{d_{12}^{\ast}d_{21}^{\ast}}+p\bigg) \subset \fp^\lambda &\qquad\text{if $\tilde w_{f-1-j} = \fW t_{(2,1)}$},\\
    (c_{11},c_{12},c_{21},c_{22},d_{22}) \subset \fp^\lambda &\qquad\text{if $\tilde w_{f-1-j} = t_{(1,2)}$},
  \end{align*}
  as ideals of $R$, where we omit the superscripts $(j)$ for readability.
  Moreover, the sum of the ideals on the left equals $\fp^\lambda$ if $\lambda_{f-1-j} = (2,1)$ for all $j$ (by dimension reasons or since
  the monodromy condition is vacuous in this case).  
\end{rem}

\begin{cor}\label{cor:special-fibre-def-ring}
  For each $\lambda = (\lambda_j) \in \{(3,0), (2,1)\}^f$ and $\tld{w}\in \Adm^\vee(t_{(\un 2,\un 1)})$ the special fibre of $\Spec R^{\lambda, \tau_{\tld{w}}}_{\rhobar}$
  is reduced and all its irreducible components are formally smooth over $\F$.
\end{cor}

\begin{proof}
  Referring back to the proof of Proposition~\ref{prop:def:ring} as well as Lemma~\ref{lem:automorphism} we have an isomorphism $R^{\leq(3,0),\tau,\nabla}_{\ovl{\fM},\ovl{\beta}} \cong R/I_\mathrm{poly}$ and
  \begin{equation}\label{eq:14a}
    R^{\lambda, \tau_{\tld{w}}}_{\rhobar}\bbra{X_1,\dots,X_{2f}}\cong\left(\widehat{\bigotimes}_{\cO, 0\leq j\leq f-1}R^{(j)}/\fp^{(j),\lambda_{f-1-j}}_\mathrm{poly}\right)\bbra{Y_1,\dots,Y_4},
  \end{equation}
  where $\fp^{(j),\lambda_{f-1-j}}_\mathrm{poly}$ is the ideal of $R^{(j)}_\mathrm{poly}$ generated by the elements of rows 7 and 8 in Tables~\ref{Table1FV}--\ref{Table3FV}
  \emph{without their $O(p^{N-8})$ tails}.

  From~\eqref{eq:14a} and right exactness of completed tensor products we obtain
  \begin{equation}\label{eq:14b}
    \begin{aligned}
      (R^{\lambda, \tau_{\tld{w}}}_{\rhobar}/\varpi)\bbra{X_1,\dots,X_{2f}} %
      &\cong\left(\left(\widehat{\bigotimes}_{\cO, 0\leq j\leq f-1}R^{(j)}/\fp^{(j),\lambda_{f-1-j}}_\mathrm{poly}\right)\Big/(\varpi)\right)\bbra{Y_1,\dots,Y_4}
      \\
      &\cong
      \left(\widehat{\bigotimes}_{\F, 0\leq j\leq f-1}R^{(j)}/(\varpi,\fp^{(j),\lambda_{f-1-j}}_\mathrm{poly})\right)\bbra{Y_1,\dots,Y_4}.
    \end{aligned}
  \end{equation}
  By Tables~\ref{Table1FV}--\ref{Table3FV} we see that $R^{(j)}/(\varpi,\fp^{(j),\lambda_{f-1-j}}_\mathrm{poly}) \cong
  \F\bbra{Z_1,\dots,Z_{3+m}}/(Z_1Z_2,\dots,Z_{2m-1}Z_{2m})$ for some $m \le 1$.
  It follows from~\eqref{eq:14b} and Lemma~\ref{lm:hamann} that 
  \[ R^{\lambda, \tau_{\tld{w}}}_{\rhobar}/\varpi \cong \F\bbra{U_1,\dots,U_{f+4+m}}/(U_1U_2,\dots,U_{2m-1}U_{2m})\]
  for some $m \le f$.
\end{proof}

\subsection{Deformation rings II: multiple types}
\label{sec:deformation-rings2}

Inspired by the techniques of \cite[\S 3.2]{DanWild} we now compute some multi-type deformation rings. 

We suppose that $\rhobar$ is as in \S\ref{sec:setup}.
For $\sigma \in W(\rhobar)$ let $R^{\leq(3,0),\sigma}_{\rhobar}$ denote the maximal reduced, $\cO$-flat quotient of $R^{\square}_{\rhobar}$
that parametrizes lifts of $\rhobar$ of Hodge--Tate weights $\le (3,0)$ in each embedding and tame inertial type $\tau$ for some $\tau$ such
that $\sigma \in \JH\left(\ovl{\sigma(\tau)}\otimes_\F N_{k/\Fp}\circ \det\right)$.
Letting $\tld w_\sigma \defeq \theta(\sigma)$ via the bijection $\theta$ of
Lemma~\ref{lem:inter} and
$$X(\sigma) \defeq \{ \tld w \in \Adm^\vee(t_{(\un 2,\un 1)}) : \tld{w}_j\neq (\tld w_\sigma)_j \ \forall\, j\},$$ we see that
$\Spec R^{\leq(3,0),\sigma}_{\rhobar}$ is the flat closure of
$\bigcup_{\tld w \in X(\sigma)} \Spec R^{\leq(3,0),\tau_{\tld w}}_{\rhobar}[1/p]$ inside $\Spec R^{\square}_{\rhobar}$.
Also, define a bijection $i : \Adm^\vee(t_{(\un{2},\un{1})}) \to \{1,2,3\}^f$ by letting $i(\tld w)$ be the $f$-tuple given by
\[
i(\tld{w})_j \defeq \begin{cases}
1&\text{if }\tld{w}_j=t_{(2,1)}\\
2&\text{if }\tld{w}_j=\fW t_{(2,1)}\\
3&\text{if }\tld{w}_j=t_{(1,2)}
\end{cases}
\]
for all $0\leq j\leq f-1$.

\begin{prop}\label{prop:multitype-def-ring}
We have an isomorphism
\[
R^{\leq(3,0),\sigma}_{\rhobar}\bbra{X_1,\dots,X_{2f}}\cong\bigg(
S/\bigcap_{\tld w \in X(\sigma)} \sum_j I_{\tld w}^{(j)}
\bigg)\bbra{Y_1,\dots,Y_4},
\]
where $S \defeq \widehat{\bigotimes}_{\cO, 0\leq j\leq f-1}S^{(j)}$
and the $\cO$-algebras $S^{(j)}$ and the ideals $I_{\tld{w}}^{(j)}$ of $S$ are as in Table \ref{Table4FV} if
$(\tld w_\sigma)_{f-1-j} = t_{(1,2)}$, whereas $S^{(j)}$ and the ideals $I_{\tld{w}}^{(j)}$ of $S$ are 
as in Table \ref{Table4} if $(\tld w_\sigma)_{f-1-j} = t_{(2,1)}$.
The irreducible components of $\Spec R^{\leq(3,0), \sigma}_{\rhobar}$ are given by the $\Spec R^{\lambda, \tau_{\tld{w}}}_{\rhobar}$,
where $\lambda = (\lambda_j) \in \{(3,0), (2,1)\}^f$ and $\tld w \in X(\sigma)$.

More precisely, via the above isomorphism, for any choice of $\lambda = (\lambda_j) \in \{(3,0), (2,1)\}^f$ and $\tld w \in X(\sigma)$ the kernel of the natural surjection
$R^{\leq(3,0),\sigma}_{\rhobar}\bbra{X_1,\dots,X_{2f}} \onto R^{\lambda, \tau_{\tld{w}}}_{\rhobar}\bbra{X_1,\dots,X_{2f}}$ is
generated by the prime ideal $\sum_{j=0}^{f-1} \fp_{\tld w}^{(j),\lambda_{f-1-j}}$ of $S/\bigcap_{\tld w \in X(\sigma)} \sum_j I_{\tld w}^{(j)}$, where the ideals $\fp_{\tld w}^{(j),\lambda_{f-1-j}}$ of $S/\bigcap_{\tld w \in X(\sigma)} \sum_j I_{\tld w}^{(j)}$ are found in Tables \ref{Table4FV}--\ref{Table4}.
\end{prop}

\begin{proof}
Recall that $\rhobar|_{I_K}\cong \taubar(s,\mu)$. The proof of Lemma~\ref{lem:ss:Kisin} shows that the \'etale $\phz$-module
associated to $\rhobar|_{G_{K_\infty}}$ is given by $\Mat(\phz^{(j)})=(Ds^*t_{\mu^*})_j$ in some basis, for some $D = (D_j)\in
\un{T}(\F)$. Define $\delta_{12}^{(j)}$, $\delta_{21}^{(j)} \in \cO^\times$ to be the Teichm\"uller lifts of the diagonal entries
of $D_{f-1-j}$. Also let $\mu'_j\defeq \mu_j-(1,1) = (r_j+1,0)$.

Let $\oS\defeq S/\bigcap_{\tld w \in X(\sigma)} \sum_j I_{\tld w}^{(j)}$.
Consider the \'etale $\phz$-module $\cM$ over $\cO_{\mathcal{E},\oS}$ given by
\[
\Mat(\phz_{\cM}^{(f-1-j)})=
\begin{pmatrix}
(v+p)(\delta_{12}^{(j)}+x_{12}^{\ast(j)})+c_{12}^{(j)}+\frac{b_{12}^{(j)}}{v}&\frac{1}{v}\big((v+p)d_{11}^{(j)}+c_{11}^{(j)}\big)\\
(v+p)d_{22}^{(j)}+c_{22}^{(j)}&
(v+p)(\delta_{21}^{(j)}+x_{21}^{\ast(j)})+c_{21}^{(j)}+\frac{b_{21}^{(j)}}{v}
\end{pmatrix}s_{j}^{-1}v^{\mu_j'} %
\]
in a suitable basis, where $b_{21}^{(j)}\defeq 0$ if $(\tld{w}_\sigma)_{f-1-j}=t_{(1,2)}$ and $b_{12}^{(j)}\defeq 0$ if
$(\tld w_\sigma)_{f-1-j} = t_{(2,1)}$.
Write $\oS\bbra{\un{Y}} \defeq \oS\bbra{Y_1,\dots,Y_4}$ for short and
define the $\phz$-module $\cM_{\oS\bbra{\un{Y}}}\defeq \cM\widehat{\otimes}_\oS \oS\bbra{\un{Y}}$ over $\cO_{\mathcal{E},\oS\bbra{\un{Y}}}$. %

Let $\cM_\F\defeq \cM\otimes_\oS\F$. As every variable in $S^{(j)}$ gets sent to zero in $\F$ and $\mu_j = (r_j+2,1)$, we see that $\bV^*_K(\cM_\F)\cong \rhobar|_{G_{K_\infty}}$.
Fix an $\F$-basis $\gamma_\F$ of $\bV^*_K(\cM_\F)\cong \rhobar|_{G_{K_\infty}}$.
If $\rhobar$ is reducible, we demand moreover that $\gamma_{\F,1}$, $\gamma_{\F,2}$ each span $G_{K_\infty}$-stable lines.

Fix an $\oS$-basis $\gamma$ of $\bV^*_K(\cM)$ that lifts $\gamma_\F$.
Then the $G_{K_\infty}$-representation $\bV^*_K(\cM_{\oS\bbra{\un{Y}}})$ together with the basis $\big(1+\begin{pmatrix}Y_1&Y_2\\Y_3&Y_4\end{pmatrix}\big)\big(\gamma\otimes1\big)$ gives rise to a homomorphism $\psi_0: R^{\Box}_{\rhobar|_{G_{K_\infty}}}\rightarrow \oS\bbra{\un{Y}}$.

For notational convenience, rename the variables $(X_1,\dots,X_f)$ as $\un{X}' \defeq (X_0',\dots,X_{f-1}')$ and $(X_{f+1},\dots,X_{2f})$ as $\un{X}'' \defeq (X_0'',\dots,X_{f-1}'')$. 
Extend $\psi_0$ to a homomorphism $\psi: R^{\Box}_{\rhobar|_{G_{K_\infty}}}\bbra{\un{X}',\un{X}''}\rightarrow \oS\bbra{\un{Y}}$ as follows:
\begin{align*}
\psi(X'_j)&=
\begin{cases}x_{12}^{\ast (j)}&\text{if $0\leq j<f-1$ or $\rhobar$ is irreducible;}\\
Y_1&\text{if $j=f-1$ and $\rhobar$ is reducible;}
\end{cases}
\\
\psi(X''_j)&=\begin{cases}x_{21}^{\ast (j)}&\text{if $0\leq j<f-1$;}\\
Y_4&\text{if $j=f-1$.}\end{cases}
\end{align*}

\paragraph{\textit{Claim 1.}} The map $\psi: R^{\Box}_{\rhobar|_{G_{K_\infty}}}\bbra{\un{X}',\un{X}''}\rightarrow \oS\bbra{\un{Y}}$ is surjective.

We will check it is injective on reduced tangent vectors, i.e.\ on $\F[\eps]/(\eps^2)$-points.
Pick any continuous homomorphism $t: \oS\bbra{\un{Y}}\rightarrow \F[\eps]/(\eps^2)$, let $t_0:\oS\bbra{\un{Y}}\rightarrow \F\rightarrow \F[\eps]/(\eps^2)$ be the zero vector, and suppose that $t\circ\psi=t_0\circ\psi$.
Abusing notation, we will write $t(b_{ik}^{(j)})=\eps b_{ik}^{(j)}$ for some $b_{ik}^{(j)}\in \F$ on the right, and similarly $t(c_{ik}^{(j)})=\eps c_{ik}^{(j)}$, $t(d_{ik}^{(j)})=\eps d_{ik}^{(j)}$, $t(x_{ik}^{\ast (j)})=\eps x_{ik}^{(j)}$, $t(Y_{i})=\eps y_{i}$.
From the definition of $\psi$ (and $t\circ\psi=t_0\circ\psi$) we deduce $x_{12}^{(j)}=x_{21}^{(j)}=0$ for $0\leq j<f-1$, $y_4=0$, and 
\begin{equation}
\begin{cases}
x_{12}^{(f-1)}=0&\text{if $\rhobar$ is irreducible},\\
y_1=0&\text{if $\rhobar$ is reducible.}
\end{cases}\label{eq:3}
\end{equation}
Also, since {\it a fortiori} $t\circ\psi_0=t_0\circ\psi_0$, we see that there is an isomorphism
\begin{equation}
\lambda:\cM_{\oS\bbra{\un{Y}}}\widehat{\otimes}_{\oS\bbra{\un{Y}},t}\F[\eps]/(\eps^2)\congto \cM_{\oS\bbra{\un{Y}}}\widehat{\otimes}_{\oS\bbra{\un{Y}},t_0}\F[\eps]/(\eps^2)\label{eq:4}
\end{equation}
such that $\bV^*_{K}(\lambda)$ sends the basis $(1+\eps\begin{pmatrix}y_1&y_2\\ y_3&y_4\end{pmatrix})(\gamma\otimes 1)$ to $\gamma\otimes 1$.
In particular $\lambda\mod \eps$ is the identity of $\cM_\F$.

Hence, as $\cO_{\mathcal{E},\F[\eps]/(\eps^2)} \cong \prod_{j=0}^{f-1} \F\ppar v[\eps]/(\eps^2)$, the isomorphism $\lambda$ is realized by change of basis matrices of the form
\[
1+\eps M_{f-1-j}\in\GL_2(\F\ppar v[\eps]/(\eps^2)),
\]
for some $M_{f-1-j}\in\M_2(\F\ppar{v})$. In other words,
\begin{equation}
\label{eq:chg:basis}
\begin{split}
&(1+\eps M_{j-1})\begin{pmatrix}\delta_{12}^{(j)}&\\&\delta_{21}^{(j)}\end{pmatrix}s_j^{-1}v^{\mu'_j}(1-\eps \phz(M_{j}))=\\
&\qquad=
\begin{pmatrix}
\delta_{12}^{(j)}+\eps(x_{12}^{(j)}+c_{12}^{(j)}v^{-1}+b_{12}^{(j)}v^{-2})&
\eps(d_{11}^{(j)}v^{-1}+c_{11}^{(j)}v^{-2})\\
\eps(d_{22}^{(j)}+c_{22}^{(j)}v^{-1})&
\delta_{21}^{(j)}+\eps(x_{21}^{(j)}+c_{21}^{(j)}v^{-1}+b_{21}^{(j)}v^{-2})
\end{pmatrix}s_j^{-1}v^{\mu_j'},
\end{split}
\end{equation}
where we have divided by $v$, and $j$ is considered in $\ZZ/f\ZZ$, as usual.

Let $k_j\in \ZZ$ be minimal such that $v^{k_j} M_j\in \M_2(\F\bbra{v})$. 
Consider
\begin{align*}
&1-\eps \phz(M_{j})=v^{-\mu'_j}s_j\begin{pmatrix}\delta_{12}^{(j)}&\\&\delta_{21}^{(j)}\end{pmatrix}^{-1}
(1-\eps M_{j-1}) \cdot \\ %
&\qquad \cdot \begin{pmatrix}
\delta_{12}^{(j)}+\eps(x_{12}^{(j)}+c_{12}^{(j)}v^{-1}+b_{12}^{(j)}v^{-2})&
\eps(d_{11}^{(j)}v^{-1}+c_{11}^{(j)}v^{-2})\\
\eps(d_{22}^{(j)}+c_{22}^{(j)}v^{-1})&
\delta_{21}^{(j)}+\eps(x_{21}^{(j)}+c_{21}^{(j)}v^{-1}+b_{21}^{(j)}v^{-2})
\end{pmatrix}s_j^{-1}v^{\mu_j'}.
\end{align*}
Then multiplying the right-hand side by $v^{r_j+1}\cdot v^{k_{j-1}}\cdot v^2$ makes it $v$-integral, hence $pk_j\leq k_{j-1}+r_j+3<k_{j-1}+p-1$ by genericity.
This implies $p\max_j k_j<\max_j k_j + p-1$, so $\max_j k_j<1$, meaning $M_{j}\in \M_2(\F\bbra{v})$ for all $j$.

From (\ref{eq:chg:basis}) we get by multiplying on the right by $v^{-\mu_j'}s_j$:
\begin{equation}
\label{eq:M:int}
\begin{aligned}
&M_{j-1}\begin{pmatrix}\delta_{12}^{(j)}&\\&\delta_{21}^{(j)}\end{pmatrix} - \begin{pmatrix}\delta_{12}^{(j)}&\\&\delta_{21}^{(j)}\end{pmatrix}s_j^{-1}v^{\mu'_j}\phz(M_{j})v^{-\mu_j'}s_j=\\
&\qquad=
\begin{pmatrix}
x_{12}^{(j)}+c_{12}^{(j)}v^{-1}+b_{12}^{(j)}v^{-2}&
d_{11}^{(j)}v^{-1}+c_{11}^{(j)}v^{-2}\\
d_{22}^{(j)}+c_{22}^{(j)}v^{-1}&
x_{21}^{(j)}+c_{21}^{(j)}v^{-1}+b_{21}^{(j)}v^{-2}
\end{pmatrix}.
\end{aligned}
\end{equation}
Recall that we assumed $s_j=1$ for all $0<j\leq f-1$ and $s_0=1$ if and only if $\rhobar$ is reducible (see the beginning of \S\ref{sec:setup}).

As the $(1,1)$ and $(2,2)$-entries of the left-hand side of (\ref{eq:M:int}) are $v$-integral,
we deduce that $c_{12}^{(j)}=b_{12}^{(j)}=c_{21}^{(j)}=b_{21}^{(j)}=0.$
From the $(2,1)$-entry of (\ref{eq:M:int}) when $s_j=1$ (resp.~the $(1,2)$-entry of (\ref{eq:M:int}) when $s_j\neq 1$) and from $2 < r_j+1 < p$ we deduce that $v\mid (M_j)_{21}$
for all $j$. This implies that the left-hand side of~(\ref{eq:M:int}) is $v$-integral and its $(2,1)$-entry is divisible by $v$. In particular,
$d_{11}^{(j)}=c_{11}^{(j)}=d_{22}^{(j)}=c_{22}^{(j)}=0$ for all $j$.

If $s_j=1$ (e.g.~if $j\neq 0$) we have by~(\ref{eq:M:int}) and the previous paragraph
\begin{equation}
\begin{cases}
x_{12}^{(j)}=\delta_{12}^{(j)}\big((M_{j-1})_{11}-(M_j)_{11}\big)|_{v=0},\\
x_{21}^{(j)}=\delta_{21}^{(j)}\big((M_{j-1})_{22}-(M_j)_{22}\big)|_{v=0}.
\end{cases}\label{eq:2}
\end{equation}
If $s_{j}\neq 1$ then we have by~(\ref{eq:M:int}) and the previous paragraph
\begin{equation}
\begin{cases}
x_{12}^{(j)}=\delta_{12}^{(j)}\big((M_{j-1})_{11}-(M_j)_{22}\big)|_{v=0}, \\
x_{21}^{(j)}=\delta_{21}^{(j)}\big((M_{j-1})_{22}-(M_j)_{11}\big)|_{v=0}.
\end{cases}\label{eq:1}
\end{equation}

Recall that $x_{12}^{(j)}=x_{21}^{(j)}=0$ for all $0 \le j < f-1$.
If $\rhobar$ is reducible (i.e.~$s_0=1$) we deduce by~(\ref{eq:2}) %
that $\sum_j (\delta_{12}^{(j)})^{-1} x_{12}^{(j)}=\sum_j (\delta_{21}^{(j)})^{-1} x_{21}^{(j)}=0$ and hence that $x_{12}^{(j)}=x_{21}^{(j)}=0$ for all $j$.
Otherwise (i.e.~$s_0\neq 1$), we deduce from~(\ref{eq:1}) that $\sum_j \big((\delta_{12}^{(j)})^{-1} x_{12}^{(j)} + (\delta_{21}^{(j)})^{-1} x_{21}^{(j)}\big)=0$ and hence by~(\ref{eq:3}) that $x_{12}^{(j)}=x_{21}^{(j)}=0$ for all $j$.
As a result, the right-hand side of (\ref{eq:M:int}) vanishes and we conclude that $(M_{f-1-j})_j\in \End_{\phz\text{-mod}}(\cM_{\F})$.
Denote this endomorphism by $\xi$. From the properties of $\lambda$ (see~(\ref{eq:4}) and the line after) we have $(1+\eps \bV^*_K(\xi)) (1+\eps\begin{pmatrix}y_1&y_2\\ y_3&y_4\end{pmatrix})(\gamma\otimes 1) = \gamma \otimes 1$, so $\bV^*_K(\xi)=-\begin{pmatrix}y_1&y_2\\y_3&y_4\end{pmatrix}$ with respect to the basis $\gamma_\F$.
On the other hand, $\End_{\phz\text{-mod}}(\cM_{\F})\cong \End_{G_{K_\infty}}(\rhobar|_{K_\infty}) \cong \End_{G_K}(\rhobar)$ by Lemma~\ref{lem:ff}.

If $\rhobar$ is (absolutely) irreducible, then $\End_{\phz\text{-mod}}(\cM_{\F})=\F$.
As $y_4=0$ we conclude from the formula for $\bV^*_K(\xi)$ that $y_i=0$ for all $i$.
 
If $\rhobar$ is reducible, then $\End_{\phz\text{-mod}}(\cM_{\F}) \cong \F\times \F$.
By our condition that $\gamma_{\F,1}$, $\gamma_{\F,2}$ each span $G_{K_\infty}$-stable lines, we conclude that $y_2=y_3=0$.
Using~(\ref{eq:3}) we also have $y_1=y_4=0$.

We have shown that $t = t_0$, completing the proof of Claim 1.

We consider now the surjections
\[
R^{\Box}_{\rhobar|_{G_{K_\infty}}}\onto R^{\Box}_{\rhobar}\onto R^{\leq (3,0),\sigma}_{\rhobar}.
\]
(For the first, see \cite[Prop.\ 3.12]{LLLM} and use that $\ad(\rhobar)$ is cyclotomic free.)
\paragraph{\textit{Claim 2.}} The map $\psi_0: R^{\Box}_{\rhobar|_{G_{K_\infty}}}\rightarrow \oS\bbra{\un{Y}}$ factors through the surjection $R^{\Box}_{\rhobar|_{G_{K_\infty}}}\onto R^{\leq(3,0),\sigma}_{\rhobar}$.

By $\cO$-flatness it is enough to check that any closed point $x$ of $\Spec \oS\bbra{\un{Y}}[1/p]$ is sent to the closed subscheme
$\Spec R^{\leq(3,0),\sigma}_{\rhobar}[1/p]$ of $\Spec R^{\Box}_{\rhobar|_{G_{K_\infty}}}[1/p]$.
Let $\fp_x$ be the maximal ideal of $\oS\bbra{\un{Y}}[1/p]$ corresponding to $x$.
Its residue field $\kappa(x)$ is a finite extension of $E$.

By definition,
\[
\bigcap_{\tld w \in X(\sigma)} \sum_j I_{\tld w}^{(j)} = 0
\] 
in $\oS$, hence there exists some $\tld{w}\in X(\sigma)$ such that $\sum_j I_{\tld w}^{(j)} \subseteq \fp_x$.

We now observe that we have a canonical isomorphism
\begin{equation}\label{eq:comparison-S-and-R}
  S/\sum_j I_{\tld w}^{(j)} \cong R/\sum_j I^{(j)},
\end{equation}
where the right-hand side is the ring of Proposition~\ref{prop:def:ring} (for the type $\tau_{\tld w}$), using the change of variables in Figure~\ref{fig:changeofvar} (where we omit the superscripts $(j)$ for readability).
(We caution that the constants $a_i$ and the $O(p^{N-8})$ tails in Tables \ref{Table1FV}--\ref{Table3FV} depend on $\tld w$ and not just on $\tld w_{f-1-j}$.
Moreover, recall that the $O(p^{N-8})$ tails in Tables \ref{Table1FV}--\ref{Table3FV} involve variables of all embeddings, so the change of variables of $I^{(j)}$ really depends on $\tld w$ and not just on $\tld w_{f-1-j}$.)

\begin{figure}[t]
\caption{Change of variables between the tables}\label{fig:changeofvar}
\centering
\adjustbox{max width=\textwidth}{
\begin{tabular}{| c | c | c | c | c | c | c | c | c | }
\hline%
Table \ref{Table1FV}&$e_{11}^*$&$d_{11}$&$c_{11}$&$d_{21}$&$c_{12}$&$c_{21}$&$d_{22}^*$&$c_{22}$
\\
\hline%
Table \ref{Table4FV}&$d_{12}^*$&$c_{12}-pd_{12}^*$&$b_{12}-pc_{12}$&$d_{22}$&$d_{11}$&$c_{22}$&$d_{21}^*$&$c_{21}$\\
\hline
\end{tabular}}

\vspace{5mm}

\adjustbox{max width=\textwidth}{
\begin{tabular}{| c | c | c | c | c | c | c | c | c | }
\hline
Table \ref{Table3FV}&$d_{11}^*$&$c_{11}$&$d_{12}$&$c_{12}$&$c_{21}$&$e_{22}^*$&$d_{22}$&$c_{22}$
\\
\hline
Table \ref{Table4}&$d_{12}^*$&$c_{12}$&$d_{11}$&$c_{11}$&$d_{22}$&$d_{21}^*$&$c_{21}-pd_{21}^*$&$b_{21}-pc_{21}$\\
\hline
\end{tabular}
}
\end{figure}%

Importantly, under the isomorphism \eqref{eq:comparison-S-and-R} the $\phz$-module in Tables \ref{Table4FV}--\ref{Table4} becomes identified with the $\phz$-module described in Tables \ref{Table1FV}--\ref{Table3FV}.
Thus the $\phz$-module $\cM_{\oS\bbra{\un{Y}}}\widehat{\otimes}_{\oS\bbra{\un{Y}}}\kappa(x)$ is one of the $\phz$-modules described in Tables \ref{Table1FV}--\ref{Table3FV} for the type $\tau_{\tld{w}}$, at least after replacing $\cO$ by $\cO_{\kappa(x)}$.
In particular, by the proof of Proposition~\ref{prop:def:ring} we know that $\bV^*_K(\cM_{\oS\bbra{\un{Y}}}\widehat{\otimes}_{\oS\bbra{\un{Y}}}\kappa(x))$ is the restriction to $G_{K_\infty}$ of a potentially crystalline representation $\rho_x$ of $G_K$ over $\kappa(x)$, of inertial types $\tau_{\tld{w}}$ and Hodge--Tate weights $\leq (3,0)$.
Together with the basis $\gamma\otimes_x 1$, $\rho_x|_{G_{K_\infty}}$ is a framed deformation of $\rhobar|_{G_{K_\infty}}$.
By Corollary \ref{cor:ff}, $\rho_x$ is a framed deformation of $\rhobar$, completing the proof of Claim 2.

\paragraph{\textit{Claim 3.}} The ring $\oS$ is reduced, $\cO$-flat, and has $4^f$ irreducible components, each of relative dimension $3f$ over $\cO$.

For short, let $I_{\tld w} \defeq \sum_j I_{\tld w}^{(j)}$ for any $\tld w \in X(\sigma)$. Recall that $\# X(\sigma) = 2^f$.
As $\oS = S/\bigcap_{\tld w \in X(\sigma)} I_{\tld w}$ and each $S/I_{\tld w}$ is, by construction, identified with the ring
$R/\sum_j I^{(j)}$ of Proposition~\ref{prop:def:ring} (for the type $\tau_{\tld{w}}$), we deduce that $\oS$ is reduced and $\cO$-flat and that, in order to establish the claim about irreducible
components, it suffices to show that the ideals $I_{\tld w}$ are pairwise relatively prime in $S[1/p]$.
Pick $\tld w \ne \tld w'$ in $X(\sigma)$ and choose $j$ such that $\tld w_{f-1-j} \ne \tld w'_{f-1-j}$. Assume $(\tld w_\sigma)_{f-1-j} = t_{(1,2)}$, so we are in the
setting of Table \ref{Table4FV} at embedding $j$.
Hence by Table \ref{Table4FV} the ideal $I_{\tld{w}}^{(j)}+I^{(j)}_{\tld{w}'}$ contains an element of the form  (recall we omit the superscripts $(j)$)
\begin{align*}
&\bigg(
c_{12}-pd_{12}^*+(a_1-2)\frac{d_{11}d_{22}}{d_{21}^*}
\bigg) - \bigg(
c_{12}-a_2d_{12}^*\Big(\frac{d_{11}d_{22}}{d_{12}^*d_{21}^*}+p\Big)
\bigg)+O(p^{N-8})\\
&=p(a_2-1)d_{12}^*+(a_1+a_2-2)\frac{d_{11}d_{22}}{d_{21}^*}+O(p^{N-8}).
\end{align*}
As $a_2\not\equiv 1\pmod p$, $a_1+a_2\equiv 2\pmod p$ (see the explicit formulas below Tables~\ref{Table1FV}--\ref{Table2FV}), and $N \geq 10$ we deduce that $p\in I_{\tld{w}}^{(j)}+I_{\tld{w}'}^{(j)}$,
which in turn is contained in $I_{\tld w}+I_{\tld w'}$.
The case where $(\tld{w}_\sigma)_{f-1-j}=t_{(2,1)}$ is analogous, checking that $p\in I_{\tld{w}}^{(j)}+I_{\tld{w}'}^{(j)}$ by using the two elements of the form $c_{21}+\dots$ from Table~\ref{Table4}. This establishes Claim 3.

\paragraph{\textit{Conclusion of the proof.}}
By Claims 1 and 2 we have a surjective morphism $R^{\leq (3,0),\sigma}_{\rhobar}\bbra{\un{X}',\un{X}''}\onto \oS\bbra{\un{Y}}$.
By \cite[Thm.\ (3.3.8)]{KisinPSS} the ring $R^{\leq (3,0),\sigma}_{\rhobar}$ is reduced, $\cO$-flat, and each irreducible component is of relative dimension $f+4$ over $\cO$.
By Proposition \ref{prop:def:ring} it has precisely $4^f$ irreducible components.
By Claim 3 and as $(f+4)+2f=3f+4$ we deduce that $R^{\leq (3,0),\sigma}_{\rhobar}\bbra{\un{X}',\un{X}''}\cong \oS\bbra{\un{Y}}$.

The identification of irreducible components follows from Proposition \ref{prop:def:ring}, as for any $\tilde w \in X(\sigma)$ 
the isomorphism $R^{\leq (3,0),\sigma}_{\rhobar}\bbra{\un{X}',\un{X}''}\cong \oS\bbra{\un{Y}}$
factors through the isomorphism $R^{\leq (3,0),\tau_{\tld w}}_{\rhobar}\bbra{\un{X}',\un{X}''}\cong S/I_{\tld w}\bbra{\un{Y}}$
of Proposition \ref{prop:def:ring} (keeping in mind the change of variables discussed in the proof of Claim 2).
\end{proof}

\begin{lem}\label{lem:lowest-hodge-type}
  If $(\tld w_\sigma)_{f-1-j} = t_{(1,2)}$ let
  \begin{align*}
    \fq^{(j),(2,1)}_1 &\defeq (b_{12}-pc_{12},c_{11},c_{12}-pd_{12}^*,c_{21},c_{22},d_{11}), \\
    \fq^{(j),(2,1)}_2 &\defeq \bigg(b_{12},c_{11}, c_{12},c_{21},c_{22}, \frac{d_{11}d_{22}}{d_{12}^*d_{21}^*}+p\bigg)\\
  \noalign{\noindent and if $(\tld w_\sigma)_{f-1-j} = t_{(2,1)}$ let}
    \fq^{(j),(2,1)}_2 &\defeq \bigg(b_{21},c_{11}, c_{12},c_{21},c_{22}, \frac{d_{11}d_{22}}{d_{12}^*d_{21}^*}+p\bigg), \\
    \fq^{(j),(2,1)}_3 &\defeq (b_{21}-pc_{21},c_{11},c_{12},c_{21}-pd_{21}^*,c_{22},d_{22}),
  \end{align*}
  where we omit the superscripts $(j)$ for readability and we consider these as ideals of $S^{(j)}$. 
  Let $\tld{w}\in X(\sigma)$.
  Then $\fq^{(j),(2,1)}_{i(\tld w)_{f-1-j}} \subset \sum_{j'=0}^{f-1} \fp^{(j'),\lambda_{f-1-j'}}_{\tld w}$
  whenever $\lambda_{f-1-j} = (2,1)$ and $\sum_{j'=0}^{f-1} \fq^{(j'),(2,1)}_{i(\tld w)_{f-1-j'}} = \sum_{j'=0}^{f-1} \fp^{(j'),(2,1)}_{\tld w}$
  \emph{(}as ideals of $S$\emph{)}.
\end{lem}

\begin{proof}
  We can check the containment after reducing modulo $I_{\tld w} = \sum_{j'=0}^{f-1} I_{\tld w}^{(j')} \subset \sum_{j'=0}^{f-1} \fp^{(j'),\lambda_{f-1-j'}}_{\tld w}$.
  Hence this follows from Remark~\ref{rem:lowest-hodge-type} (and the identification~\eqref{eq:comparison-S-and-R} in the proof of Proposition~\ref{prop:multitype-def-ring}).
  The final equality follows by dimension reasons.
\end{proof}

Recall that $\rhobar: G_K\ra \GL_2(\F)$ is such that $\rhobar|_{I_K}\cong \taubar(s,\mu)$, where $\mu-\eta$ is $N$-deep with {$N\geq 12$} (see item \ref{it:sf:2} in \S \ref{sec:setup}).

\begin{prop}
\label{prop:p:in:inter}
  Keep the hypotheses of Proposition~\ref{prop:multitype-def-ring} and the definitions of Lemma~\ref{lem:lowest-hodge-type}. Then for any $0 \le j \le f-1$ and any $\tld{w}\in X(\sigma)$ such that $i(\tld{w})_{f-1-j}=2$ we have
  $p \in \fq^{(j),(2,1)}_1 \cap \fq^{(j),(2,1)}_2 + \fp^{(j),(3,0)}_{\tld{w}}$ if $(\tld w_\sigma)_{f-1-j} = t_{(1,2)}$ and 
  $p \in \fq^{(j),(2,1)}_2 \cap \fq^{(j),(2,1)}_3 + \fp^{(j),(3,0)}_{\tld{w}}$ if $(\tld w_\sigma)_{f-1-j} = t_{(2,1)}$.
\end{prop}

\begin{proof}
Suppose that $(\tld w_\sigma)_{f-1-j} = t_{(1,2)}$.
We will systematically omit superscripts $(j)$ and write $\fp_2^{(3,0)}$ instead of $\fp_{\tld{w}}^{(j),(3,0)}$ for readability.
From rows 4 and 8 of Table \ref{Table4FV} note that the following elements are in $\fp_2^{(3,0)}$:
\begin{align*}
&c_{21}+(a_2-1)d_{21}^*\bigg(\frac{d_{11}d_{22}}{d_{12}^*d_{21}^*}+p\bigg)+O(p^{N-8}),\\
&\frac{d_{11}d_{22}}{d_{12}^*d_{21}^*}+p\frac{(a_2-2)(a_2+1)}{a_2(a_2-1)}+O(p^{N-8})
  = \bigg(\frac{d_{11}d_{22}}{d_{12}^*d_{21}^*}+p\bigg)-\frac{2p}{a_2(a_2-1)}+O(p^{N-8}).\\
\end{align*}
By eliminating $\frac{d_{11}d_{22}}{d_{12}^*d_{21}^*}+p$ using the last element we get
\begin{align*}
c_{21}+\frac{2p}{a_2}d_{21}^*+O(p^{N-8})&\in \fp_2^{(3,0)}.
\end{align*}
Noting that $c_{21}$ is in $\fq_1^{(2,1)}\cap \fq_2^{(2,1)}$ we deduce that
\begin{align*}
\fp_2^{(3,0)}+\fq_1^{(2,1)}\cap \fq_2^{(2,1)}&\ni\frac{2p}{a_2}d_{21}^*+O(p^{N-8})\\
&=p\bigg(\frac{2}{a_2}d_{21}^*+O(p^{N-9})\bigg).&
\end{align*}
As $N\geq 10$, the factor in parentheses is a unit in $S$, so we obtain $p\in \fp_2^{(3,0)}+\fq_1^{(2,1)}\cap \fq_2^{(2,1)}$.

The case $(\tld w_\sigma)_{f-1-j} = t_{(2,1)}$ is completely analogous, using from rows 3 and 6 of Table~\ref{Table4} that 
\begin{align*}
&c_{12}-\frac{2p}{a_2-1}d_{12}^*+O(p^{N-8})\in\fp^{(3,0)}_2, \\
&c_{12}\in\fq^{(2,1)}_2\cap\fq^{(2,1)}_3.
\end{align*}
(Alternatively, we mention that the element $\begin{pmatrix}0 & 1\\v &0\end{pmatrix}$
normalizing the Iwahori interchanges shapes $t_{(2,1)}$ and $t_{(1,2)}$ and preserves $\fW t_{(2,1)}$. It can then be seen
that Tables 1 and 3, and likewise Tables 4 and 5, are interchanged under the transformation sending $c_{ik}$, $d_{ik}$, \dots\
to $c_{3-i,3-k}$, $d_{3-i,3-k}$, \dots\ and $a_i$ to $1-a_{4-i}$. In this way we can reduce the second case of this proposition
to the first.)
\end{proof}

\newpage

\begin{table}[ht]
\captionsetup{justification=centering}
\caption[Foo content]{\textbf{Shape $\tld{w}_{f-1-j}=t_{(2,1)}$, i.e.~$\ovl{A}^{(f-1-j)}=\begin{pmatrix}\ovl{e_{11}^*}v^2&0\\0&\ovl{d_{22}^*}v\end{pmatrix}$.
}
}
\label{Table1FV}
\centering
\adjustbox{max width=\textwidth}{
\begin{tabular}{| c | c |}
\hline
&\\
$A^{(f-1-j)}$ & 
$\begin{pmatrix} (v+p)^2 e_{11}^*+(v+p)d_{11}+c_{11}& c_{12} \\ v((v+p)d_{21}+c_{21}) &  (v+p)d_{22}^*+c_{22}\end{pmatrix}$\\
&\\
\hline
&\\
$\varphi$-module at the &
\multirow{2}{*}{$
\begin{pmatrix} \frac 1v\big((v+p)^2 e_{11}^*+(v+p)d_{11}+c_{11}\big) & c_{12} \\ (v+p)d_{21}+c_{21} &  (v+p)d_{22}^*+c_{22}\end{pmatrix}s_{j}^{-1}\begin{pmatrix}v^{r_{j}+1}&0\\0&1
\end{pmatrix}
$}\\
$(f-1-j)$-th embedding& \\
&\\
\hline
&\\
$R^{(j)}$& $\cO\bbra{c_{11},d_{11},x_{11}^*,c_{12},c_{21},d_{21},c_{22},x_{22}^*}$\\
&\\
\hline
&\\
$I^{(j),\leq (3,0)}$&
$\begin{aligned}
&c_{11}c_{22}+pc_{12}c_{21},\\
&d_{11}c_{22}-c_{12}c_{21}+c_{11}d_{22}^*+pc_{12}d_{21},\\
&e_{11}^*c_{22}+d_{11}d_{22}^*-c_{12}d_{21}
\end{aligned}$\\
&\\
\hline
&\\
$I^{(j),\nabla}$&
$\begin{aligned}
&(a_1-1)d_{11}c_{22}+a_1c_{11}d_{22}^*+p(d_{11}d_{22}^*+2e_{11}^*c_{22})+O(p^{N-4}),\\
&c_{22}(a_1c_{11}+pd_{11})+O(p^{N-3}),\\
&c_{12}((a_1-1)d_{11}+2pe_{11}^*)+O(p^{N-4}),\\
&c_{12}(a_1c_{11}+pd_{11})+O(p^{N-3}),\\
&(a_1-1)c_{21}c_{22}-p\big((a_1-3)d_{21}c_{22}+(a_1+1)c_{21}d_{22}^*\big)+O(p^{N-4}),\\
&p\big((a_1-1)c_{21}c_{22}+p(d_{21}c_{22}-c_{21}d_{22}^*)\big)+O(p^{N-3}),\\
&(a_1-1)c_{12}c_{21}+c_{11}d_{22}^*-p\big((a_1-3)c_{12}d_{21}+d_{11}d_{22}^*\big)+O(p^{N-4}),\\
&p\big((a_1-1)c_{12}c_{21}+c_{11}d_{22}^*+pc_{12}d_{21}\big)+O(p^{N-3})\\
\end{aligned}$\\
&\\
\hline
&\\
$I^{(j)}$&
$\begin{aligned}
& d_{11}+(a_1-2)\frac{c_{12}d_{21}}{d_{22}^*}+O(p^{N-8}),
\\
& c_{22}-(a_1-1)\frac{c_{12}d_{21}}{e_{11}^*}+O(p^{N-8}),\\
& c_{21} + \frac{(a_1-1)(a_1-2)}{a_1}\frac{c_{12}(d_{21})^2}{e_{11}^*d_{22}^*}+O(p^{N-8}),\\
& c_{11} - \frac{c_{12}d_{21}}{d_{22}^*} \bigg(\frac{(a_1-1)^2 (a_1-2)}{a_1}\frac{c_{12}d_{21}}{e_{11}^*d_{22}^*} - p\bigg)+O(p^{N-8}),\\
& \Big(c_{12}+O(p^{N-8})\Big)\bigg((a_1-1)(a_1-2)\frac{c_{12}d_{21}}{e_{11}^*d_{22}^*}-2p+O(p^{N-8})\bigg)
\end{aligned}$
$\begin{aligned}
\end{aligned}$\\
&\\
\hline
&\\
$\fp^{(j),(2,1)}$& $I^{(j)}+\Big(c_{12}+O(p^{N-8})\Big)$ %
\\
&\\
\hline
&\\
$\fp^{(j),(3,0)}$& $I^{(j)}+\bigg((a_1-1)(a_1-2)\displaystyle\frac{c_{12}d_{21}}{e_{11}^*d_{22}^*}-2p+O(p^{N-8})\bigg)$
\\
&\\
\hline
\end{tabular}}
\captionsetup{justification=raggedright,
singlelinecheck=false
}
\caption*{\tiny{
Here, $a_1 \in \ZZ_{(p)}$ and $a_1\equiv -\langle s_{j}^{-1}(\mu_{j})-(2,1),\alpha_{j}^\vee\rangle \equiv -\sgn(s_{j})(r_{j}+1)+1 \pmod p$.}
For readability we write $a_1$, $c_{ik}$, etc.\ instead of $a_1^{(j)}$, $c_{ik}^{(j)}$, etc.
Also, note that $x_{11}^* \defeq e_{11}^*-[\ovl{e^*_{11}}]$ and $x_{22}^* \defeq d_{22}^*-[\ovl{d^*_{22}}]$.

Note that both $a_1$ and the $O(p^{N-8})$ tails \emph{depend on the whole $f$-tuple} $\tld w$ and not just on $\tld w_{f-1-j}$. 
Also, the $O(p^{N-8})$ tails involve variables of all embedding and $I^{(j),\nabla}$, $I^{(j)}$, $\fp^{(j),(2,1)}$ and $\fp^{(j),(3,0)}$ are not ideals of $R^{(j)}$ in general.
A similar comment applies to Tables \ref{Table2FV}--\ref{Table4} below.
}
\end{table}

\begin{table}[ht]
\captionsetup{justification=centering}
\caption[Foo content]{\textbf{Shape $\tld{w}_{f-1-j}=\fW t_{(2,1)}$, i.e.~$\ovl{A}^{(f-1-j)}=\begin{pmatrix}0&\ovl{d_{12}^*}v\\\ovl{d_{21}^*}v^2&0\end{pmatrix}$.}} 
\label{Table2FV}
\centering
\adjustbox{max width=\textwidth}{
\begin{tabular}{| c | c |}
\hline
&\\
$A^{(f-1-j)}$ & 
$\begin{pmatrix} (v+p)d_{11}+c_{11}& (v+p)d_{12}^*+c_{12} \\ v((v+p)d_{21}^*+c_{21}) &  (v+p)d_{22}+c_{22}\end{pmatrix}$\\
&\\
\hline
&\\
$\varphi$-module at the & 
\multirow{2}{*}{$\begin{pmatrix} (v+p)d_{12}^*+c_{12} & \frac{1}{v} \big((v+p)d_{11}+c_{11}\big)\\(v+p)d_{22}+c_{22} &(v+p)d_{21}^*+c_{21} \end{pmatrix}s_{j}^{-1}\begin{pmatrix}v^{r_{j}+1}&0\\0&1
\end{pmatrix}$}\\
$(f-1-j)$-th embedding & \\
&\\
\hline
&\\
$R^{(j)}$& $\cO\bbra{c_{11},d_{11},c_{12},x^*_{12},c_{21},x^*_{21},c_{22},d_{22}}$\\
&\\
\hline
&\\
$I^{(j),\leq (3,0)}$&
$\begin{aligned}
&d_{11}d_{22}-(c_{12}d_{21}^*+d_{12}^*c_{21})+pd_{12}^*d_{21}^*,\\
&c_{12}c_{21}-d_{11}c_{22}-c_{11}d_{22}-p(c_{12}d_{21}^*+d_{12}^*c_{21}),\\
&c_{11}c_{22}+pc_{12}c_{21}
\end{aligned}$\\
&\\
\hline
&\\
$I^{(j),\nabla}$&
$\begin{aligned}
&(a_2-1)d_{11}c_{22}+a_2 c_{11}d_{22}+p(d_{11}d_{22}-2d^*_{12}c_{21}+pd_{12}^*d_{21}^*)+O(p^{N-4}),\\
&a_2c_{11}c_{22}+p(d_{11}c_{22}+pd_{12}^*c_{21})+O(p^{N-3}),\\
&(a_2+1)c_{11}d_{12}^*+(a_2-1)d_{11}c_{12}+O(p^{N-4}),\\
&a_2c_{11}c_{12}+p(d_{11}c_{12}-c_{11}d_{12}^*)+O(p^{N-3}),\\
&(a_2-1)c_{21}c_{22}-p\big((a_2-3) d_{21}^*c_{22}+(a_2+1) c_{21}d_{22}\big)+O(p^{N-4}),\\
&p\big((a_2-1)c_{21}c_{22}+p (d_{21}^*c_{22}-c_{21}d_{22})\big)+O(p^{N-3}),\\
&(a_2-1) c_{12}c_{21}+c_{11}d_{22}-p\big((a_2-3) c_{12}d_{21}^*+(a_2-1) d_{12}^*c_{21} \\ &\qquad\qquad + d_{11}d_{22} + pd_{12}^*d_{21}^*\big)+O(p^{N-4}),\\
&p\big((a_2-1)c_{12}c_{21}+c_{11}d_{22}+p c_{12}d_{21}^*\big)+O(p^{N-3})\\
\end{aligned}$\\
&\\
\hline
&\\
$I^{(j)}$&
$\begin{aligned}
& c_{21}+(a_2-1)d_{21}^*\bigg(\frac{d_{11}d_{22}}{d_{12}^*d_{21}^*}+p\bigg)+O(p^{N-8}),\\
& c_{12}-a_2d_{12}^*\bigg(\frac{d_{11}d_{22}}{d_{12}^*d_{21}^*}+p\bigg)+O(p^{N-8}),\\
& c_{11}+\frac{a_2(a_2-1)}{a_2+1}d_{11}\bigg(\frac{d_{11}d_{22}}{d_{12}^*d_{21}^*}+p\bigg)+O(p^{N-8}),\\
& c_{22}-\frac{a_2(a_2-1)}{a_2-2}d_{22}\bigg(\frac{d_{11}d_{22}}{d_{12}^*d_{21}^*}+p\bigg)+O(p^{N-8}),\\
& \bigg(\frac{d_{11}d_{22}}{d_{12}^*d_{21}^*}+p+O(p^{N-8})\bigg)\bigg(\frac{a_2(a_2-1)}{(a_2-2)(a_2+1)}\frac{d_{11}d_{22}}{d_{12}^*d_{21}^*}+p+O(p^{N-8})\bigg)
\end{aligned}$\\
&\\
\hline&\\
$\fp^{(j),(2,1)}$& $I^{(j)}+\bigg(\displaystyle\frac{d_{11}d_{22}}{d_{12}^*d_{21}^*}+p+O(p^{N-8})\bigg)$ %
\\
&\\
\hline
&\\
$\fp^{(j),(3,0)}$& $I^{(j)}+\bigg(\displaystyle\frac{a_2(a_2-1)}{(a_2-2)(a_2+1)}\frac{d_{11}d_{22}}{d_{12}^*d_{21}^*}+p+O(p^{N-8})\bigg)$
\\
&\\
\hline
\end{tabular}}
\captionsetup{justification=raggedright,
singlelinecheck=false
}
\caption*{\footnotesize{
Here, $a_2 \in \ZZ_{(p)}$ and $a_2\equiv-\langle \fW s_{j}^{-1}(\mu_{j})-(2,1),\alpha_{j}^\vee\rangle\equiv  \sgn(s_{j})(r_{j}+1)+1 \pmod p$.
For readability we write $a_2$, $c_{ik}$, etc.\ instead of $a_2^{(j)}$, $c_{ik}^{(j)}$, etc.
Also, note that $x_{12}^* \defeq d_{12}^*-[\ovl{d^*_{12}}]$ and $x_{21}^* \defeq d_{21}^*-[\ovl{d^*_{21}}]$.
}}
\end{table}

\begin{table}[ht]
\captionsetup{justification=centering}
\caption[Foo content]{\textbf{Shape $\tld{w}_{f-1-j}=t_{(1,2)}$, i.e.~$\ovl{A}^{(f-1-j)}=\begin{pmatrix}\ovl{d_{11}^*}v&0\\0&\ovl{e_{22}^*}v^2\end{pmatrix}$.}}
\label{Table3FV}
\centering
\adjustbox{max width=\textwidth}{
\begin{tabular}{| c | c |}
\hline
&\\
$A^{(f-1-j)}$ & 
$\begin{pmatrix} (v+p)d_{11}^*+c_{11}& (v+p)d_{12}+c_{12} \\ vc_{21} &  (v+p)^2e_{22}^*+(v+p)d_{22}+c_{22}\end{pmatrix}$\\
&\\
\hline
&\\
$\varphi$-module at the & 
\multirow{2}{*}{$\begin{pmatrix} (v+p)d_{11}^*+c_{11}& \frac{1}{v}\big((v+p)d_{12}+c_{12}\big) \\ vc_{21} &  \frac{1}{v}\big((v+p)^2e_{22}^*+(v+p)d_{22}+c_{22}\big)\end{pmatrix}s_{j}^{-1}\begin{pmatrix}v^{r_{j}+1}&0\\0&1\end{pmatrix}$}\\
$(f-1-j)$-th embedding&\\
&\\
\hline
&\\
$R^{(j)}$& $\cO\bbra{c_{11},x_{11}^*,c_{12},d_{12},c_{21},c_{22},d_{22},x_{22}^*}$\\
&\\
\hline
&\\
$I^{(j),\leq (3,0)}$&
$\begin{aligned}
&c_{11}c_{22}+pc_{12}c_{21},\\
&c_{11}d_{22}-c_{12}c_{21}+d_{11}^*c_{22}+pd_{12}c_{21},\\
&c_{11}e_{22}^*+d_{11}^*d_{22}-d_{12}c_{21}
\end{aligned}$\\
&\\
\hline
&\\
$I^{(j),\nabla}$&
$\begin{aligned}
&a_3c_{11}d_{22}+(a_3-1)d^*_{11}c_{22}-p(d^*_{11}d_{22}+2c_{11}e^*_{22})+O(p^{N-4}),\\
&c_{11}\big((a_3-1)c_{22}-pd_{22}\big)+O(p^{N-3}),\\
&c_{21}(a_3d_{22}-2pe_{22}^*)+O(p^{N-4}),\\
&c_{21}\big((a_3-1)c_{22}-pd_{22}\big)+O(p^{N-3}),\\
&a_3c_{11}c_{12}-p\big((a_3+2)c_{11}d_{12}+(a_3-2)d_{11}^*c_{12}\big)+O(p^{N-4}),\\
&p\big(a_3c_{11}c_{12}-p(c_{11}d_{12}-d_{11}^*c_{12})\big)+O(p^{N-3}),\\
&a_3c_{12}c_{21}-d_{11}^*c_{22}-p\big((a_3+2)d_{12}c_{21}-d_{11}^*d_{22}\big)+O(p^{N-4}),\\
&p\big(a_3c_{12}c_{21}-d_{11}^*c_{22}-pd_{12}c_{21}\big)+O(p^{N-3})\\
\end{aligned}$\\
&\\
\hline
&\\
$I^{(j)}$&
$\begin{aligned}
&d_{22}-(a_3+1)\frac{d_{12}c_{21}}{d_{11}^*}+O(p^{N-8}),\\
&c_{11}+a_3\frac{d_{12}c_{21}}{e_{22}^*}+O(p^{N-8}),
\\
&c_{12}-\frac{a_3(a_3+1)}{a_3-1}\frac{(d_{12})^2c_{21}}{d_{11}^*e_{22}^*}+O(p^{N-8}),\\
&c_{22}-\frac{d_{12}c_{21}}{d_{11}^*}\bigg(\frac{(a_3)^2(a_3+1)}{a_3-1}\frac{d_{12}c_{21}}{d_{11}^*e_{22}^*}-p\bigg)+O(p^{N-8}),\\
&\Big(c_{21}+O(p^{N-8})\Big)\bigg(a_3(a_3+1)\frac{d_{12}c_{21}}{d_{11}^*e_{22}^*}-2p+O(p^{N-8})\bigg)
\end{aligned}$\\
&\\
\hline
&\\
$\fp^{(j),(2,1)}$& $I^{(j)}+\Big(c_{21}+O(p^{N-8})\Big)$%
\\
&\\
\hline
&\\
$\fp^{(j),(3,0)}$& $I^{(j)}+\bigg(a_3(a_3+1)\displaystyle\frac{d_{12}c_{21}}{d_{11}^*e_{22}^*}-2p+O(p^{N-8})\bigg)$
\\
&\\
\hline
\end{tabular}}
\captionsetup{justification=raggedright,
singlelinecheck=false
}
\caption*{\footnotesize{
Here, $a_3\in \ZZ_{(p)}$ and $a_3\equiv-\langle s_{j}^{-1}(\mu_{j})-(1,2),\alpha_{j}^\vee\rangle\equiv- \sgn(s_{j})(r_{j}+1)-1 \pmod p$.
For readability we write $a_3$, $c_{ik}$, etc.\ instead of $a_3^{(j)}$, $c_{ik}^{(j)}$, etc.
Also, note that $x_{11}^* \defeq d_{11}^*-[\ovl{d^*_{11}}]$ and $x_{22}^* \defeq e_{22}^*-[\ovl{e^*_{22}}]$.
}}
\end{table}

\begin{table}[ht]
\captionsetup{justification=centering}
\caption[Foo content]{\textbf{Multi-type deformations: shapes $\tld w_{f-1-j} = t_{(2,1)}$ and $\tld{w}_{f-1-j}=\fW t_{(2,1)}$.}}
\label{Table4FV}
\centering
\adjustbox{max width=\textwidth}{
\begin{tabular}{| c | c | }
\hline
&\\
Multi-type $\varphi$-module at& 
\multirow{2}{*}{$\begin{pmatrix}
(v+p)d_{12}^*+c_{12}+\frac{b_{12}}{v}&\frac{1}{v}\big((v+p)d_{11}+c_{11}\big)
\\
(v+p)d_{22}+c_{22}&
(v+p)d_{21}^*+c_{21}
\end{pmatrix}
s_{j}^{-1}\begin{pmatrix}v^{r_{j}+1}&0\\0&1
\end{pmatrix}$}\\
the $(f-1-j)$-th embedding&\\
&\\
\hline
&\\
$S^{(j)}$& $\cO\bbra{c_{11},d_{11},b_{12},c_{12},x^*_{12},c_{21},x^*_{21},c_{22},d_{22}}$\\
&\\
\hline
&\\
$I_{\tld{w}}^{(j)},\quad i(\tld{w})_{f-1-j}=1$
&
$\begin{aligned}
&c_{11}+pd_{11},\\
&c_{12}-pd_{12}^*+(a_1-2)\frac{d_{11}d_{22}}{d_{21}^*}+O(p^{N-8}),\\
&c_{21}-(a_1-1)\frac{d_{11}d_{22}}{d_{12}^*}+O(p^{N-8}),\\
&c_{22}+\frac{(a_1-1)(a_1-2)}{a_1}\frac{d_{11}(d_{22})^2}{d_{12}^*d_{21}^*}+O(p^{N-8}),\\
&b_{12}-pc_{12}-\frac{d_{11}d_{22}}{d_{21}^*}\bigg(\frac{(a_1-1)^2(a_1-2)}{a_1}\frac{d_{11}d_{22}}{d_{12}^*d_{21}^*}-p\bigg)+O(p^{N-8}),\\
&\Big(d_{11}+O(p^{N-8})\Big)\bigg((a_1-1)(a_1-2)\frac{d_{11}d_{22}}{d_{12}^*d_{21}^*}-2p+O(p^{N-8})\bigg)
\end{aligned}$\\
&\\
\hline
&\\
$I_{\tld{w}}^{(j)},\quad i(\tld{w})_{f-1-j}=2
$
&
$\begin{aligned}
&b_{12},\\
&c_{21}+(a_2-1)d_{21}^*\bigg(\frac{d_{11}d_{22}}{d_{12}^*d_{21}^*}+p\bigg)+O(p^{N-8}),\\
&c_{12}-a_2d_{12}^*\bigg(\frac{d_{11}d_{22}}{d_{12}^*d_{21}^*}+p\bigg)+O(p^{N-8}),\\
&c_{11}+\frac{a_2(a_2-1)}{a_2+1}d_{11}\bigg(\frac{d_{11}d_{22}}{d_{12}^*d_{21}^*}+p\bigg)+O(p^{N-8}),\\
&c_{22}-\frac{a_2(a_2-1)}{a_2-2}d_{22}\bigg(\frac{d_{11}d_{22}}{d_{12}^*d_{21}^*}+p\bigg)+O(p^{N-8}),\\
&\bigg(\frac{d_{11}d_{22}}{d_{12}^*d_{21}^*}+p+O(p^{N-8})\bigg)\bigg(\frac{a_2(a_2-1)}{(a_2-2)(a_2+1)}\frac{d_{11}d_{22}}{d_{12}^*d_{21}^*}+p+O(p^{N-8})\bigg)
\end{aligned}$\\
&\\
\hline
&\\
$\fp_{\tld{w}}^{(j),(2,1)}, i(\tld{w})_{f-1-j}=1$& $I_{\tld{w}}^{(j)}+\Big(d_{11}+O(p^{N-8})\Big)$ %
\\
&\\
\hline
&\\
$\fp_{\tld{w}}^{(j),(3,0)}, i(\tld{w})_{f-1-j}=1$& $I_{\tld{w}}^{(j)}+\bigg((a_1-1)(a_1-2)\displaystyle\frac{d_{11}d_{22}}{d_{12}^*d_{21}^*}-2p+O(p^{N-8})\bigg)$
\\
&\\
\hline
&\\
$\fp_{\tld{w}}^{(j),(2,1)}, i(\tld{w})_{f-1-j}=2$& $I_{\tld{w}}^{(j)}+\bigg(\displaystyle\frac{d_{11}d_{22}}{d_{12}^*d_{21}^*}+p+O(p^{N-8})\bigg)$ %
\\
&\\
\hline
&\\
$\fp_{\tld{w}}^{(j),(3,0)}, i(\tld{w})_{f-1-j}=2$& $I_{\tld{w}}^{(j)}+\bigg(\displaystyle\frac{a_2(a_2-1)}{(a_2-2)(a_2+1)}\frac{d_{11}d_{22}}{d_{12}^*d_{21}^*}+p+O(p^{N-8})\bigg)$
\\
&\\
\hline
\end{tabular}}
\captionsetup{justification=raggedright,
singlelinecheck=false
}
\caption*{\tiny{
For readability we write $a_i$, $c_{ik}$, etc.\ instead of $a_i^{(j)}$, $c_{ik}^{(j)}$, etc.
Also, note that $x_{12}^* \defeq d_{12}^*-[\ovl{d^*_{12}}]$ and $x_{21}^* \defeq d_{21}^*-[\ovl{d^*_{21}}]$, where $\ovl{d^*_{12}}$, $\ovl{d^*_{21}} \in \F^\times$.
Note that the constants $a_1, a_2$ and the $O(p^{N-8})$ tails coming from Tables~\ref{Table1FV}--\ref{Table2FV} (by the change of variables in Figure~\ref{fig:changeofvar}) \emph{depend on the whole $f$-tuple} $\tld{w}\in X(\sigma)$.
}}
\end{table}

\begin{table}[ht]
\captionsetup{justification=centering}
\caption[Foo content]{\textbf{Multi-type deformations: shapes $\tld w_{f-1-j} = \fW t_{(2,1)}$ and $\tld{w}_{f-1-j}=t_{(1,2)}$.}}
\label{Table4}
\centering
\adjustbox{max width=\textwidth}{
\begin{tabular}{| c | c | }
\hline
&\\
Multi-type $\varphi$-module at& 
\multirow{2}{*}{$\begin{pmatrix}
(v+p)d_{12}^*+c_{12}&\frac{1}{v}\big((v+p)d_{11}+c_{11}\big)
\\
(v+p)d_{22}+c_{22}&(v+p)d_{21}^*+c_{21}+\frac{b_{21}}{v}
\end{pmatrix}s_{j}^{-1}\begin{pmatrix}v^{r_{j}+1}&0\\0&1
\end{pmatrix}
$}
\\
the $(f-1-j)$-th embedding&\\
&\\
\hline
&\\
$S^{(j)}$& $\cO\bbra{c_{11},d_{11},c_{12},x^*_{12},b_{21},c_{21},x^*_{21},c_{22},d_{22}}$\\
&\\
\hline
&\\
$I_{\tld{w}}^{(j)},\quad i(\tld{w})_{f-1-j}=2$
&
$\begin{aligned}
&b_{21},\\
&c_{21}+(a_2-1)d_{21}^*\bigg(\frac{d_{11}d_{22}}{d_{12}^*d_{21}^*}+p\bigg)+O(p^{N-8}),\\
&c_{12}-a_2d_{12}^*\bigg(\frac{d_{11}d_{22}}{d_{12}^*d_{21}^*}+p\bigg)+O(p^{N-8}),\\
&c_{11}+\frac{a_2(a_2-1)}{a_2+1}d_{11}\bigg(\frac{d_{11}d_{22}}{d_{12}^*d_{21}^*}+p\bigg)+O(p^{N-8}),\\
&c_{22}-\frac{a_2(a_2-1)}{a_2-2}d_{22}\bigg(\frac{d_{11}d_{22}}{d_{12}^*d_{21}^*}+p\bigg)+O(p^{N-8}),\\
&\bigg(\frac{d_{11}d_{22}}{d_{12}^*d_{21}^*}+p+O(p^{N-8})\bigg)\bigg(\frac{a_2(a_2-1)}{(a_2-2)(a_2+1)}\frac{d_{11}d_{22}}{d_{12}^*d_{21}^*}+p+O(p^{N-8})\bigg)
\end{aligned}$\\
&\\
\hline
&\\
$I_{\tld{w}}^{(j)},\quad i(\tld{w})_{f-1-j}=3$
&
$\begin{aligned}
&c_{22}+pd_{22},\\
&c_{21}-pd_{21}^*-(a_3+1)\frac{d_{11}d_{22}}{d_{12}^*}+O(p^{N-8}),\\
&c_{12}+a_3\frac{d_{11}d_{22}}{d_{21}^*}+O(p^{N-8}),\\
&c_{11}-\frac{a_3(a_3+1)}{a_3-1}\frac{(d_{11})^2 d_{22}}{d_{12}^*d_{21}^*}+O(p^{N-8}),\\
&b_{21}-pc_{21}-\frac{d_{11}d_{22}}{d_{12}^*}\bigg(\frac{(a_3)^2(a_3+1)}{a_3-1}\frac{d_{11}d_{22}}{d_{12}^*d_{21}^*}-p\bigg)+O(p^{N-8}),\\
&\Big(d_{22}+O(p^{N-8})\Big)\bigg(a_3(a_3+1)\frac{d_{11}d_{22}}{d_{12}^*d_{21}^*}-2p+O(p^{N-8})\bigg)
\end{aligned}$\\
&\\
\hline
&\\
$\fp_{\tld{w}}^{(j),(2,1)}, i(\tld{w})_{f-1-j}=2$& $I_{\tld{w}}^{(j)}+\bigg(\displaystyle\frac{d_{11}d_{22}}{d_{12}^*d_{21}^*}+p+O(p^{N-8})\bigg)$%
\\
&\\
\hline
&\\
$\fp_{\tld{w}}^{(j),(3,0)}, i(\tld{w})_{f-1-j}=2$& $I_{\tld{w}}^{(j)}+\bigg(\displaystyle\frac{a_2(a_2-1)}{(a_2-2)(a_2+1)}\frac{d_{11}d_{22}}{d_{12}^*d_{21}^*}+p+O(p^{N-8})\bigg)$
\\
&\\
\hline
&\\
$\fp_{\tld{w}}^{(j),(2,1)}, i(\tld{w})_{f-1-j}=3$& $I_{\tld{w}}^{(j)}+\Big(d_{22}+O(p^{N-8})\Big)$%
\\
&\\
\hline
&\\
$\fp_{\tld{w}}^{(j),(3,0)}, i(\tld{w})_{f-1-j}=3$& $I_{\tld{w}}^{(j)}+\bigg(a_3(a_3+1)\displaystyle\frac{d_{11}d_{22}}{d_{12}^*d_{21}^*}-2p+O(p^{N-8})\bigg)$
\\
&\\
\hline
\end{tabular}}
\captionsetup{justification=raggedright,
singlelinecheck=false
}
\caption*{\tiny{
For readability we write $a_i$, $c_{ik}$, etc.\ instead of $a_i^{(j)}$, $c_{ik}^{(j)}$, etc.
Also, note that $x_{12}^* \defeq d_{12}^*-[\ovl{d^*_{12}}]$ and $x_{21}^* \defeq d_{21}^*-[\ovl{d^*_{21}}]$, where $\ovl{d^*_{12}}$, $\ovl{d^*_{21}} \in \F^\times$.
Note that the constants $a_2, a_3$ and the $O(p^{N-8})$ tails coming from Tables~\ref{Table1FV}--\ref{Table2FV} (by the change of variables in Figure~\ref{fig:changeofvar}) \emph{depend on the whole $f$-tuple} $\tld{w}\in X(\sigma)$.
}}
\end{table}

\clearpage{}%
\clearpage{}%
\section{Gelfand--Kirillov dimension and representations of the Iwahori}

We introduce an analog of the Gelfand--Kirillov dimension for smooth modulo $p$ representations of $p$-adic analytic groups and prove Corollary~\ref{cor:GKdim} which gives an upper bound for this dimension in the case of representations of the Iwahori subgroup of $\GL_2(L)$, $L$ unramified, satisfying a ``multiplicity one'' assumption in the first three layers of their socle filtration.

Let $\F$ be a finite field of characteristic $p$. If $H$ is a compact $p$-adic analytic group, we define
\[ \ZZ_p\bbra{H}\defeq \varprojlim_{H'\subset H}\ZZ_p[H/H'], \qquad \F\bbra{H}\defeq \F\otimes_{\ZZ_p}\ZZ_p\bbra{H},\]
for $H'$ varying among open normal subgroups of $H$. If $H$ is moreover a pro-$p$-group, $\F\bbra{H}$ is a complete noetherian local ring whose maximal ideal is denoted by $\mathfrak{m}_H$. We let $\gr_\mathfrak{m}\F\bbra{H}$ be the graded ring of $\F\bbra{H}$ for the $\mathfrak{m}_H$-adic filtration
\[ \gr_\mathfrak{m}\F\bbra{H}\defeq \bigoplus_{n\geq0}\mathfrak{m}_H^n/\mathfrak{m}_H^{n+1}.\]

\subsection{Review of Gelfand--Kirillov dimension}
\label{sec:kirillov}

We recall the notion of Gelfand--Kirillov dimension of an admissible smooth $\F$-representation of a $p$-adic analytic group. General references for this part are \cite{Venjakob} and \cite{Ardakov-Brown}. We recall here some useful definitions and results for the reader.

Let $H$ be a compact $p$-adic analytic group and let $M$ be a finitely generated $\F\bbra{H}$-module. Its \emph{grade} $j_H(M)$ is the smallest integer $d$ such that $\Ext^d_{\F\bbra{H}}(M,\F\bbra{H})\neq0$ (with the convention that the smallest element of the empty set is $+\infty$). Moreover, if $M\neq0$, we have
\[0\leq j_H(M)\leq \dim (H),\]
where $\dim(H)$ is the dimension of $H$ as a $\Qp$-analytic variety. This is a consequence of the following two facts:
\begin{enumerate}[(i)]
\item if $H'\subset H$ is an open subgroup of $H$, the $\F\bbra{H'}$-module $M$ is finitely generated and we have $j_H(M)=j_{H'}(M)$, as follows from \cite[Prop.~2.7]{Venjakob};
\item if $H$ is $p$-torsion free, $\F\bbra{H}$ is of finite injective dimension equal to $\mathrm{cd}_p(H)$ \cite[Thm.~3.30(ii)]{Venjakob} and $\mathrm{cd}_p(H)=\dim(H)$ \cite[Cor.~1]{Serrecd}.
\end{enumerate}

We also define a \emph{dimension function} by $\dim_H(M)\defeq \dim(H)-j_H(M)$.

When $H$ is a uniform pro-$p$-group, the graded $\F$-algebra $\gr_\mathfrak{m}\F\bbra{H}$ is commutative isomorphic to the polynomial algebra in $\dim(H)$ variables over $\F$ (see the paragraph after Remark 3.31 in \cite{Venjakob}). If $M$ is a finitely generated $\F\bbra{H}$-module, its graded module $\gr_\mathfrak{m}M$ for the $\mathfrak{m}_H$-adic filtration is a finitely generated $\gr_\mathfrak{m}\F\bbra{H}$-module and $\dim_H(M)$ is equal to the dimension of the support of $\gr_\mathfrak{m}M$ in $\Spec(\gr_\mathfrak{m}\F\bbra{H})$ (see \cite[Thm.~3.21.(ii)]{Venjakob}).

Let $G$ be a $p$-adic analytic group and $\pi$ an admissible smooth $\F$-representation of $G$. For each compact open subgroup $H$ of $G$, the dual $\pi^\vee\defeq \Hom_{\F}(\pi,\F)$ of $\pi$ is a finitely generated $\F\bbra{H}$-module. Its grade does not depend on the choice of $H$ and is denoted $j_G(\pi^\vee)$. The \emph{dimension}, or Gelfand--Kirillov dimension, of $\pi$ is then $\dim_G(\pi)\defeq \dim(G)-j_G(\pi^\vee)=\dim_H(\pi^\vee)$.

\begin{rem}\label{GKdimisnotGKdim}
Let $H$ be some open uniform subgroup of $G$. Then $\dim_G(\pi)$ is the Gelfand--Kirillov dimension of the graded module of $\pi^\vee$ for the $\mathfrak{m}_H$-adic topology (see \cite[\S5.4]{Ardakov-Brown}) but it does not coincide in general with the Gelfand--Kirillov dimension of $\pi^\vee$ as an $\F\bbra{H}$-module [\emph{loc.~cit.}, \S5.6]. However we have the following description of $\dim_G(\pi)$ (see \cite[Prop.~2.18]{EmertonPask}). For $n\geq1$, let $H^{p^n}$ be the subgroup of $p^n$-th powers of elements of $H$. There exist real numbers $a\geq b\geq\tfrac{1}{(\dim_G(\pi))!}$ such that
\begin{equation}\label{eq:EmPask} bp^{n\dim_G(\pi)}+O(p^{n(\dim_G(\pi)-1)})\leq\dim_{\F}\Big(\pi^{H^{p^n}}\Big)\leq ap^{n\dim_G(\pi)}+O(p^{n(\dim_G(\pi)-1)}).\end{equation}
For this reason, the integer $0\leq\dim_G(\pi)\leq\dim(G)$ (or $-\infty$ if $\pi=0$) is also called the \emph{Gelfand--Kirillov dimension} of $\pi$.
\end{rem}

\begin{lem}
\label{lem:GvsG/Z}
Let $G$ be a $p$-adic analytic group and $N$ a closed normal subgroup of $G$. Let $\pi$ be an admissible smooth $\F$-representation of $G$ such that $N$ acts trivially on $\pi$. Then we have $\dim_G(\pi)=\dim_{G/N}(\pi)$.
\end{lem}

\begin{proof}
By replacing $G$ by an open subgroup and $N$ by the intersection we may assume that $G$ is uniform \cite[Cor.\ 8.34]{DDMS}.
Then by Exercise 14 in \cite[\S4]{DDMS} there exists an open uniform pro-$p$-group $H\subset G$ such that $H\cap N$ is uniform. The result is then a direct consequence of the characterization given by \eqref{eq:EmPask}.
\end{proof}

\begin{lem}\label{prop:grsemiabelian}
Let $G$ be an analytic pro-$p$-group without $p$-torsion. Assume that the graded ring $\gr_\mathfrak{m}\F\bbra{G}$ is Auslander-regular \emph{(}see for example \cite[\S III.2.1,~Def.~7]{LiOy} for the precise definition\emph{)}. Let $I$ be a two-sided ideal of $\gr_\mathfrak{m}\F\bbra{G}$ generated by a sequence of $r$ central elements which is $\gr_\mathfrak{m}\F\bbra{G}$-regular \emph{(}where $\gr_\mathfrak{m}\F\bbra{G}$ is considered as a module over its center\emph{)} and such that $\gr_\mathfrak{m}\F\bbra{G}/I$ is isomorphic to a polynomial ring in $\dim(G)-r$ variables. Let $M$ be a finitely generated $\F\bbra{G}$-module such that $\gr_\mathfrak{m}M$ is annihilated by $I$. Then $\dim_G(M)$ is equal to the dimension of the support of $\gr_\mathfrak{m}M$ in $\Spec(\gr_\mathfrak{m}\F\bbra{G}/I)$.
\end{lem}

\begin{proof}
For a ring $A$ and a left $A$-module $N$, we recall the notation
\[ j_A(N)\defeq \min\set{n\in\NN : \Ext^n_A(N,A)\neq0}\]
(with the usual convention that the minimum of the empty set is $+\infty$). Let $A\defeq \gr_\mathfrak{m}\F\bbra{G}$. It follows from \cite[\S III.2.5,~Thm.~2]{LiOy} that $j_G(M)=j_A(\gr_\mathfrak{m}M)$ if $M$ is a finitely generated $\F\bbra{G}$-module. (Note that $\F\bbra{G}$ is a left and right Zariski ring by \cite[II.2.2, Prop.~1]{LiOy}.)

As $A/I$ is a polynomial ring in $\dim(G)-r$ variables, it follows from \cite[\S III.4.1,~Thm.~7]{LiOy} that $j_{A/I}(\gr_\mathfrak{m}M)$ is equal to $\dim(G)-r-\dim_{\textrm{Kr}}\Big(\Supp_{\Spec(A/I)}\big(\gr_\mathfrak{m}M\big)\Big)$, where $\dim_{\textrm{Kr}}$ denotes the \emph{Krull} dimension.

Since $\gr_\mathfrak{m}M$ is annihilated by $I$, there is a spectral sequence
\[ E_2^{p,q}=\Ext^p_{A/I}(\gr_\mathfrak{m}M,\Ext^q_A(A/I,A))\Rightarrow\Ext^{p+q}_A(\gr_\mathfrak{m}M,A).\]
Let $(h_1,\dots,h_r)$ be an $A$-regular generating sequence of central elements in $I$. For all $i\in\ZZ$, we have $\Ext^i_A(A,A)\cong A$ if $i=0$ and $0$ if $i\neq0$. By induction on $r$, we can use the long exact sequence of cohomology to prove that $\Ext^i_A(A/I,A)\cong A/I$ if $i=r$ and $0$ if $i\neq r$. This implies that the spectral sequence degenerates and that $\Ext^p_{A/I}(\gr_\mathfrak{m}M,A/I)\cong \Ext^{p+r}_A(\gr_\mathfrak{m}M,A)$ for all $p \in \ZZ$. We deduce that $j_{A/I}(\gr_\mathfrak{m}M)=j_A(\gr_\mathfrak{m}M)-r$. Consequently we have
\[ j_A(\gr_\mathfrak{m}M)=\dim(G)-\dim_{\textrm{Kr}}\Big(\Supp_{\Spec(A/I)}\big(\gr_\mathfrak{m}M\big)\Big)\]
and we deduce
\[ \dim_G (M) = \dim(G)-j_G(M)=\dim(G)-j_A(\gr_\mathfrak{m}M)=\dim_{\textrm{Kr}}\Big(\Supp_{\Spec(A/I)}\big(\gr_\mathfrak{m}M\big)\Big).\qedhere\]
\end{proof}

\subsection{Recollection of results of Lazard}\label{sec:recoLazard}

Let $G$ be a group with unit element $e_G$. A \emph{$p$-valuation} \cite[III.2.1.2]{Lazard} on $G$ is a map
\[ \omega : G\longrightarrow\RR_{>0}\cup\set{+\infty}\]
such that, for all $x,y\in G$,
\begin{itemize}
\item $\omega(xy^{-1})\geq\min(\omega(x),\omega(y))$;
\item $\omega(x^{-1}y^{-1}xy)\geq\omega(x)+\omega(y)$;
\item $\omega(x)=+\infty\Leftrightarrow x=e_G$;
\item $\omega(x) >\frac{1}{p-1}$;
\item $\omega(x^p)=\omega(x)+1$.
\end{itemize}
A $p$-valuation $\omega$ on $G$ is \emph{saturated} \cite[III.2.1.5]{Lazard} if, for all $x\in G$,
\[ \omega(x)>\frac{p}{p-1}\ \Longleftrightarrow\ \exists y\in G,\ y^p=x.\]

Now we assume that there exists, and we fix it, a saturated $p$-valuation $\omega$ on $G$. For $\nu\in\RR_{>0}$, we define
\[ G_\nu\defeq \set{x\in G : \omega(x)\geq\nu}, \quad G_{\nu^+}\defeq \set{x\in G : \omega(x)>\nu}, \quad \gr_\nu G\defeq G_\nu/G_{\nu^+}.\]
The sets $G_\nu$ and $G_{\nu^+}$ are normal subgroups of $G$. They form a fundamental system of neighborhoods of $e_G$ for a structure of topological group on $G$. The direct sum $\gr G\defeq \bigoplus_\nu \gr_\nu G$ is a graded Lie algebra \cite[II.1.1.7]{Lazard}. If $x\in G\setminus\set{e_G}$, we define $\gr(x)$ as being the image of $x$ in $\gr_{\omega(x)}G\subset\gr G$. We assume that the topological group $G$ is compact so that $\omega(G)$ is discrete in $\RR_{>0}\cup\set{+\infty}$ \cite[Prop.~III.2.2.6]{Lazard}.

Let $\Zp\bbra{G}\defeq \varprojlim_\nu\Zp[G/G_\nu]$ be the completed group algebra of $G$. Note that when $G$ is a compact $p$-adic analytic group, the topology induced by a $p$-valuation is the profinite topology of $G$ \cite[III.3.1.4]{Lazard}.

The map $\gr(x)\mapsto\gr(x^p)$ from $\gr_\nu$ to $\gr_{\nu+1}$ induces an endomorphism of degree $1$ of the graded Lie algebra $\gr G$. Let $\F_p[\eps]$ be the graded polynomial algebra in $\eps$ with $\eps$ in degree $1$. Then there is a unique structure of graded $\F_p[\eps]$-Lie algebra on $\gr G$ such that $\eps$ acts via $\gr(x)\mapsto\gr(x^p)$. The graded $\F_p[\eps]$-module $\gr G$ is then a graded-free $\F_p[\eps]$-module \cite[III.2.1.3]{Lazard}. 
If $G$ is a compact $p$-adic analytic group, this $\F_p[\eps]$-module has finite rank $d = \dim(G)$ \cite[Prop.~III.3.1.3]{Lazard}.

From now on we assume that $G$ is a compact $p$-adic analytic group (and still that it has a saturated $p$-valuation). We fix a family $(x_i)_{1\leq i\leq d}$ of elements of $G$ such that $(\gr(x_i))_{1\leq i\leq d}$ is a basis of the $\F_p[\eps]$-module $\gr G$ (so that $x_i\neq 1$ for all $i$). We call the family $(x_i)_{1\leq i\leq d}$ an \emph{ordered basis} of $G$.

Let $\alpha=(\alpha_i)_{1\leq i\leq d}\in\NN^d$.
We define $z^{\alpha}\defeq \prod_{i=1}^d(x_i-1)^{\alpha_i}\in\Zp[G]$ and $\tau(\alpha)\defeq \sum_{i=1}^d\alpha_i\omega(x_i)$. Following Lazard, we define a valuation $w : \Zp[G]\rightarrow\RR_{>0}\cup\set{+\infty}$ as the (pointwise) infimum of the set of all $\Zp$-algebra valuations $w$ such that, for all $x\in G$, $w(x-1)\geq\omega(x)$. Actually Lazard takes the (pointwise) infimum of all \emph{filtrations} \cite[III.2.3.1.2]{Lazard} but in our case this last infimum is a valuation, so that our definition is equivalent \cite[Thm.~III.2.3.3, Cor.~III.2.3.4]{Lazard}. Moreover by \emph{loc.~cit.}, the $\Zp$-algebra $\Zp\bbra{G}$ is isomorphic to the completion of $\Zp[G]$ for $w$. We have the following description of $\Zp\bbra{G}$ and $w$ \cite[III.2.3.8.8, III.2.3.9]{Lazard}:
\begin{gather*} \Zp\bbra{G}=\set*{\sum_{\alpha\in\NN^d} \lambda_\alpha z^\alpha : \lambda_\alpha\in\Zp}; \\
w\left(\sum_{\alpha\in\NN^d}\lambda_\alpha z^\alpha\right)=\inf\set{v_p(\lambda_\alpha)+\tau(\alpha)}.\end{gather*}

The valuation $w$ extends immediately to $\Qp[G]$ and we define $\mathcal{D}_G$ as the completion of $\Qp[G]$ for the valuation $w$ (or equivalently for the multiplicative norm $\norm{\cdot}=p^{-w(\cdot)}$) which extends canonically to $\mathcal{D}_G$. This is the $\Qp$-algebra named $\mathrm{Sat}\,\Zp[G]$ in \cite[IV.1.2.7]{Lazard}. We deduce from the previous description that:
\begin{gather*} \mathcal{D}_G=\set*{\sum_{\alpha\in\NN^{d}} \lambda_\alpha z^\alpha : \lambda_\alpha\in\Qp,\ v_p(\lambda_\alpha)+\tau(\alpha)\rightarrow +\infty\ \text{as $\tau(\alpha)\rightarrow+\infty$}} \end{gather*}
and that the closure of $\Zp[G]$ in $\mathcal{D}_G$ is isomorphic to the completed group algebra $\Zp\bbra{G}$.

Let $U_{\F_p[\eps]}(\gr G)$ be the enveloping algebra of the $\F_p[\eps]$-Lie algebra $\gr G$. As $\gr G$ is graded, the $\F_p[\eps]$-algebra $U_{\F_p[\eps]}(\gr G)$ is canonically a graded $\F_p[\eps]$-algebra. Namely the tensor algebra $\mathcal{T}_{\Fp[\eps]}(\gr G)$ of the $\Fp[\eps]$-module $\gr G$ inherits a grading from $\gr G$ (see \cite[I.3.3.2]{Lazard}) and, for $x,y\in\gr G$ two homogeneous elements, the element $x\otimes y-y\otimes x-[x,y]$ is homogeneous in $\mathcal{T}_{\Fp[\eps]}(\gr G)$. Consequently $U_{\F_p[\eps]}(\gr G)$ is a quotient of a graded algebra by an homogeneous ideal and is a graded algebra (see \cite[IV.2.1.4]{Lazard}).

Let $\gr\Zp[G]$ be the graded algebra of $\Zp[G]$ with respect to the valuation $w$ which is naturally a graded $\Fp[\eps]$-algebra \cite[I.2.3.2, I.2.3.11]{Lazard}. By definition of $w$, there is a morphism of graded $\F_p[\eps]$-Lie algebras $\gr G\rightarrow\gr\Zp[G]$ given by $\gr(g)\mapsto \gr(g-1)$ for $g\in G$ \cite[III.2.3.2]{Lazard}. In particular, we have $\gr(g^p)\mapsto\eps\gr(g-1)$ for $g\in G$. By the universal property of the enveloping algebra, it extends to a morphism of graded algebras $U_{\F_p[\eps]}(\gr G)\rightarrow\gr\Zp[G]$. It follows from \cite[Thm.~III.2.3.3]{Lazard} that this morphism is an isomorphism. As $\Zp\bbra{G}$ is the completion of $\Zp[G]$ for the valuation $w$, we can identify $\gr\Zp[G]$ and $\gr\Zp\bbra{G}$.

Let $\overline{w}$ be the quotient \emph{filtration} (in the sense of \cite[I.2.1.7]{Lazard}) on $\Fp\bbra{G}=\F_p\otimes_{\Zp}\Zp\bbra{G}$. It is defined by $\overline{w}(x)\defeq \sup\set{w(\tilde{x})\in\RR\cup\set{+\infty} : \tilde{x}\in\Zp\bbra{G},\  \tilde{x}\equiv x \mod p}$. We have
\[ \overline{w}\left(\sum_{\alpha\in\NN^d}\lambda_\alpha z^\alpha\right)=\inf\set{\tau(\alpha) : \lambda_\alpha\neq0}.\]
If $x\in\Zp\bbra{G}$, we have $w(px)=w(x)+1$ so that $\gr(px)=\eps\gr(x)$ and finally $\gr(p\Zp\bbra{G})=\eps\gr(\Zp\bbra{G})$ inside $\gr(\Zp\bbra{G})$. This implies that the short exact sequence of filtered modules is strict \cite[I.2.3.8.2]{Lazard}
\[ 0\longrightarrow (p\Zp\bbra{G},w|_{p\Zp\bbra{G}})\longrightarrow(\Zp\bbra{G},w)\longrightarrow (\Fp\bbra{G},\overline{w})\longrightarrow0.\]
Combined with the isomorphism $U_{\Fp[\eps]}(\gr G)\simeq \gr\Zp\bbra{G}$, this implies the existence of an isomorphism of graded algebras
\[ U_{\F_p[\eps]}(\gr G)\otimes_{\F_p[\eps]}\F_p\simeq \gr\Fp\bbra{G}.\]
Let $\overline{\gr G}$ be the graded Lie algebra $\gr G\otimes_{\F_p[\eps]}\F_p$. We deduce an isomorphism of graded algebras
\begin{equation}\label{eq:iso-gradedmodp-algebra} U_{\Fp}(\overline{\gr G})\simeq\gr\Fp\bbra{G}.\end{equation}

We now give a convenient way to compute $\overline{\gr G}$. Actually we rather compute $\gr G$ and deduce $\overline{\gr G}$ after quotienting by $\eps$.

Let $\mathcal{L}$ be a $\Zp$-Lie algebra. A \emph{$p$-valuation} on $\mathcal{L}$ is a map $w : \mathcal{L}\rightarrow\RR_{>0}\cup\set{+\infty}$ such that for all $\lambda\in\Zp$ and $x,y\in \mathcal{L}$:
\begin{itemize}
\item $w(\lambda x)=v_p(\lambda)+w(x)$;
\item $w(x+y)\geq\inf(w(x),w(y))$;
\item $w([x,y])\geq w(x)+w(y)$.
\end{itemize}
If $(\mathcal{L},w)$ is a $p$-valued Lie algebra, the set $\gr \mathcal{L}$ has a canonical structure of graded Lie algebra. Moreover the map $\gr(x)\mapsto\gr(px)$ extends to a degree $1$ morphism $\gr \mathcal{L}\rightarrow\gr \mathcal{L}$ and to a structure of graded $\F_p[\eps]$-Lie algebra on $\gr \mathcal{L}$.

If $x\in G$, the series
\[ \log_{\mathcal{D}_G}(x)\defeq \sum_{n\geq0}\frac{(-1)^{n-1}}{n}(x-1)^n\]
converges in $\mathcal{D}_G$. The associative algebra $\mathcal{D}_G$
with its valuation $w$ is a $p$-valued Lie algebra for the commutator bracket. The subset $\mathcal{L}_G\defeq \set{\log_{\mathcal{D}_G}(x) : x\in G}$ of $\mathcal{D}_G$ is then a $p$-valued sub-$\ZZ_p$-Lie algebra of $\mathcal{D}_G$. Moreover there is canonical isomorphism of graded $\F_p[\eps]$-Lie algebras $\gr \mathcal{L}_G\simeq\gr G$ (this is a consequence of \cite[Thm.~IV.3.2.5 and IV.1.3.5]{Lazard}).

\subsection{The case of the pro-\texorpdfstring{$p$}{p}-Iwahori of \texorpdfstring{$\GL_2$}{GL\_2}}\label{sec:propIwahoriGL2}

We compute the graded ring of the completed group algebra of the pro-$p$-Iwahori subgroup $I_1$ of $\GL_2(L)$ for unramified $L$ and introduce an interesting ideal which allows us to control the Gelfand--Kirillov dimension of representations of $I_1$.

Let $L$ be an unramified extension of $\Qp$ of degree $f$ with ring of integers $\cO_L$ and residue field $k$. We are interested in the particular case of the group $I_1/Z_1$ which is the quotient of the (upper) pro-$p$-Iwahori subgroup of $\GL_2(\mathcal{O}_L)$ by its center. This group is isomorphic to the subgroup $G\defeq  I_1\cap\SL_2(\mathcal{O}_L)$ of $I_1$ since $p>2$. The following results can also be deduced from \cite{Clozel}. However we prefer to follow \cite{Lazard} in order to emphasize that the graded ring naturally has the structure of an enveloping algebra (see~\eqref{isograded}).

We follow \cite[III.3.2.7]{Lazard} to define a saturated $p$-valuation on $G$. We assume that $p>3$. 
Let $L'=L(\sqrt{p})$ and $v : \M_2(L')\rightarrow\RR_{>0}\cup\set{+\infty}$ be the valuation defined by
\[ v((m_{i,j}))\defeq \min\set{v_p(m_{i,j})}.\]
Let $D$ be the diagonal matrix $\left(\begin{smallmatrix}1 & 0 \\ 0& \sqrt{p}\end{smallmatrix}\right)$ in $\M_2(\mathcal{O}_{L'})$. We define, for $x\in G$:
\[ \omega(x)\defeq  v(D^{-1}xD-I_2).\]
It follows from \cite[III.3.2.7]{Lazard} that $\omega$ is a saturated $p$-valuation on $G$ (here we are using that $p>3$). Explicitly, for $a,b,c,d\in\mathcal{O}_L$
such that $(1+pa)(1+pd)-pbc=1$:
\[ \omega\bigg(\!\!\begin{pmatrix} 1+pa & b \\ pc & 1+pd\end{pmatrix}\!\!\bigg)=\min\set{1+v_p(a),\frac{1}{2}+v_p(b), \frac{1}{2}+v_p(c), 1+v_p(d)}.\]

Let $\mathfrak{g}_{\Zp}$ be the sub-$\Zp$-Lie algebra of $\mathfrak{sl}_{2,\Zp}$ defined by
\[ \mathfrak{g}_{\Zp}\defeq \set*{\begin{pmatrix} pa & b \\ pc & -pa \end{pmatrix} : (a,b,c)\in\Zp^3}.\]

\begin{lem}\label{LieIwahori}
We have an isomorphism of $p$-valued Lie algebras $\cL_{G}\simeq\mathcal{O}_L\otimes_{\Zp}\mathfrak{g}_{\Zp}$ with valuation, for $a,b,c\in\mathcal{O}_L$,
\begin{equation}\label{eq:explvaluation} w\bigg(\!\!\begin{pmatrix} pa & b \\ pc & -pa \end{pmatrix}\!\!\bigg)=\min\set{1+v_p(a),\frac{1}{2}+v_p(b), \frac{1}{2}+v_p(c)}.\end{equation}
\end{lem}

\begin{proof}
Let $G'$ be the subgroup of $\GL_2(L')$ defined by
\[ G'=\set*{x\in \M_2(L') : v(x-I_2)\geq\frac{1}{2}}.\]
As $p-1>2$, it follows from \cite[II.8.4, Prop.~4]{BourbakiLie2et3} that $\log_{\M_2(L')}(G')$ is the sub-Lie algebra of $\M_2(L')$ defined by
\[ \log_{\M_2(L')}(G')=\set*{x\in \M_2(L') : v(x)\geq\frac{1}{2}}.\]
For $x\in G'$, we have $\log_{\M_2(L')}(\Ad(D)x)=\Ad(D)\log_{\M_2(L')}(x)$. As $G=\Ad(D)(G')\cap \M_2(L)$, we have
\begin{equation}\label{eq:isoLielog} \log_{\M_2(L')}(G)=\set*{x\in \M_2(L) : v(\Ad(D)^{-1}x)\geq\frac{1}{2}}=\mathcal{O}_L\otimes_{\Zp}\mathfrak{g}_{\Zp}.\end{equation}
We use the notation to denote the valuation on $\mathcal{D}_G$ associated to $\omega$ as in section \ref{sec:recoLazard}. Let $\log_{\mathcal{D}_G}$ be the logarithm map on $\mathcal{D}_G$:
\[ \set*{x\in\mathcal{D}_G : w(x-1)>\frac{1}{p-1}}\longrightarrow\set*{x\in\mathcal{D}_G : w(x)>\frac{1}{p-1}}.\]
The inclusion $G\subset \M_2(\mathcal{O}_{L'})$ is continuous and extends to a continuous morphism of $\Zp$-algebras $h : \Zp[G]\rightarrow \M_2(\mathcal{O}_{L'})$ and a morphism of $\Qp$-algebras $\Qp[G]\rightarrow \M_2(L')$. By definition of $w$, we have the inequality $w(x)\leq v(\Ad(D^{-1})h(x))$ for $x\in\Zp[G]$, since $v\circ\Ad(D^{-1})\circ h$ is a valuation $w'$ on $\Zp[G]$ such that $w'(x-1)=\omega(x)$ for $x\in G$ and $w$ is defined as the pointwise infimum of valuations $w''$ with $w''(x-1) \ge \omega(x)$ for $x \in G$. As $w$ and $v$ are valuations of $\Qp$-algebras, we deduce that this inequality is true for all $x\in \Qp[G]$. As $\M_2(L')$ is complete, we can extend $h$ to a morphism of valued $\Qp$-algebras $(\mathcal{D}_G,w)\rightarrow (\M_2(L'),v\circ\Ad(D)^{-1})$. Now, by continuity of $h$, the composite
\[ G\xrightarrow{\log_{\mathcal{D}_G}}\mathcal{D}_G\xrightarrow{h} \M_2(L')\]
is the logarithm computed in $\M_2(L')$. This implies that the restriction of $h$ to $\log_{\mathcal{D}_G}(G)$ is an isomorphism of Lie algebras
\begin{equation}
  \cL_G = \log_{\mathcal{D}_G}(G)\simeq \log_{\M_2(L')}(G).\label{eq:9}
\end{equation}
Finally both valuations $w$ and $v\circ\Ad(D)^{-1}$ take value $\omega(x)$ at $x-1$ for $x\in G$. By \cite[III.1.1.5]{Lazard} the condition $\omega(x)>\tfrac{1}{p-1}$ for $x\in G$ implies then
\[ w(\log_{\mathcal{D}_G}(x))=\omega(x)=v(\Ad(D^{-1})\log_{\M_2(L')}(x)),\]
proving that~\eqref{eq:9} is an isomorphism of \emph{valued} Lie algebras. The conclusion follows from \eqref{eq:isoLielog} and from the fact that the valuation $v\circ\Ad(D^{-1})$ restricted to $\log_{\M_2(L')}(G)=\mathcal{O}_L\otimes_{\Zp}\mathfrak{g}_{\Zp}$ is given by \eqref{eq:explvaluation}.
\end{proof}

We endow the Lie algebra $\mathfrak{g}_{\Zp}$ with the restriction of the valuation $w$ and we let $\mathfrak{g}\defeq \gr\mathfrak{g}_{\Zp}$. The Lie algebra $\cL_{G}$ is an $\mathcal{O}_L$-Lie algebra and, for $a\in\mathcal{O}_L$ and $x\in \cL_{G}$, we have $w(ax)=v_p(a)+w(x)$. Hence the graded $\F_p[\eps]$-Lie algebra $\gr G \cong \gr \cL_G$ has the structure of a $k[\eps]$-graded Lie algebra and is isomorphic to $k\otimes_{\F_p}\mathfrak{g}$. Consequently the graded $\F_p$-Lie algebra $\overline{\gr G}=\gr G\otimes_{\F_p[\eps]}\F_p$ is isomorphic to $k\otimes_{\F_p}\overline{\mathfrak{g}}$, where $\overline{\mathfrak{g}}\defeq\Fp\otimes_{\Fp[\eps]}\mathfrak{g}$, and has a natural structure of graded $k$-Lie algebra.

We want to show that $\gr\Fp\bbra{G}$, defined by the valuation $\overline{w}$ associated to $\omega$, and $\gr_\mathfrak{m}\Fp\bbra{G}$ (the graded ring for the $\mathfrak{m}_G$-adic filtration of $\Fp\bbra{G}$)
are isomorphic up to rescaling indices. We will need the following lemma:

\begin{lem}\label{lem:usualmodpcomputation}
Let $G$ be a pro-$p$-group. Then for $g$ and $h$ in $G$, we have
\[ gh-1\equiv (g-1)+(h-1) \mod \mathfrak{m}_G^2, \quad (g^{-1}-1)\equiv -(g-1)\mod \mathfrak{m}_G^2\]
in $\Fp\bbra{G}$. Moreover if $g\in G$, $(g^p-1)\in\mathfrak{m}_G^p$.
\end{lem}

\begin{proof}
The first two assertions are consequences of the equality $(g-1)(h-1)=(gh-1)-(g-1)-(h-1)$ and from the fact that $g-1\in\mathfrak{m}_G$. The last one comes from $(g^p-1)=(g-1)^p$.
\end{proof}

\begin{prop}\label{comparison_filtrations}
We have, for $j\in\tfrac{1}{2}\NN$,
\[ \mathfrak{m}_{G}^{2j}=\set{x\in\Fp\bbra{G} : \overline{w}(x)\geq j}.\]
\end{prop}

\begin{proof}
Let $a\in\mathcal{O}_L$ such that $\F_p[a]=k$, hence $\mathcal{O}_L=\Zp[a]$. Using Lemma \ref{LieIwahori} (and its proof) we see that we can choose an ordered basis $(x_1,\dots,x_{3f})$ of $G$ whose elements are
\[ E_i=\left(\begin{smallmatrix} 1 & a^i \\ 0 & 1 \end{smallmatrix}\right), \quad F_i=\left(\begin{smallmatrix}
1 & 0 \\ pa^i & 1 \end{smallmatrix}\right), \quad H_i=\left(\begin{smallmatrix}
(1-a^ip)^{-1} & 0 \\ 0 & 1-a^ip \end{smallmatrix}\right)\]
for $0\leq i \leq f-1$.

For $j\in\tfrac{1}{2}\NN$, $\set{x\in\Fp\bbra{G} : \overline{w}(x)\geq j}$ is the ideal generated by monomials $z^{\alpha}=\prod_{i=1}^{3f}(x_i-1)^{\alpha_i}$ with $\tau(\alpha)=\sum_{i=1}^{3f}\omega(x_i)\alpha_i\geq j$. For $0\leq i\leq f-1$, we have $E_i-1\in\mathfrak{m}_{G}$, $F_i-1\in\mathfrak{m}_G$. Let's prove that $H_i-1\in\mathfrak{m}_G^2$. We have
\[ E_iF_0E_i^{-1}F_0^{-1}=H_i\begin{pmatrix}
1 &-(1-pa^i)a^{2i}\\0 & 1\end{pmatrix}^p \begin{pmatrix}
1 & 0 \\ pa^i(1-pa^i)^{-1} & 1
\end{pmatrix}^p.\]
Using Lemma \ref{lem:usualmodpcomputation}, this implies that
\[ E_iF_0E_i^{-1}F_0^{-1}-1\equiv H_i-1\mod \mathfrak{m}_G^2\]
and finally that
\begin{align*}
H_i-1&\equiv E_i-1+F_0-1-(E_i-1)-(F_0-1)\mod\mathfrak{m}_G^2 \\
&\equiv 0\mod\mathfrak{m}_G^2.
\end{align*}
Since $\omega(E_i)=\omega(F_i)=1/2$ and $\omega(H_i)=1$, this proves that $z^{\alpha}\in\mathfrak{m}_G^{2j}$ when $\tau(\alpha)\geq j$, i.e.\ $\set{x\in\Fp\bbra{G} : \overline{w}(x)\geq j}\subset\mathfrak{m}_G^{2j}$. 

Noticing that $\mathfrak{m}_{G}=\set{x\in\Fp\bbra{G} : w(x)\geq 1/2}$, we have, conversely,
\[ \mathfrak{m}_G^j\subset\set{x\in\Fp\bbra{G} : \overline{w}(x)\geq 1/2}^j\subset\set{x\in\Fp\bbra{G} : \overline{w}(x)\geq j/2},\]
the last inclusion being deduced from the properties of a valuation.
\end{proof}

Proposition \ref{comparison_filtrations} suggests that we should rescale the gradings of $\mathfrak{g}$ and $\overline{\mathfrak{g}}$ by replacing the valuation $w$ on $\mathfrak{g}_{\Zp}$ with $2w$, and this is what we do from now on. 
Therefore, the multiplication by $\eps$ on $\mathfrak{g}$ now has degree $2$. We deduce from Proposition \ref{comparison_filtrations} and isomorphism \eqref{eq:iso-gradedmodp-algebra} that we have
an isomorphism of $\Fp$-Lie algebras
\begin{equation}\label{eq:iso-maxgradedmodp-algebra} \gr_\mathfrak{m}\Fp\bbra{G}\simeq U_{\Fp}(k\otimes_{\Fp}\overline{\mathfrak{g}}).\end{equation}

We now determine $\overline{\mathfrak{g}}$ explicitly. The $\Zp$-Lie algebra $\mathfrak{g}_{\Zp}$ has a $\Zp$-basis given by
\[ e=\begin{pmatrix}0 & 1 \\ 0 & 0 \end{pmatrix}, \quad f=\begin{pmatrix}
0 & 0 \\ p & 0 \end{pmatrix}, \quad h=\begin{pmatrix}
p & 0 \\ 0 & -p \end{pmatrix}\]
with relations
\[ [e,f]=h, \quad [h,e]=2pe, \quad [h,f]=-2pf\]
and valuations $2w(e)=2w(f)=1$, $2w(h)=2$. Hence the graded $\F_p[\eps]$-Lie algebra $\mathfrak{g} = \gr \mathfrak{g}_{\Zp}$ is
\[ \mathfrak{g}=\F_p[\eps]e\oplus\F_p[\eps]f\oplus\F_p[\eps]h\]
with $e$ and $f$ in degree $1$ and relations
\[ [e,f]=h,\quad [h,e]=2\eps e, \quad [h,f]=-2\eps f,\]
and the graded $\Fp$-Lie algebra $\overline{\mathfrak{g}}$ is
\[ \overline{\mathfrak{g}}=\F_pe\oplus\F_pf\oplus\F_ph\]
with $e$ and $f$ in degree $1$, $h$ in degree $2$ and relations
\begin{equation}\label{commutationrelations} [e,f]=h,\quad [h,e]=[h,f]=0.\end{equation}

Let $H$ be the (prime-to-$p$) torsion subgroup of the diagonal torus of $\GL_2(\mathcal{O}_L)$. Then $H$ is a finite subgroup of the ``upper'' Iwahori subgroup $I$ of $\GL_2(\mathcal{O}_L)$. It normalizes $I_1$ and $G$. Therefore the group $H$ acts on every object considered so far: $\Fp\bbra{G}$, $\cL_G$, $\mathfrak{g}$, $\overline{\mathfrak{g}}$, \dots\ and the isomorphism \eqref{eq:iso-maxgradedmodp-algebra} is equivariant for this action of $H$.  %
Note that the action of $H$ on $\cL_G$, $\mathfrak{g}$ and $\overline{\mathfrak{g}}$ is $k$-linear. More precisely, we have, for $g=\left(\begin{smallmatrix} a & 0 \\ 0 & d \end{smallmatrix}\right)\in H$, and $\alpha\in k$:
\[ g(\alpha\otimes e)=(ad^{-1}\alpha)\otimes e,\quad g(\alpha\otimes f)=((ad^{-1})^{-1}\alpha)\otimes f,\quad g(\alpha\otimes h)=\alpha\otimes h.\]

Let $\F$ be a field of characteristic $p$. %
Recall from the introduction that if $\F$ is an extension of $\F_p$ such that $k$ embeds into $\F$, we label the embeddings $\sigma_j = \sigma_0 \circ \varphi^{j}$, so the set $\mathcal{J}$ of embeddings $k\hookrightarrow\F$ is identifed with $\set{0,\dots,f-1}$. 
In this case, for $0\leq j\leq f-1$, we define $\mathfrak{g}_j\defeq \F \otimes_{\sigma_j,k} \gr G$ and $\overline{\mathfrak{g}}_j\defeq \F \otimes_{\sigma_j,k} \ovl{\gr G}$. Then we have a decomposition
\begin{equation*}\label{splittingLie} \F\otimes_{\F_p}\overline{\gr G}\simeq\bigoplus_{j=0}^{f-1}\overline{\mathfrak{g}}_j\end{equation*}
and canonical isomorphisms $\mathfrak{g}_j\simeq\F\otimes_{\F_p}\mathfrak{g}$ as well as $\overline{\mathfrak{g}}_j\simeq\F\otimes_{\F_p}\overline{\mathfrak{g}}$. Using also~\eqref{eq:iso-maxgradedmodp-algebra} we deduce an isomorphism of graded $\F$-algebras
\begin{equation}\label{isograded} \gr_\mathfrak{m}\F\bbra{G}\simeq\F\otimes_{\F_p}\gr_\mathfrak{m}\F_p\bbra{G}\simeq\bigotimes_{j=0}^{f-1}U_{\Fp}(\overline{\mathfrak{g}}_j)\simeq U_{\Fp}(\overline{\mathfrak{g}})_\F^{\otimes f}.\end{equation}
For $0\leq j\leq f-1$ let $e_j,f_j,h_j \in \ovl{\mathfrak{g}}_j$ denote the images of $1\otimes e, 1\otimes f, 1\otimes h$ under the isomorphism $\F\otimes_{\F_p}\overline{\mathfrak{g}}\simeq \overline{\mathfrak{g}}_j$. Then we have, for $g=\left(\begin{smallmatrix} a & 0 \\ 0 & d \end{smallmatrix}\right)\in H$, and for $0\leq j\leq f-1$,
\[ ge_j=\sigma_j(ad^{-1})e_j,\quad gf_j=\sigma_j(ad^{-1})^{-1}f_j,\quad gh_j=h_j.\]

 Let $I_G$ be the left ideal of $\gr_\mathfrak{m}\F\bbra{G}$ generated by the elements $(1\otimes e)(1\otimes f)$ and $1\otimes h$ (of degree $2$). We easily see that $I_G$ is in fact a $2$-sided ideal of $\gr_\mathfrak{m}\F\bbra{G}$. If $k$ embeds in $\F$, then $I_G$ is the left ideal generated by $(e_jf_j,h_j;\, 0\leq j\leq f-1)$ via the isomorphism \eqref{isograded}. 
 
\begin{thm}\label{quotientalg}
Let $\F$ be a field of characteristic $p$. The graded ring $\gr_\mathfrak{m}\F\bbra{G}$ is Auslander-regular and $(\gr_\mathfrak{m}\F\bbra{G})/I_G$ is a commutative Cohen--Macaulay $\F$-algebra of dimension $f$. More precisely, if we assume moreover that $k$ embeds in $\F$, then

\begin{enumerate}
\item the sequence $(h_0,\dots,h_{f-1})$ is a regular sequence of central elements of $\gr_\mathfrak{m}\F\bbra{G}$ and $\gr_\mathfrak{m}\F\bbra{G}/(h_0,\dots,h_{f-1})$ is isomorphic to $\F[e_j,f_j ;\, 0\leq j\leq f-1]$, a polynomial ring in $2f$ variables;
\item we have an isomorphism
\[ (\gr_\mathfrak{m}\F\bbra{G})/I_G\simeq\F[e_j,f_j ;\, 0\leq j\leq f-1]/(e_jf_j ;\, 0\leq j\leq f-1).\]
\end{enumerate}
\end{thm}

\begin{proof}
By \cite[\S III.2.4.4]{LiOy}, the graded ring $\gr_\mathfrak{m}\F\bbra{G}$ is Auslander-regular since it is isomorphic to an enveloping algebra. Assume now that $k$ embeds in $\F$.

(i) It follows from \eqref{commutationrelations} that $h_0,\dots,h_{f-1}$ are central elements of $\gr_\mathfrak{m}\F\bbra{G}$. For $0\leq i\leq f-1$, the ring $(\gr_\mathfrak{m}\F\bbra{G})/(h_0,\dots,h_i)$ is isomorphic to the enveloping algebra of the quotient of the Lie algebra $\F\otimes_{\Fp}\overline{\gr G}$ by the ideal generated by $h_0,\dots,h_i$ and is therefore a ring without zero divisors by the Poincar\'e--Birkhoff--Witt Theorem. This proves that $h_{i+1}$ is a regular element of $(\gr_\mathfrak{m}\F\bbra{G})/(h_0,\dots,h_i)$ and that $(h_0,\dots,h_{f-1})$ is a regular sequence of central elements of $\gr_\mathfrak{m}\F\bbra{G}$. The last assertion is clear by \eqref{commutationrelations}.

(ii) Using the isomorphism of $\F$-algebras
\[ (\gr_\mathfrak{m}\F\bbra{G})/I_G\simeq\bigotimes_{0\leq j\leq f-1} (U_{\Fp}(\overline{\mathfrak{g}}_j)/(e_jf_j,h_j)),\]
the assertion is a consequence of (i).
The sequence $(e_jf_j ;\, 0\leq j\leq f-1)$ is a regular sequence in $\F[e_j,f_j ;\, 0\leq j\leq f-1]$, so the ring $(\gr_\mathfrak{m}\F\bbra{G})/I_G$ is Cohen--Macaulay of dimension $f$.

In general (if $k$ does not embed in $\F$), we can find a finite extension $\F'/\F$ such that $k$ embeds in $\F'$. By what precedes, the ring $\F'\otimes_{\F}((\gr_\mathfrak{m}\F\bbra{G})/I_G)\simeq\gr_{\mathfrak{m}}(\F'\bbra{G}/(\F'\otimes_{\F}I_G))$ is Cohen--Macaulay of dimension $f$, hence so is $(\gr_\mathfrak{m}\F\bbra{G})/I_G$ \cite[Cor.\ (6.7.8)]{EGAIV}.
\end{proof}

\begin{cor}\label{cor:GKdim}
Let $\pi$ be an admissible smooth representation of $I/Z_1$ over $\F$. %
Assume that for each character such that $\Hom_I(\chi,\pi)\neq0$, the natural injection
\[\Hom_I(\chi,\pi)\into\Hom_{I}(W_{\chi,3},\pi)\]
is an isomorphism, where $W_{\chi,3}$ is defined in~\eqref{eq:W-chi-n}
below. Then $\dim_I(\pi)=\dim_{I/Z_1}(\pi)\leq f$.
\end{cor}

\begin{proof}
By increasing $\F$ we may assume that $k$ embeds in $\F$. As $\pi$ is an admissible representation of $I/Z_1$, it is an admissible representation of $G\simeq I_1/Z_1$ and $\pi^\vee$ is a finitely generated $\F\bbra{G}$-module. Moreover the socle filtration on $\pi$ coincides with the socle filtration on $\pi|_G$ and with the dual of the $\mathfrak{m}_G$-adic filtration on $\pi^\vee$ so that $(\soc_i \pi/\soc_{i-1}\pi)^\vee\simeq\gr_{\mathfrak{m}}^i\pi^\vee$. Moreover the graded $\gr_\mathfrak{m}\F\bbra{G}$-module $\gr_\mathfrak{m}\pi^\vee$ is generated by its homogeneous elements of degree $0$.

Let $I_G$ be the graded ideal of $\gr_\mathfrak{m}\F\bbra{G}$ defined above and let $I_G^{(2)}$ be its homogeneous component of degree $2$. Note that $H$ acts trivially on $I_G^{(2)}$.
If $\Hom_I(\chi,\gr_\mathfrak{m}^0\pi^\vee)\neq0$, then by assumption $\Hom_I(\chi,\gr_\mathfrak{m}^2\pi^\vee)=0$, so we have $I_G^{(2)}(\gr_\mathfrak{m}^0\pi^\vee)=0$. As $\gr_\mathfrak{m}\pi^\vee$ is generated by $\gr_\mathfrak{m}^0\pi^\vee$ and $I_G$ by $I_G^{(2)}$, we deduce that $I_G(\gr_\mathfrak{m}\pi^\vee)=0$ and that $\gr_\mathfrak{m}\pi^\vee$ is actually a $\gr_\mathfrak{m}\F\bbra{G}/I_G$-module. Theorem \ref{quotientalg} implies that the dimension of its support is $\leq f$. We can therefore apply Lemma \ref{prop:grsemiabelian} (with $I=(h_0,\dots,h_{f-1})$) to conclude that $\dim_{I/Z_1}(\pi)=\dim_G(\pi)\leq f$. The equality $\dim_I(\pi)=\dim_{I/Z_1}(\pi)$ follows from Lemma \ref{lem:GvsG/Z}.
\end{proof}

Using \eqref{isograded} and the Poincar\'e--Birkhoff--Witt Theorem, we can write down explicitly the structure of the first three graded pieces of $\gr_\mathfrak{m}\F\bbra{I_1/Z_1}$ as $I$-representations, assuming that $k$ embeds in $\F$:
\begin{gather}
\begin{gathered}
\label{eq:expldescription} \gr_\mathfrak{m}^0\F\bbra{I_1/Z_1}=\F, \quad \gr_\mathfrak{m}^1\F\bbra{I_1/Z_1}\simeq\bigoplus_{i=0}^{f-1}(\F\alpha_i\oplus\F\alpha_i^{-1}), \\ \gr_\mathfrak{m}^2\F\bbra{I_1/Z_1}\simeq\F^{2f}\oplus\bigoplus_{0\leq i \le j\leq f-1}\F\alpha_i\alpha_j\oplus \bigoplus_{0\leq i \le j\leq f-1}\F\alpha_i^{-1}\alpha_j^{-1}\oplus \bigoplus_{0\leq i \neq j\leq f-1}\F\alpha_i\alpha_j^{-1},
\end{gathered}
\end{gather}
where $\alpha_j$ is the character $\left(\begin{smallmatrix} a & 0 \\ 0 & d \end{smallmatrix}\right)\mapsto\sigma_j(ad^{-1})$. As a consequence, each nontrivial character appears with multiplicity at most one as a Jordan--H\"older factor of $\F\bbra{I_1/Z_1}/\mathfrak{m}_{I_1/Z_1}^3$.

\section{On smooth representations of \texorpdfstring{$\GL_2$}{GL\_2}}
\label{sec:smooth:rep}

The aim of this section is to prove Theorem \ref{thm:GKdim-criterion} below which provides a useful criterion for bounding the Gelfand--Kirillov dimension of an admissible smooth representation of $\GL_2(L)$.

 We keep the notation of \S\ref{sec:propIwahoriGL2}: $L$ is a finite unramified extension of $\Qp$ of degree $f$ with ring of integers $\mathcal{O}_L$ and residue field $k$, $I$ (resp.~$I_1$) is the upper (resp.~upper pro-$p$) Iwahori subgroup of $K\defeq \GL_2(\mathcal{O}_L)$ and $Z_1$ is the center of $I_1$. We set $K_1\defeq 1+p\M_2(\mathcal{O}_L)\subset I_1$.

If $H$ is a compact $p$-adic analytic group and if $V$ is an admissible smooth $\F$-representation of $H$ we denote by $\Inj_H V$ an injective envelope of $V$ in the category of admissible smooth representations of $H$; it is unique up to \emph{nonunique} isomorphism. As an $\F\bbra{H}$-module, the dual $V^\vee$ is finitely generated and we denote by $\Proj_H V^\vee$ a projective cover of $V^\vee$ in the category of pseudocompact $\F\bbra{H}$-modules. The radical $\rad M$ of a pseudocompact $\F\bbra{H}$-module is the submodule $\mathfrak{m}_H M$.

If $G$ is a $p$-adic analytic group, $H$ a closed subgroup of $G$ and $V$ a smooth $H$-representation over $\F$, we denote by $\Ind_H^GV$ the $\F$-vector space of smooth functions $f : G\rightarrow V$ such that $f(hg)=hf(g)$ for all $g\in G$ and $h\in H$. The group $G$ acts on $\Ind_H^GV$ by translation on the right. If $H$ is cocompact in $G$, the representation $\Ind_H^GV$ is smooth and if moreover $V$ is admissible, it is admissible.

If $\lambda\in X^*(\un{T})$ we use the notation $\chi_\lambda$ to denote the character $T(k)\to \un{T}(\F) \xrightarrow{\lambda}\F^\times$, where the first map is the inclusion. We use the same notation $\chi_\lambda$ to denote the character of $I$ obtained by composition with $I\twoheadrightarrow T(k)$. Equivalently $\chi_\lambda$ is the character of $I$ acting on $F(\lambda)^{I_1}$. 

In this section, we always assume that $p>3$.

\subsection{On some %
 representations of the Iwahori}
\label{sec:some-repr-iwah}

\

Let $\alpha_i : T(k) \to \F^\times$ denote also the character $\chi_{\alpha_i}$, i.e.~the character sending $\left(\begin{smallmatrix} a & 0 \\ 0 & d \end{smallmatrix}\right)\in T(k)$ to $\sigma_i(ad^{-1})$.
In particular, $\alpha_i=\alpha_0^{p^i}$ as characters of $T(k)$ for $0\leq i\leq f-1$. 

Let $\chi:I\ra\F^{\times}$ be a smooth character. For any $n\geq 1$, we set
\begin{equation}\label{eq:W-chi-n} W_{\chi,n}\defeq(\Proj_{I/Z_1}\chi)/\m_{I_1}^n.\end{equation}
(Note that via the natural map $\F\bbra{I}\rightarrow\F\bbra{I/Z_1}$ the actions of $\m_{I_1}^n$ and $\m_{I_1/Z_1}^n$ coincide on $\Proj_{I/Z_1}\chi$; similar comment will apply later on for pseudocompact $\F\bbra{K/Z_1}$-modules.)

Let $\chi_0$ be the trivial character of $I$. As any smooth character $\chi : I\rightarrow\F^\times$ is trivial on $I_1$, there is an isomorphism of $\F\bbra{I/Z_1}$-modules
\[ \Proj_{I/Z_1}\chi\simeq\chi\otimes_{\F}\Proj_{I/Z_1}\chi_0\]
and an isomorphism of $\F\bbra{I/Z_1}$-modules $\Proj_{I/Z_1}\chi_0\simeq\F\bbra{I_1/Z_1}$. 
(Note that the decomposition $I= I_1\rtimes H$ with $H$ as in \S\ref{sec:propIwahoriGL2} gives a natural left action of $I$ on $\F\bbra{I_1/Z_1}$, where $I_1$ acts by left translation and $H$ by conjugation.)
Consequently for any $n\geq1$, we have an isomorphism of $I$-representations $W_{\chi,n}\simeq\chi\otimes_{\F}(\F\bbra{I_1/Z_1}/\mathfrak{m}_{I_1}^n)$. From the description of $\gr_{\mathfrak{m}}\F\bbra{I_1/Z_1}$ in \eqref{eq:expldescription}, we can deduce the following result.

\begin{lem}
\label{lemma-Ext1=dim1}
We keep the above hypotheses.
\begin{enumerate}%
\item 
\label{it2-Ext1=dim1}
For any $\chi'\neq \chi$, $[W_{\chi,3}:\chi']\leq 1$.
\item\label{it1-Ext1=dim1}
Suppose that $\chi, \chi': I\ra\F^{\times}$ are smooth characters such that $\Ext^1_{I/Z_1}(\chi,\chi')\neq0$.
Then $\chi'\in \{\chi\alpha_i^{\pm1} : 0\leq i\leq f-1\}$ and we have $\dim_{\F}\Ext^1_{I/Z_1}(\chi,\chi')=1$.
Letting $E_{\chi',\chi}$ denote the unique nonsplit $I$-extension 
\begin{equation}
  0\ra \chi'\ra E_{\chi',\chi}\ra\chi\ra0,\label{eq:10}
\end{equation}
the group $K_1$ acts trivially on $E_{\chi',\chi}$ if and only if $\chi'=\chi\alpha_i$ for some $0\leq i\leq f-1$.
\end{enumerate}
\end{lem}

\begin{proof}
Part \ref{it2-Ext1=dim1} follows from equation \eqref{eq:expldescription} by twisting and part \ref{it1-Ext1=dim1} follows from \cite[Lemma~2.4]{yongquan-algebra} (i) and (ii).
\end{proof}

Now, let $\chi'$ be a character such that $\Ext^1_{I/Z_1}(\chi,\chi')\neq0$. Since $[W_{\chi,3}:\chi']=1$ and $\chi'$ occurs as a subquotient in $\rad_{I_1}(W_{\chi,3})$ which is killed by $\mathfrak{m}_{I_1}^2$, there is a unique (up to scalar) nonzero $I$-equivariant morphism $W_{\chi',2}\ra W_{\chi,3}$.

\begin{lem}\label{lemma-chi''isinimage}
If $\Ext^1_{I/Z_1}(\chi,\chi')\neq0$, then any nonzero $I$-equivariant  morphism $W_{\chi',2}\ra W_{\chi,3}$ is injective.
\end{lem}

\begin{proof}
By twisting, it is sufficient to consider the case where $\chi$ is the trivial character $\chi_0$. In this case, there is an $I$-equivariant isomorphism $\F\bbra{I_1/Z_1}\simeq \Proj_{I/Z_1}\chi_0$. Let $e\in \gr_{\mathfrak{m}}^1\F\bbra{I_1/Z_1}$ be an eigenvector of weight $\chi'$. There is a unique degree $1$ morphism of graded $\gr_\mathfrak{m}\F\bbra{I_1/Z_1}$-modules $f : \gr_\mathfrak{m}\F\bbra{I_1/Z_1}\rightarrow\gr_\mathfrak{m}\F\bbra{I_1/Z_1}$ sending $1$ to $e$. As $\gr_\mathfrak{m}\F\bbra{I_1/Z_1}$ is isomorphic to an enveloping algebra over a field by \eqref{isograded}, the Poincar\'e--Birkhoff--Witt Theorem implies that it has no zero divisor so that the map $f$ is injective. Let $\tilde{e}\in \fm_{I_1/Z_1}$ such that $\gr_\mathfrak{m}(\tilde{e})=e$. We define a degree 1 morphism of filtered $\F\bbra{I_1/Z_1}$-modules $\tilde{f} : \F\bbra{I_1/Z_1}\rightarrow\F\bbra{I_1/Z_1}$ sending $x$ to $x\tilde{e}$. Obviously we have $f=\gr_\mathfrak{m}(\tilde{f})$. Moreover, if we choose for $\tilde{e}$ a $\chi'$-eigenvector for the action of the group $H$, then $\tilde{f}$ induces an $H$-equivariant map $\tilde{f}': \chi'\otimes_{\F}\F\bbra{I_1/Z_1}\rightarrow\F\bbra{I_1/Z_1}$. As $I= I_1\rtimes H$, the map $\tilde{f}'$ is $I$-equivariant. 
Since $\tilde{f}'$ is injective on graded modules for the $\mathfrak{m}_{I_1}$-adic filtration, it induces an $I$-equivariant injective map
\[ W_{\chi',2}=\Proj_{I/Z_1}\chi'/\mathfrak{m}_{I_1}^2\hookrightarrow\Proj_{I/Z_1}\chi_0/\mathfrak{m}_{I_1}^3=W_{\chi_0,3}.\]
Since $\Hom_{I}(W_{\chi',2},W_{\chi_0,3})$ has dimension $1$, this finishes the proof.
\end{proof}

For an integer $0 \le \ell \le q-1$ we let $\ell_i$ denote the $i$-th base $p$ digit of $\ell$, so $\ell = \sum_{i=0}^{f-1} \ell_i p^i$.

\begin{lem}\label{lm:inj-env-borel}
  Let $\cI_\chi \defeq \Inj_{B(k)}\chi$. Then $\cI_\chi$ has socle and cosocle isomorphic to $\chi$, and its remaining Jordan--H\"older factors
  $\chi \alpha_0^{-j}$, $0 < j < q-1$,
  occur with multiplicity 1. Its submodule structure is determined by the following property: the unique proper submodule of
  $\cI_\chi$ with cosocle $\chi \alpha_0^{-j}$ \emph{(}$0 \le j < q-1$\emph{)}
  has Jordan--H\"older factors $\chi \alpha_0^{-\ell}$, where $0 \le \ell < q-1$ and $\ell_i \le j_i$ for all $i$.
\end{lem}

\begin{proof}
  The claim about socle and cosocle are true for injective envelopes of any finite group.

  We first observe that $\cI_\chi \cong \Ind_{T(k)}^{B(k)} \chi$. The latter representation is injective by Frobenius reciprocity
  (as any $T(k)$-representation is injective). It has the correct socle and cosocle by Frobenius reciprocity, hence indeed
  $\cI_\chi \cong \Ind_{T(k)}^{B(k)} \chi$.
  
  As the kernel of $B(k) \onto T(k)$ is a normal $p$-subgroup, every irreducible $B(k)$-representation is trivial on it. To
  determine Jordan--H\"older factors we may thus restrict to $T(k)$.  By Mackey's formula,
  $(\Ind_{T(k)}^{B(k)} \chi)|_{T(k)} \cong \chi \oplus (\Ind_{Z(k)}^{T(k)} \chi)|_{Z(k)}$, where $Z$ is the center of $\GL_2$. Thus the
  irreducible constituents of $\cI_\chi$ are all the characters $\chi'$ of $T(k)$ such that $\chi'|_{Z(k)} = \chi|_{Z(k)}$, or
  equivalently $\chi' = \chi \alpha_0^{-j}$ for some $0\leq j< q-1$, as well as one more copy of $\chi$.

  As in \cite[\S2]{BP} we define
  $f_j \defeq \sum_{\lambda\in k} \lambda^j \big(\begin{smallmatrix}1 & \lambda \\ 0 & 1 \end{smallmatrix}\big) \phi$,
  where $\phi\in \Ind_{T(k)}^{B(k)} \chi$ is some function whose support equals $T(k)$.
  It follows that $f_j$ is a $T(k)$-eigenvector with eigenvalue $\chi \alpha_0^{-j}$.

  Assume now that $j < q-1$.
  An explicit calculation shows that
  $\big(\begin{smallmatrix}1 & x \\ 0 & 1 \end{smallmatrix}\big) f_j = \sum_{\ell=0}^j \binom{j}{\ell} (-x)^{j-\ell} f_\ell$. Hence
  the $B(k)$-representation $W$ generated by $f_j$ has basis $f_\ell$ for $\ell$ such that $\binom{j}{\ell} \ne 0$ or equivalently
  $\ell_i \le j_i$ for all $i$. In particular, $W \ne \cI_\chi$ since $j< q-1$. On the other hand, $W$ is a quotient of
  $\Ind_{T(k)}^{B(k)} \chi \alpha_0^{-j}$, so $W$ is the unique proper subrepresentation of $\cI_\chi$ with cosocle
  $\chi \alpha_0^{-j}$.
\end{proof}

The element $\big(\begin{smallmatrix}0 & 1 \\ p & 0 \end{smallmatrix}\big)\in\GL_2(L)$
normalizes $I$ and its square is central. Let $\chi^s$ denote the conjugate of $\chi$ by $\big(\begin{smallmatrix}0 & 1 \\ p & 0 \end{smallmatrix}\big)\in\GL_2(L)$. By conjugating $\cI_\chi$ by $\big(\begin{smallmatrix}0 & 1 \\ p & 0 \end{smallmatrix}\big)\in\GL_2(L)$ we obtain
the following corollary.

\begin{cor}\label{cor:iwahori-reps}
  Given $\chi : T(k) \to \F^\times$ there is a \emph{(}finite-dimensional\emph{)}
  smooth representation $\cJ_\chi$ of $I$ with the following properties.  The socle and cosocle of $\cJ_\chi$ are isomorphic
  to $\chi^s$, and the remaining Jordan--H\"older factors of $\cJ_\chi$ are $\chi^s \alpha_0^{j}$ for $0 < j < q-1$,
  each occurring with multiplicity 1. The unique proper
  submodule of $\cJ_\chi$ with cosocle $\chi^s \alpha_0^{j}$ \emph{(}$0 \le j < q-1$\emph{)}
  has Jordan--H\"older factors $\chi^s \alpha_0^{\ell}$, where $0 \le \ell < q-1$ and $\ell_i \le j_i$ for all $i$.
  Moreover, $\cJ_\chi$ admits a central character.
\end{cor}

\begin{rem}
On $\mathcal{J}_\chi$ the action of $I$ does not factor through its quotient $B(k)$, contrary to the case $\mathcal{I}_\chi$
(cf.\ Lemma~\ref{lemma-Ext1=dim1}).
\end{rem}

\subsection{On some indecomposable representations of \texorpdfstring{$K$}{K}}
\label{sec:some-repr-K}

\

We will use again the notation of section \ref{sec:ext:graph}. 
In particular, recall that we have identified $\cJ=\Hom(k,\F)$ with $\set{0,1,\dots, f-1}$ and that $\eta_J\defeq\sum_{i\in J}\eta_i$ for $J\subseteq \cJ$. %
Also, for $\lambda\in X^*(\un{T})$ recall the injective map
\[ \t_\lambda : \Lambda^\lambda_W \rightarrow X_{\reg}(\un{T})/(p-\pi)X^0(\un{T}). \]

Let $\sigma'$ be a Serre weight appearing in $\Inj_{\GL_2(k)}F(\lambda)$. 
It follows from \cite[Cor.\ 3.12]{BP} that there exists a unique subrepresentation of $\Inj_{\GL_2(k)}F(\lambda)$, denoted by $I(F(\lambda),\sigma')$, with cosocle $\sigma'$ and such that $[I(F(\lambda),\sigma'):F(\lambda)]=1$. Moreover, $I(F(\lambda),\sigma')$ is multiplicity-free. As a consequence, if $W$ is a subrepresentation of $\Inj_{\GL_2(k)}F(\lambda)$ such that $[W:\sigma']\neq0$, then $W$ contains $I(F(\lambda),\sigma')$ as a subrepresentation. 
Dually, we have similar statements for quotients of $\Proj_{\GL_2(k)}F(\lambda)$.

\begin{lem}\label{lm:princ-series}
We keep the above hypotheses.
\begin{enumerate}%
\item 
\label{it:princ-series-1}
Suppose that $0 < \langle\lambda,\alpha_i^\vee\rangle < p-1$ for all $i$. Then $\Ind_{I}^{K}\chi_\lambda^s$ is multiplicity-free with Jordan--H\"older factors $\{F(\t_\lambda(-\ovl{\eta}_J)) : J \subset \cJ\}$. %
\item
\label{it:princ-series-2}
Suppose that $0 < \langle\lambda,\alpha_i^\vee\rangle < p-2$ for all $i$. The Jordan--H\"older factors of $\Inj_{\GL_2(k)}F(\lambda)$ are the $\set{F(\t_\lambda(\sum_{i\in \cJ}a_i\ovl{\eta}_i)) : (a_i)_{i\in \cJ}\in \set{0,\pm1}^\cJ}$, up to multiplicity.
\item
\label{it:princ-series-3}
Suppose that $0 < \langle\lambda,\alpha_i^\vee\rangle < p-2$ for all $i$.
Let $\sigma'=F(\t_{\lambda}(\sum_{i\in\cJ}a_i\ovl{\eta}_i))$ for some $(a_i)\in\{0,\pm1\}^{\cJ}.$ The Jordan--H\"older factors of $I(F(\lambda),\sigma')$ are $\big\{F(\t_{\lambda}(\sum_{i\in J}a_i\ovl{\eta}_i)): J\subset\cJ\big\}.$ As a consequence, the length of $I(F(\lambda),\sigma')$ is equal to $2^{|\{i\in\cJ:a_i\neq0\}|}$.
\end{enumerate}
\end{lem}

By Remark \ref{rk:t_lambda}\ref{it:t_lambda:2} the condition on $\lambda$ in \ref{it:princ-series-1} is precisely that all weights $\t_\lambda(-\ovl{\eta}_J)$ lie in $C_0$.  %
Also note in part~\ref{it:princ-series-3} that the Jordan--H\"older factors correspond via $\t_\lambda$ precisely to the weights lying on geodesics between
$0$ and $\sum_{i\in\cJ}a_i\ovl{\eta}_i$.

\begin{proof}
Part \ref{it:princ-series-1} is almost a special case of Proposition \ref{prop:JH:graph} (with $sw^{-1} = 1$, $\nu = \eta$, and $\mu-\eta = \lambda$), but the hypothesis is weaker here.

If $\nu \in X^0(\un{T})$, then from the definition, $F(\t_{\lambda+\nu}(\omega)) \cong F(\t_\lambda(\omega)) \otimes_{\F} F(\nu)$. (Note that $F(\nu)$ is one-dimensional.) We may therefore assume that $\lambda_i$ is of the form $(a_i,0)$ for some integers $0 < a_i < p-1$.

Recall from Remark \ref{rk:t_lambda}\ref{it:t_lambda:1} the notation $w_{0,J} = \prod_{i+1 \in J} \fW_{i} \in \un{W}$, where $\fW_{i}$ denotes the Weyl group element which is nontrivial exactly in the $i$-th embedding. We first calculate $\t_\lambda(-\ovl{\eta}_J) \equiv \mu_J \mod (p-\pi)X^0(\un{T})$, where $\mu_J = (t_{\pi^{-1}(\eta_J)} w_{0,J}) \cdot (\lambda-\eta_J) \in X^*(\un{T})$.
We have
\begin{align*}
\mu_{J,i} &= \begin{cases}
\lambda_i - \delta_J(i) (1,0) & \text{if $i+1 \not\in J$,} \\
w_0\cdot \big(\lambda_i + (0,p)-\delta_J(i) (1,0)\big) & \text{if $i+1 \in J$,}
\end{cases}\\
&= \begin{cases}
(a_i,0) - \delta_J(i) (1,0) & \text{if $i+1 \not\in J$,} \\
(p-1,a_i+1)-\delta_J(i) (0,1) & \text{if $i+1 \in J$,}
\end{cases}
\end{align*}
where $\delta_J$ is the characteristic function of $J$ (cf.~equation (\ref{eq:expl:tmu})). 
Replacing $J$ by the set $K \defeq \{ i \in \cJ : i+1 \not \in J \}$, we obtain precisely the formula for the composition factors listed in \cite[Prop.\ 1.1]{diamond-durham}.

Part \ref{it:princ-series-2} follows similarly from \cite[Lemma~3.2]{BP}, and part \ref{it:princ-series-3} follows from \cite[Cor.\ 4.11]{BP}.
\end{proof}

\begin{prop}\label{prop:K-rep-by-ind}
  Fix $\lambda \in X^*(\un{T})$. Suppose that integers $B_i \in \ZZ_{\ge 0}$ and signs $\eps_i \in \{\pm1\}$ \emph{(}$0\leq i\leq f-1$\emph{)} satisfy the following conditions:
  \begin{enumerate}
  \item $B_i \equiv \frac{1-\eps_{i-1}}2 \pmod 2$;
  \item\label{item:1} if $\eps_i = -1$, then $B_i \le \ang{\lambda,\alpha_i^\vee} \le p-2-\frac{1+\eps_{i-1}}2$;
  \item\label{item:2} if $\eps_i = 1$, then $B_i \le p-2-\ang{\lambda,\alpha_i^\vee} \le p-2-\frac{1+\eps_{i-1}}2$.
  \end{enumerate}
  Then there exists a multiplicity-free representation $V$
  of $K/Z_1$ with Jordan--H\"older constituents $\sigma_{\un{a}} \defeq F(\t_\lambda(\sum \eps_i a_i \ovl{\eta}_i))$, where $0 \le a_i \le B_i$
  and whose submodule structure is determined as follows: the unique subrepresentation with cosocle $\sigma_{\un{a}}$ has constituents
  $\sigma_{\un{b}}$ for all $\un{b}$ such that $0 \le b_i \le a_i$ for all $i$. In particular, the socle of $V$ is isomorphic to $F(\lambda)$.
\end{prop}

\begin{proof}
  As a first step we consider the case where $\eps_i = -1$ for all $i$. Let $b_i \defeq \frac{B_i-1}2 \in \ZZ_{\ge 0}$ for $0\leq i\leq f-1$. Note that $\t_\lambda(-\sum_i a_i\ovl{\eta}_i)\in\un{C_0}$ for all $0\leq a_i\leq B_i$ is \emph{equivalent} to condition \ref{item:1} (cf.\ Remark \ref{rk:t_lambda}\ref{it:t_lambda:2}). Let $\chi \defeq \chi_\lambda$.
  Corollary~\ref{cor:iwahori-reps} gives us a representation $W \subset \cJ_\chi$ of $I$ with constituents
  $\chi^s \alpha_0^{j}$, where $0 \le j_i \le b_i$ for all $i$, and such that the unique subrepresentation of $W$ with cosocle
  $\chi^s \alpha_0^{j}$ has constituents $\chi^s \alpha_0^{\ell}$, where $0 \le \ell_i \le j_i$ for all $i$.  Let
  $V \defeq \Ind_I^{K} W$. By Lemma~\ref{lm:princ-series} and Remark \ref{rk:t_lambda}\ref{it:t_lambda:3}, this representation is multiplicity-free with constituents
  $F(\t_\lambda(-\sum c_i \ovl{\eta}_i))$, where $0 \le c_i \le 2b_i+1 = B_i$ for all $i$.

  To determine the submodule structure, by Lemma~\ref{lm:ext1} it is enough to show that for any $(c_i)_i$ as above and any $j$ such
  that $c_j < 2b_j+1$ there exists a length 2 subquotient with socle $F(\t_\lambda(-\sum c_i \ovl{\eta}_i))$ and cosocle
  $F(\t_\lambda(-\ovl{\eta}_j-\sum c_i \ovl{\eta}_i))$.
  To see this, write $c_i = 2d_i+r_i$ with $0 \le r_i \le 1$. Observe that
\[F(\t_\lambda(-\sum c_i \ovl{\eta}_i)) = F(\t_\lambda(-\sum r_i \ovl{\eta}_i-\sum d_i \alpha_i)) = F(\t_{\lambda-\sum
    d_i \alpha_i}(-\sum r_i \ovl{\eta}_i))\]
  by applying Remark~\ref{rk:t_lambda}\ref{it:t_lambda:3}. By Lemma~\ref{lm:princ-series} this is a
  constituent of $\Ind_B^G \chi'^s$, where $\chi'^s = \chi^s_{\lambda-\sum d_i \alpha_i} = \chi^s_\lambda \alpha_0^{\sum d_ip^{i}}$.

  If $r_j = 0$, then $F(\t_\lambda(-\ovl{\eta}_j-\sum c_i \ovl{\eta}_i))$ is a constituent of $\Ind_{I}^{K} \chi'^s$ as well, 
  and we are done by Lemma~\ref{lm:princ-series}, as $V$ admits $\Ind_I^{K} \chi'^s$ as subquotient.

  If $r_j = 1$, then $F(\t_\lambda(-\ovl{\eta}_j-\sum c_i \ovl{\eta}_i))$ is a constituent of
  $\Ind_{I}^{K} \chi'^s \alpha_0^{p^{j}}$. Letting the other $r_i$ vary in $\set{0,1}$, we need to check the existence of the $2^{f-1}$ nonsplit
  extensions inside $V$ between constituents of $\Ind_{I}^{K} \chi'^s \alpha_0^{p^{j}}$ and $\Ind_{I}^K \chi'^s$ given by
  Lemma~\ref{lm:ext1}. When $f = 1$ this is obvious, as we can compute the cosocle of $\Ind_{I}^K (E_{\chi'^s,\chi'^s \alpha_0^{p^{j}}})$
  by Frobenius reciprocity (cf.\ Lemma~\ref{lemma-Echi'=minus}). When $f \ge 2$ then \cite[Lemme 2.12(i)]{yongquan-algebra} confirms there are $2^{f-1}$ nonsplit extensions, as required
  (in the notation of that reference the condition is $J(\lambda) = J(\theta) \sqcup \{j-1\}$).

  Finally we treat the general case. Let $J \defeq \{ 0\leq i \leq f-1 : \eps_{i-1} = 1\}$. Set $\mu=\t_\lambda(w_{0,J}(\ovl{\eta}_J))$. Using Lemma~\ref{lm:change-origin} 
  and Remark~\ref{rk:t_lambda}\ref{it:t_lambda:1} we compute
  $\t_\lambda(\sum \eps_i a_i \ovl{\eta}_i) = \t_\mu(-\sum (a_i+\delta_J(i))\ovl{\eta}_i)$ for integers $a_i$.
  Note that $\delta_J(i) = \frac{1+\eps_{i-1}}2$.
  
  We apply the first step of the proof with the weight $\mu$, the bounds $B_i+\delta_{J}(i)$ and all signs $-1$. We obtain a representation $V'$ with socle $F(\mu)$ satisfying the desired hypotheses with signs $-1$ for all $i$ and $B_i+\delta_J(i)$ in place of $B_i$. We note that its unique quotient with socle $F(\lambda)$ has the desired properties with signs $\eps_i$ and bounds $B_i$. We just have to check that we can apply the first step in this case. Namely it suffices to check that $\t_{\mu}(-\sum a'_i\ovl{\eta}_i) \in \un{C_0}$ for $0 \le a'_i \le B_i+\delta_{J}(i)$, noting that
  $B_i+\delta_{J}(i) = B_i+\frac{1+\eps_{i-1}}2$ is odd for all $i$. Equivalently, we need that
  $\t_\lambda(\sum \eps_i a_i \ovl{\eta}_i) \in \un{C_0}$ for $-\delta_{J}(i) \le a_i \le B_i$, i.e.\ 
  $0 \le \ang{\lambda,\alpha_i^\vee}+\eps_i a_i \le p-2$ for $-\delta_{J}(i) \le a_i \le B_i$ and all $i$.
  This is equivalent to conditions (ii) and (iii) that we assumed.
\end{proof}

Assume that $\lambda$ is $1$-deep in alcove $\un{C}_0$, i.e.\ $1\leq\ang{\lambda,\alpha_i^\vee}\leq p-3$ for all $i$. Let $V$ be the representation of Proposition \ref{prop:K-rep-by-ind} with $B_i\in\set{0,1}$ for all $i$. Let $\un{a}$ be such that $0 \le a_i \le B_i$ for all $i$. Then the subrepresentation of $V$ with cosocle $\sigma_{\un{a}}$ of Proposition \ref{prop:K-rep-by-ind} is isomorphic to the representation $I(F(\lambda),\sigma_{\un{a}})$ of \cite[Cor.~3.12]{BP}.

\begin{lem}\label{lm:K1-invt}
  Suppose that $V$ is a finite-dimensional smooth representation of $K$ that has irreducible $K$-socle $\sigma = F(\lambda)$
  with $2 < \ang{\lambda,\alpha_i^\vee} < p-3$ for all $i$.
  If $[V:\sigma] = 1$ and all constituents of $V$ occur in $\Inj_{\GL_2(k)} \sigma$, then $V$ is $K_1$-invariant.
\end{lem}

\begin{proof}
  By writing $V$ as a quotient of $\Proj_K(\cosoc_K V)$ and decomposing $\cosoc_K V$ as a direct sum of irreducible representations,
  we see that $V$ is the sum of all subrepresentations with irreducible cosocle. We may thus assume that $V$ itself has irreducible cosocle $\tau$,
  and we argue by induction on the length $\ell(V)$ of $V$. If $\ell(V) = 1$ there is nothing to show. By induction, $\rad_K V$ is $K_1$-invariant, so $V[\mathfrak{m}_{K_1}^2] = V$.
  By \cite[Thm.\ 2.23]{HuWang2} we know that $V$ is $K_1$-invariant.
\end{proof}

\begin{prop}\label{prop:J-fil}
  Fix $\lambda \in X^*(\un{T})$. Suppose that integers $B_i \in \ZZ_{\ge 0}$ and signs $\eps_i \in \{\pm1\}$ \emph{(}$0\leq i\leq f-1$\emph{)} satisfy the following conditions:
  \begin{enumerate}
  \item $B_i \equiv \frac{1-\eps_{i-1}}2 \pmod 2$;
  \item\label{item:J1} if $\eps_i = -1$, then $3+2\lfloor B_i/2\rfloor \le \ang{\lambda,\alpha_i^\vee} \le p-4$;
  \item\label{item:J2} if $\eps_i = 1$, then $3 \le \ang{\lambda,\alpha_i^\vee} \le p-4-2\lfloor B_i/2\rfloor$.
  \end{enumerate}
  Let $V$ be the $K$-representation defined by this choice of $\lambda$, $B_i$, $\eps_i$ in Proposition~\ref{prop:K-rep-by-ind}.

  Then for $0 \le n-1 \le \sum \lfloor B_i/2\rfloor$ we have that $V[\mathfrak{m}_{K_1}^n]$ is the unique subrepresentation of $V$ with cosocle $\bigoplus \sigma_{\un{a}}$, where the sum runs over all $\un{a}$ such that $0 \le a_i \le B_i$ and
  \begin{enumerate}
  \item $a_i$ is odd or $a_i = B_i$,
  \item $\sum \lfloor a_i/2\rfloor = n-1$.
  \end{enumerate}
\end{prop}

\begin{proof}
  We proceed by induction on $n \ge 1$ and denote by $V_n$ the unique subrepresentation in the statement. For convenience let $V_0 = 0$. We need to show that $V_n/V_{n-1} = (V/V_{n-1})^{K_1}$.
  The constituents of $V_n/V_{n-1}$ (resp.\ $V/V_{n-1}$) are all Serre weights $\sigma_{\un{a}}$ with $0 \le a_i \le B_i$ and $\sum \lfloor a_i/2\rfloor = n-1$
  (resp.\ $\sum \lfloor a_i/2\rfloor \ge n-1$).
  Using the submodule structure of $V$ given by Proposition \ref{prop:K-rep-by-ind}, we see that $V_n/V_{n-1}$ is a direct sum of indecomposable representations $W_{\un{a}}$, where the index set is the
  same as in the statement of the proposition and the constituents of $W_{\un{a}}$ are all $\sigma_{\un{b}}$ with $0 \le b_i \le B_i$ and
  $\lfloor b_i/2\rfloor = \lfloor a_i/2\rfloor$ for all $i$ (and the submodule structure is described by the usual
  partial order). Note that $\soc_K W_{\un{a}} \cong \sigma_{\un{b}}$, where $b_i = 2\lfloor a_i/2\rfloor$.

  By Lemma~\ref{lm:K1-invt}, $V_n/V_{n-1}$ is $K_1$-invariant (the given bounds guarantee that the lemma applies by Remark \ref{rk:t_lambda}\ref{it:t_lambda:2},
  see also Lemma~\ref{lm:princ-series}\ref{it:princ-series-2}). On the
  other hand, $(V/V_{n-1})^{K_1}$ has to inject into the injective envelope $\Inj_{\GL_2(k)} (\soc_{K} (V/V_{n-1}))$. By 
  Lemma~\ref{lm:princ-series}\ref{it:princ-series-2} we deduce that $(V/V_{n-1})^{K_1} \subset V_n/V_{n-1}$.
  (Note that our genericity bounds are stronger.)
\end{proof}

\subsection{A result on maximal representations of \texorpdfstring{$K$}{K} with prescribed socle}

In this section, we prove a structure result for certain representations of $K$ killed by $\mathfrak{m}_{K_1}^2$.

We begin with some preliminary lemmas concerning Jordan--H\"older factors of subrepresentations of some parabolically induced representations. Recall from~\eqref{eq:10} the representation $E_{\chi',\chi}$ for two characters $\chi,\chi'$ of $I$ such that $\Ext^1_{I/Z_1}(\chi,\chi')\neq0$.

\begin{lem} \label{lemma-Echi'=minus}
Assume $\chi'=\chi\alpha_i^{-1}$ for some $0\leq i\leq f-1$. The cosocle of $\Ind_I^{K}E_{\chi',\chi}$ is equal to the cosocle of $\Ind_I^{K}\chi$.
\end{lem}

\begin{proof}
Let $\sigma$ be a Serre weight and assume there exists a surjection $f : \Ind_I^{K}E_{\chi',\chi}\twoheadrightarrow \sigma$. Then Frobenius reciprocity induces a nonzero $I$-equivariant morphism $f'\in \Hom_I(E_{\chi',\chi},\sigma|_{I})$. Since $K_1$ acts trivially on $\sigma$ but not on $E_{\chi',\chi}$ (see Lemma \ref{lemma-Ext1=dim1}\ref{it1-Ext1=dim1}), $f'$ cannot be injective. In other words, $f'$ factors through $E_{\chi',\chi}\twoheadrightarrow \chi\hookrightarrow\sigma|_I$, i.e.\ $f$ factors through $\Ind_I^{K}E_{\chi',\chi}\twoheadrightarrow\Ind_I^{K}\chi$.
\end{proof}
\begin{rem}
For the explicit structure of $\Ind_I^{K}E_{\chi',\chi}$ when $\chi'=\chi\alpha_i^{-1}$ (resp.\ $\chi'=\chi\alpha_i$), see \cite[\S18]{BP} or \cite[Lemma 3.7]{HuWang2} (resp.\ \cite[Lemma 3.8]{HuWang2}). 
\end{rem}

Given $\chi$ satisfying $\chi\neq \chi^s$, we denote by $\sigma_{\chi}$ the unique Serre weight such that $I$ acts on $\sigma_{\chi}^{I_1}$ via $\chi$. Recall that in this case $\Ind_I^{K}\chi$ has irreducible cosocle $\sigma_{\chi}$ and irreducible socle $\sigma_{\chi^s}$ (see e.g.\ \cite[Thm.\ 2.4]{BP}). Given a Serre weight $\sigma$, we denote by $\chi_{\sigma}$ the  character of $I$ acting on $\sigma^{I_1}$.

 \begin{lem}\label{lemma-IndW2-multione}
Assume that $0<\ang{\lambda,\alpha_i^\vee}<p-3$ for all $0\leq i\leq f-1$.
Then the $K$-representation $\Ind_I^KE_{\chi\alpha_i,\chi}$ is multiplicity-free for any $i$.
Suppose moreover that $\chi = \chi_\lambda$ with $2<\ang{\lambda,\alpha_i^\vee}<p-3$ for all $0\leq i\leq f-1$. Then the $K$-representation $\Ind_I^KW_{\chi,2}$ is multiplicity-free, where $W_{\chi,2}$ is defined in \eqref{eq:W-chi-n}.
\end{lem}
 
\begin{proof}
This is a direct check using Remark \ref{rk:t_lambda}\ref{it:t_lambda:3} and Lemma \ref{lm:princ-series}\ref{it:princ-series-1}.
The assumption that $0<\ang{\lambda,\alpha_i^\vee}<p-3$ for all $0\leq i\leq f-1$ ensures that the hypothesis of Lemma \ref{lm:princ-series}\ref{it:princ-series-1} applies to $\Ind_I^K\chi\alpha_i$ and $\Ind_I^K\chi$.
If furthermore $\lambda$ satisfies the stronger condition $2<\ang{\lambda,\alpha_i^\vee}<p-3$ for all $0\leq i\leq f-1$ then again the hypothesis of Lemma \ref{lm:princ-series}\ref{it:princ-series-1} applies to all $\Ind_I^{K}\chi'$ with $\chi'\in\JH(W_{\chi,2})$.
\end{proof}

From now on we fix $\chi=\chi_\lambda$ with $\lambda\in X_1(\un{T})$ such that $0<\ang{\lambda,\alpha_i^\vee}<p-3$ for all $0\leq i\leq f-1$.
Let $\chi'\defeq \chi\alpha_i$. Lemma \ref{lemma-IndW2-multione} implies that $\Ind_I^K E_{\chi',\chi}$ is multiplicity-free. 

On the other hand, $K_1$ acts trivially on $\Ind_I^KE_{\chi',\chi}$ by Lemma~\ref{lemma-Ext1=dim1}\ref{it1-Ext1=dim1}.
Hence there is a unique (up to scalar) nonzero map $f : \Proj_{\GL_2(k)}\sigma_\chi\rightarrow\Ind_I^KE_{\chi',\chi}$. Observe that the composite map
\[\Proj_{\GL_2(k)}\sigma_{\chi}\xrightarrow{f}\Ind_{I}^KE_{\chi',\chi}\twoheadrightarrow \Ind_I^K\chi\]
is surjective, since it is surjective on $K$-cosocles.

\begin{lem}\label{lm:IndE}
Suppose that $\chi = \chi_\lambda$ with $0<\ang{\lambda,\alpha_i^\vee}<p-3$ for all $0\leq i\leq f-1$. %
Assume $\chi'=\chi\alpha_i$ for some $i\in\cJ$. We have
\begin{equation}\label{eq:JH-imf} \JH(\Image(f))=\JH(\Ind_I^KE_{\chi',\chi})\cap\JH(\Proj_{\GL_2(k)}\sigma_\chi).\end{equation}
\end{lem}

\begin{proof}
Observe that the $K$-socle of $\Ind_I^KE_{\chi',\chi}$ is isomorphic to $\sigma_{\chi'^s}\oplus\sigma_{\chi^s}$, i.e.\ the direct sum of the socles of $\Ind_I^K\chi'$ and $\Ind_I^K\chi$. Indeed, it is clear that \[\sigma_{\chi'^s}\subset\soc_K(\Ind_I^KE_{\chi',\chi})\subset \sigma_{\chi'^s}\oplus\sigma_{\chi^s},\] 
so it suffices to prove that $\Hom_{K}(\sigma_{\chi^s},\Ind_I^KE_{\chi',\chi})\neq0$, or equivalently $\Hom_{I}(\sigma_{\chi^s}|_I,E_{\chi',\chi})\neq0$, by Frobenius reciprocity. This can be checked directly, by writing down the standard basis of $\sigma_{\chi^s}$. 

Let $V\defeq\Image(f)$. We claim that $V\cap\Ind_I^K\chi'\neq0$. Otherwise, the composite morphism $V\hookrightarrow \Ind_I^KE_{\chi',\chi}\twoheadrightarrow \Ind_I^K\chi$ would be injective, and also surjective as remarked before the lemma. Thus, we would have a $K$-equivariant decomposition $\Ind_I^KE_{\chi',\chi}\simeq\Ind_I^K\chi\oplus\Ind_I^K\chi'$, which is not possible (see for example \cite[\S8, Lemma~6(5)]{Alperin}). As a consequence of the claim, $\sigma_{\chi'^s}$ appears in $V$ (as a subobject), and therefore $V$ admits a quotient isomorphic to $I(\sigma_{\chi'^s},\sigma_{\chi})$ (we recall that this representation was defined in \S\ref{sec:some-repr-K}).

Now we prove \eqref{eq:JH-imf}. The inclusion $\subseteq$ is obvious. Let $\sigma$ be a Serre weight lying in the right-hand side of \eqref{eq:JH-imf}. If $\sigma\in \JH(\Ind_I^K\chi)$, then clearly $\sigma\in \JH(V)$ because $\Ind_I^K\chi$ is a quotient of $V$. So we may assume $\sigma\in\JH(\Ind_I^K\chi')$. Then, by Lemma \ref{lm:princ-series}\ref{it:princ-series-1} and Remark \ref{rk:t_lambda}\ref{it:t_lambda:3}, $\sigma$ is of the form $F(\t_{\lambda+\alpha_i}(-\ovl{\eta}_J))=F(\t_{\lambda}(2\ovl{\eta}_i-\ovl{\eta}_J))$ for some $J\subset\cJ$. 
It follows from Lemma \ref{lm:princ-series}\ref{it:princ-series-2}, \ref{it:princ-series-3} and Remark \ref{rk:t_lambda}\ref{it:t_lambda:3} that such a Serre weight is a Jordan--H\"older factor of $\Proj_{\GL_2(k)}\sigma_{\chi}$ if and only if $i\in J$, if and only if it is a Jordan--H\"older factor of $I(\sigma_{\chi'^s},\sigma_\chi)$. 
(Note that $\sigma_\chi \cong F(\lambda)$ and $\sigma_{\chi'^s} \cong F(\t_{\lambda}(2\ovl{\eta}_i-\ovl{\eta}_{\cJ}))$.)
Since $I(\sigma_{\chi'^s},\sigma_{\chi})$ is a quotient of $V$, this finishes the proof.
\end{proof}

\begin{lem}\label{lemma-Echi'=plus}
Suppose that $\chi = \chi_\lambda$ with  $0<\ang{\lambda,\alpha_i^\vee}<p-3$ for all $0\leq i\leq f-1$.
Assume $\chi'=\chi\alpha_i$ for some $i\in\cJ$. 
Let $Q$ be a quotient of $\Ind_I^{K}E_{\chi',\chi}$ such that $[Q:\sigma_{\chi}]=0$, then $\Ext^1_{K}(\sigma,\sigma_{\chi})=0$ for any $\sigma\in \JH(Q)$.
\end{lem}
\begin{proof}

Let $M$ be the kernel of $\Ind_{I}^{K}E_{\chi',\chi}\twoheadrightarrow Q$. By Lemma \ref{lemma-IndW2-multione}, $\Ind_I^KE_{\chi',\chi}$ is multiplicity-free.
Since $[Q:\sigma_{\chi}]=0$ by assumption, we have $[M:\sigma_{\chi}]=1$. As a consequence, the natural morphism $M\ra \Ind_I^{K}\chi$ is surjective (as $\sigma_{\chi}$ is the cosocle of $\Ind_I^K\chi$), and therefore $Q$ is a quotient of $\Ind_I^{K}\chi'$ by the snake lemma. By Lemma \ref{lm:princ-series}\ref{it:princ-series-1}, the Jordan--H\"older factors of $\Ind_I^{K}\chi'$ are of the form $F(\t_{\lambda+\alpha_i}(-\ovl{\eta}_J))$ for $J\subset\cJ$. It follows from Lemma \ref{lm:ext1} that the existence of $\sigma\in\JH(Q)$ such that $\Ext^1_{K}(\sigma,\sigma_\chi)\neq0$ implies the existence of $J\subset\cJ$
and $j\in\cJ$ such that $\sigma=F(\t_{\lambda+\alpha_i}(-\ovl{\eta}_J))\in\JH(Q)$ and $\t_{\lambda+\alpha_i}(-\ovl{\eta}_J)=\t_\lambda(\pm \ovl{\eta}_j)$. By Remark \ref{rk:t_lambda}\ref{it:t_lambda:3} we get $2\ovl{\eta}_i-\ovl{\eta}_J=\pm\ovl{\eta}_j$, i.e.\ we must have $J=\set{i}$ and $j=i$, and hence $\sigma=F(\t_{\lambda}(\ovl{\eta}_i))$.

Consider again the unique (up to a scalar) nonzero map
\[ f : \Proj_{\GL_2(k)}\sigma_{\chi}\rightarrow\Ind_I^{K}E_{\chi',\chi}.\]
By Lemma \ref{lm:IndE} we have $F(\t_\lambda(\ovl{\eta}_i))\in\JH(\Image(f))$. However, $\sigma_\chi\in \JH(M)$, thus by uniqueness of $f$, we must have $\Image(f)\subset M$. Then the Serre weight $F(\t_\lambda(\ovl{\eta}_i))$ is a subquotient of both $M$ and $Q$. This contradicts the fact that $\Ind_I^KE_{\chi',\chi}$ is multiplicity-free. %
\end{proof}

We fix $\un{\eps}\in\set{-1,0,1}^\cJ$ and define
\begin{equation}\label{eq:D-lambda-eps}
D_{\lambda,\un{\eps}}\defeq I\Big(F(\lambda),F\big(\t_\lambda(\sum_{i\in\cJ}\eps_i\ovl{\eta}_i)\big)\Big).
\end{equation}
Its Jordan--H\"older factors are given by $F(\t_\lambda(\sum_{i\in J}\eps_i\ovl{\eta}_i))$ for $J\subset J_{\un{\eps}}$ by Lemma \ref{lm:princ-series}\ref{it:princ-series-3}, where \[J_{\un{\eps}} \defeq \set{i\in \cJ:\eps_i\neq 0}.\]
In particular, $D_{\lambda,\un{\eps}}$ has length $2^{|J_{\un{\eps}}|}$.
 
\begin{rem}\label{rem:D-eps} %
Keep the previous hypotheses and setting. 
\begin{enumerate}%
\item We have
\[ \Ind_{I}^{K}\chi_\lambda^s\simeq D_{\lambda,\un{-1}},\]
as follows from Lemma \ref{lm:princ-series}\ref{it:princ-series-1}.
\item Let $\rhobar$ be a $2$-dimensional semisimple Galois representation which is $2$-generic (see Definition \ref{def:rhobar:gen}).
Then the $\GL_2(k)$-representation $D_0(\rhobar)$ attached to $\rhobar$ as in \cite[\S14]{BP} is a direct sum of such $D_{\lambda,\un{\eps}}$, where $\un{\eps}\in\set{\pm 1}^\cJ$; see Theorem 14.8 in \emph{loc.~cit.}
\end{enumerate}
\end{rem}

We want to understand the structure of $D_{\lambda,\un{\eps}}\otimes_{\F} F(\alpha_j)$. 

\begin{lem}\label{lemmJH}
Suppose that $\chi = \chi_\lambda$ with $2<\ang{\lambda,\alpha_i^\vee}<p-4$ for all $0\leq i\leq f-1$.
The Jordan--H\"older factors of $D_{\lambda,\un{\eps}}\otimes_{\F} F(\alpha_j)$ have multiplicity one and are given by $F(\t_{\lambda}(2\eps'\ovl{\eta}_j+\sum_{i\in J}\eps_i\ovl{\eta}_i))$ for $J\subset J_{\un{\eps}}$ and $\eps'\in\set{-1,0,1}$.
\end{lem}

\begin{proof}
First note that we have $F(\lambda) \otimes_{\F} F(\alpha_j) \cong \bigoplus_{\eps\in\set{-1,0,1}} F(\lambda+\eps \alpha_j)$ by \cite[Prop.~5.4]{BP} or \cite[Prop.~3.3(1)]{LMS}.
We then obtain the Jordan--H\"older factors using Remark~\ref{rk:t_lambda}\ref{it:t_lambda:3}.
The multiplicity one property then follows from the injectivity of $\t_\lambda$. Namely if $2\eps'_1\ovl{\eta}_j+\sum_{i\in J_1}\eps_i\ovl{\eta}_i=2\eps'_2\ovl{\eta}_j+\sum_{i\in J_2}\eps_i\ovl{\eta}_i$  for some subsets $J_1$, $J_2$ of $J_{\un{\eps}}$, then $J_1 = J_2$ by passing to $\Lambda_W/2\Lambda_W$, so also $\eps'_1=\eps'_2$.
\end{proof}

\begin{lem}\label{socle}
Suppose that $\chi = \chi_\lambda$ with $2<\ang{\lambda,\alpha_i^\vee}<p-4$ for all $0\leq i\leq f-1$.
We have
\begin{eqnarray*}
\soc_{\GL_2(k)}(D_{\lambda,\un{\eps}}\otimes_{\F} F(\alpha_j))&\cong&\bigoplus_{\eps'\in\set{-1,0,1}} F(\t_\lambda(2\eps'\ovl{\eta}_j)),\\
\cosoc_{\GL_2(k)}(D_{\lambda,\un{\eps}}\otimes_{\F} F(\alpha_j))&\cong&\bigoplus_{\eps'\in\set{-1,0,1}} F(\t_\lambda(2\eps'\ovl{\eta}_j+\sum_{i\in J_{\un{\eps}}}\eps_i\ovl{\eta}_i)).
\end{eqnarray*}
\end{lem}
\begin{proof}
Let $I_\lambda\defeq \Inj_{\GL_2(k)}F(\lambda)$.
We have inclusions $F(\lambda)\subset D_{\lambda,\un{\eps}}\subset I_{\lambda}$, which induces inclusions
\[F(\lambda)\otimes_{\F}F(\alpha_j)\subset D_{\lambda,\un{\eps}}\otimes_{\F}F(\alpha_j)\subset I_{\lambda}\otimes_{\F}F(\alpha_j), \]
and also inclusions of the corresponding $K$-socles.
It follows from %
\cite[Prop.\ 3.3(2)]{LMS} that, if $1<\ang{\lambda,\alpha_i^\vee}<p-3$ for all $0\leq i\leq f-1$ when $f\geq 2$, or $2<\ang{\lambda,\alpha_0^\vee}<p-3$ when $f=1$, then  
\[ I_\lambda\otimes_{\F} F(\alpha_j)\simeq\bigoplus_{\eps'\in\set{-1,0,1}} I_{\lambda+\eps'\alpha_j}.\]
In particular, the $K$-socle of $I_{\lambda}\otimes_{\F}F(\alpha_j)$ is isomorphic to $\bigoplus_{\eps'\in\set{-1,0,1}} F(\t_\lambda(2\eps'\ovl{\eta}_j))$, which itself is isomorphic to $F(\lambda)\otimes_{\F}F(\alpha_j)$. This proves the result in these cases. %

The assertion on the cosocle follows by a dual argument, using surjections $P_{\t_\lambda(\sum_{i}\eps_i\ovl{\eta}_i)} \onto D_{\lambda,\un{\eps}}\onto F(\lambda)$, where $P_\mu \defeq \Proj_{\GL_2(k)}F(\mu) (\cong I_\mu)$.
By \cite[Prop.\ 3.3(2)]{LMS},
\begin{equation}\label{eq:proj-covers-tensor}
  P_{\t_\lambda(\sum_{i}\eps_i\ovl{\eta}_i)} \otimes_{\F} F(\alpha_j)\simeq\bigoplus_{\eps'\in\set{-1,0,1}} P_{\t_\lambda(2\eps'\ovl{\eta}_j+\sum_{i}\eps_i\ovl{\eta}_i)},
\end{equation}
unless there is some $\eps'\in\set{-1,0,1}$ such that $\ang{\t_\lambda(2\eps'\ovl{\eta}_j+\sum_{i}\eps_i\ovl{\eta}_i),\alpha_{i'}^\vee} = 0$ for all $i'$.
This exceptional case can only happen when $f = 1$, in which case the condition is equivalent to $\ang{\lambda+2\eps'\ovl{\eta}_0+\eps_0\ovl{\eta}_0,\alpha_0^\vee} = p-2$ (by Remark~\ref{rk:t_lambda}\ref{it:t_lambda:1} the element $\tld w'$ in the definition of $\t_\lambda$ is not a translation).
Combined with our genericity condition we obtain $\ang{\lambda,\alpha_0^\vee} = p-5$ and $\eps_0 = \eps' = +1$.
In this case, by \cite[Prop.\ 3.3(2)]{LMS} an extra direct summand that is irreducible of dimension $p$ appears on the right-hand side of~\eqref{eq:proj-covers-tensor},
and we conclude as there is no $p$-dimensional constituent in $\JH(D_{\lambda,\un{\eps}}\otimes_{\F} F(\alpha_0))$.
\end{proof}

\begin{lem}\label{lemmaPS}
Suppose that $0\leq j\leq f-1$ and that $\chi = \chi_\lambda$, where %
$\lambda$ is $4$-deep in $\un{C}_0$, i.e.\ $3<\langle \lambda,\alpha_i^{\vee}\rangle < p-5$ for all $0\leq i\leq f-1$.
Let $\eps\in\set{-1,1}$ and write $V$ for the unique nonsplit extension of $F(\t_\lambda(\eps\ovl{\eta}_j))$ by $F(\lambda)$:
\[ 0\rightarrow F(\lambda)\rightarrow V \rightarrow F(\t_\lambda(\eps\ovl{\eta}_j))\rightarrow0.\]
Then $V\otimes_{\F} F(\alpha_j)$ has a $3$-step increasing filtration whose successive graded pieces are $V_1$, $V_2$, $V_3$, where 
\begin{enumerate}
\item[$\bullet$] $V_1$ is a nonsplit extension of $F(\t_\lambda(3\eps\ovl{\eta}_j))$ by $F(\t_\lambda(2\eps\ovl{\eta}_j))$,
\item[$\bullet$] $V_2$ is a nonsplit extension of $F(\t_\lambda(\eps\ovl{\eta}_j))$ by $F(\lambda)$ (i.e.\ $V_2\cong V$), and
\item[$\bullet$] $V_3$ is a nonsplit extension of $F(\t_\lambda(-\eps\ovl{\eta}_j))$ by $F(\t_\lambda(-2\eps\ovl{\eta}_j))$. 
\end{enumerate}
As a consequence, $F(\t_\lambda(\eps\ovl{\eta}_j))$ is not contained in the socle of $(V\otimes_{\F} F(\alpha_j))/F(\t_\lambda(2\eps\ovl{\eta}_j))$.

Moreover, the corresponding extensions of $V_2$ by $V_1$, and $V_3$ by $V_2$, are nonsplit.
\end{lem}

The structure of $V\otimes_{\F}F(\alpha_j)$ can be illustrated by the extension graph
\begin{equation}\label{eq:ext-graph-tensor}
  \begin{gathered}
    \xymatrix{%
      F(\t_{\lambda}(3\eps\ovl{\eta}_j))\ar@{-}[d]&F(\t_{\lambda}(\eps\ovl{\eta}_j))\ar@{-}[dl]\ar@{-}[d]&F(\t_{\lambda}(-\eps\ovl{\eta}_j))\ar@{-}[d]\ar@{-}[dl]\\
      F(\t_{\lambda}(2\eps\ovl{\eta}_j))&F(\lambda)&F(\t_{\lambda}(-2\eps\ovl{\eta}_j))}
  \end{gathered}
\end{equation}
where the bottom (resp.~top) row corresponds to the socle (resp.~cosocle) of $V\otimes_{\F} F(\alpha_j)$. 

\begin{proof}
We note that $V \cong D_{\lambda,\un\eps}$, where $\eps_i = \eps$ if $i = j$ and $\eps_i = 0$ otherwise.
By Lemmas~\ref{lemmJH} and \ref{socle} we see that $V\otimes_{\F}F(\alpha_j)$ has the 6 Jordan--H\"older factors listed in~\eqref{eq:ext-graph-tensor}, and that its socle (resp.\ cosocle) is given by the bottom (resp.\ top) row of~\eqref{eq:ext-graph-tensor}.
In particular, $V\otimes_{\F}F(\alpha_j)$ is multiplicity-free and has Loewy length $2$.

Let us begin with the case where $\eps=-1$, and we will only assume $3<\langle \lambda,\alpha_i^{\vee}\rangle < p-4$ for all $0\leq i\leq f-1$.
We define $V_1$ as the image of the unique (up to scalar) nonzero map $\Proj_{\GL_2(k)}F(\t_\lambda(-3\ovl{\eta}_j))\rightarrow V\otimes_{\F} F(\alpha_j)$. 
Since $F(\t_\lambda(-3\ovl{\eta}_j))$ does not occur in the $K$-socle of $V\otimes_{\F}F(\alpha_j)$, $V_1$ also has Loewy length $2$, and each Serre weight occurring in $\rad_K(V_1)$ must have a nonsplit extension with $F(\t_\lambda(-3\ovl{\eta}_j))$. %
Comparing Jordan--H\"older factors of $V\otimes_{\F} F(\alpha_j)$ and using Lemma \ref{lm:ext1}, we find that $V_1$ has length two with socle $F(\t_\lambda(-2\ovl{\eta}_j))$ and cosocle $F(\t_\lambda(-3\ovl{\eta}_j))$.
We define $V_2\subset (V\otimes_{\F} F(\alpha_j))/V_1$ as the image of a nonzero map $\Proj_{\GL_2(k)}F(\t_\lambda(-\ovl{\eta}_j))\rightarrow (V\otimes_{\F} F(\alpha_j))/V_1$, and $V_3$ as the quotient of $(V\otimes_{\F}F(\alpha_j))/V_1$ by $V_2$.

Using the fact that $\eps=-1$ and Lemma \ref{lm:princ-series}\ref{it:princ-series-1} and \ref{it:princ-series-3}, we know that $V$ is a subrepresentation of the principal series $\Ind_{I}^{K}\chi$ with $\chi=\chi_{\lambda}^s$. Therefore, $V\otimes_{\F} F(\alpha_j)$ is a subrepresentation of \[\big(\Ind_{I}^{K}\chi\big)\otimes_{\F} F(\alpha_j)\simeq\Ind_{I}^{K}\big(\chi\otimes_{\F}F(\alpha_j)|_{I}\big).\] 
We deduce from the exactness of induction that $\Ind_{I}^{K}(\chi\otimes_{\F}F(\alpha_j)|_{I})$ has a $3$-step increasing filtration whose successive graded pieces are \[\Ind_{I}^{K}\chi\alpha_j, \ \ \Ind_{I}^{K}\chi,\ \ \Ind_{I}^{K}\chi\alpha_j^{-1}.\] 
We claim that 
\[\JH(V_1)= \JH(V\otimes_{\F}F(\alpha_j))\cap\JH(\Ind_I^K\chi\alpha_j). \]
Indeed, recalling $\chi=\chi_{\lambda}^s$, the Jordan--H\"older factors of $\Ind_I^K\chi\alpha_j=\Ind_I^K(\chi_{\lambda}\alpha_j^{-1})^s$ are of the form $F(\t_{\lambda-\alpha_j}(-\ovl{\eta}_J))=F(\t_{\lambda}(-2\ovl{\eta}_j-\ovl{\eta}_J))$ for $J\subset\cJ$, and the claim is checked  by comparing with $\JH(V\otimes_{\F}F(\alpha_j))$ given in the first paragraph of the proof.
Since $(\Ind_I^K\chi)\otimes_{\F}F(\alpha_j)$ is a subquotient of $\Ind_I^KW_{\chi,2}$, it is multiplicity-free by Lemma \ref{lemma-IndW2-multione}, so we deduce that  
\begin{equation}\label{eq:V1}V_1=(V\otimes_{\F}F(\alpha_j))\cap(\Ind_I^K\chi\alpha_j)\end{equation}
and hence an embedding
\[(V\otimes_{\F}F(\alpha_j))/V_1\hookrightarrow \Ind_I^K(\chi\otimes_{\F}F(\alpha_j)|_I)/\Ind_I^K\chi\alpha_j\cong\Ind_I^KE_{\chi,\chi\alpha_j^{-1}},\]
where the isomorphism holds because $(\chi\otimes_{\F}F(\alpha_j)|_I)/\chi\alpha_j$ is isomorphic to $E_{\chi,\chi\alpha_j^{-1}}$ as $I$-representa\-tion.

 As in the proof of Lemma \ref{lm:IndE}, the $K$-socle of $\Ind_{I}^KE_{\chi,\chi\alpha_j^{-1}}$ is equal to $F(\lambda)\oplus F(\t_{\lambda}(2\ovl{\eta}_j))$.  
In particular, $F(\t_\lambda(-\ovl{\eta}_j))$ is not a subrepresentation of $V_2$. As $F(\t_\lambda(\ovl{\eta}_j))$ and $F(\t_\lambda(2\ovl{\eta}_j))$ are not Jordan--H\"older factors of $\Proj_{\GL_2(k)}F(\t_\lambda(-\ovl{\eta}_j))$ (cf.~Lemmas \ref{lm:princ-series}\ref{it:princ-series-2} and \ref{lm:change-origin}), this implies that the socle of $V_2$ is equal to $F(\lambda)$ and hence $V_2$ is a nonsplit extension of $F(\t_{\lambda}(-\ovl{\eta}_j)) $ by $F(\lambda)$, as desired.
We deduce that $\JH(V_3)=\{F(\t_{\lambda}(2\ovl{\eta}_j)),F(\t_{\lambda}(\ovl{\eta}_j))\}$. Since $V_3$ has cosocle $F(\t_{\lambda}(\ovl{\eta}_j))$ by the first paragraph of the proof, $V_3$ has to be a nonsplit extension of $F(\t_{\lambda}(\ovl{\eta}_j))$ by $F(\t_{\lambda}(2\ovl{\eta}_j))$ as desired. 

Now we prove the last assertion (still when $\eps=-1$). We only prove that the extension of $V_2$ by $V_1$ %
is nonsplit, the other case being analogous. It suffices to prove that $V\otimes_{\F} F(\alpha_j)$ admits a subquotient isomorphic to the (unique) nonsplit extension $\mathcal{E}$ of $F(\t_{\lambda}(-\ovl{\eta}_j))$ by $F(\t_{\lambda}(-2\ovl{\eta}_j))$. %
As $V\otimes_{\F} F(\alpha_j)$ embeds in $(\Ind_I^K\chi)\otimes_{\F}F(\alpha_j)$, which is multiplicity-free, we are reduced to prove that $(\Ind_I^K\chi)\otimes_{\F}F(\alpha_j)$ admits a subquotient isomorphic to $\mathcal{E}$.
It follows from the proof of Lemma \ref{lm:IndE} that $I(F(\t_{\lambda}(-2\ovl{\eta}_j)),\sigma_{\chi})$ is isomorphic to a subquotient of $\Ind_I^K E_{\chi\alpha_j,\chi}$. Note that $\sigma_{\chi}\cong F(\t_{\lambda}(-\ovl{\eta}_{\cJ}))$ by Lemma \ref{lm:princ-series}\ref{it:princ-series-1}, so $F(\t_{\lambda}(-\ovl{\eta}_j))$ is a Jordan--H\"older factor of $I(F(\t_{\lambda}(-2\ovl{\eta}_j)),F(\t_{\lambda}(-\ovl{\eta}_{\cJ})))$ by Lemma~\ref{lm:princ-series}\ref{it:princ-series-3}.
This finishes the proof in the case $\eps=-1$.

To deal with the case $\eps=+1$, 
assume first $f\geq 2$. As $\eps=+1$, we know that $V$ is a quotient of $\Ind_I^K \chi_\mu$, where $\mu \defeq \t_\lambda(\ovl\eta_j)$ (use Lemma~\ref{lm:change-origin} and note that
$\lambda = \t_\mu(-\ovl\eta_j)$).
As $\lambda$ is $4$-deep in $\un{C}_0$ we have $3<\langle\mu,\alpha_i^\vee\rangle<p-4$ for all $0\leq i\leq f-1$, and we can use a similar argument as in the case $\eps=-1$
(this is where we need the assumption that $\lambda$ is $4$-deep rather than just $3$-deep). 

The case $\eps = +1$ is a little subtler when $f=1$ (i.e.\ $k=\F_p$), because $V$ is neither a subrepresentation nor a quotient of any principal series. To handle this case, we note the following exact sequence (see \cite[\S3]{BP})
\begin{equation*}
0\ra V\ra \Inj_{\GL_2(\F_p)}F(\lambda)\ra V'\ra0,
\end{equation*}
where $V'=\Ind_{I}^K\chi_{\lambda}$ is a principal series, and the decomposition (\cite[Prop.~3.3(2)]{LMS})
\[(\Inj_{\GL_2(\F_p)}F(\lambda))\otimes_{\F}F(\alpha_0)\cong \Inj_{\GL_2(\F_p)}F(\t_\lambda(2\ovl{\eta}_0))\oplus \Inj_{\GL_2(\F_p)}F(\lambda)\oplus \Inj_{\GL_2(\F_p)}F(\t_\lambda(-2\ovl{\eta}_0)). \] 
By tensoring with $F(\alpha_0)$ we obtain
\begin{equation}\label{eq:V-V'}
0\ra V\otimes_{\F}F(\alpha_0)\ra (\Inj_{\GL_2(\F_p)}F(\lambda))\otimes_{\F}F(\alpha_0)\ra V'\otimes_{\F}F(\alpha_0)\ra0.
\end{equation}
Like in the case $\eps=-1$, let $V_1$ denote the image of the unique (up to scalar) nonzero map $\Proj_{\GL_2(k)}F(\t_\lambda(3\ovl{\eta}_0))\rightarrow V\otimes_{\F} F(\alpha_0)$.
Dually, we define $V_3$ as the image of the unique nonzero map $V\otimes_{\F} F(\alpha_0)\rightarrow \Inj_{\GL_2(k)}F(\t_\lambda(-2\ovl{\eta}_0))$ extending $F(\t_\lambda(-2\ovl{\eta}_0)) \hookrightarrow \Inj_{\GL_2(k)}F(\t_\lambda(-2\ovl{\eta}_0))$ (and using that $F(\t_\lambda(-2\ovl{\eta}_0))\hookrightarrow V\otimes_{\F} F(\alpha_0)$).
Comparing Jordan--H\"older factors and using again the first paragraph of the proof, we see that $V_1$ and $V_3$ are as in the statement of this lemma.
Let $R$ denote the kernel of $V\otimes_{\F} F(\alpha_0)\onto V_3$ and let $V_2 \defeq R/V_1$.
It remains to show that three nonsplit extensions occur as subquotients of $V\otimes_{\F} F(\alpha_0)$, namely the nonsplit extensions of $F(\t_\lambda(\ovl{\eta}_0))$ by $F(\t_\lambda(2\ovl{\eta}_0))$, resp.\ $F(\t_\lambda(\ovl{\eta}_0))$ by $F(\lambda)$ (i.e.\ $V_2$ is nonsplit), resp.\ $F(\t_\lambda(-\ovl{\eta}_0))$ by $F(\lambda)$, cf.\ the extension graph \eqref{eq:ext-graph-tensor}.

If the nonsplit extension of $F(\t_\lambda(-\ovl{\eta}_0))$ by $F(\lambda)$ does not occur in $V\otimes_{\F} F(\alpha_0)$, then the image of the composition
\begin{equation*}
V\otimes_{\F} F(\alpha_0) \into (\Inj_{\GL_2(\F_p)}F(\lambda))\otimes_{\F}F(\alpha_0) \onto \Inj_{\GL_2(\F_p)}F(\lambda)
\end{equation*}
is contained in the unique subrepresentation isomorphic to $V$, so by~\eqref{eq:V-V'} we get an induced surjection $V'\otimes_{\F} F(\alpha_0) \onto (\Inj_{\GL_2(\F_p)}F(\lambda))/V$, where the latter representation is a nonsplit extension of $F(\lambda)$ by $F(\t_\lambda(-\ovl{\eta}_0))$, which contradicts the case $\eps = -1$.
(Note that $V'$ satisfies the relaxed genericity condition assumed in the case $\eps = -1$ above, as $\ang{\t_{\lambda}(-\eta_0), \alpha_0^\vee}=p-1-\ang{\lambda,\alpha_0^\vee}$.)
A similar argument applies to show the other two nonsplit extensions $\cE$ occur, using always~\eqref{eq:V-V'} and the projection $(\Inj_{\GL_2(\F_p)}F(\lambda))\otimes_{\F}F(\alpha_0) \onto \Inj_{\GL_2(\F_p)} (\soc_{\GL_2(\Fp)} \cE)$.
\end{proof}

\begin{prop}\label{prop:main}
Suppose that $\chi = \chi_\lambda$, where $\lambda$ is $4$-deep in $\un{C}_0$, i.e.\ $3<\ang{\lambda,\alpha_i^\vee}< p-5$ for all $0\leq i\leq f-1$. 
Let $\un\eps \in \set{-1,0,1}^{\cJ}$ and $0\leq j\leq f-1$. 
If $\eps_j \ne 0$, then there is an increasing $3$-step filtration of $D_{\lambda,\un{\eps}}\otimes_{\F} F(\alpha_j)$ whose successive graded pieces are:
\begin{equation}\label{eq:incr-filtr}
  D_{\lambda+\eps_j\alpha_j,\un{\eps}}, \quad D_{\lambda,\un{\eps}}, \quad D_{\lambda-\eps_j\alpha_j,\un{\eps}}.
\end{equation}
As a consequence, there is an embedding $D_{\lambda+\eps_j\alpha_j,\un{\eps}}\hookrightarrow D_{\lambda,\un{\eps}}\otimes_{\F} F(\alpha_j)$ whose cokernel has socle $F(\lambda)\oplus F(\t_\lambda(-2\eps_j\ovl{\eta}_j))$.

Assume moreover that $\lambda$ is $5$-deep in $\un{C}_0$. 
If $\sigma$, $\sigma'$ are irreducible constituents of $D_{\lambda,\un{\eps}}\otimes_{\F} F(\alpha_j)$ such that $\Ext^1_{\GL_2(k)}(\sigma,\sigma') \ne 0$,
  then either the nonsplit extension of $\sigma$ by $\sigma'$ or the nonsplit extension of $\sigma'$ by $\sigma$ occurs as subquotient of $D_{\lambda,\un{\eps}}\otimes_{\F} F(\alpha_j)$.
\end{prop}

We remark that if $\eps_j = 0$, then the arguments in the proof simplify and show that $D_{\lambda,\un{\eps}}\otimes_{\F} F(\alpha_j) \cong D_{\lambda+\alpha_j,\un{\eps}} \oplus D_{\lambda,\un{\eps}} \oplus D_{\lambda-\alpha_j,\un{\eps}}.$

\begin{proof}
By Lemma \ref{socle}, we know what are the socle and cosocle of $D_{\lambda,\un{\eps}}\otimes_{\F} F(\alpha_j)$.

During this proof, we will use the notation $\eta'_{\cJ}\defeq\sum_{i\in \cJ}\eps_i\eta_i$ %
(note that $\eta'_{\cJ}$ depends on the sign $\un{\eps}$). We recall that $\t_\lambda(2\eps_j\ovl{\eta}_j+\ovl{\eta}'_{\cJ})=\t_{\lambda+\eps_j\alpha_j}(\ovl{\eta}'_{\cJ})$ by Remark \ref{rk:t_lambda}\ref{it:t_lambda:3}. 
By Lemma \ref{lemmJH}, there exists a unique (up to scalar) nonzero map \[\Proj_{\GL_2(k)}F(\t_\lambda(2\eps_j\ovl{\eta}_j+\ovl{\eta}'_{\cJ}))\rightarrow D_{\lambda,\un{\eps}}\otimes_{\F} F(\alpha_j);\] let $W_1$ be its image. The socle of $W_1$ is contained in the socle of $D_{\lambda,\un{\eps}}\otimes_{\F} F(\alpha_j)$. But $F(\t_\lambda(2\eps_j\ovl{\eta}_j))$ is the only constituent of this socle which is also a constituent of $\Proj_{\GL_2(k)}F(\t_\lambda(2\eps_j\ovl{\eta}_j+\ovl{\eta}'_{\cJ}))$, cf.~Lemmas \ref{lm:princ-series}\ref{it:princ-series-2} and \ref{lm:change-origin}. This implies that $W_1$ is a quotient of $\Proj_{\GL_2(k)}F(\t_\lambda(2\eps_j\ovl{\eta}_j+\ovl{\eta}'_{\cJ}))$ with socle $F(\t_\lambda(2\eps_j\ovl{\eta}_j))$ and such that $[W_1:F(\t_\lambda(2\eps_j\ovl{\eta}_j))]=1$. We conclude that $W_1$ is isomorphic to $D_{\lambda+\eps_j\alpha_j,\un{\eps}}$. Let $Q$ be the quotient of $D_{\lambda,\un{\eps}}\otimes_{\F} F(\alpha_j)$ by $W_1$. Then $Q$ has cosocle isomorphic to the direct sum of $F(\t_\lambda(\ovl{\eta}'_\cJ))$ and $F(\t_\lambda(-2\eps_j\ovl{\eta}_j+\ovl{\eta}'_\cJ))$. Let $W_2$ be the image in $Q$ of the unique nonzero map $\Proj_{\GL_2(k)}F(\t_\lambda(\ovl{\eta}'_\cJ))\rightarrow Q$ and let $W_3\defeq Q/W_2$. Then $W_3$ is a quotient of $\Proj_{\GL_2(k)}F(\t_\lambda(-2\eps_j\ovl{\eta}_j+\ovl{\eta}'_\cJ))$.

We claim that $F(\lambda)$ is in the socle of $W_2$. Let's assume it for now. As $W_2$ is multiplicity-free, it has a unique quotient with socle $F(\lambda)$, namely $W_2$ has a quotient isomorphic to $D_{\lambda,\un{\eps}}$.

We can check that the Serre weight $F(\t_\lambda(-2\eps_j\ovl{\eta}_j))$ is not a subquotient of $\Proj_{\GL_2(k)}F(\t_\lambda(\ovl{\eta}'_\cJ))$ (again, by Lemmas \ref{lm:princ-series}\ref{it:princ-series-2} and \ref{lm:change-origin}) so that $F(\t_\lambda(-2\eps_j\ovl{\eta}_j))$ is a constituent of the socle of $W_3$. As above, we can conclude that $W_3$ has a quotient isomorphic to $D_{\lambda-\eps_j\alpha_j,\un{\eps}}$. It follows from length considerations that we must have $W_2\simeq D_{\lambda,\un{\eps}}$ and $W_3\simeq D_{\lambda-\eps_j\alpha_j,\un{\eps}}$.

We still have to prove that $F(\lambda)$ is contained in the socle of $W_2$ or equivalently that $F(\lambda)$ is a subquotient of $W_2$. Assume it is not the case. Let $\widetilde{W}_2$ be the image in $D_{\lambda,\un{\eps}}\otimes_{\F} F(\alpha_j)$ of the unique nonzero map $\Proj_{\GL_2(k)}F(\t_\lambda(\ovl{\eta}'_\cJ))\rightarrow D_{\lambda,\un{\eps}}\otimes_{\F} F(\alpha_j)$. Then $W_2$ is a quotient of $\widetilde{W}_2$ and the kernel of $\widetilde{W}_2\rightarrow W_2$ is contained in $W_1$. Thus $F(\lambda)$ is not a subquotient of $\widetilde{W}_2$ (as $\eps_j \ne 0$). The socle of $\widetilde{W}_2$ is contained in the socle of $D_{\lambda,\un{\eps}}\otimes_{\F} F(\alpha_j)$, which itself is equal to $F(\t_{\lambda}(2\eps_j\ovl{\eta}_j))\oplus F(\lambda)\oplus F(\t_{\lambda}(-2\eps_j\ovl{\eta}_j))$ by Lemma \ref{socle} (as $\eps_j \ne 0$). However, $F(\lambda)$ does not appear in the socle of $\widetilde{W}_2$ by hypothesis, neither does $F(\t_\lambda(-2\eps_j\ovl{\eta}_j))$ since it is not a subquotient of $\Proj_{\GL_2(k)}F(\t_\lambda(\ovl{\eta}'_\cJ))$. The socle of $\widetilde{W}_2$ is then equal to $F(\t_\lambda(2\eps_j\ovl{\eta}_j))$. By multiplicity-freeness, we have $\widetilde{W}_2\simeq I(F(\t_\lambda(2\eps_j\ovl{\eta}_j)),F(\t_\lambda(\ovl{\eta}'_\cJ)))$. Consequently $\widetilde{W}_2/F(\t_\lambda(2\eps_j\ovl{\eta}_j))$ contains $F(\t_\lambda(\eps_j\ovl{\eta}_j))$ in its socle by Lemma~\ref{lm:princ-series}\ref{it:princ-series-3}. This contradicts Lemma \ref{lemmaPS}. Namely if $V$ is the unique nonsplit extension of $F(\t_\lambda(\eps_j\ovl{\eta}_j))$ by $F(\lambda)$ (using $\eps_j \ne 0$), then $V\subset D_{\lambda,\un{\eps}}$ and $V\otimes_\F F(\alpha_j)\subset D_{\lambda,\un{\eps}}\otimes_\F F(\alpha_j)$ and Lemma \ref{lemmaPS} shows that $F(\t_{\lambda}(\eps_j\ovl{\eta}_j))$ occurs in $V\otimes_\F F(\alpha_j)$ but is not contained in the socle of $(V\otimes_\F F(\alpha_j))/F(\t_\lambda(2\eps_j\ovl{\eta}_j))$.

The assertion about the socle of cokernel of $D_{\lambda+\eps_j\alpha_j,\un{\eps}}\hookrightarrow D_{\lambda,\un{\eps}}\otimes_{\F} F(\alpha_j)$ is a consequence because we get the lower bound from Lemma~\ref{socle} and the upper bound from the filtration~\eqref{eq:incr-filtr}.

We now justify the final assertion of the proposition.
Recall that $J_{\un\eps} = \set{ i \in \cJ : \eps_i \ne 0 }$, which contains $j$ by assumption.
We first note that the irreducible constituents of $D_{\lambda+\eps_j\alpha_j,\un{\eps}}$, resp.\ $D_{\lambda,\un{\eps}}$, resp.\ $D_{\lambda-\eps_j\alpha_j,\un{\eps}}$ are given by $F(\t_\lambda(\sum_{i\in J_{\un\eps}}\eps_i a_i\ovl{\eta}_i))$, where $0 \le a_i \le 1$ for all $i \ne j$ and where $2 \le a_j \le 3$, resp.\ $0 \le a_j \le 1$, resp.\ $-2 \le a_j \le -1$.
By Lemma~\ref{lm:ext1} we deduce that $\sigma$ and $\sigma'$ occur either in the same or in consecutive graded pieces in \eqref{eq:incr-filtr}.
If they occur in the same graded piece, a nonsplit extension between $\sigma$ and $\sigma'$ has to occur in $D_{\lambda,\un{\eps}}\otimes_{\F}F(\alpha_j)$ by definition of $D_{\lambda,\un{\eps}}$ \eqref{eq:D-lambda-eps}.
If they occur in consecutive graded pieces, we may assume by symmetry %
and by Lemma~\ref{lm:ext1} that $\sigma \cong F(\t_\lambda(\sum_{i\in J_{\un\eps}}\eps_i a_i\ovl{\eta}_i))$ and $\sigma' \cong F(\t_\lambda(\sum_{i\in J_{\un\eps}}\eps_i a_i\ovl{\eta}_i - \eps_j \ovl{\eta}_j))$, where $0 \le a_i \le 1$ for all $i \ne j$ and where $a_j \in \{0,2\}$.
Let $V$ be the unique nonsplit extension of $F(\t_\lambda(\sum_{i\in J_{\un\eps}\setminus\{j\}}\eps_i a_i\ovl{\eta}_i + \eps_j \ovl{\eta}_j))$ by $F(\t_\lambda(\sum_{i\in J_{\un\eps}\setminus\{j\}}\eps_i a_i\ovl{\eta}_i))$ (using $\eps_j \ne 0$), which is a subquotient of $D_{\lambda,\un{\eps}}$.
By Lemma~\ref{lemmaPS} (in particular, the extension graph \eqref{eq:ext-graph-tensor} below it), we deduce that the nonsplit extension of $\sigma'$ by $\sigma$ occurs in $V \otimes_{\F} F(\alpha_j)$ and hence in $  D_{\lambda,\un{\eps}}\otimes_{\F}F(\alpha_j)$.
\end{proof}

\begin{thm}\label{thm:elimination}%
  Fix $\lambda \in X_1(\un{T})$ which is $7$-deep in $\un{C}_0$ and $\un{\eps}\in\{\pm1\}^{\cJ}$. We set
  \[\mathcal{W}_{-\un{\eps}}\defeq\set{F(\t_\lambda(-\sum_{j\in J}\eps_j\ovl{\eta}_j)) : J\subset\cJ}.\]
  There exists a largest subrepresentation $W$ of $(\Inj_{K/Z_1}F(\lambda))[\mathfrak{m}_{K_1}^2]$ satisfying $[W : \tau]=\delta_{F(\lambda),\tau}$ for $\tau\in\mathcal{W}_{-\un{\eps}}$. Moreover it has the following properties:
  \begin{enumerate}%
  \item 
  \label{it:elimination-1}
  $W^{K_1}=D_{\lambda,\un{\eps}}$;
  \item 
    \label{it:elimination-2}
  the representation $W$ is an extension of $\bigoplus_{0\leq i\leq f-1}D_{\lambda+\eps_i\alpha_i,\un{\eps}}$ by $D_{\lambda,\un{\eps}}$;
  \item
    \label{it:elimination-3}
   the representation $W$ is multiplicity-free;
  \item 
    \label{it:elimination-4}
  the cosocle of $W$ is isomorphic to $\bigoplus_{0\leq j\leq f-1}F(\t_\lambda(2\eps_j\ovl{\eta}_j+\sum_{0\leq i\leq f-1}  \eps_i\ovl{\eta}_i))$;
  \item 
    \label{it:elimination-5}
  its submodule structure is determined by: for $0\leq a_i\leq 3$ such that $\sigma_{\un{a}}= F(\t_\lambda(\sum \eps_i a_i \ovl{\eta}_i))$ is a subquotient of $W$, the unique subrepresentation of $W$ with cosocle $\sigma_{\un{a}}$ has constituents $\sigma_{\un{b}}$ for all $\un{b}$ such that $0 \le b_i \le a_i$ for all $i$.
  \end{enumerate}
\end{thm}

\begin{rem}\label{rk:elimination}
  The proof shows that $\lambda$ only needs to be $4$-deep in $\un{C}_0$ for $W$ to exist and for part~\ref{it:elimination-1} to hold. 
  In particular, in this case $W^{K_1}=D_{\lambda,\un{\eps}}$ is the largest subrepresentation of $(\Inj_{K/Z_1}F(\lambda))[\mathfrak{m}_{K_1}] = \Inj_{\GL_2(k)} F(\lambda)$ satisfying $[W^{K_1} : \tau]=\delta_{F(\lambda),\tau}$ for $\tau\in\mathcal{W}_{-\un{\eps}}$.
\end{rem}

\begin{proof}
  Let $I_\lambda\defeq \Inj_{\GL_2(k)}F(\lambda)$ and
  let $\widetilde{I}_\lambda \defeq (\Inj_{K/Z_1}F(\lambda))[\mathfrak{m}_{K_1}^2]$, which is finite-dimensional by dualising and using Nakayama's lemma. We have $I_\lambda=\widetilde{I}_\lambda[\mathfrak{m}_{K_1}]$.

The existence of a largest subrepresentation $W\subset\widetilde{I}_{\lambda}$ satisfying the desired hypothesis follows exactly as in \cite[Prop.\ 13.1]{BP}. As the representation $D_{\lambda,\un{\eps}}$ satisfies $[W : \tau]=\delta_{F(\lambda),\tau}$ for $\tau\in\mathcal{W}_{-\un{\eps}}$ by Lemma \ref{lm:princ-series}\ref{it:princ-series-3}, we have $D_{\lambda,\un{\eps}}\subset W^{K_1}$. 
Conversely, note that $W^{K_1}$ is a subrepresentation of $\widetilde{I}_{\lambda}^{K_1} \cong I_{\lambda}$. As $[W^{K_1}:F(\lambda)] = 1$ it follows by
\cite[Prop.\ 3.6 \& Cor.~3.11]{BP} that $W^{K_1}$ is multiplicity-free. By Lemma~\ref{lm:princ-series}\ref{it:princ-series-3} and our hypothesis on multiplicities,
$\JH(W^{K_1}) \subset \JH(D_{\lambda,\un{\eps}})$. Hence $W^{K_1}=D_{\lambda,\un{\eps}}$, proving \ref{it:elimination-1}.

Consider the short exact sequence:
\[ 0\rightarrow D_{\lambda,\un{\eps}}\rightarrow W\rightarrow W/D_{\lambda,\un{\eps}}\rightarrow0.\]
The long exact sequence of $K_1/Z_1$-invariants gives an injection 
\[ W/D_{\lambda,\un{\eps}}=(W/D_{\lambda,\un{\eps}})^{K_1}\hookrightarrow H^1(K_1/Z_1,D_{\lambda,\un{\eps}})\simeq D_{\lambda,\un{\eps}}\otimes_{\F} H^1(K_1/Z_1,\F),\]
where the last isomorphism holds because $K_1$ acts trivially on $D_{\lambda,\un{\eps}}$. 
Using the isomorphism $H^1(K_1/Z_1,\F)\simeq\bigoplus_{j=0}^{f-1} F(\alpha_j)$ (see \cite[Prop.\ 5.1]{BP}), we have:
\[ W/D_{\lambda,\un{\eps}}\hookrightarrow \bigoplus_{j=0}^{f-1} (D_{\lambda,\un{\eps}}\otimes_{\F} F(\alpha_j)).\]
For each $0\leq j\leq f-1$, we have a decomposition:
\[ 0\rightarrow D_{\lambda+\eps_j\alpha_j,\un{\eps}}\rightarrow D_{\lambda,\un{\eps}}\otimes_{\F} F(\alpha_j)\rightarrow Q_j\rightarrow0\]
with $\soc_{\GL_2(k)} Q_j=F(\lambda)\oplus F(\t_{\lambda}(-2\eps_j\ovl{\eta}_j))$ by Proposition \ref{prop:main}.

The assumption $[W : F(\lambda)]=1$ implies that
\[ \soc_K(W/D_{\lambda,\un{\eps}})=\soc_K(W/W^{K_1})\hookrightarrow \bigoplus_i F(\t_{\lambda}(\pm2\eps_j\ovl{\eta}_j)).\]
For $0\leq j\leq f-1$, Lemma \ref{lm:ext1} implies that the representation $F(\t_\lambda(-2\eps_j\ovl{\eta}_j))$ has no extension with Jordan--H\"older factors of $D_{\lambda,\un{\eps}}$, consequently the Serre weights $F(\t_\lambda(-2\eps_j\ovl{\eta}_j))$ are not in the socle of $W/D_{\lambda,\un{\eps}}$. We conclude that the image of $W/D_{\lambda,\un{\eps}}$ in $Q_j$ is zero and that $W/D_{\lambda,\un{\eps}}\subset\bigoplus_{j=0}^{f-1} D_{\lambda+\eps_j\alpha_j,\un{\eps}}$.

Let $V$ be the representation of $K$ constructed in Proposition \ref{prop:K-rep-by-ind}. Note that the deepness assumption on $\lambda$ allows us to apply it with $B_i=4$ if $\eps_{i-1}=1$ and $B_i=3$ if $\eps_{i-1}=-1$. Let $W'=V[\mathfrak{m}_{K_1}^2]$. By Proposition \ref{prop:K-rep-by-ind} we have $[W':\tau]=\delta_{F(\lambda),\tau}$ for $\tau\in\mathcal{W}_{-\un{\eps}}$ so that $W'\subset W$ by maximality of $W'$. It follows from Proposition \ref{prop:J-fil} 
with $n=2$ and $n=1$ that
\[ \cosoc_K(W')=\bigoplus_{0\leq j\leq f-1} F(\t_\lambda(2\eps_j\ovl{\eta}_j+\sum_i\eps_i\ovl{\eta}_i))\]
and $W'^{K_1}=D_{\lambda,\un{\eps}}=W^{K_1}$. By what precedes we have an inclusion
\[ W'/W'^{K_1}\subset W/W^{K_1}\subset\bigoplus_{j=0}^{f-1} D_{\lambda+\eps_j\alpha_j,\un{\eps}}.\]
However, the outside terms have the same cosocle, so these inclusions are equalities. From $W^{K_1}=W'^{K_1}$ and $W'/W'^{K_1}= W/W^{K_1}$ we deduce that $W'=W$. This also proves that $W/D_{\lambda,\un{\eps}}$ is isomorphic to $\bigoplus_{j=0}^{f-1} D_{\lambda+\eps_j\alpha_j,\un{\eps}}$ and gives \ref{it:elimination-2}. We then deduce properties \ref{it:elimination-3} to \ref{it:elimination-5} from the properties of $V$ given by Proposition \ref{prop:K-rep-by-ind}.
\end{proof}

\begin{cor}\label{cor:J-fil}
  Let $\rhobar : G_L\rightarrow\GL_2(\F)$ be a tame Galois representation such that $\rhobar|_{I_L}\simeq\overline{\tau}(s,\mu)$ such that $\mu-\eta$ is $8$-deep in $\un{C}_0$.
 \begin{enumerate}%
  \item 
  \label{it:J-fil-1}
  Let $\tau$ be a finite-dimensional semisimple representation of $K$ over $\F$ of the form $\tau\cong\bigoplus_{\sigma\in W(\rhobar)}\sigma^{\oplus m_{\sigma}}$, with $m_{\sigma}\geq 1$ for all $\sigma$. Then there exists a largest $K$-subrepresentation $V$ inside $(\Inj_{K/Z_1} \tau)[\mathfrak{m}_{K_1}^2]$ with $\soc_K V = \tau$ such that
    for all $\sigma \in W(\rhobar)$,
    \[ [V : \sigma]=[\tau : \sigma]=m_\sigma.\]
    Moreover $V \cong \bigoplus_{\sigma \in W(\rhobar)} V_\sigma^{\oplus m_{\sigma}}$, where $V_\sigma \subset (\Inj_{K/Z_1} \sigma)[\mathfrak{m}_{K_1}^2]$ denotes the largest $K$-subrep\-resenta\-tion of $(\Inj_{K/Z_1}\sigma)[\mathfrak{m}_{K_1}^2]$ such that $[V_\sigma:\sigma'] = \delta_{\sigma, \sigma'}$ for all $\sigma' \in W(\rhobar)$.
  \item 
  \label{it:J-fil-2}
  Fix $\sigma \in W(\rhobar)$ and choose $\lambda \in X_1(\un{T})$ such that $\sigma \cong F(\lambda)$. There exists $\un{\eps}=(\eps_i) \in \{\pm 1\}^{\cJ}$ such that
    $W(\rhobar) = \{ F(\t_\lambda(-\sum_{i \in J} \eps_i \ovl{\eta}_i)) : J \subset \cJ\}$. Then $V_\sigma$ is
    multiplicity-free and $V_\sigma^{K_1}\simeq D_{\lambda,\un{\eps}}$. Moreover the Jordan--H\"older constituents of $V_\sigma$ are the $\sigma_{\un{a}} = F(\t_\lambda(\sum \eps_i a_i \ovl{\eta}_i))$,
    where $a_i \ge 0$ and $\sum_i \lfloor a_i/2\rfloor \le 1$, with submodule structure determined as follows: the unique
    subrepresentation of $V_{\sigma}$ with cosocle $\sigma_{\un{a}}$ has constituents $\sigma_{\un{b}}$ for all $\un{b}$ such that
    $0 \le b_i \le a_i$ for all $i$.
    \item
    \label{it:J-fil-3} If $\sigma$ and $\sigma'$ are both in $W(\rhobar)$ and nonisomorphic, the sets $\JH(V_\sigma)$ and $\JH(V_{\sigma'})$ are disjoint.
  \end{enumerate}
\end{cor}

\begin{rem}
In Corollary \ref{cor:J-fil}\ref{it:J-fil-2} the condition $a_i \ge 0$ and $\sum_i \lfloor a_i/2\rfloor \le 1$ means exactly that $a_i \in \set{0,1,2,3}$ and that at most one of them is $\geq2$.
\end{rem}

\begin{proof}
 Part \ref{it:J-fil-1} follows by the same argument as in the proof of \cite[Prop.\ 13.1]{BP}. For the existence of $V$ we have to prove that, if $V_1$ and $V_2$ are two subrepresentations of $(\Inj_{K/Z_1} \tau)[\mathfrak{m}_{K_1}^2]$ such that $\Hom_{K}(\sigma,V_i)\simeq\Hom_{K}(\Proj_{K}\sigma,V_i)$ for all $\sigma \in W(\rhobar)$, then $V_1+V_2$ has the same property. This follows from the exactness of the sequence
  \begin{multline*} 0\longrightarrow\Hom_{K}(\Proj_{K/Z_1}\sigma,V_1\cap V_2) \\ 
  \longrightarrow \Hom_{K}(\Proj_{K/Z_1}\sigma,V_1)\oplus \Hom_{K}(\Proj_{K/Z_1}\sigma,V_2) \\ \longrightarrow\Hom_{K}(\Proj_{K/Z_1}\sigma,V_1+V_2)\longrightarrow0.\end{multline*}
  By assumption, we have
  \[ \dim_\F\big(\Hom_{K}(\Proj_{K/Z_1}\sigma,V_i)\big)=\dim_\F\big(\Hom_{K}(\Proj_{K/Z_1}\sigma,V_1\cap V_2)\big)=m_{\sigma}\]
  so that
  \[ \dim_\F\big(\Hom_{K}(\Proj_{K/Z_1}\sigma,V_1+V_2)\big)=m_{\sigma}=\dim_\F\big(\Hom_{K}(\sigma,V_1+V_2)\big).\]
  As $\tau\simeq\bigoplus_{\sigma\in W(\rhobar)}\sigma^{\oplus m_{\sigma}}$, there is a $K$-equivariant inclusion
  \[V\hookrightarrow\bigoplus_{\sigma\in W(\rhobar)}(\Inj_{K/Z_1}\sigma)^{\oplus m_{\sigma}}[\mathfrak{m}_{K_1}^2]\]
  and, by maximality of $V$, we have
  \[ \bigoplus_{\sigma\in W(\rhobar)}V_\sigma^{\oplus m_{\sigma}}\subset V\subset\bigoplus_{\sigma\in W(\rhobar)} (\Inj_{K/Z_1}\sigma)^{\oplus m_{\sigma}}[\mathfrak{m}_{K_1}^2].\]
By definition of $V_\sigma$, the socle of $(\Inj_{K/Z_1}\sigma)[\mathfrak{m}_{K_1}^2]/V_\sigma$ contains only Serre weights of $W(\rhobar)$. Hence the socle of $V/(\bigoplus_{\sigma\in W(\rhobar)}V_\sigma^{\oplus m_{\sigma}})$ has the same property. However it follows from the exactness of $\Hom_{K}(\Proj_{K/Z_1}\sigma,-)$ that we have for all $\sigma\in W(\brho)$
  \[  \Hom_{K}\Big(\Proj_{K/Z_1}\sigma,V/\big(\bigoplus_{\sigma\in W(\rhobar)}V_\sigma^{\oplus m_{\sigma}}\big)\Big)=0,\]
so that $\soc_K(V/(\bigoplus_{\sigma\in W(\rhobar)}V_\sigma^{\oplus m_{\sigma}}))=0$ and
  \[ V=\bigoplus_{\sigma\in W(\rhobar)}V_\sigma^{\oplus m_{\sigma}}.\]
  
  Now we prove part \ref{it:J-fil-2}. 
  By Proposition \ref{prop:SW:graph} the elements of $W(\rhobar)$ are of the form $F(\t_{\mu-\eta}(s\ovl{\eta}_{J'}))$ for $J'\subseteq \cJ$ and we let $J\subset\cJ$ be such that $\sigma \cong F(\t_\lambda(0)) \cong F(\t_{\mu-\eta}(s\ovl{\eta}_J))$.
  In particular, all elements of $W(\rhobar)$ are 7-deep in $\un{C}_0$ (for example, by Remark~\ref{rk:t_lambda}\ref{it:t_lambda:4}). 
  By Remark~\ref{rk:graph-auto} there exists $\un{\eps}=(\eps_i) \in \{\pm 1\}^{\cJ}$ such that $W(\rhobar) = \{ F(\t_\lambda(-\sum_{i \in J'} \eps_i \ovl{\eta}_i)) : J' \subset \cJ\}$.
The properties of $V_\sigma$ are then immediate consequences of Theorem \ref{thm:elimination}\ref{it:elimination-1}, \ref{it:elimination-3}, and \ref{it:elimination-5}.
  
  For part \ref{it:J-fil-3}, let $\lambda, \lambda'\in X_1(\un{T})$ be such that $\sigma\cong F(\lambda)$, $\sigma'\cong F(\lambda')$ and $\un{\eps}$ such that
  \[ W(\rhobar)=\set{F(\t_\lambda(\sum_{i\in J} -\eps_i\ovl{\eta}_i)) : J\subset\cJ}.\]
  Then 
  \begin{equation*}
    \JH(V_\sigma) = \{ F(\t_\lambda(\sum_i \eps_i a_i \ovl{\eta}_i)) : a_i \ge 0, \sum_i \lfloor a_i/2\rfloor \le 1 \}.\label{eq:11}
  \end{equation*}
  Choose $J\subset\cJ$ such that $F(\lambda')\cong F(\t_\lambda(-\sum_{i\in J}\eps_i\ovl{\eta}_i))$. 
  Then by part \ref{it:J-fil-2} and Remark~\ref{rk:graph-auto} we see that
  \begin{equation*}
    \JH(V_{\sigma'}) = \{ F(\t_\lambda(- \sum_J \eps_i (b_i+1)\ovl{\eta}_i + \sum_{\cJ \backslash J} \eps_i b_i \ovl{\eta}_i)) : b_i \ge 0, \sum_i \lfloor b_i/2\rfloor \le 1 \}.\label{eq:12}
  \end{equation*}
  (Note that $W(\rhobar)$ is obtained by putting $-1 \le b_i \le 0$.) If $\JH(V_\sigma)$ and $\JH(V_{\sigma'})$ are not disjoint, then $J = \emptyset$ (as $b_j+1 > 0$),
  contradicting $\sigma \not\cong \sigma'$.
\end{proof}

\begin{cor}\label{cor:V-K1}
Let $\rhobar$, $m_\sigma$ and $V$ be as in Corollary \ref{cor:J-fil}. Then 
\[V[\mathfrak{m}_{K_1}]=\bigoplus_{\sigma\in W(\rhobar)}D_{0,\sigma}(\brho)^{\oplus m_{\sigma}},\]
where $D_{0,\sigma}(\brho)$ is the representation of $\GL_2(k)$ constructed in \cite[\S13]{BP}.
\end{cor}
\begin{proof}
This follows from Corollary \ref{cor:J-fil}\ref{it:J-fil-1} and \ref{it:J-fil-2}, as well as Remark~\ref{rk:elimination}.
\end{proof}

\subsection{Multiplicity one result for the pro-\texorpdfstring{$p$}{p}-Iwahori}\label{sec:multiplicityoneprop}

The aim of this subsection is to prove that some multiplicity one assumption on the first two layers of the $K_1$-socle filtration implies a multiplicity one result on the first three layers of the $I_1$-socle filtration of an admissible smooth representation of $\GL_2(L)$.

\begin{prop}\label{prop-forspecialW}
Suppose that $\chi = \chi_\lambda$ with $2<\ang{\lambda,\alpha_i^\vee}<p-3$ for all $0\leq i\leq f-1$. 
Let $W$ be a smooth and finite length representation of $I$ over $\F$ satisfying the following conditions:
\begin{enumerate}
\item[$\bullet$] both the socle and cosocle of $W$ are irreducible and isomorphic to $\chi$;
\item[$\bullet$] we have $\soc_I(W)\subsetneq \rad_I(W)$ and $\rad_I(W)/\soc_I(W)$ is semisimple; in other words, the Loewy length of $W$ is equal to $3$.
\end{enumerate} 
Let $Q$ be a nonzero quotient of $\Ind_I^{K}W$ such that $[Q:\sigma_{\chi}]=1$. Then the composition 
\[\chi=\soc_I(W)\hookrightarrow W\overset{f}{\ra} Q|_I\]
is zero, where $f$ is induced by Frobenius reciprocity.
\end{prop}

\begin{proof}
Assume that $f|_{\soc_I(W)}$ is nonzero, or equivalently $f$ is injective, for a contradiction. Then the image of $
\Ind_I^{K}\soc_I(W)\ra Q$ is nonzero and has cosocle $\sigma_{\chi}$ (recall that $\sigma_{\chi}$ is the cosocle of $\Ind_I^K\chi$). Since $[Q:\sigma_{\chi}]=1$ by assumption, we may replace $Q$ by the image of the unique (up to scalar) nonzero morphism $Q\ra \Inj_{K/Z_1}\sigma_{\chi}$, and therefore assume $\soc_{K}(Q)\cong \sigma_{\chi}$. Indeed, letting $Q'$ be this image, we have $[\ker(Q\rightarrow Q') : \sigma_\chi]=0$. Since $\sigma_\chi$ is a Jordan--H\"older factor of the image of $\Ind_I^{K}\soc_I (W)$ in $Q$, the map from $\Ind_I^{K}\soc_I (W)$ to $Q'$ is nonzero and hence the composite $\soc_I(W)\rightarrow Q \onto Q'$ is nonzero. From now on we suppose that $\soc_{K}(Q)\cong \sigma_\chi$. Note that, the image of the map
\[ \Ind_I^{K}\soc_I(W)\longrightarrow Q\]
is then exactly $\soc_{K} Q=\sigma_\chi$. Also note that $Q/\sigma_{\chi}\neq0$, otherwise $f$ could not be injective because $[W:\chi]=2$ while $[\sigma_{\chi}|_I:\chi]=1$.

Using Lemma \ref{lemma-Ext1=dim1}, we deduce that $\rad_I(W)/\soc_I(W)$ is isomorphic to a direct sum of characters of the form $\chi\alpha_i^{\pm1}$, each appearing at most once. Let $S_+$ (resp.\ $S_-$) be the set of characters appearing in $\rad_I(W)/\soc_I(W)$ and of the form $\chi\alpha_i$ (resp.\ $\chi\alpha_i^{-1}$). Also let $W'\subset W$ be the subrepresentation defined by
\[0\ra \chi\ra W'\ra \bigoplus_{\chi'\in S_-}\chi'\ra0,\]
and $W''=W/W'$ so that 
\[0\ra \bigoplus_{\chi'\in S_+}\chi'\ra W''\ra\chi\ra0.\] 
Note that both $W'$ and $W''$ are fixed by $K_1$, see Lemma \ref{lemma-Ext1=dim1}\ref{it1-Ext1=dim1}. 

We claim that $f(W')$ is contained in $\sigma_{\chi}$. This is equivalent to showing that the morphism 
$\Ind_I^{K}W'\ra Q$ (induced from $f$ by Frobenius reciprocity)
has image contained in (and hence equal to) $\sigma_{\chi}$. Let $Q'$ denote the image of $\Ind_I^KW'$. Clearly, $Q'$ is contained in $Q^{K_1}$, which itself is a subrepresentation of $\Inj_{\GL_2(k)}\sigma_{\chi}$. If $\sigma_{\chi}\subsetneq Q'$, then, as $f(\soc_I W) \subset \sigma_\chi$, we would obtain a nonzero morphism $\Ind_I^{K}(W'/\chi)\onto Q'/\sigma_{\chi}\hookrightarrow (\Inj_{\GL_2(k)}\sigma_{\chi})/\sigma_{\chi}$. However, one checks that no Jordan--H\"older factors of $\Ind_{I}^K\chi'$ for $\chi'\in S_-$ can appear in $\Inj_{\GL_2(k)}\sigma_{\chi}$, using Lemma \ref{lm:princ-series}. Hence we have $Q'=\sigma_{\chi}$.

We obtain a surjective morphism 
\[\Ind_I^{K}W''\twoheadrightarrow Q''\defeq Q/\sigma_{\chi}\neq0.\]
Since $[Q'':\sigma_{\chi}]=0$, Lemma \ref{lemma-Echi'=plus} implies that no Jordan--H\"older factors of $Q''$ have nonsplit extensions with $\sigma_{\chi}$. However, as $Q$ has irreducible socle $\sigma_{\chi}$ we obtain a contradiction. 
\end{proof}

\begin{definit}\label{definit-fullset}
Let $V$ be a \emph{semisimple} smooth representation of $I$ over $\F$. We say $V$ is \emph{connected} if the following condition is satisfied: for any two smooth characters $\chi\neq \chi''$ of $I$ occurring in $V$ such that $\chi''\in \soc_I(W_{\chi,3})$, there exists a character $\chi'$ occurring in $V$ such that $\Ext^1_{I/Z_1}(\chi',\chi'')\neq0$ and $\Ext^1_{I/Z_1}(\chi,\chi')\neq0$. 
\end{definit}

The motivation of the above definition comes from the following result.

\begin{lem}\label{lem:connected}
Let $\brho:G_L\ra\GL_2(\F)$ be a $6$-generic representation,
not necessarily semisimple. Let $D_0(\brho)$ be the $\GL_2(k)$-representation constructed in \cite[\S13]{BP}. Then $D_1(\brho)\defeq D_{0}(\brho)^{I_1}$ is connected in the sense of Definition \ref{definit-fullset}. As a consequence, if $V$ is a semisimple representation of $I$ such that $\JH(V)=\JH(D_1(\brho))$ up to multiplicity, then $V$ is connected. 
\end{lem}

\begin{proof}
We first note the general fact that up to multiplicity \[\JH(D_0(\rhobar))=\JH\big(\bigoplus_{\sigma\in W(\rhobar)}\Inj_{\GL_2(k)}\sigma\big)\]
Indeed, the inclusion ``$\subseteq$'' is trivial and ``$\supseteq$'' follows from \cite[Lemma 12.8, Prop.\ 13.4]{BP}. As a consequence, we have
\[\JH(D_0(\rhobar))\subseteq \JH(D_0(\rhobar^{\rm ss})). \] 
We write $\rhobar^\mathrm{ss}|_{I_L}\simeq\overline{\tau}(s,\mu)$ such that $\mu-\eta$ is $6$-deep in $\un{C}_0$.
As in the proof of Corollary~\ref{cor:J-fil}(ii) we know that $W(\rhobar^{\rm ss}) = \{ F(\t_{\mu-\eta}(\sum_J \eps_i \ovl\eta_i)) : J \subset \cJ \}$ for some choice of $\eps_i \in \{\pm 1\}$.
By using Remarks~\ref{rk:elimination} and \ref{rk:graph-auto} we see that $\JH(D_0(\rhobar^{\rm ss})) = \{ F(\t_{\mu-\eta}(\sum \eps_i a_i \ovl\eta_i)) : -1 \le a_i \le 2 \}$.

Suppose $\chi$ and $\chi''$ are as in Definition~\ref{definit-fullset} for $V = D_1(\brho)$. By Lemma \ref{lemma-Ext1=dim1}, $\chi''$ has the form $\chi\alpha_{i_1}^{\pm1}\alpha_{i_2}^{\pm1}$ for some $0\leq i_1,i_2\leq f-1$. 
Say $\chi = \sigma^{I_1}$ and $\chi'' = (\sigma'')^{I_1}$ for some $\sigma$, $\sigma'' \in \JH(D_0(\brho))$.
By the discussion in last paragraph, we may write $\sigma \cong F(\t_{\mu-\eta}(\sum \eps_i a_i \ovl\eta_i))$ and $\sigma'' \cong F(\t_{\mu-\eta}(\sum \eps_i a_i'' \ovl\eta_i))$
for some $-1 \le a_i, a''_i \le 2$.

First suppose that $i_1=i_2$. 
Recalling that $F(\lambda)^{I_1} = \chi_\lambda$ and $\t_{\lambda\pm 2\alpha_i}(\omega) = \t_{\lambda}(\omega \pm 4\ovl{\eta}_i)$ we see that 
$\sum \eps_i a_i'' \ovl\eta_i = \sum \eps_i a_i \ovl\eta_i \pm 4\ovl{\eta}_{i_1}$ for some $-1 \le a_i, a''_i \le 2$; contradiction.
(The 6-deepness of $\mu-\eta$ guarantees that we are staying inside $\Lambda_W^{\mu-\eta}$.)

Now suppose $i_1 \ne i_2$. As in the previous case we know that $|a_i-a''_i| = 2$ if $i \in \{i_1,i_2\}$ and $a_i = a''_i$ otherwise.
We let $a'_i \defeq a_i$ for $i \ne i_1$, $a'_{i_1} \defeq a''_{i_1}$,
$\sigma' \defeq F(\t_{\mu-\eta}(\sum \eps_i a'_i \ovl\eta_i))$, and $\chi' \defeq (\sigma')^{I_1}$. We claim that $\chi' \in D_1(\rhobar)^{I_1}$.
Equivalently we need to show that the unique principal series with cosocle $\sigma'$ contains an element of $W(\rhobar)$ as constituent
(then the principal series admits a quotient that contains precisely one element of $W(\rhobar)$ and that as its socle).
By Lemma~\ref{lm:princ-series}\ref{it:princ-series-1} and Remark~\ref{rk:graph-auto} the principal series with cosocle $\sigma$ has constituents $F(\t_{\mu-\eta}(\sum \eps_i a_i \ovl\eta_i + \sum_J \eps'_i \ovl\eta_i))$
($J \subset \cJ$) for certain signs $\eps'_i \in \{\pm 1\}$.
By Remark~\ref{rk:t_lambda}\ref{it:t_lambda:3} the same is true for the principal series with cosocle $\sigma'$ (resp.\ $\sigma''$), by replacing $a_i$ by $a'_i$ (resp.\ $a''_i$).
The claim follows, since the condition of containing a Serre weight of $W(\rhobar)$ is checked separately for each embedding.
(Use Proposition \ref{prop:SW:graph} if $\rhobar$ is semisimple and \cite[Prop.\ 3.2]{DanWild}, as well as \cite[Def.\ 3.5]{LMS}, otherwise.)

The last assertion immediately follows from the first one, because by definition the connectedness of $V$ depends only on $\JH(V)$ up to multiplicity. 
\end{proof}
We now consider an admissible smooth $G$-representation $\pi$ satisfying the following properties:
\begin{enumerate}[(a)]
\item 
\label{it:pi:properties-1}
$\pi[\mathfrak{m}_{K_1}^2]|_K$ is isomorphic to a subrepresentation of a direct sum
\[ \bigoplus_{\sigma\in \cW}\widetilde{D}_{\sigma}^{\oplus m_\sigma}\]
for some set of Serre weights $\cW$, some $K$-representations $\widetilde{D}_{\sigma}$ with $\soc_K \widetilde{D}_{\sigma} \cong \sigma$, and some integers $m_\sigma\geq 1$;
\item
\label{it:pi:properties-2}
the $K$-representation
\[\widetilde{D}\overset{\rm def}{=}\bigoplus_{\sigma\in \cW}\widetilde{D}_{\sigma}\] 
is multiplicity-free and for each Jordan--H\"older factor $\sigma'$ of $\widetilde D$ we have $\chi_{\sigma'} \ne \chi_{\sigma'}^s$
(equivalently, $1 < \dim_{\F} (\sigma') < q$).
\end{enumerate} 

In our application below we will have $\cW = W(\rhobar)$ for some tame mod $p$ Galois representation $\rhobar$.
Note that if $\chi\in \widetilde{D}^{I_1}$, then Frobenius reciprocity induces a nonzero morphism $\Ind_I^K\chi\ra \widetilde{D}^{K_1}$.
By condition \ref{it:pi:properties-2}, $\Ind_I^K \chi$ has irreducible cosocle $\sigma_\chi$, so there is a unique $\sigma \in \cW$
such that $\sigma_\chi$ occurs in $\widetilde{D}_{\sigma}^{K_1}$ (or equivalently, such that $\chi$ occurs in $\widetilde{D}_{\sigma}^{I_1}$).
In particular, $\sigma_{\chi}$ does not occur as a subquotient of $\widetilde{D}/\widetilde{D}^{K_1}$. 

We also note that $\wt{D}^{I_1}$ is multiplicity-free: for a character $\chi$ of $I$ we have $\Hom_I(\chi,\wt D^{I_1}) \cong \Hom_K(\Ind_I^K \chi, \wt D)$. 
If $\chi\in\widetilde{D}^{I_1}$, then $\Ind_I^K \chi$ has an irreducible cosocle as seen above. As moreover $\wt D$ is multiplicity-free,
we deduce that $\Hom_K(\Ind_I^K \chi, \wt D)$ is one-dimensional.

\begin{lem}\label{lemma-W2topi=dim1}
Let $\pi$ and $\widetilde{D}$ be as above satisfying the conditions {\upshape\ref{it:pi:properties-1}, \ref{it:pi:properties-2}}. 
Suppose $\chi\in\pi^{I_1}$ is of the form $\chi_\lambda$ with $2<\ang{\lambda,\alpha_i^\vee}<p-3$ for all $0\leq i\leq f-1$. 
Then the natural quotient morphism $W_{\chi,2}\twoheadrightarrow \chi$ induces an isomorphism
\[\Hom_{I}(\chi,\pi)\simto\Hom_I(W_{\chi,2},\pi).\] 
\end{lem}

\begin{proof}
Since $W_{\chi,2}$ is killed by $\mathfrak{m}_{I_1}^2$, any morphism $W_{\chi,2}\ra \pi|_I$ has image contained in
\[\pi[\mathfrak{m}_{I_1}^2]\subset\pi[\mathfrak{m}_{K_1}^2].\]

Let $f:W_{\chi,2}\ra \pi|_I$ be an $I$-equivariant morphism. For $\sigma\in \cW$, consider the map $f_\sigma:W_{\chi,2}\ra\widetilde{D}_\sigma^{\oplus m_\sigma}|_I$ obtained by composing $f$ with the projection to the corresponding direct factor in condition \ref{it:pi:properties-1}. 

Let $\chi'$ be a character in $\soc_I(W_{\chi,2})$. By Lemma \ref{lemma-Ext1=dim1}, there exists $i\in\cJ$ such that $\chi'=\chi\alpha_i^{\pm1}$ and the $\chi'$-isotypic subspace is $1$-dimensional.

We first consider the case where $\chi'$ is of the form $\chi\alpha_i^{-1}$ for some $i\in\cJ$. Assume for contradiction that $f$ is nonzero on the (one-dimensional) $\chi'$-isotypic space of $W_{\chi,2}$. Then there exists at least one $\sigma\in \cW$ such that $f_\sigma$ is nonzero on the $\chi'$-isotypic subspace of $W_{\chi,2}$. 

As a consequence of Lemma \ref{lemma-IndW2-multione} (and Frobenius reciprocity), no character $\psi$ of $\soc_I(W_{\chi,2})$ other than $\chi'$ can occur in $\widetilde{D}_{\sigma}^{I_1}$, otherwise $\sigma$ would be a common irreducible subquotient of both $\Ind_{I}^K\chi'$ and $\Ind_I^K\psi$. Hence, the map $f_\sigma$ factors through the quotient $E_{\chi',\chi}$ of $W_{\chi,2}$ and induces an embedding $E_{\chi',\chi}\hookrightarrow \widetilde{D}_{\sigma}^{\oplus m_\sigma}|_I$. Let \[\widetilde{f}_\sigma:\Ind_I^{K}E_{\chi',\chi}\ra \widetilde{D}_{\sigma}^{\oplus m_\sigma}\] 
be the induced morphism by Frobenius reciprocity. Lemma \ref{lemma-Echi'=minus} implies that the cosocle of $\Ind_I^{K}E_{\chi',\chi}$ is equal to that of $\Ind_I^{K}\chi$, i.e.\ $\sigma_{\chi}$, hence so is the cosocle of $\mathrm{Im}(\widetilde{f}_\sigma)$. Since $E_{\chi',\chi}$ is not $K_1$-invariant, neither is $\mathrm{Im}(\widetilde{f}_\sigma)$ because the morphism $E_{\chi',\chi}\ra \mathrm{Im}(\widetilde{f}_\sigma)|_I$ is injective. We deduce that $\sigma_{\chi}$ occurs in $\widetilde{D}_{\sigma}/\widetilde{D}_{\sigma}^{K_1}$. This contradicts \ref{it:pi:properties-2}, as remarked just before this lemma.

We conclude that the map $f$ is zero on all $\chi'$-isotypic subspaces of $W_{\chi,2}$ for $\chi'=\chi\alpha_i^{-1}$, $i\in\cJ$.

The general case can be reduced to the above case, using the fact that $\pi$ carries an action of $ t\defeq\smatr{0}1p0$. Namely let $f'$ be the map from $W_{\chi,2}^t$ (conjugate representation by $t$) to $\pi$ defined by $t\circ f$. As $f$ is $I$-equivariant, the map $f'$ is $I$-equivariant. As $W_{\chi,2}^t\simeq W_{\chi^s,2}$ and as the $\chi'$-isotypic subspace of $W_{\chi,2}$ coincides with the $\chi'^s$-isotypic subspace of $W_{\chi,2}^t$, it follows from the first case that $t\circ f$, and hence $f$, is zero on the $\chi'$-isotypic subspace of $W_{\chi,2}$ for $\chi'=\chi\alpha_i$ with $i\in\cJ$. As a consequence, $f$ is zero on $\soc_I(W_{\chi,2})$.
\end{proof}

We will not use the following corollary of Lemma \ref{lemma-W2topi=dim1} but we state it since the result can be useful.

\begin{cor}\label{cor-noembedding}
Let $\pi$ and $\widetilde{D}$ be as above satisfying the conditions {\upshape\ref{it:pi:properties-1}, \ref{it:pi:properties-2}}. 
Suppose $\chi\in\pi^{I_1}$ is of the form $\chi_\lambda$ with $2<\ang{\lambda,\alpha_i^\vee}<p-3$ for all $0\leq i\leq f-1$. 
Then for any character $\chi'\in\pi^{I_1}$ such that $\Ext^1_{I/Z_1}(\chi,\chi')\neq0$ there exists no $I$-equivariant embedding
\[E_{\chi',\chi}\hookrightarrow \pi|_I.\]
\end{cor}

We now make an additional assumption on $\pi$:
\begin{enumerate}[(c)]
\item
\label{it:pi:properties-3}
 $\pi^{I_1}$ is connected (cf.~Definition \ref{definit-fullset}). 
\end{enumerate}

\medskip
 
\begin{prop}\label{prop-W3topi=dim1}
Let $\pi$ and $\widetilde{D}$ be as above satisfying the conditions {\upshape\ref{it:pi:properties-1}, \ref{it:pi:properties-2}, \ref{it:pi:properties-3}}. 
Suppose $\chi\in\pi^{I_1}$ is of the form $\chi_\lambda$ with $2<\ang{\lambda,\alpha_i^\vee}<p-3$ for all $0\leq i\leq f-1$. 
Then the natural quotient morphism $W_{\chi,3}\twoheadrightarrow \chi$ induces an isomorphism
\[\Hom_{I}(\chi,\pi)\simto\Hom_{I}(W_{\chi,3},\pi).\]
\end{prop}

\begin{proof}
Let $f: W_{\chi,3}\ra \pi|_I$ be a nonzero $I$-equivariant morphism. 
It suffices to prove that $f$ factors through the cosocle $W_{\chi,3}\twoheadrightarrow \chi$. Let's assume this is not the case and derive a contradiction. Note that this implies that $f|_{\soc_I(W_{\chi,3})}$ is nonzero by Lemma \ref{lemma-W2topi=dim1}.

Step 1. We first show that $f$ is zero when restricted to $X''\defeq\oplus {\chi''}$, where the direct sum is taken over all characters $\chi''$ in $\soc_I(W_{\chi,3})$ which are different from $\chi$ (recall that $[W_{\chi,3}:\chi'']=1$ for such a $\chi''$). Indeed, if there exists such a $\chi''$ such that $f$ is nonzero when restricted to $\chi''$, then in particular $\chi''\in \pi^{I_1}$. Since $\pi^{I_1}$ is assumed to be connected by \ref{it:pi:properties-3}, we can find $\chi'\in \pi^{I_1}$ as in Definition \ref{definit-fullset}. By construction, $\chi'$ occurs in the second layer of the socle filtration of $W_{\chi,3}$ and Lemma \ref{lemma-chi''isinimage} shows that $\chi''$ occurs in the socle of the image of any nonzero morphism
\[W_{\chi',2}\ra W_{\chi,3}.\] But, the composition $W_{\chi',2}\ra W_{\chi,3}\overset{f}{\ra} \pi$ gives a morphism that does not factor through its cosocle $\chi'$, which contradicts Lemma \ref{lemma-W2topi=dim1}. As a consequence, $f$ factors through the quotient $W_{\chi,3}/X''$. 
Note that $W_{\chi,3}/X''$ is killed by $\mathfrak{m}_{K_1}^2$, because we may define a suitable subrepresentation $W'$ of $W_{\chi,3}/X''$, with quotient $W''$, such that both $W'$ and $W''$ are killed by $\mathfrak{m}_{K_1}$ (cf.\ the proof of Proposition \ref{prop-forspecialW}). Hence, $\mathrm{Im}(f)$ is contained in $\pi[\mathfrak{m}_{K_1}^2]$.

Step 2. Since $f|_{\soc_I(W_{\chi,3})}$ is nonzero, combining with Step 1, we deduce that $\chi$ occurs in the socle of $\mathrm{Im}(f)$. By \ref{it:pi:properties-1}, $\pi[\mathfrak{m}_{K_1}^2]\subset \bigoplus_{\sigma\in \cW} \widetilde{D}_\sigma^{\oplus m_\sigma}$, so there exists a projection $\mathrm{pr}:\bigoplus_{\sigma\in \cW} \widetilde{D}_\sigma^{\oplus m_\sigma}\twoheadrightarrow \widetilde{D}_\sigma$ such that $\mathrm{pr}\circ f$ remains nonzero on the $\chi$-isotypic part of $\soc_I(W_{\chi,3})$.
By Frobenius reciprocity $\sigma_{\chi}$ occurs as a subquotient in $\widetilde{D}_{\sigma}[\mathfrak{m}_{K_1}]$. Consider the composite morphism 
\[f_{\sigma}: W_{\chi,3}\overset{f}{\ra} \pi[\mathfrak{m}_{K_1}^2]|_I\overset{\mathrm{pr}}{\ra} \widetilde{D}_{\sigma}|_I.\]
Let $W\defeq\mathrm{Im}(f_{\sigma})$ and $Q$ be the image of the induced morphism $\Ind_{I}^KW_{\chi,3}\ra\widetilde{D}_{\sigma}$. 
By Lemma \ref{lemma-IndW2-multione}, any $\chi'$ with $\Ext^1_{I/Z_1}(\chi,\chi')\neq0$ cannot occur in $\widetilde{D}_{\sigma}^{I_1}$, otherwise $\sigma$ would be a common Jordan--H\"older factor of both $\Ind_I^K\chi$ and $\Ind_I^K\chi'$. Combining with Step 1, we deduce that $\soc_I(W)$ is $\chi$-isotypic (being a subrepresentation of $\widetilde{D}_{\sigma}^{I_1}$). Since $\widetilde{D}_{\sigma}^{I_1}$ is multiplicity-free by \ref{it:pi:properties-2} (as observed above), we must have $\soc_I(W)=\chi$. Since $[Q:\sigma_{\chi}]=1$ (as $\widetilde{D}_{\sigma}$ is multiplicity-free by~\ref{it:pi:properties-2}), Proposition \ref{prop-forspecialW} provides the desired contradiction.
\end{proof}

\bigskip

We can now prove the main theorem of this section. 
Let $\rhobar : G_L\rightarrow\GL_2(\F)$ be a tame Galois representation such that $\rhobar|_{I_L}\simeq\overline{\tau}(s,\mu)$ (cf.\ Definition \ref{def:tau}) with $\mu-\eta$ being $8$-deep in $\un{C}_0$ (\S\ref{sec:GT:prel}).

\begin{thm}\label{thm:GKdim-criterion}
Let $\pi$ be an admissible smooth $\GL_2(L)$-representation over $\F$ with a central character. Assume that:
\begin{enumerate}%
\item we have $\JH(\soc_K(\pi))=W(\rhobar)$ \emph{(}up to multiplicity\emph{)};
\item for all $\sigma\in W(\rhobar)$, we have $[\pi[\mathfrak{m}_{K_1}^2] : \sigma]=[\soc_K(\pi) : \sigma]$;
\item we have $\JH(\pi^{I_1})=\JH(D_1(\rhobar))$ \emph{(}up to multiplicity\emph{)}.
\end{enumerate}
Then $\dim_{\GL_2(L)}(\pi)\leq f$.
 \end{thm}

\begin{proof}
As $\pi$ has a central character, the group $Z_1$ acts trivially on $\pi$.
Therefore, by Corollary \ref{cor:J-fil}, Corollary \ref{cor:V-K1} and Lemma \ref{lem:connected}, the representation $\pi$ satisfies hypotheses \ref{it:pi:properties-1}, \ref{it:pi:properties-2}, \ref{it:pi:properties-3} above. Then Proposition \ref{prop-W3topi=dim1} shows that $\Hom_I(\chi,\pi)\simeq\Hom_I(W_{\chi,3},\pi)$ for all characters $\chi$ occurring in $\pi^{I_1}$. We can then apply Corollary \ref{cor:GKdim} to conclude that $\dim_{I}(\pi|_I)\leq f$ and thus that $\dim_{\GL_2(L)}(\pi)\leq f$ (since $I$ is open in $\GL_2(L)$).
\end{proof}

\section{Construction of a lattice}
\label{sec:lattices}

In this section we construct a $\GL_2(\cO_L)$-stable lattice with simple cosocle in some particular locally algebraic representation of $\GL_2(L)$.

We keep the notation of section \ref{sec:smooth:rep}.
Hence, $L$ is a finite unramified extension of $\Qp$ of degree $f$, ring of integers $\mathcal{O}_L$, residue field $k$.
Recall that we have set $K\defeq \GL_2(\mathcal{O}_L)$, $K_1\defeq1+p\M_2(\mathcal{O}_L)$ and $Z_1\defeq Z(\cO_L)\cap K_1$.

Let $\sigma$ be a Serre weight for $\un{G}_0\times_{\Zp}\Fp$. 
We write $P_\sigma\defeq \Proj_{\GL_2(k)}\sigma$ for the projective cover of $\sigma$ in the category of $\F[\GL_2(k)]$-modules and we let $\wtld{P}_\sigma$ be the projective $\cO[\GL_2(k)]$-module lifting $P_\sigma$. 
Then $\wtld{P}_\sigma\otimes_\cO E$ is a (semisimple) finite-dimensional representation of $\GL_2(k)$ over $E$. By inflation, we view it as $K$-representation on which the subgroup $K_1$ acts trivially.

The space $\mathfrak{sl}_{2,L}$ of $2\times 2$ matrices of trace zero with coefficients in $L$ is endowed with the adjoint action of ${\GL_2}_{/L}$, which is isomorphic to $V(\alpha)_{/L}\cong {\rm Sym}^2(L^2)\otimes {\det}^{-1}$. In particular it has an action of $K$. 
The goal of this section is to show the existence of a $K$-stable lattice $V^\circ$ in the locally $\Qp$-algebraic representation $\mathfrak{sl}_{2,L}\otimes_{\Qp}\wtld{P}_\sigma$ such that $(V^\circ/\varpi V^\circ)_{K_1}$ is isomorphic to $P_\sigma$ (and hence such that $\sigma$ is the $K$-cosocle of $V^\circ$) under some mild genericity assumption on $\sigma$.

As $\wtld{P}_\sigma$ is defined over $W(\F)$, and since $\Hom_{\text{$\Qp$-alg}}(L,W(\F)[1/p])$ has $[L:\Qp]$ elements, we may assume that $E$ is unramified over $\Qp$.

\emph{Throughout this section, $E$ is assumed to be unramified over $\Qp$.} We recall that, as before, we assume $p>2$.

\subsection{Locally algebraic lattices}

Let $V^\circ$ be some $K$-stable $\cO$-lattice in some continuous finite-dimensional representation $(V,\rho)$ of $K/Z_1$ over $E$. We assume that the group $K_1$ acts trivially on $V^\circ/p V^\circ$.

As $p>2$, the map $x\mapsto\exp(px)$ induces a bijection $\mathfrak{sl}_{2,\cO_L}\simto K_1/Z_1$ (note that since $p>2$, the map $K_1\cap\SL_2(L)\rightarrow K_1/Z_1$ is an isomorphism) and a group isomorphism
\begin{equation}\label{eq:lieK1} \mathfrak{sl}_{2,\cO_L}/p\mathfrak{sl}_{2,\cO_L}\simto (K_1/Z_1)/(K_1/Z_1)^p.\end{equation}
(See \cite[III.1.1.4, III.1.1.5, III.1.1.8]{Lazard}.)

By assumption, we have $\rho(k)\in \Id_{V^\circ}+p \End_{\cO}(V^\circ)$ for $k\in K_1$. For $x\in\mathfrak{sl}_{2,k}$ and $v\in V^\circ/p V^\circ$, we choose lifts $\tilde{x}\in\mathfrak{sl}_{2,\cO_L}$ of $x$ and $\tilde{v}\in V^\circ$ of $v$ and we define:
\[ \beta'_{V^\circ}(x,v)\defeq p^{-1}(\rho(\exp(p\tilde{x}))\tilde{v}-\tilde{v}) \bmod p V^\circ.\]
Note that $\beta'_{V^\circ}(x,v)$ does not depend on the choices of $\tilde{x}$ and $\tilde{v}$ and is $\Fp$-linear in $x$ and $\F$-linear in $v$. The independence and linearity in $x$ is a consequence of \eqref{eq:lieK1} and of the fact that if $g\in K_1$, we have $[g^p]-1\in\mathfrak{m}_{K_1}^2$ in $\F\bbra{K_1}$.

Therefore there exists a unique $\F$-linear map
\[ \beta_{V^\circ} : \mathfrak{sl}_{2,k}\otimes_{\Fp}(V^\circ/p V^\circ)\longrightarrow V^\circ/p V^\circ\]
such that $\beta_{V^\circ}(x\otimes v)=\beta'_{V^\circ}(x,v)$ for $x\in\mathfrak{sl}_{2,k}$ and $v\in V^\circ/p V^\circ$.
(Alternatively, one can verify that the natural Lie algebra action of $\mathfrak{sl}_{2,\cO_L}$ on $V$ preserves $V^\circ$ and gives rise to
$\beta_{V^\circ}$ upon reduction modulo $p$.)

The map $\beta_{V^\circ}$ measures the defect of exactness of the functor $(-)_{K_1}$ on finite quotients of $V^\circ$. It is a particular case of a Bockstein homomorphism in some homology long exact sequence. More precisely, we have the following lemma.

\begin{lem}\label{lm:homology-exact-sq}
The following sequence is exact:
\[ \mathfrak{sl}_{2,k}\otimes_{\Fp} (V^\circ/p V^\circ)\xrightarrow{\beta_{V^\circ}} V^\circ/p V^\circ\xrightarrow{p} (V^\circ/p^2V^\circ)_{K_1}\longrightarrow V^\circ/p V^\circ\longrightarrow0,\]
where the last map is the reduction mod $p$ \emph{(}recall that $(V^\circ/pV^\circ)_{K_1}=V^\circ/pV^\circ$\emph{)}.
\end{lem}

\begin{proof}
As the functor of $K_1$-coinvariants is right exact and since $(V^\circ/pV^\circ)_{K_1}=V^\circ/pV^\circ$, it is sufficient to check that the kernel of the second map coincides with the image of $\beta_{V^\circ}$.

Let $x\in\mathfrak{sl}_{2,k}$ and $v\in V^\circ/p V^\circ$ and choose $\tilde{x}\in\mathfrak{sl}_{2,\cO_L}$ and $\tilde{v}\in V^\circ$ lifting $x$ and $v$. By definition we have:
\[ p\beta_{V^\circ}(x\otimes v)=\rho(\exp(p\tilde{x}))\tilde{v}-\tilde{v} \bmod p^2 V^\circ \in\ker((V^\circ/p^2 V^\circ)\rightarrow (V^\circ/p^2 V^\circ)_{K_1}).\]
This implies that the composite $p\beta_{V^\circ}$ is zero.

Conversely let $v\in V^\circ/p V^\circ$ be such that $p v$ is zero in $(V^\circ/p^2 V^\circ)_{K_1}$. This implies that there exist $k_1,\dots,k_r$ in $K_1$ and $\tilde{v}_1,\dots,\tilde{v}_r$ in $V^\circ$ such that
\[ p v=\sum_{i=1}^r (\rho(k_i)-1)\tilde{v}_i\bmod p^2 V^\circ.\]
Then there exist $\tilde{x}_1,\dots,\tilde{x}_r$ in $\mathfrak{sl}_{2,\cO_L}$ such that $k_i=\exp(p\tilde{x}_i)$ and we have $\beta_{V^\circ}(\sum_i {x}_i \otimes {v}_i)=v$ in $V^\circ/p V^\circ$, where $x_i \in \mathfrak{sl}_{2,k}$, $v_i \in V^\circ/p V^\circ$ are the images of $\tilde{x}_i$, $\tilde{v}_i$.\end{proof}

Recall that the group $K$ acts by the adjoint action on $\mathfrak{sl}_{2,L}$ and induces a $\Qp$-algebraic $E$-linear representation of $K$ on $\mathfrak{sl}_{2,L}\otimes_{\Qp}E$. There is a decomposition
\[ \mathfrak{sl}_{2,L}\otimes_{\Qp}E\simeq\bigoplus_{i=0}^{f-1}\mathfrak{sl}_{2,E},\]
where $K$ acts on the $i$-th summand by the adjoint action via the embedding $K\hookrightarrow\GL_2(E)$ given by $\sigma_i : L \into E$ on the coefficients. The sub-$\cO$-module $\mathfrak{sl}_{2,\cO_L}\otimes_{\Zp}\cO$ is a $K$-stable lattice and the action of $K$ on $(\mathfrak{sl}_{2,\cO_L}\otimes_{\Zp}\cO)/p(\mathfrak{sl}_{2,\cO_L}\otimes_{\Zp}\cO)\simeq \mathfrak{sl}_{2,k}\otimes_{\Fp}\F$ factors through $\GL_2(k)$ so that $K_1$ acts trivially on this quotient.

Now we compute $\beta_{V^\circ}$ in the case where $V^\circ$ is the lattice $\mathfrak{sl}_{2,\cO_L}\otimes_{\Zp}\cO$ in the locally algebraic representation $\mathfrak{sl}_{2,L}\otimes_{\Qp}E$.

\begin{lem}\label{lm:cocycle-adjoint}
Assume that $V^\circ=\mathfrak{sl}_{2,\cO_L}\otimes_{\Zp}\cO$. Then $V^\circ/p V^\circ\simeq\mathfrak{sl}_{2,k}\otimes_{\Fp}\F$ and the map $\beta_{V^\circ}$ is given explicitly by
\[ \beta_{V^\circ}(x\otimes y\otimes z)=[x,y]\otimes z\]
for $x,y\in\mathfrak{sl}_{2,k}$ and $z\in\F$.
\end{lem}

\begin{proof}
Let $\tilde{x}$ and $\tilde{y}$ in $\mathfrak{sl}_{2,\cO_L}$ lifting $x$ and $y$. We have:
\[ \exp(p\tilde{x})\tilde{y}\exp(p\tilde{x})^{-1}-\tilde{y}\equiv p\tilde x\tilde y-p\tilde y\tilde x\pmod {p^2\mathfrak{sl}_{2,\cO_L}}\]
so that $\beta_{\mathfrak{sl}_{2,\cO_L}\otimes_{\Zp}\cO}(x\otimes y\otimes 1)=[x,y]$ and we conclude by $\F$-linearity.
\end{proof}

\begin{rem}\label{rem:devissage-beta}
By construction of $\beta_{V^\circ}$ we can check that $\beta_{V_1^\circ\oplus V_2^\circ}=\beta_{V_1^\circ}\oplus\beta_{V_2^\circ}$ and, if $W^\circ$ is another lattice on which $K_1$ acts trivially, $\beta_{V^\circ\otimes_\cO W^\circ}=\beta_{V^\circ}\otimes\Id_{W^\circ/pW^\circ}$.
\end{rem}

We leave to the reader the task to verify the following lemma along the lines of the proof of Lemma \ref{lm:homology-exact-sq}.

\begin{lem}\label{lm:sublattice}
Let $W\subset V^\circ/p V^\circ$ be a sub-$\F$-vector space stable under $K$ and let $V_1^\circ \subset V^\circ$ be the inverse image of $W$ in $V^\circ$. We have a commutative diagram with exact rows:
\[ \begin{tikzcd}[column sep=3.5em]
\mathfrak{sl}_{2,k}\otimes_{\Fp} W \ar[r,"\beta_{V^\circ}|_{\mathfrak{sl}_{2,k}\otimes_{\Fp} W}"] \ar[d,hookrightarrow] & V^\circ/pV^\circ \ar[r,"p"] \ar[d,equal] & (V_1^\circ/p^2 V^\circ)_{K_1} \ar[r] \ar[d]& W \ar[r] \ar[d,hookrightarrow] & 0 \\
\mathfrak{sl}_{2,k}\otimes_{\Fp} (V^\circ/pV^\circ) \ar[r,"\beta_{V^\circ}"] & V^\circ/pV^\circ \ar[r,"p"] & (V^\circ/p^2 V^\circ)_{K_1} \ar[r] &V^\circ/pV^\circ \ar[r] & 0. \end{tikzcd}\]
\end{lem}

\subsection{Preliminary computations}

In this technical subsection, we make some explicit computations with $\mathfrak{sl}_{2,\F}$-representations and deduce that a certain endomorphism of a direct sum of Serre weights is actually an automorphism.

If $G$ is an algebraic group over $\F$, we use the notion of $G$-module $M$ as defined in \cite[I.2.7]{RAGS}. Such an object has an underlying structure of an $\F$-vector space. It has moreover a natural structure of a module over the Lie algebra $\Lie(G)$ such that the structure map $\Lie(G)\otimes_\F M\rightarrow M$ is a morphism of $G$-modules, where $\Lie(G)$ is considered as a $G$-module for the adjoint action (\cite[I.7.11 \& I.7.18(1)]{RAGS}).

Given $\lambda\in X^*(T)$ (resp.~$\lambda\in X^*(\un{T})$), as in \S\ref{sec:inert-local-langl} we let $L(\lambda)_{/\F}$ be the irreducible algebraic representation of ${\GL_2}_{/\F}$ (resp.~of $\un{G}$) of highest weight $\lambda$.
We write $L(\lambda)$ instead of $L(\lambda)_{/\F}$ in order not to overload notation.

If $\lambda=(\lambda_i)_{0\leq i\leq f-1}$ with $\lambda_i\in X_1(T)$, we have
\[ L(\lambda)\simeq\bigotimes_{i=0}^{f-1} L(\lambda_i)^{(i)},\]
where $L(\lambda_i)^{(i)}$ is the inflation of the ${\GL_2}_{/\F}$-module $L(\lambda_i)$ to $\un{G}$ via the map $\un{G}\cong\prod_{\cJ}\GL_2\stackrel{\pi_i}{\onto}\GL_2$ corresponding to the $i$-th projection. 

Moreover $L(\lambda)$ inherits an action of the group $\un{G}(\F)=\GL_2(k\otimes_{\Fp}\F)$ and $F(\lambda)=L(\lambda)|_{\GL_2(k)}$ via the inclusion $\GL_2(k)\hookrightarrow\un{G}(\F)=\GL_2(k\otimes_{\Fp}\F)$ corresponding to the ring homomorphism $k\rightarrow k\otimes_{\Fp}\F$, $a\mapsto a\otimes 1$ (see \S\ref{sec:inert-local-langl}).

We fix the following $\F$-basis $(e,h,f)$ of $\mathfrak{sl}_{2,\F}$:
\[ e=\begin{pmatrix} 0 & 1 \\ 0 & 0 \end{pmatrix}, \quad h=\begin{pmatrix}
1 & 0 \\ 0 & -1 \end{pmatrix}, \quad f=\begin{pmatrix}
0& 0 \\ 1 & 0 \end{pmatrix}.\]
Recall that the space $\mathfrak{sl}_{2,\F}$ is a ${\GL_2}_{/\F}$-module for the adjoint action and if $p>2$ we have $\alpha\in X_1(T)$ and $\mathfrak{sl}_{2,\F}$ is isomorphic to $L(\alpha)$.

Let $\lambda\in X_1(T)$. We recall that $L(\lambda)$ has a structure of $\mathfrak{sl}_{2,\F}$-module. Let $v_\lambda$ be a highest weight vector of $L(\lambda)$. Then the $\F$-vector space $L(\lambda)$ has a basis given by $(f^iv_\lambda)_{0\leq i\leq r}$ with $r\defeq \ang{\lambda,\alpha^\vee}$ and the action of $\GL_2(\F)$ is given, for $v\in L(\lambda)$, by
\[ \begin{pmatrix}1 & a \\ 0 & 1\end{pmatrix}v=\sum_{n\geq0} a^n\frac{e^n}{n!}v, \quad \begin{pmatrix}1 & 0 \\ a & 1\end{pmatrix}v=\sum_{n\geq0} a^n\frac{f^n}{n!}v.\]
(See \cite[II.1.19(6)]{RAGS} and note that here the sum over $0 \le n \le p-1$ suffices.)

Assume from now on that $\lambda$ is $2$-deep in the lowest alcove, i.e.\ $2 \le r \le p-4$.
Then we have an isomorphism of ${\GL_2}_{/\F}$-modules (see \cite[Lemma]{humphreys-generic}):
\begin{equation}\label{eq:dec-2deep} \mathfrak{sl}_{2,\F}\otimes_\F L(\lambda)\simeq L(\alpha)\otimes_\F L(\lambda) \simeq L(\lambda)\oplus L(\lambda+\alpha)\oplus L(\lambda-\alpha),\end{equation}
noting that the weights $\lambda+\alpha$ and $\lambda-\alpha$ are $p$-restricted. %
We note that the vector $2(e\otimes fv_\lambda)+r(h\otimes v_\lambda)$ is annihilated by $e$ and is a weight vector of weight $\lambda$, it therefore generates the submodule isomorphic to $L(\lambda)$ in $\mathfrak{sl}_{2,\F}\otimes_\F L(\lambda)$. The vector $e\otimes v_\lambda$ (resp.~$e\otimes f^2v_\lambda+(r-1)h\otimes fv_\lambda-r(r-1)f\otimes v_\lambda$) is annihilated by $e$ and is a weight vector of weight $\lambda+\alpha$ (resp.~$\lambda-\alpha$) and generates the submodule isomorphic to $L(\lambda+\alpha)$ (resp.~$L(\lambda-\alpha)$).

We denote by $d_\lambda$ the unique map of ${\GL_2}_{/\F}$-modules $L(\lambda)\hookrightarrow \mathfrak{sl}_{2,\F}\otimes_\F L(\lambda)$ sending $v_\lambda$ to $2(e\otimes fv_\lambda)+r(h\otimes v_\lambda)$. Note that this is the unique (up to scalar) nonzero map between these ${\GL_2}_{/\F}$-modules.

\begin{lem}\label{lm:computation-GL2}
The composite map of ${\GL_2}_{/\F}$-modules
\[ \psi_\lambda : \mathfrak{sl}_{2,\F}\otimes_\F L(\lambda) \xrightarrow{\Id_{\mathfrak{sl}_{2,\F}}\otimes d_\lambda} \mathfrak{sl}_{2,\F}\otimes_\F\mathfrak{sl}_{2,\F}\otimes_\F L(\lambda) \xrightarrow{[-,-]\otimes\Id_{L(\lambda)}} \mathfrak{sl}_{2,\F}\otimes_\F L(\lambda)\]
is an isomorphism.
\end{lem}

\begin{proof}
As both sides have the same dimension, it is sufficient to prove that this map is injective. As a ${\GL_2}_{/\F}$-module, $\mathfrak{sl}_{2,\F}\otimes_\F L(\lambda)$ is a direct sum of distinct simple modules by \eqref{eq:dec-2deep}, it is therefore sufficient to prove that the map $\psi_\lambda$ is nonzero on some well chosen vector of each direct summand. We will check this for each of these modules.

The submodule isomorphic to $L(\lambda+\alpha)$ contains the vector $e\otimes v_\lambda$. We have
\begin{align*}
\psi_\lambda(e\otimes v_\lambda)&=([-,-]\otimes \Id_{L(\lambda)})(e\otimes (2(e\otimes fv_\lambda)+r(h\otimes v_\lambda))) \\
&=2 [e,e]\otimes fv_\lambda +r [e,h]\otimes v_\lambda\\
&=-2r e\otimes v_\lambda\neq0\end{align*}
since $2r\neq0$ in $\F$.

The submodule isomorphic to $L(\lambda)$ contains the vector $d_\lambda(v_\lambda)=2(e\otimes fv_\lambda)+r(h\otimes v_\lambda)$. Note that
\begin{align*} d_\lambda(f v_\lambda)&=f(2e\otimes f v_\lambda+r h\otimes v_\lambda) \\
&=2[f,e]\otimes fv_\lambda+2e\otimes f^2v_\lambda+r[f,h]\otimes v_\lambda+r h\otimes fv_\lambda \\
&=-2h\otimes fv_\lambda+2e\otimes f^2v_\lambda+2r f\otimes v_\lambda+r h\otimes fv_\lambda \\
&=2e\otimes f^2v_\lambda+(r-2)h\otimes fv_\lambda+2r f \otimes v_\lambda.
\end{align*}
We have
\begin{align*}
\psi_\lambda(d_\lambda(v_\lambda))&=([-,-]\otimes \Id_{L(\lambda)})(2e\otimes d_\lambda(fv_\lambda)+rh\otimes d_\lambda(v_\lambda)) \\
&=4[e,e]\otimes f^2v_\lambda+2(r-2)[e,h]\otimes fv_\lambda+4r[e,f]\otimes v_\lambda+2r[h,e]\otimes fv_\lambda+r^2[h,h]\otimes v_\lambda \\
&=-4(r-2) e\otimes fv_\lambda+4rh\otimes v_\lambda+4re\otimes fv_\lambda \\
&=8e\otimes f v_\lambda+4rh\otimes v_\lambda\neq0\end{align*}
since, for example, $8\neq0$ in $\F$.

The submodule isomorphic to $L(\lambda-\alpha)$ contains the vector $e\otimes f^2v_\lambda+(r-1)h\otimes fv_\lambda-r(r-1)f\otimes v_\lambda$. We first check that
\[ d_\lambda(f^2v_\lambda)=2e\otimes f^3v_\lambda+(r-4)h\otimes f^2v_\lambda+4(r-1)f\otimes fv_\lambda.\]
Then we have
\begin{align*}
\psi_\lambda(e&\otimes f^2v_\lambda+(r-1)h\otimes fv_\lambda-r(r-1)f\otimes v_\lambda)\\
&=2(r+2)e\otimes f^2v_\lambda+2(r-1)(r+2)h\otimes fv_\lambda -2r(r-1)(r+2)f\otimes v_\lambda
\end{align*}
and this is nonzero, since $2\leq r\leq p-4$. This proves the lemma.
\end{proof}

Let $\sigma$ be a Serre weight for $\un{G}_0\times_{\Zp}\Fp$.
It is an absolutely irreducible representation of $\un{G}_0(\Fp)=\GL_2(k)$.
There exists a $p$-restricted weight $\lambda\in X_1(\un{T})$ such that $\sigma\cong F(\lambda)= L(\lambda)|_{\GL_2(k)}\simeq\bigotimes_{i=0}^{f-1}L(\lambda_i)^{(i)}|_{\GL_2(k)}$ (see \S \ref{sec:inert-local-langl}). 

Assume from now on that $\lambda$ is $2$-deep in $\un{C}_0$. Then the weights $\lambda$, $\lambda\pm\alpha_i$ are $p$-restricted, hence we have an isomorphism of $\GL_2(k)$-representations 
\[ \mathfrak{sl}_{2,k}\otimes_{k,\sigma_i} F(\lambda) \cong F(\lambda) \oplus F(\lambda+\alpha_i) \oplus F(\lambda-\alpha_i), \]
where the summands on the right-hand side are irreducible and pairwise nonisomorphic.
For each $i$, we choose a nonzero map $d_{\sigma,i}\in \Hom_{\GL_2(k)}(\sigma,\mathfrak{sl}_{2,k}\otimes_{k,\sigma_i}\sigma)$.
By comparing with \eqref{eq:dec-2deep} it follows that that the map $d_{\sigma,i}$ is a nonzero multiple of the map $\Id_{\bigotimes_{j\neq i}L(\lambda_j)^{(j)}}\otimes d_{\lambda_i}^{(i)}$ and we define $\un{d}_\sigma\defeq(d_{\sigma,i})$ which is a $\GL_2(k)$-equivariant map from $\sigma$ to $\mathfrak{sl}_{2,k}\otimes_{\Fp}\sigma\simeq\bigoplus_i (\mathfrak{sl}_{2,k}\otimes_{k,\sigma_i}\sigma)$.
(Note that $\mathfrak{sl}_{2,k}\otimes_{k,\sigma_i}\sigma$ is isomorphic to the $\GL_2(k)$-restriction of $(\mathfrak{sl}_{2,\F}\otimes_{\F} L(\lambda_i))^{(i)}\bigotimes_{j\neq i}L(\lambda_j)^{(j)}$ or, equivalently, of $L(\alpha_i)\otimes_{\F}L(\lambda)$.)

\begin{prop}\label{prop:computation-Gbar}
Assume that $\lambda$ is $2$-deep in $\un{C}_0$. Then the map of $\GL_2(k)$-representations
\[ \Psi : \mathfrak{sl}_{2,k}\otimes_{\Fp}\sigma\xrightarrow{\Id_{\mathfrak{sl}_{2,k}}\otimes\un{d}_\sigma} \mathfrak{sl}_{2,k}\otimes_{\Fp}\mathfrak{sl}_{2,k}\otimes_{\Fp}\sigma\xrightarrow{[-,-]\otimes\Id_\sigma}\mathfrak{sl}_{2,k}\otimes_{\Fp}\sigma\]
is an isomorphism.
\end{prop}

\begin{proof}
As the map $[-,-]$ is $k$-bilinear, the map $[-,-]\otimes\Id_\sigma$ factors through
\[ \mathfrak{sl}_{2,k}\otimes_{\Fp}\mathfrak{sl}_{2,k}\otimes_{\Fp}\sigma\twoheadrightarrow \mathfrak{sl}_{2,k}\otimes_k\mathfrak{sl}_{2,k}\otimes_{\Fp}\sigma.\]
Therefore, the map $\Psi$ is the direct sum of the maps $\Psi_i$, where $\Psi_i$ is the $\F$-linear composite map
\[ \mathfrak{sl}_{2,k}\otimes_{k,\sigma_i}\sigma\xrightarrow{\Id_{\mathfrak{sl}_{2,k}}\otimes d_{\sigma,i}} \mathfrak{sl}_{2,k}\otimes_k\mathfrak{sl}_{2,k}\otimes_{k,\sigma_i}\sigma\xrightarrow{[-,-]\otimes\Id_\sigma} \mathfrak{sl}_{2,k}\otimes_{k,\sigma_i}\sigma.\]
First of all we remark that all the modules involved in the statement are actually restrictions to $\GL_2(k)$ of $\un{G}$-modules. Namely, $\sigma=L(\lambda)|_{\GL_2(k)}$ and the action of $\GL_2(k)$ on $\mathfrak{sl}_{2,k}\otimes_{k,\sigma_i}\F$ is the restriction to $\GL_2(k)$ of the action of $\un{G}$ on $\mathfrak{sl}_{2,\F}^{(i)}$. Moreover the maps $d_{\sigma,i}$ and $[-,-]$ are maps of $\un{G}$-modules. As $L(\lambda)\simeq\bigotimes_i L(\lambda_i)^{(i)}$ and $\mathfrak{sl}_{2,k}\otimes_{k,\sigma_i}\sigma\simeq\bigotimes_{j\neq i}L(\lambda_j)^{(j)}\otimes_\F (\mathfrak{sl}_{2,\F}\otimes_\F L(\lambda_i))^{(i)}$, we have $\Psi_i=\bigotimes_j \psi_{i,j}^{(j)}$, where $\psi_{i,j}$ is the identity of $L(\lambda_j)$ when $j\neq i$ and $\psi_{i,i}$ is a nonzero scalar multiple of the endomorphism $\psi_{\lambda_i}$ (where $\psi_{\lambda_i}$ is defined in Lemma \ref{lm:computation-GL2}). By Lemma \ref{lm:computation-GL2}, the map $\Psi_i$ is an isomorphism, hence so is $\Psi$.
\end{proof}

\subsection{Construction of the lattice}
\label{sec:constructionlattice}

Let $\sigma$ be a Serre weight. We recall that we denote by $P_\sigma$ the projective cover of $\sigma$ in the category of $\F[\GL_2(k)]$-modules and $\wtld{P}_\sigma$ the projective $\cO[\GL_2(k)]$-module lifting $P_\sigma$. Then $\wtld{P}_\sigma\otimes_\cO E$ is a (semisimple) finite-dimensional representation of $\GL_2(k)$ over $E$. By inflation, we view it as a $K$-representation on which the subgroup $K_1$ acts trivially.

We set $R_1\defeq\wtld{P}_\sigma$ and we recall that we have the $\Qp$-algebraic action of the group $K$ on $\mathfrak{sl}_{2,L}\otimes_{\Qp}E$ by the adjoint action. The $\cO$-module $R_2\defeq\mathfrak{sl}_{2,\cO_L}\otimes_{\Zp}\wtld{P}_\sigma$ is a $K$-stable lattice of $R_2[1/p]$ such that $K_1$ acts trivially on $R_2/p R_2$. As the group $K_1$ acts trivially on $\wtld{P}_\sigma$, Remark \ref{rem:devissage-beta} implies that $\beta_{R_2}=\beta_{\mathfrak{sl}_{2,\cO_L}\otimes_{\Zp}\cO}\otimes \Id_{P_\sigma}$. From Lemma \ref{lm:cocycle-adjoint}, we deduce that
\[ \beta_{R_2}=[-,-]\otimes\Id_{P_\sigma} : \mathfrak{sl}_{2,k}\otimes_{\Fp}\mathfrak{sl}_{2,k}\otimes_{\Fp}P_\sigma\longrightarrow\mathfrak{sl}_{2,k}\otimes_{\Fp}P_\sigma.\]

Let $R_{2,i}\defeq\mathfrak{sl}_{2,\cO_L}\otimes_{\cO_L,\sigma_i}
\wtld{P}_\sigma$ so that $R_2\simeq\bigoplus_i R_{2,i}$. Let
$\lambda\in X_1(\un{T})$ be such that $\sigma\simeq F(\lambda)$ and
assume that $\lambda$ is $3$-deep in $\un{C}_0$. 
For $0\leq i\leq f-1$, there exists an isomorphism of $K$-representations (see for example \cite[Prop.~3.3(2)]{LMS}):
\begin{equation}\label{decomp}
R_{2,i}/p R_{2,i}\simeq \mathfrak{sl}_{2,k}\otimes_{k,\sigma_i}P_\sigma\simeq P_{\sigma}\oplus P_{\sigma_{1,i}}\oplus P_{\sigma_{2,i}},
\end{equation}
where $\sigma_{1,i}=F(\lambda-\alpha_i)$ and $\sigma_{2,i}=F(\lambda+\alpha_i)$. We fix such an isomorphism and use it to define a $K$-equivariant injection $\iota_i : P_\sigma\hookrightarrow R_{2,i}/pR_{2,i}$. We let $\iota$ denote the ``diagonal'' embedding of $P_\sigma$:
\[ \iota : x\mapsto (\iota_i(x))_i\in R_2/pR_2\simeq\bigoplus_{i=0}^{f-1} R_{2,i}/p R_{2,i}.\]

As a first step, we consider a modification of the lattice $R_2$. We define a new lattice in $R_2[1/p]$ as follows:
\[ R_2'\defeq\set{x\in R_2 : (x\bmod pR_2)\in\iota(P_\sigma)}.\]
Note that $pR_2\subset R'_2$. As $K_1$ acts trivially on $P_\sigma$, the map $R'_2/pR'_2\onto P_\sigma$ sending $x$ to $\iota^{-1}(x\bmod p)$ factors through $R'_2/pR'_2\onto (R'_2/pR'_2)_{K_1}$ and gives rise to a $K$-equivariant surjective map $(R'_2/pR'_2)_{K_1}\onto P_\sigma$.

\begin{prop}\label{prop:alg-lattice}
For $x\in R_2$, we can find elements $k_1,\dots,k_r\in K_1$ and $x_1,\dots,x_r$ in $R'_2$ such that
\[ \sum_{i=1}^r (k_i-1)x_i\equiv px\pmod {p^2R_2}.\]
Hence the $K$-equivariant map $(R'_2/pR'_2)_{K_1}\onto P_\sigma$ is an isomorphism. 
\end{prop}

\begin{proof}
By Lemmas \ref{lm:homology-exact-sq} and \ref{lm:sublattice}, we have a commutative diagram with exact rows:
\[ \begin{tikzcd}
\mathfrak{sl}_{2,k}\otimes_{\Fp} P_\sigma \ar[r] \ar[d,hookrightarrow,"\Id_{\mathfrak{sl}_{2,k}}\otimes\iota"'] \ar[rd] & R_2/pR_2 \ar[r,"p"] \ar[d,equal] & (R_2'/p^2 R_2)_{K_1} \ar[r] \ar[d]& P_\sigma \ar[r] \ar[d,hookrightarrow,"\iota"] & 0 \\
\mathfrak{sl}_{2,k}\otimes_{\Fp} (R_2/pR_2) \ar[r,"\beta_{R_2}"] & R_2/pR_2 \ar[r,"p"] & (R_2/p^2R_2)_{K_1} \ar[r] & R_2/pR_2 \ar[r] & 0. \end{tikzcd}\]
We will prove that the diagonal map is surjective (equivalently, an isomorphism, for dimension reasons). This is equivalent to the first statement of the proposition, and the second statement immediately follows.

As $R_2/pR_2\simeq\mathfrak{sl}_{2,k}\otimes_{\Fp} P_\sigma$ and $\beta_{R_2}=[-,-]\otimes \Id_{P_\sigma}$, we need to prove that the composite map $([-,-]\otimes\Id_{P_\sigma})\circ (\Id_{\mathfrak{sl}_{2,k}}\otimes \iota)$ is surjective:
\[ \mathfrak{sl}_{2,k}\otimes_{\Fp} P_\sigma\xrightarrow{\Id_{\mathfrak{sl}_{2,k}}\otimes \iota}  \mathfrak{sl}_{2,k}\otimes_{\Fp}  \mathfrak{sl}_{2,k}\otimes_{\Fp} P_\sigma \xrightarrow{[-,-]\otimes\Id_{P_\sigma}} \mathfrak{sl}_{2,k}\otimes_{\Fp} P_\sigma.\]

For dimension reasons, it is equivalent to prove that it is injective. This can be checked on the socle.

The socle of $P_\sigma$ is isomorphic to $\sigma$ and the nonzero map (unique up to scalar) $\sigma\hookrightarrow P_\sigma$ induces a $K$-equivariant map $\mathfrak{sl}_{2,k}\otimes_{\Fp}\sigma\hookrightarrow \mathfrak{sl}_{2,k}\otimes_{\Fp}P_\sigma$ whose image is the socle of $\mathfrak{sl}_{2,k}\otimes_{\Fp}P_\sigma$ (see Lemma \ref{lm:socle-tensor} below).

To summarize, we have a commutative diagram
\[ \begin{tikzcd}[column sep=3.8em]
\mathfrak{sl}_{2,k}\otimes_{\Fp} \sigma\ar[r,"\Id_{\mathfrak{sl}_{2,k}}\otimes \iota|_{\sigma}"] \ar[d,hookrightarrow] &  \mathfrak{sl}_{2,k}\otimes_{\Fp}  \mathfrak{sl}_{2,k}\otimes_{\Fp} \sigma \ar[r,"{[-,-]}\otimes\Id_{\sigma}"] \ar[d,hookrightarrow]
\ar[d,hookrightarrow] & \mathfrak{sl}_{2,k}\otimes_{\Fp} \sigma \ar[d,hookrightarrow] \\
\mathfrak{sl}_{2,k}\otimes_{\Fp} P_\sigma \ar[r,"\Id_{\mathfrak{sl}_{2,k}}\otimes \iota"] & \mathfrak{sl}_{2,k}\otimes_{\Fp}  \mathfrak{sl}_{2,k}\otimes_{\Fp} P_\sigma \ar[r,"{[-,-]}\otimes\Id_{P_\sigma}"]  & \mathfrak{sl}_{2,k}\otimes_{\Fp} P_\sigma. \end{tikzcd}\]
We need to prove that the composition of the maps of the top row is injective and we will be done.

In the decomposition $\mathfrak{sl}_{2,k}\otimes_{\Fp}\sigma\simeq\bigoplus_{i=0}^{f-1}(\mathfrak{sl}_{2,k}\otimes_{k,\sigma_i}\sigma)$, the map $\iota|_\sigma$ corresponds to $(\iota_i|_\sigma)_{0\leq i\leq f-1}$. As $\iota_i$ is injective and $\sigma$ is the socle of $P_\sigma$, we have that $\iota_i|_\sigma$ is nonzero. We can apply Proposition \ref{prop:computation-Gbar} to conclude that the composite map in the top row of the diagram above is an isomorphism.
\end{proof}

\begin{lem}\label{lm:socle-tensor}
The $\GL_2(k)$-equivariant map $\sigma\hookrightarrow P_\sigma$ \emph{(}resp.~$P_\sigma\twoheadrightarrow \sigma$\emph{)} induces a $\GL_2(k)$-equiva\-riant map $\mathfrak{sl}_{2,k}\otimes_{\Fp}\sigma \hookrightarrow \mathfrak{sl}_{2,k}\otimes_{\Fp} P_\sigma$ \emph{(}resp.~$\mathfrak{sl}_{2,k}\otimes_{\Fp} P_\sigma\twoheadrightarrow \mathfrak{sl}_{2,k}\otimes_{\Fp} \sigma$\emph{)} whose image is the socle \emph{(}resp.~cosocle\emph{)} of $\mathfrak{sl}_{2,k}\otimes_{\Fp} P_\sigma$.
\end{lem}

\begin{proof}
As the map $\mathfrak{sl}_{2,k}\otimes_{\Fp} \sigma\hookrightarrow \mathfrak{sl}_{2,k}\otimes_{\Fp} P_\sigma$ is $k\otimes\F$-linear, it can be decomposed as the direct sum of the maps $\mathfrak{sl}_{2,k}\otimes_{k,\sigma_i}\sigma\rightarrow\mathfrak{sl}_{2,k}\otimes_{k,\sigma_i}P_\sigma$. Therefore it is sufficient to prove that the image of the map $\mathfrak{sl}_{2,k}\otimes_{k,\sigma_i}\sigma\rightarrow\mathfrak{sl}_{2,k}\otimes_{k,\sigma_i}P_\sigma$ is the socle of the right-hand side for each $0\leq i\leq f-1$. We observe that the left-hand side is semisimple (by \eqref{eq:dec-2deep}), the map is injective and the socle of the right-hand side has the same dimension as the left-hand side (by \eqref{decomp}). This implies the result. The case of the cosocle is similar.
\end{proof}

Using Proposition \ref{prop:alg-lattice}, we identify $(R'_2/pR'_2)_{K_1}$ with $P_\sigma$ and we define the lattice $R$ by ``glueing'' $R_1$ and $R_2'$ along $P_\sigma$:
\begin{equation}\label{eq:R}
\begin{aligned}
 R&\defeq\set{(x_1,x_2)\in R_1\oplus R_2' : (x_1\bmod p)=({\rm image\ of\ }x_2\bmod p)\ {\rm in\ }P_\sigma\simeq(R'_2/pR'_2)_{K_1}} \\
&=\set{(x_1,x_2)\in R_1\oplus R_2 : (x_2\bmod p)=\iota (x_1\bmod p)\in R_2/pR_2}
\end{aligned}
\end{equation}
(equivalently, $R\cong R_1\times_{P_\sigma}R'_2$). This is a $K$-stable lattice in $R_1[1/p]\oplus R_2[1/p]$. We define $r$ to be the map $R\rightarrow P_\sigma$ sending $(x_1,x_2)$ to $(x_1\bmod p)$.

\begin{thm}\label{thm:lattice-coinvariants}
There exists a short exact sequence of $K$-representations
\begin{equation}\label{eq:dec-lattice-modp}0 \longrightarrow R_2/p R_2\longrightarrow R/p R\xrightarrow{r} P_\sigma\longrightarrow0.\end{equation}
Moreover the map $r : R/ p R\onto P_\sigma$ induces an isomorphism $(R/p R)_{K_1}\xrightarrow{\sim}P_\sigma$.
\end{thm}

\begin{proof}
As $p R_2\subset \ker(r)\subset R$ we have $p^2 R_2\subset p R$ and the inclusion of $pR_2$ in $\ker(r)$ induces a map $pR_2/p^2R_2\rightarrow \ker(r)/pR$. This map is actually a $K$-equivariant isomorphism
\[p R_2/p^2 R_2\simto \ker(r)/p R.\]
Namely these two representations are finite-dimensional over $\F$ and have the same dimension. It is therefore sufficient to prove that $p R\cap p R_2=p^2 R_2$. The right-hand side is clearly included in the left-hand side. Conversely let $(p x_1,p x_2)$ be some element in the left-hand side. We have $\iota(x_1\bmod p)=(x_2\bmod p)$ in $R_2/pR_2$. As $x_1=0$, we have $x_2\in p R_2$, which proves the assertion. This gives us the short exact sequence \eqref{eq:dec-lattice-modp}.

Now we prove the second assertion. We define $\overline{r} : R/p R\onto P_\sigma$ as the factorization of $r$ by $R/p R$. As $K_1$ acts trivially on $P_\sigma$ and $\overline{r}$ is $K$-equivariant, the map $\overline{r}$ factors as $(R/p R)_{K_1}\onto P_\sigma$. We need to prove that the kernel of $\overline{r}$ is contained in the kernel of $R/p R\onto (R/p R)_{K_1}$, i.e.\ that each element of $\ker(\ovl r)$ can be written as a finite sum $\sum_j (k_j-1)y_j$ with $k_j\in K_1$ and $y_j\in R/pR$.

Let $x\in \ker(\overline{r})$. By what precedes, there exists $y\in R_2$ such that $p y$ reduces to $x$ modulo $p R$. By Proposition \ref{prop:alg-lattice} we can find $k_1,\dots,k_r$ in $K_1$ and $x_1,\dots,x_r$ in $R'_2$ such that
\[ py\equiv \sum_{j=1}^r (k_j-1)x_j \pmod{p^2R_2}.\]
Let $z_1,\dots,z_r$ in $R_1$ be such that $\iota(z_j\bmod p)=(x_j\bmod p)$ for all $1\leq j\leq r$. Then $(z_j,x_j)\in R$ for all $1\leq j\leq r$. Since $K_1$ acts trivially on $R_1$, we have $(k_j-1)(z_j,x_j)=(0,(k_j-1)x_j)$ so that
\begin{equation}\label{spouf} \sum_{j=1}^r (k_j-1)(z_j,x_j)=(0,py+p^2u)\end{equation}
for some $u\in R_2$. 
Let $y_j$ be the image of $(z_j,x_j) \in R$ in $R/pR$. Reducing \eqref{spouf} modulo $p R$, we obtain
\[ \sum_{j=1}^r(k_j-1)y_j=x,\]
proving that $\overline{r}$ induces an isomorphism $(R/p R)_{K_1}\xrightarrow{\sim} P_\sigma$.
\end{proof}

\begin{cor}\label{rpr}
The $K$-cosocle of $R/pR$ is isomorphic to $\sigma$. Moreover the $K$-representations $(\Proj_{K/Z_1}\sigma)/\mathfrak{m}_{K_1}^2(\Proj_{K/Z_1}\sigma)$ and $R/pR$ are isomorphic.
\end{cor}

\begin{proof}
As $K_1$ is a normal pro-$p$-subgroup of $K$, the group $K_1$ acts trivially on every semisimple representation of $K$. Therefore the $K$-cosocle of $R/pR$ is the $\GL_2(k)(=K/K_1)$-cosocle of $(R/pR)_{K_1}$. As $(R/pR)_{K_1}$ is isomorphic to $P_\sigma$ by Theorem \ref{thm:lattice-coinvariants}, we obtain
\[ \cosoc_K(R/pR)\simeq \cosoc_{\GL_2(k)}\big((R/pR)_{K_1}\big)\simeq \cosoc_{\GL_2(k)}(P_\sigma)\simeq\sigma.\]
Note that $Z_1$ acts trivially on $R_1$ and $R_2$, and hence also on $R$.
This implies that there exists a $K$-equivariant map $\theta : \Proj_{K/Z_1}\sigma\rightarrow R/pR$ which is surjective on cosocles and is hence surjective. Note that $R_2/pR_2$ is killed by $\mathfrak{m}_{K_1}$ %
so that Theorem \ref{thm:lattice-coinvariants} implies that $R/pR$ is killed by $\mathfrak{m}_{K_1}^2$. The map $\theta$ factors through the quotient $(\Proj_{K/Z_1}\sigma)/\mathfrak{m}_{K_1}^2(\Proj_{K_1/Z_1}\sigma)$ and gives rise to a surjective map
\[ (\Proj_{K_1/Z_1}\sigma)/\mathfrak{m}_{K_1}^2(\Proj_{K_1/Z_1}\sigma)\twoheadrightarrow R/pR.\]
We now prove that this map is an isomorphism. Namely, since $R$ is a lattice of $\wtld{P}_\sigma[1/p]\oplus\bigoplus_{i=0}^{f-1} (\mathfrak{sl}_{2,L}\otimes_{\cO_L,\sigma_i}\wtld{P}_\sigma)$, we have
\begin{eqnarray*}
\dim_\F (R/pR)&=&\dim_E\Big(\wtld{P}_\sigma[1/p]\oplus\bigoplus_{i=0}^{f-1} (\mathfrak{sl}_{2,L}\otimes_{\cO_L,\sigma_i}\wtld{P}_\sigma)\Big)\\
&=&(3f+1)\dim_E\big(\wtld{P}_\sigma[1/p]\big)\ =\ (3f+1)\dim_{\F}(P_\sigma).
\end{eqnarray*}
On the other hand, the isomorphism $(\Proj_{K/Z_1}\sigma)/\mathfrak{m}_{K_1}(\Proj_{K/Z_1}\sigma)\simeq P_\sigma$ induces an exact sequence
\[ 0 \to (\mathfrak{m}_{K_1/Z_1}/\mathfrak{m}_{K_1/Z_1}^2)\otimes_{\F}P_\sigma\longrightarrow (\Proj_{K/Z_1}\sigma)/\mathfrak{m}_{K_1}^2(\Proj_{K/Z_1}\sigma)\longrightarrow P_\sigma\longrightarrow0.\]
(Note that $\Proj_{K/Z_1}\sigma$ is projective in the category of pseudocompact $K_1/Z_1$-modules, since $K_1$ is an open subgroup of $K$.)
As the group $K_1/Z_1$ is uniform of dimension $3f$, we have $\dim_{\F}(\mathfrak{m}_{K_1/Z_1}/\mathfrak{m}_{K_1/Z_1}^2)=3f$, and hence 
\[ \dim_{\F}\big((\Proj_{K/Z_1}\sigma)/\mathfrak{m}_{K_1}^2(\Proj_{K/Z_1}\sigma)\big)= (3f+1)\dim_{\F}(P_\sigma).\]
This implies that $\dim_{\F}((\Proj_{K/Z_1}\sigma)/\mathfrak{m}_{K_1}^2(\Proj_{K/Z_1}\sigma))= \dim_{\F}(R/pR)$, so the map $\theta$ is an isomorphism. 
\end{proof}

\subsection{Projectivity} 

We prove several results which will be  used in the gluing process in \S\ref{sec:freeness}.  

\begin{prop}\label{prop:commutative}
Assume $\sigma\cong F(\lambda)$ where $\lambda\in X^*(\un{T})$ satisfies $2< \langle \lambda,\alpha_i^\vee\rangle <p-4$ for all $i\in\cJ$.
The endomorphism ring $\End_{K}(\Proj_{K/Z_1}\sigma/\mathfrak{m}_{K_1}^2(\Proj_{K/Z_1}\sigma))$ is commutative.
\end{prop}

\begin{proof}
By Corollary \ref{rpr}, it is equivalent to show that $\End_K(R/pR)$ is commutative. 

Note that $\wtld{P}_\sigma\otimes_{\cO}E$ is isomorphic to a direct sum of $2^f$ absolutely irreducible pairwise non-isomorphic $K$-representations, see the last paragraph of the proof of Lemma \ref{lem:inter}.
Thus, $R\otimes_{\cO}E$ is semisimple and isomorphic to a direct sum of $2^f(f+1)$ absolutely irreducible and pairwise non-isomorphic $K$-representations.
We conclude that $\End_{E[K]}(R\otimes_{\cO}E)$ is a commutative ring
of dimension $2^f(f+1)$. Since
$\End_{\cO[K]}(R)\otimes_{\cO}E=\End_{E[K]}(R\otimes_{\cO}E)$ and
$\End_{\cO[K]}(R)$ is $p$-torsion free,   $\End_{\cO[K]}(R)$ is also a commutative ring and is a free $\cO$-module of rank $2^f(f+1)$. The  exact sequence 
$0\ra R\overset{\times p}{\lra} R\ra R/pR\ra0$
induces 
\[0\ra \End_{\cO[K]}(R)\overset{\times p}{\lra} \End_{\cO[K]}(R)\overset{\gamma}\ra \Hom_{\cO[K]}(R,R/pR)=\End_{\F[K]}(R/pR).\]
From the construction of $R$, see \eqref{decomp} and \eqref{eq:R}, and using the fact that $[P_{\sigma}:\sigma]=2^f$ and $[P_{\sigma_{1,i}}:\sigma]=[P_{\sigma_{2,i}}:\sigma]=0$ (the latter justified by Proposition \ref{lm:princ-series}\ref{it:princ-series-2} and the assumption on $\lambda$),  we get $[R/pR:\sigma]=2^f(f+1)$, and so 
\[\dim_{\F}\End_{\F[K]}(R/pR)=\dim_{\F}\Hom_{\F[K]}(\Proj_{K/Z_1}\sigma,R/pR)=2^f(f+1)\]
by Corollary \ref{rpr}. 
Hence $\gamma$ is surjective, and the result follows.
\end{proof}

We assume from now on that $5< \langle \lambda,\alpha_i^\vee\rangle <p-7$.
Letting $\tau$ be a Serre weight occurring in $\JH(R/pR)$, we denote by $R_{\tau}$  the object $R$ constructed in \S\ref{sec:constructionlattice} with $\sigma$ replaced by $\tau$. Then $\End_{K}(R_{\tau}/pR_{\tau})$ is also commutative by Proposition \ref{prop:commutative} and the assumption on $\lambda$. 

\begin{lem}\label{lem:tau-cyclic}
As an $\End_{K}(R_{\tau}/pR_{\tau})$-module, $\Hom_K(R_{\tau}/pR_{\tau},R/pR)$ is cyclic.
\end{lem}
\begin{proof}
By \cite[Thm.~2.30]{HuWang2} (which generalizes \cite[Cor.~3.12]{BP}),
there is a unique quotient of $R/pR$, denoted by $I(\tau,\sigma)$,
such that $\soc_K I(\tau,\sigma)=\tau$ and
$[I(\tau,\sigma):\sigma]=1$; moreover $I(\tau,\sigma)$ is multiplicity
free. The projectivity of $R_{\tau}/pR_{\tau}$ in the category
of $\F[\![K/Z_1]\!]/\mathfrak{m}_{K_1}^2$-modules then gives a morphism
$\phi_{\tau}:R_{\tau}/pR_{\tau}\ra R/pR$ which makes the following
diagram commutative 
\[\xymatrix{R_{\tau}/pR_{\tau}\ar@{-->}_{\phi_{\tau}}[d]\ar@{->>}[r]&\tau\ar@{^(->}[d]\\
R/pR\ar@{->>}[r]& I(\tau,\sigma).}\]
We have $[\mathrm{coker}(\phi_{\tau}):\tau]=0$, because any quotient of $R/pR$ in which $\tau$ occurs must admit $I(\tau,\sigma)$ as a quotient by \cite[Thm. 2.30]{HuWang2}. We deduce the result and also the fact that $\phi_{\tau}$ is a generator of $\Hom_K(R_{\tau}/pR_{\tau},R/pR)$ over $\End_K(R_{\tau}/pR_{\tau})$.
\end{proof}

\begin{prop}\label{prop:projectivity}
Let $Q$ be a quotient of $R/pR$. Then $Q$ satisfies the following property: for any subquotient $Q'$ of $Q$, the projection $R/pR\twoheadrightarrow Q$ induces an isomorphism
\[\Hom_{K}(Q,Q')\simto \Hom_{K}(R/pR,Q').\]
In particular, if $\cosoc_K(Q')\cong \sigma$, then there exists a $K$-equivariant surjection $Q\twoheadrightarrow Q'$.
\end{prop} 
 
\begin{rem} \label{rem:projectivity} Proposition \ref{prop:projectivity} can be interpreted as saying that $Q$ is a projective object in the smallest abelian subcategory of  $\F[\![K/Z_1]\!]/\mathfrak{m}_{K_1}^2$-modules which contains all subquotients of $Q$. 
\end{rem}
\begin{proof}
Let $\tau\in \JH(R/pR)$.  The projectivity of $R_{\tau}/pR_{\tau}$ implies a surjection \[\Hom_K(R_{\tau}/pR_{\tau},R/pR)\twoheadrightarrow \Hom_K(R_{\tau}/pR_{\tau},Q),\] so that $\Hom_K(R_{\tau}/pR_{\tau},Q)$ is a cyclic $\End_K(R_{\tau}/pR_{\tau})$-module generated by the composite map
\[\phi_{\tau,Q}: R_{\tau}/pR_{\tau}\overset{\phi_{\tau}}{\lra} R/pR\twoheadrightarrow Q \]
where $\phi_{\tau}$ is as in the proof of Lemma \ref{lem:tau-cyclic} and the second map is the natural projection. Moreover, the annihilator ideal \[\mathfrak{a}_{\tau,Q}\defeq \{h\in \End_K(R_{\tau}/pR_{\tau}): \phi_{\tau,Q}\circ h=0\}\] is identified with 
  $\Hom_K(R_{\tau}/pR_{\tau},\ker(\phi_{\tau,Q}))$.  By the projectivity of $R_{\tau}/pR_{\tau}$, $\Hom_K(R_{\tau}/pR_{\tau},Q')$ is a subquotient  of $\Hom_K(R_{\tau}/pR_{\tau},Q)$ as $\End_K(R_{\tau}/pR_{\tau})$-modules, so it is also annihilated by $\mathfrak{a}_{\tau,Q}$. Here we use the commutativity of $\End_K(R_{\tau}/pR_{\tau})$ in Proposition \ref{prop:commutative}.
This means that any $f_{\tau}\in \Hom_K(R_{\tau}/pR_{\tau},Q')$ is zero on the image of the evaluation map
\[\Hom_K(R_{\tau}/pR_{\tau},\ker(\phi_{\tau,Q}))\otimes R_{\tau}/pR_{\tau} \ra \ker(\phi_{\tau,Q}).\]
The projectivity of $R_{\tau}/pR_{\tau}$ shows that the above image   is identified with the largest submodule of $\ker(\phi_{\tau,Q})$ whose cosocle is $\tau$-isotypic; we denote it by $\ker(\phi_{\tau,Q})^{\tau}$. 

Let $f\in \Hom_K(R/pR,Q')$. We need to prove that $f$ factors through $Q$, equivalently that $f$ is zero on $\ker_Q\defeq\ker(R/pR\twoheadrightarrow Q)$. The snake lemma gives the following exact sequence 
\[0\ra \ker(\phi_{\tau}) \ra \ker(\phi_{\tau,Q})\overset{\phi_{\tau}}{\lra} \ker_Q\ra \mathrm{coker}(\phi_{\tau}).\]
By the last paragraph (applied to $f_{\tau}=f\circ\phi_{\tau}$), $f$ is zero on the image of $\ker(\phi_{\tau,Q})^{\tau}$ in $\ker_Q$.
Since $[\mathrm{coker}(\phi_{\tau}):\tau]=0$ (see the proof of Lemma \ref{lem:tau-cyclic}), any morphism $R_{\tau}/pR_{\tau}\ra \ker_Q$ must factor through $\phi_{\tau}$, hence the image of $\ker(\phi_{\tau,Q})^{\tau}$ is equal to the largest submodule of $\ker_Q$ whose cosocle is $\tau$-isotypic.   Since $\tau$ is arbitrary, $f$ must be identically zero on $\ker_Q$.

The last assertion is obvious, because under the assumption on $Q'$ there exists a $K$-equivariant surjection $R/pR\twoheadrightarrow Q'$ which must factor through $Q$ by the first assertion. 
\end{proof}
\section{Global applications}\label{sec:globalapplications}

We prove our main global results: Theorem \ref{mainpatching}, Theorem \ref{mainpatching2}, Theorem \ref{largest}, Corollary \ref{padiclanglands} and Corollary \ref{mainmain}.

\subsection{Patching functors}\label{patching}

We introduce the global background and the patching functors that we will use (following \cite[\S 6.2]{EGS}). We assume $p>5$ (for the main theorem, we will in fact need {$p>23$}) and $E$ unramified, i.e.\ ${\mathcal O}=W(\F)$. We use the notation and conventions of \S\ref{sec:preliminaries}.

We fix $F$ a totally real number field, and denote by ${\mathcal O}_F$ its ring of integers and $S_p$ the set of places of $F$ above $p$. We assume $F$ is unramified at each place in $S_p$. For each place $w$ of $F$ we denote by $F_w$ the completion of $F$ at $w$, ${\mathcal O}_{F_w}$ its ring of integers and ${\rm Frob}_w$ a geometric Frobenius element at $w$. We denote by ${\mathbb A}_F^{\infty}$ the finite ad\`eles of $F$. For any finite place $w$ of $F$, let $q_w$ denote the cardinality of the residue field of $F_w$.

We fix $D/F$ a quaternion algebra of center $F$ which is split at all places in $S_p$ and at no more than one infinite place of $F$ (in the sequel we call the two cases the ``indefinite case'' and the ``definite case''). In the indefinite case we assume $(F,D)\ne ({\mathbb Q},\GL_2)$ (our main result is already known in the case $(F,D)= ({\mathbb Q},\GL_2)$). We denote by $S_D$ the set of finite places where $D$ ramifies. We fix a maximal order ${\mathcal O}_D$ in $D$ and isomorphisms $({\mathcal O}_D)_w\buildrel\sim\over\rightarrow \M_2({\mathcal O}_{F_w})$ for $w\notin S_D$, where $({\mathcal O}_D)_w\defeq {\mathcal O}_D\otimes_{{\mathcal O}_F}{\mathcal O}_{F_w}$.

We fix $\rbar:G_F\ra \GL_2(\F)$ a continuous representation and set $\rbar_w\defeq \rbar|_{G_{F_w}}$ for each finite place $w$ of $F$. We assume that $\rbar |_{G_{F(\!\sqrt[p]{1})}}$ is absolutely irreducible and $\rbar_w$ is generic in the sense of \cite[Def.\ 11.7]{BP} (or \cite[Def.\ 2.1.1]{EGS}) for $w\in S_p$. We let $S_{\rbar}$ be the set of (finite) places where $\rbar$ is ramified (hence $S_p\subseteq S_{\rbar}$ by the previous genericity) and we moreover assume that the universal framed deformation ring $R_{\rbar_w}$ of $\rbar_w$ over $W(\F)$ is formally smooth over $W(\F)$ if $w\in (S_D \cup S_{\rbar})\backslash S_p$ (see Remark \ref{shotton} below). We let $\psi:G_F\rightarrow W(\F)^\times$ be the Teichm\"uller lift of $\omega \det \rbar$ and set $\psi_w\defeq \psi|_{G_{F_w}}$.

Assume first that we are in the indefinite case. For a compact open
subgroup $V$ of $(D\otimes_F{\mathbb A}_F^\infty)^\times$ let $X_V$ be
the associated smooth projective algebraic Shimura curve over $F$ (see
e.g.\ \cite[\S 3.1]{BD} and the references therein). We choose the
convention $\eps=-1$ as in \cite{BDJ} to define $X_V$. This is not the
convention of \cite{BD}, but we point out that the results of \cite{BD} that we will use
below do not depend on this choice. We assume that there exists $V$ such that
\begin{equation}\label{Vnonzero}
\Hom_{G_F}\!\big(\rbar,H^1_{{\rm \acute et}}(X_V \times_F \overline F, \F)\big)\ne 0.
\end{equation}
Then one can always take $V$ of the following form: $V=\prod V_w$ with $V_w\subseteq ({\mathcal O}_D)_{w}^\times$ for all $w$, $V_w=({\mathcal O}_D)_{w}^\times$ for $w\notin S_D \cup S_{\rbar}$ and $V_w=1+p\M_2({\mathcal O}_{F_w})$ for $w\in S_p$ (see e.g.\ \cite[Thm.\ 3.2.2]{BD} or the proof of \cite[Cor. 3.2.3]{BD}). For Serre weights $(\sigma_w)_{w\in S_p}$ and any $V=\prod V_w$ such that (\ref{Vnonzero}) holds and $V_w\subseteq 1+p\M_2({\mathcal O}_{F_w})$ is normal in $({\mathcal O}_D)_{w}^\times$ for $w\in S_p$ we have by \cite[\S 5.5]{gee-kisin}:
\begin{equation}\label{h1nonzero}
\Hom_{\GL_2({\mathcal O}_F\otimes_{\mathbb Z}\Zp)}\Big(\!\bigotimes_{\F,w}\sigma_w,\Hom_{G_F}\!\big(\rbar,H^1_{\rm \acute et}(X_V \times_F \overline F, \F)\big)\Big)\ne 0 \Longleftrightarrow \sigma_w\in W(\rbar_w^\vee) \ \forall w\in S_p,
\end{equation}
where we recall that $W(\rbar_w^\vee)$ is defined as in \cite[\S 3]{BDJ} (with $\rho$ there being $\overline r_w^\vee$), cf.\ \S\ref{sec:inert-local-langl}.

We now fix
\begin{enumerate}
\item\label{it:asum1} a finite place $v\in S_p$ such that $\rbar_v$ is semisimple of one of the following forms up to twist:
\begin{enumerate}[label=(\alph*)]
\item\label{it:asum1a} $\rbar_v\vert_{I_{F_v}}\cong \begin{pmatrix}\omega_f^{(r_0+1)+\cdots+p^{f-1}(r_{f-1}+1)}&0\\0&1\end{pmatrix}$\ \ {$12\leq r_i\leq p-15$},
\item\label{it:asum1b} $\rbar_v\vert_{I_{F_v}}\cong\begin{pmatrix}\omega_{2f}^{(r_0+1)+\cdots+p^{f-1}(r_{f-1}+1)}&0\\0&\omega_{2f}^{q_v({\rm same})}\end{pmatrix}$\ {$13\leq r_0\leq p-14$, $12\leq r_i\leq p-15$ for $i>0$},
\end{enumerate}
(equivalently, $\rbar_v^\vee$ satisfies the same hypothesis; note that, up to twist, $\rbar_v$ is of the form described at the beginning of \S \ref{sec:setup});
\item\label{it:asum2} a finite place $w_1\notin S_D \cup S_{\rbar}$ such that
\begin{enumerate}[label=(\alph*)]
\item\label{it:asum2a} Norm$(w_1)$ is not congruent to $1$ mod $p$,
\item\label{it:asum2b} the ratio of the eigenvalues of $\rbar({\rm Frob}_{w_1})$ is not in $\{1, {\rm Norm}(w_1), {\rm Norm}(w_1)^{-1}\}$,
\item\label{it:asum2c} for any nontrivial root of unity $\zeta$ in a quadratic extension of $F$, $w_1\nmid (\zeta + \zeta^{-1}-2)$
\end{enumerate}
(such a place $w_1$ exists by \cite[\S\S 6.2, 6.5]{EGS});
\item\label{it:asum3} a finite set of finite places $S$ such that
\begin{enumerate}[label=(\alph*)]
\item\label{it:asum3a} $S$ contains $S_D \cup S_{\rbar}$ but not $w_1$,
\item\label{it:asum3b} for $w\in S\backslash S_p$ the framed deformation ring $R_{\rbar_w^\vee}$ of $\rbar_w^\vee$ is formally smooth over $W(\F)$;
\end{enumerate}
\item\label{it:asum4} a compact open subgroup $U=\prod_wU_w\subseteq \prod_w({\mathcal O}_D)_{w}^\times$ such that 
\begin{enumerate}[label=(\alph*)]
\item\label{it:asum4a} $U_w=({\mathcal O}_D)_{w}^\times=\GL_2({\mathcal O}_{F_{w}})$ for $w\notin S \cup \{w_1\}$ or $w\in S_p$,
\item\label{it:asum4b} \eqref{Vnonzero} holds for $V=\big(\prod_{w\notin S_D \cup S_{\rbar}}({\mathcal O}_D)_{w}^\times\big)\big(\prod_{(S_D \cup S_{\rbar})\backslash S_p}U_w\big)\big(\prod_{w\in S_p}1+p\M_2({\mathcal O}_{F_w})\big)$,
\item\label{it:asum4c} $U_{w_1}$ is contained in the subgroup of $({\mathcal O}_D)_{w_1}^\times=\GL_2({\mathcal O}_{F_{w_1}})$ of matrices that are upper-triangular unipotent mod $w_1$.
\end{enumerate} 
\end{enumerate}

\begin{rem}\label{shotton}
Using \cite[\S 5]{Shotton} one can make assumption \ref{it:asum3}\ref{it:asum3b} above completely explicit. For instance, if Norm$(w)$ is not congruent to $\pm 1$ mod $p$, then $R_{\rbar_w^\vee}$ (or equivalently $R_{\rbar_w}$, the two rings are isomorphic by duality) is always formally smooth, except when $\rbar_w\cong \begin{pmatrix} \omega & 0 \\ 0 & 1\end{pmatrix}$ up to twist.
\end{rem}

The following lemma due to Hamann \cite[Thm.\ 4]{Hamann} will be convenient below.

\begin{lem}\label{lm:hamann}
Suppose that $R$, $S$ are local rings. If $R\bbra x \cong S\bbra x$, then $R \cong S$.
\end{lem}

For each $w\in S_p\backslash \{v\}$ we fix a tame inertial type $\tau_w$ such that $\JH(\ovl{\sigma(\tau_w)^\vee})=\JH(\ovl{\sigma(\tau_w^\vee)})$ contains exactly one Serre weight in $W(\rbar_w^\vee)$ (\cite[Prop.\ 3.5.1]{EGS}) and we fix a $\GL_2({\mathcal O}_{F_{w}})$-invariant lattice $\sigma^0(\tau_w^\vee)$ in $\sigma(\tau_w^\vee)=\sigma(\tau_w)^\vee$ (so, increasing $\F$ if necessary, $\sigma^0(\tau_w^\vee)$ is a free $W(\F)$-module, see the last statement in \cite[Lemma\ 3.1.1]{EGS}). As any Serre weight in $W(\rbar_w^\vee)$ has central character $(\omega^{-1}\det \rbar_w^\vee)\vert_{I_{F_w}}=\overline \psi|_{I_{F_w}}^{-1}$ and $\tau_w$ is tame, the central character of $\sigma^0(\tau_w^\vee)$ is $\psi|_{I_{F_w}}^{-1}$ and $\det \tau_w=\psi|_{I_{F_w}}$. We define a representation $\sigma_p^v$ of $\prod_{w\in S\backslash \{v\}}U_w$ over $W(\F)$ by
\begin{equation}\label{horsv}
\sigma_p^v\defeq \bigotimes_{w\in S_p\backslash \{v\}}\sigma^0(\tau_w^\vee),
\end{equation}
with $\prod_{w\in S\backslash \{v\}}U_w$ acting via $\prod_{w\in S\backslash \{v\}}U_w\twoheadrightarrow \prod_{w\in S_p\backslash \{v\}}U_w=\prod_{w\in S_p\backslash \{v\}}\GL_2({\mathcal O}_{F_{w}})$. As in \cite[\S\S 6.2, 6.4]{EGS} using $K=U$, we then define a patching functor (depending on $\sigma_p^v$)
$$M_\infty^{\sigma_p^v}:\ \sigma_v\longmapsto M_\infty(\sigma_p^v\otimes_{W(\F)}\sigma_v)$$
from the category of continuous representations $\sigma_v$ of $\GL_2({\mathcal O}_{F_{v}})$ on finite type $W(\F)$-modules with central character $\psi|_{I_{F_v}}^{-1}$ to the category of finite type $R_\infty$-modules, where (see \cite[\S 5.4.1]{gee-kisin})
\begin{equation*}
R_\infty\defeq R^{\rm loc}\bbra{X_1,\cdots,X_{q-[F:{\mathbb Q}]+|S|-1}}.
\end{equation*}
Here $q$ is an integer $\geq [F:{\mathbb Q}]$ and
$$R^{\rm loc}\defeq\Big(\widehat\bigotimes_{w\in S\backslash S_p}R_{\rbar_w}^{\psi_w}\Big)\widehat\otimes_{W(\F)}\Big(\widehat\bigotimes_{w\in S_p\backslash \{v\}}R_{\rbar_w}^{(0,-1),\tau_w,\psi_w}\Big)\widehat\otimes_{W(\F)}R_{\rbar_v}^{\psi_v},$$
where the exponent $\psi_w$ means framed deformations of $\rbar_w$ with fixed determinant $\varepsilon^{-1}\psi_w$ and where $R_{\rbar_w}^{(0,-1),\tau_w,\psi_w}$ is the reduced $p$-torsion free quotient of $R_{\rbar_w}^{\psi_w}$ parametrizing those deformations which have parallel Hodge--Tate weights $(0,-1)$ and inertial type $\tau_w$ (by local-global compatibility and the inertial local Langlands correspondence, for $w\in S_p \backslash \{v\}$ the action of $R_{\rbar_w}^{\psi_w}$ on $M_\infty(\sigma_p^v\otimes_{W(\F)}\sigma_v)$ factors through this quotient).
By assumption \ref{it:asum3}\ref{it:asum3b} above (with \cite[Rk.\ 5.2.2]{gee-kisin} and Lemma \ref{lm:hamann}) we have $R_{\rbar_w}^{\psi_w}\cong W(\F)\bbra{X_1,X_2,X_3}$ for $w\in S\backslash S_p$, and by genericity of $\rbar_v$ we have $R_{\rbar_v}^{\psi_v}\cong W(\F)\bbra{X_1,\dots,X_{3+3[F_v:\Qp]}}$. Taking the duals of representations induces a canonical isomorphism $R_{\rbar_w}^{(0,-1),\tau_w,\psi_w}\cong R_{\rbar_w^\vee}^{(1,0),\tau_w^\vee,\psi_w^{-1}}$, where the ring on the right-hand side is the more familiar quotient of $R_{\rbar_w^\vee}$ parametrizing potentially Barsotti--Tate deformations of $\rbar_w^\vee$ with inertial type $\tau_w^\vee$ and determinant $\varepsilon \psi_w^{-1}$. By \cite[Thm.\ 7.2.1(2)]{EGS} (with \cite[Rk.\ 5.2.2]{gee-kisin} and Lemma \ref{lm:hamann}) we have $R_{\rbar_w^\vee}^{(1,0),\tau_w^\vee,\psi_w^{-1}}\cong W(\F)\bbra{X_1,\dots,X_{3+[F_w:\Qp]}}$, so that we finally get
\begin{equation}\label{rinfini}
R_\infty\cong R_{\rbar_v}^{\psi_v}\bbra{X_1,\dots,X_{4(|S|-1)+q-[F_v:\Qp]}}\cong W(\F)\bbra{X_1,\dots,X_{4|S|+q-1+2[F_v:\Qp]}}.
\end{equation}

\begin{rem}\label{mEGS}
Here are several remarks on the definition of $M_\infty(\sigma_p^v\otimes_{W(\F)}\sigma_v)$ in \cite{EGS}.
\begin{enumerate}
\item\label{it:EGS1} One needs to extend the action of $U$ on $\sigma_p^v\otimes_{W(\F)}\sigma_v$ (which acts via $U\twoheadrightarrow \prod_{w\in S_p}U_w$) to an action of $U({\mathbb A}_F^{\infty})^\times$ with $({\mathbb A}_F^{\infty})^\times$ acting via
$$({\mathbb A}_F^{\infty})^\times \twoheadrightarrow ({\mathbb A}_F^{\infty})^\times/F^\times\buildrel {\rm Artin} \over \twoheadrightarrow {\rm Gal}(F^{\rm ab}/F) \buildrel \psi^{-1} \over \rightarrow W(\F)^\times.$$
Note that we believe this action of $({\mathbb A}_F^{\infty})^\times$ in \cite[\S 6.2]{EGS} should also be via $\psi^{-1}$, not $\psi$ (as it is there), otherwise there is a contradiction with (at least) $\det \tau = \psi|_{I_{F_v}}$ in \cite[\S 7.1]{EGS}, since the normalization of $\sigma(\tau)$ in \cite[\S 1.9]{EGS} is dual to the one in \cite[\S 2.1.1]{BM}.
(See also \cite[Rk.\ A.1]{CEGS-C}, as was pointed out to us by David Savitt.)
\item\label{it:EGS2} Accordingly, we need to modify the maximal ideal $\mathfrak m$ associated to $\rbar$ in \cite[\S 6.2]{EGS} as follows: ${\mathfrak m}$ is the maximal ideal generated by $T_w-S_w{\rm tr}(\rbar({\rm Frob}_w)),\ {\rm Norm}(w)-S_w{\det}(\rbar({\rm Frob}_w))$ for $w\notin S\cup \{w_1\}$ (this is the maximal ideal of \cite[\S 4]{BDJ}).
\item\label{it:EGS3} For any $V\subseteq U$ the finite group $V({\mathbb A}_F^\infty)^\times/VF^\times$ acts on $X_V$ without fixing any geometric point (see e.g.\ part (iv) of the proof of \cite[Lemme\ 3.6.2]{BD}, replacing $w_0$ there by $w_1$). In the definition of $S(\sigma)$ in \cite[\S 6.2]{EGS} in the indefinite case, one should replace the Shimura curve by its quotient by this finite action (which is still a smooth projective curve over $F$), analogously to the definite case of {\it loc.~cit.}, where $S(\sigma)$ is defined as functions $f:D^\times\backslash (D\otimes_F{\mathbb A}_F^\infty)^\times \rightarrow \sigma(\theta)^*$ such that $f(gd)=d^{-1}f(g)$ for $d\in U({\mathbb A}_F^\infty)^\times$ (not just $d\in U$). Note that replacing $X_V$ by its quotient does not change $\Hom_{G_F}\!\big(\rbar,H^1_{{\rm \acute et}}(X_V \times_F \overline F, \F)\big)$ (arguing as in the proof of \cite[Thm.\ 3.7.1]{BD}).
\end{enumerate}
\end{rem}

Denote by ${\mathfrak m}_\infty$ the maximal ideal of $R_\infty$ and for $w\in S_p\backslash \{v\}$ let $\sigma_w$ be the unique Serre weight in $W(\rbar_w^\vee)$ that appears in $\JH(\ovl{\sigma(\tau_w^\vee)})$. 
By a standard Hochschild--Serre spectral sequence (see e.g.\ the proof of \cite[Lemma\ 4.11]{BDJ} or of \cite[Lemme\ 3.6.2]{BD}) we have isomorphisms of finite-dimensional $\F$-vector spaces for any representation $\sigma_v$ of $\GL_2({\mathcal O}_{F_{v}})$ over $W(\F)$ as above such that $V_v$ acts trivially on $\sigma_v$ (see also \cite[(5.3)]{LMS}):
\begin{align}\label{h1fini}
\nonumber M_\infty^{\sigma_p^v}(\sigma_v)/{\mathfrak m}_\infty&\cong \Hom_{G_F}\!\Big(\rbar, \Hom_{U/V}\!\big((\bigotimes _{w\in S_p\backslash\{v\}} \sigma_w)\otimes \sigma_v, H^1_{{\rm \acute et}}(X_V \times_F \overline F, \F)\big)\Big)^{\!\vee}\\
\nonumber &\cong \Hom_{U/V}\!\Big((\bigotimes _{w\in S_p\backslash\{v\}} \sigma_w)\otimes \sigma_v, \Hom_{G_F}\!\big(\rbar, H^1_{{\rm \acute et}}(X_V \times_F \overline F, \F)\big)\Big)^{\!\vee}\\
&\cong \Hom_{U_v/V_v}\!\Big(\sigma_v, \Hom_{U^v/V^v}\!\Big(\!\!\bigotimes _{w\in S_p\backslash\{v\}} \sigma_w, \Hom_{G_F}\!\big(\rbar, H^1_{{\rm \acute et}}(X_V \times_F \overline F, \F)\big)\Big)\Big)^{\!\vee}
\end{align}
for any $V=\prod V_w$ such that $V_w=U_w$ if $w\notin S_p$ and $V_w\subseteq 1+p\M_2({\mathcal O}_{F_w})$ with $V_w$ normal in $\GL_2({\mathcal O}_{F_w})$ if $w\in S_p$ (and, as usual, $U^v\defeq \prod_{w \ne v} U_w$ and likewise for $V^v$). In particular, it follows from (\ref{h1nonzero}) and the exactness of the patching functor $M_\infty^{\sigma_p^v}$ in \cite[\S 6.2]{EGS} that $M_\infty^{\sigma_p^v}(\sigma_v)\ne 0$ if and only if $\JH(\ovl{\sigma_v})\cap W(\rbar_v^\vee)\ne \emptyset$.

The definite case is analogous to the indefinite one. We have the equivalence (\ref{h1nonzero}), replacing $\Hom_{G_F}(\rbar,H^1_{\rm \acute et}(X_V \times_F \overline F, \F))$ by $S(V,\F)[{\mathfrak m}]$, where $S(V,\F)\defeq \{f:D^\times\backslash (D\otimes_F{\mathbb A}_F^\infty)^\times/V \rightarrow \F\}$ and (as in Remark \ref{mEGS}\ref{it:EGS2}) $\mathfrak m$ is generated by $T_w-S_w{\rm tr}(\rbar({\rm Frob}_w)),\ {\rm Norm}(w)-S_w{\det}(\rbar({\rm Frob}_w))$ for $w\notin S\cup \{w_1\}$ such that $V_w=({\mathcal O}_D)_{w}^\times$, with $T_w$, $S_w$ acting on $S(V,\F)$ (via right translation on functions), respectively, by $V\begin{pmatrix}\varpi_w & 0\\ 0&1\end{pmatrix}V$, $V\begin{pmatrix}\varpi_w & 0\\ 0&\varpi_w\end{pmatrix}V$, where $\varpi_w$ is any uniformizer in $F_w$. In the definition of $M(\sigma_p^v\otimes_{W(\F)}\sigma_v)$ in \cite[\S 6.2]{EGS} one again modifies the maximal ideal $\mathfrak m$ as in Remark \ref{mEGS}\ref{it:EGS2}. Finally (\ref{h1fini}) becomes
\begin{equation}\label{h0fini}
M_\infty^{\sigma_p^v}(\sigma_v)/{\mathfrak m}_\infty\cong \Hom_{\GL_2({\mathcal O}_{F_v})}\!\Big(\sigma_v, \Hom_{U^v/V^v}\!\big(\!\!\bigotimes _{w\in S_p\backslash\{v\}} \sigma_w, S(V,\F)[{\mathfrak m}]\big)\Big)^{\!\vee}.
\end{equation}

For convenience, we consider the following admissible smooth representation $\pi$ of $\GL_2(F_v)$ over $\F$ with central character $\overline\psi^{-1}$:
\begin{alignat}{2}
\pi&\defeq\varinjlim_{V_v}\Hom_{U^v/V^v}\!\Big(\!\bigotimes _{w\in S_p\backslash\{v\}} \sigma_w, \Hom_{G_F}\!\big(\rbar, H^1_{{\rm \acute et}}(X_{V^vV_v} \times_F \overline F, \F)\big)\Big) &&\ {\rm in\ the\ indefinite\ case},\label{eq:pi-indef}\\
\pi&\defeq\varinjlim_{V_v}\Hom_{U^v/V^v}\!\big(\!\bigotimes _{w\in S_p\backslash\{v\}} \sigma_w, S(V^vV_v,\F)[{\mathfrak m}]\big)&&\ {\rm in\ the\ definite\ case\label{eq:pi-def}}.
\end{alignat}
Then \eqref{h1fini} and \eqref{h0fini} both become 
\begin{equation}\label{eq:Minfty-pi}M_\infty^{\sigma_p^v}(\sigma_v)/{\mathfrak m}_\infty\cong \Hom_{\GL_2(\cO_{F_v})}(\sigma_v,\pi)^{\vee}.\end{equation}

\subsection{Freeness for types}\label{tobefree1}

We prove some freeness results for $M_\infty^{\sigma_p^v}(\sigma)$ and $M_\infty^{\sigma_p^v}(\sigma)[1/p]$ for various representations $\sigma$.

We now set $K\defeq \GL_2({\mathcal O}_{F_v})$, $K_1\defeq 1+p\M_2(\mathcal{O}_{F_v})$ and we freely use the notation of \S\ref{sec:smooth:rep} (with $L=F_v$, $k$ the residue field, etc.) and of \S\ref{patching}. In order not to overload notation, we now just write $M_\infty$ for $M_\infty^{\sigma_p^v}$. If $A$ is a commutative ring and $M$ is an $A$-module, we call scheme-theoretic support of $M$ the quotient $A/{\rm Ann}_A(M)$.

\begin{lem}\label{induction}
Let $A$ be a commutative ring and $N\subseteq M$ two $A$-modules. We assume there is an integer $r\ge 1$ such that
\begin{enumerate}
\item\label{it:ind1} $N$ and $M/N$ are free of rank $r$ over their respective scheme-theoretic supports;
\item\label{it:ind2} $M$ can be generated as an $A$-module by $r$ elements;
\item\label{it:ind3} there is an isomorphism of $A$-modules ${\rm Ann}_A(M/N)/{\rm Ann}_A(M)\cong A/{\rm Ann}_A(N)$.
\end{enumerate}
Then $M$ is free of rank $r$ over its scheme-theoretic support.
\end{lem}
\begin{proof}
Replacing $A$ by $A/{\rm Ann}_A(M)$, we can assume ${\rm Ann}_A(M)=0$. Let $I\defeq {\rm Ann}_A(M/N)$ and $f:A^r \twoheadrightarrow M$ an $A$-linear surjection by \ref{it:ind2}. Then the composition of $f$ with $M\twoheadrightarrow M/N$ factors through $(A/I)^r$ and we deduce a commutative diagram of $A$-modules
\begin{equation*}
\begin{tikzcd}
0 \ar[r] & I^r \ar[r] \ar[d] & A^r \ar[r] \ar[d, twoheadrightarrow, "f"]& (A/I)^r \ar[r] \ar[d, twoheadrightarrow] & 0 \\
0 \ar[r] & N \ar[r] &M \ar[r] &M/N \ar[r] & 0. \end{tikzcd}
\end{equation*}
By~\ref{it:ind1} we have an isomorphism of $A$-modules $M/N\cong (A/I)^r$ and it follows from e.g.\ \cite[Thm.\ 2.4]{Ma} that the surjection on the right is an isomorphism. The snake lemma then shows that the vertical map on the left is surjective. Since $I\cong A/{\rm Ann}_A(N)$ by~\ref{it:ind3} (recall ${\rm Ann}_A(M)=0$) and $N\cong (A/{\rm Ann}_A(N))^r$ by~\ref{it:ind1}, \cite[Thm.\ 2.4]{Ma} again shows that the vertical map on the left is bijective, and hence all vertical maps are bijective.
\end{proof}

Recall that a finite type module $M$ over a noetherian local ring $A$ is called maximal CM over $A$ if it is Cohen--Macaulay and if its Krull dimension (which is the Krull dimension of $A/{\rm Ann}_A(M)$) is equal to the Krull dimension of $A$. In particular, $A/{\rm Ann}_A(M)$ has no embedded associated prime.

\begin{lem}\label{CM}
Let $\sigma$ be any smooth representation of $K$ on a finite length $W(\F)$-module. Then the finite type $R_\infty$-module $M_\infty(\sigma)$ is maximal CM over its scheme-theoretic support.
\end{lem}
\begin{proof}
We can assume $M_\infty(\sigma)\ne 0$. For each Serre weight $\sigma_v$ such that $M_\infty(\sigma_v)\ne 0$, it follows from \cite[Def.\ 6.1.1]{EGS} that the Krull dimension of $M_\infty(\sigma_v)$ does not depend on $\sigma_v$, call it $d$, and that $M_\infty(\sigma_v)$ is Cohen--Macaulay. By exactness of the functor $M_\infty$, the Krull dimension of $M_\infty(\sigma)$ is the maximum of the Krull dimensions of the $M_\infty(\sigma_v)$ for the constituents $\sigma_v$ of $\sigma$, hence is also $d$. In particular, each nonzero such $M_\infty(\sigma_v)$ is maximal CM over $R_\infty/{\rm Ann}_{R_\infty}(M_\infty(\sigma))$. But being maximal CM over a given noetherian local ring $A$ of residue field $\F$ is preserved by extensions of modules (as can be checked from the characterization of Cohen--Macaulay modules using ${\rm Ext}_A^i(\F,-)$). Hence $M_\infty(\sigma)$ is maximal Cohen--Macaulay.
\end{proof}

If $\tau$ is a tame inertial type and $\lambda=((a_j,b_j))_{j\in \{0,\dots,f-1\}}$, where $a_j>b_j$ are integers, we set
\begin{equation}\label{infinitame}
R_\infty^{\lambda,{\tau}}\defeq R_\infty \otimes_{R_{\rbar_v^\vee}}R_{\rbar_v^\vee}^{\lambda,{\tau}},
\end{equation}
where $R_{\rbar_v^\vee}^{\lambda,{\tau}}$ parametrizes (framed) deformations of $\rbar_v^\vee$ of inertial type $\tau$ and Hodge--Tate weights $(a_j,b_j)$ in the embedding $\sigma_j:F_v\hookrightarrow E$. Note that from the determinant condition (see (\ref{rinfini})), one must have $a_j+b_j=1$ for all $j$ in order for $R_\infty^{\lambda,{\tau}}$ to be nonzero. When $a_j=a$ and $b_j=b$ for all $j$, we write $R_\infty^{(a,b),{\tau}}$. We finally write $\overline R_\infty\defeq R_\infty/(p)$ and $\overline R_\infty^{\lambda,\tau}\defeq R_\infty^{\lambda,\tau}/(p)$.

\begin{prop}\label{starringr}
There exists an integer $r \ge 1$ such that
\begin{enumerate}
\item\label{it:star1} for all $\sigma_v\in W(\rbar_v^\vee)$ the module $M_\infty(\sigma_v)$ is free of rank $r$ over its scheme-theoretic support, which is formally smooth over $\F$;
\item\label{it:star2} for all tame inertial types $\tau$ such that $\JH(\ovl{\sigma(\tau)})\cap W(\rbar_v^\vee)\ne \emptyset$ and all $K$-invariant $W(\F)$-lattices $\sigma^0(\tau)$ in $\sigma(\tau)$ with irreducible cosocle, the module $M_\infty(\sigma^0(\tau))$ is free of rank $r$ over its scheme-theoretic support, which is a domain.
\end{enumerate}
\end{prop}
\begin{proof}
Note first that the last assertions in \ref{it:star1} and \ref{it:star2} are a consequence of \cite[Def.\ 6.1.1]{EGS}, \cite[Thm.\ 7.2.1(2), (5)]{EGS}, and \cite[Prop.\ 3.5.1]{EGS}. The strategy of the proof is very close to the one of \cite[Thm.\ 10.1.1]{EGS} (which proves the case $r=1$), and we freely use some notation from {\it loc.~cit.} (it would be too tedious to recall everything). By \cite[\S 5.1]{EGS} there is a set ${\mathcal P}_{\tau}$ of subsets of $\{0,\dots,f-1\}$ and a unique $J\in {\mathcal P}_{\tau}$ such that $\sigma^0(\tau)=\sigma^0_J(\tau)$. The constituents of $\JH(\ovl{\sigma_J^0(\tau)})\cap W(\rbar_v^\vee)$ are indexed by a certain subset $\mathcal W$ of ${\mathcal P}_{\tau}$, and for certain subsets ${\mathcal J}\subseteq {\mathcal W}$ called capped intervals (see \cite[Def.\ 10.1.4]{EGS}) there exists a subquotient $\overline\sigma^{\mathcal J}$ of $\overline{\sigma^0_J(\tau)}$ such that the irreducible constituents of $\overline\sigma^{\mathcal J}$ are exactly the constituents of $\JH(\ovl{\sigma_J^0(\tau)})\cap W(\rbar_v^\vee)$ indexed by the elements of $\mathcal J$. We first prove by induction on $|{\mathcal J}|$ that the module $M_\infty(\overline\sigma^{\mathcal J})$ is free of rank $r$ over its scheme-theoretic support for an integer $r$ which depends neither on $\tau$ nor on $\mathcal J$.

By the argument in the proof of \cite[Lemma\ 3.6.2]{LLLM2}, the ring $\overline R_\infty/{\rm Ann}_{\overline R_\infty}(M_\infty(\overline\sigma^{\mathcal J}))$ is reduced. Indeed, it is generically reduced by d\'evissage, since the scheme-theoretic supports of $M_\infty(\sigma_v)$ for Serre weights $\sigma_v \in W(\rbar_v^\vee)$ are reduced, irreducible, and pairwise distinct (of dimension independent of $\sigma_v$) and since $\overline\sigma^{\mathcal J}$ is multiplicity-free; it also has no embedded associated prime, since $M_\infty(\overline\sigma^{\mathcal J})$ is Cohen--Macaulay by Lemma \ref{CM}. Let $I_{\mathcal J}$ be the ideal of $\overline R_\infty$ defined in \cite[\S 10.1]{EGS}, it follows that
\begin{equation}\label{annulateur}
{\rm Ann}_{\overline R_\infty}(M_\infty(\overline\sigma^{\mathcal J}))=I_{\mathcal J}.
\end{equation}

If $|{\mathcal J}|\leq 2$, then by \cite[Prop.\ 3.5.1]{EGS}, \cite[Prop.\ 10.1.11]{EGS} and the very last paragraph in the proof of \cite[Lemma 10.1.12]{EGS} there is a tame inertial type $\tau'$ and a $W(\F)$-lattice $\sigma^0(\tau')$ in $\sigma(\tau')$ such that $\JH(\ovl{\sigma^0(\tau')})\cap W(\rbar_v^\vee)=\JH(\overline\sigma^{\mathcal J})$ and $M_\infty(\ovl{\sigma^0(\tau')})\cong M_\infty(\overline\sigma^{\mathcal J})$. By \cite[Thm.\ 7.2.1(2)]{EGS} (and \cite[Rk.\ 5.2.2]{gee-kisin}) the local ring $R_{\rbar_v^\vee}^{(1,0),{\tau'},\psi^{-1}_v}$ is regular, and hence also $R_\infty^{(1,0),{\tau'}}$ by (\ref{rinfini}) and (\ref{infinitame}). By \cite[Lemma\ 6.1.4]{EGS} it follows that $M_\infty(\sigma^0(\tau'))$ is free of finite type over $R_\infty^{(1,0),{\tau'}}$. Hence $M_\infty(\ovl{\sigma^0(\tau')})\cong M_\infty(\overline\sigma^{\mathcal J})$ is also free of finite type over $\overline R_\infty^{(1,0),\tau'}\cong \overline R_\infty/I_{\mathcal J}$. 

If $|{\mathcal J}|=2$, then $\overline\sigma^{\mathcal J}$ has two distinct constituents $\sigma_1$, $\sigma_2$ and the freeness of $M_\infty(\overline\sigma^{\mathcal J})$ over $\overline R_\infty/I_{\mathcal J}$ (which is a power series ring over $\F\bbra{X_1,X_2}/(X_1X_2)$) easily implies that $M_\infty(\sigma_1)$ and $M_\infty(\sigma_2)$ have the {\it same} rank over their schematic support (which is a power series ring over, respectively, $\F\bbra{X_1}$ and $\F\bbra{X_2}$). Using \cite[Prop.\ 10.1.11]{EGS} and the fact that all Serre weights in $W(\rbar_v^\vee)$ can be ``connected'' by nonsplit extensions (as follows e.g.\ from \cite[Prop.\ 3.5.2]{EGS} applied to a semisimple $\rhobar$), we obtain \ref{it:star1} for a certain integer $r\ge 1$.

If $|{\mathcal J}| > 2$ and $\mathcal J$ has a unique minimal element $J_0$ (for inclusion inside $\{0,\dots,f-1\}$), then exactly as in the analogous case of the proof of \cite[Thm.\ 10.1.1]{EGS} but using \cite[Lemma\ 4.5]{DanWild} instead of \cite[Lemma\ 10.1.13]{EGS}, we deduce that the $\overline R_\infty$-module $M_\infty(\overline\sigma^{\mathcal J})$ is generated by $r$ elements. Then one applies Lemma \ref{induction} to $M=M_\infty(\overline\sigma^{\mathcal J})$ and $N=M_\infty(\overline\sigma^{\{J_0\}})$ (the hypotheses of the lemma are satisfied, as $M/N\cong M_\infty(\overline\sigma^{{\mathcal J}\backslash \{J_0\}})$, $I_{{\mathcal J}\backslash \{J_0\}}/I_{\mathcal J}\cong \overline R_\infty/I_{\mathcal \{J_0\}}$ and using (\ref{annulateur})) together with the induction hypothesis on $|{\mathcal J}|$ to deduce that $M_\infty(\overline\sigma^{\mathcal J})$ is free of rank $r$ over $\overline R_\infty/ I_{\mathcal J}$.

If $|{\mathcal J}| > 2$ and $\mathcal J$ has at least two distinct minimal elements $J_1$, $J_2$, let ${\mathcal J}_i\defeq {\mathcal J}\backslash\{J_i\}$, $i=1,2$. Then by the induction hypothesis $M_\infty(\overline\sigma^{{\mathcal J}_1})$, $M_\infty(\overline\sigma^{{\mathcal J}_2})$ and $M_\infty(\overline\sigma^{{\mathcal J}_1\cap {\mathcal J}_2})$ are all free of rank $r$ over (respectively) $\overline R_\infty/I_{{\mathcal J}_1}$, $\overline R_\infty/I_{{\mathcal J}_2}$ and $\overline R_\infty/I_{{\mathcal J}_1\cap {\mathcal J}_2}$. Hence so is the fiber product $M_\infty(\overline\sigma^{{\mathcal J}_1})\times_{M_\infty(\overline\sigma^{{\mathcal J}_1\cap {\mathcal J}_2})}M_\infty(\overline\sigma^{{\mathcal J}_1})\cong M_\infty(\overline\sigma^{\mathcal J})$ over $\overline R_\infty/I_{{\mathcal J}_1}\times_{\overline R_\infty/I_{{\mathcal J}_1\cap {\mathcal J}_2}}\overline R_\infty/I_{{\mathcal J}_2}\cong \overline R_\infty/I_{\mathcal J}$ (see the analogous case in the proof of \cite[Thm.\ 10.1.1]{EGS}).

It remains to finish the proof of \ref{it:star2}. By the previous proof, $M_\infty(\ovl{\sigma_J^0(\tau)})\cong M_\infty(\overline\sigma^{\mathcal W})$ is free of rank $r$ over $\overline R_\infty/I_{\mathcal W}\cong \overline R_\infty^{(1,0),\tau}$. By Nakayama's lemma, we deduce a surjection of $R_\infty^{(1,0),\tau}$-modules $f:(R_\infty^{(1,0),\tau})^r\twoheadrightarrow M_\infty(\sigma_J^0(\tau))$ which is an isomorphism modulo $p$, hence satisfies $p\ker(f)=\ker(f)$ since $M_\infty(\sigma_J^0(\tau))$ has no $p$-torsion. By Nakayama's lemma again we deduce $\ker(f)=0$, which finishes the proof.
\end{proof}

\begin{cor}\label{penible}
Let $\sigma\defeq \bigoplus_{i=1}^m \sigma_i^{n_i}$, where $m, n_i\ge 1$ and the $\sigma_i=\sigma_i^{\rm smooth}\otimes_{E}\sigma_i^{\rm alg}$ are pairwise nonisomorphic absolutely irreducible locally $\Qp$-algebraic representations of $K$ over $E$ satisfying the following hypothesis: $\sigma_i^{\rm smooth}\otimes_E\Qpbar$ lies in the image of the inertial local Langlands correspondence $\tau \mapsto \sigma(\tau)$ \emph{(}after extending scalars to $\Qpbar$, see \S\ref{sec:inert-local-langl}\emph{)}
and $\bigcup_i \JH(\ovl{\sigma_i})\cap W(\rbar_v^\vee)\ne \emptyset$. Let $\sigma^0$ be any $W(\F)$-lattice in $\sigma$ preserved by $K$. Then
\begin{enumerate}
\item\label{it:pen1} $M_\infty(\sigma^0)$ is maximal CM over its scheme-theoretic support $S\defeq R_\infty/{\rm Ann}_{R_\infty}(M_\infty(\sigma^0))$, which is reduced;
\item\label{it:pen2} $M_\infty(\sigma^0)\otimes_{W(\F)}E$ is locally free over its scheme-theoretic support $S[1/p]$, which is formally smooth over $E$.
\end{enumerate}
\end{cor}
\begin{proof}
For $i\in \{1,\dots,m\}$ let $\sigma_i^0$ be any $K$-invariant $W(\F)$-lattice in $\sigma_i$. It easily follows from the exactness of the functor $M_\infty$ that there is an isomorphism of $R_{\infty}[1/p]$-modules
\begin{equation}\label{generic}
M_\infty(\sigma^0)[1/p]\cong \bigoplus_{i=1}^m M_\infty(\sigma_i^0)[1/p]^{\oplus n_i}.
\end{equation}
From the Taylor--Wiles--Kisin method, we know that the action of $R_\infty$ on $M_\infty(\sigma_i^0)$ factors through a reduced equidimensional $p$-torsion free quotient of $R_\infty$ and that the support of $M_\infty(\sigma_i^0)$ is a union of irreducible components of that quotient (see e.g.\ \cite[Lemmas 4.17, 4.18]{CEGGPS}). Hence the scheme-theoretic support of $M_\infty(\sigma_i^0)$ is also a reduced $p$-torsion free quotient $R_\infty/I_i$ of $R_\infty$. It follows from (\ref{generic}) that the support of $M_\infty(\sigma^0)[1/p]$ is $S[1/p]\cong (R_\infty/\bigcap_iI_i)[1/p]$ (as there is no $p$-torsion). Since the $\Spec\ \!(R_\infty/I_i)[1/p]$ for ${1\leq i\leq m}$ correspond to {\it disjoint} closed subschemes of $\Spec R_\infty[1/p]$ (as the locally algebraic representations $\sigma_i$ are pairwise distinct), one has by the Chinese remainder theorem
\begin{equation}\label{chinese}
S[1/p]=(R_\infty/\bigcap_iI_i)[1/p]\cong \prod_{i=1}^m (R_\infty/I_i)[1/p],
\end{equation}
which is thus reduced and formally smooth over $E$ by \cite[Thm.\ (3.3.8)]{KisinPSS}, hence regular by \cite[Thm.\ 28.7]{Ma}. Since $S$ has no $p$-torsion (as $S$ acts faithfully on $M_\infty(\sigma^0)$ which has no $p$-torsion by exactness of $M_\infty$), we deduce that $S$ is also reduced. 

The module $M_\infty(\sigma^0)/(p)\cong M_\infty(\ovl{\sigma^0})$ is a Cohen--Macaulay-module by Lemma \ref{CM}, and $p$ is a non-zero-divisor on $M_\infty(\sigma^0)$, hence $M_\infty(\sigma^0)$ is also Cohen--Macaulay, hence maximal CM over $S$. Moreover applying \cite[Thm.\ 17.3(iii)]{Ma} to $M_\infty(\sigma^0)$ we see that $M_\infty(\sigma^0)[1/p]$ is also Cohen--Macaulay as an $S[1/p]$-module. The Auslander--Buchsbaum formula applied to the localizations at prime ideals of $S[1/p]$ of the Cohen--Macaulay module $M_\infty(\sigma^0)[1/p]$ over the regular ring $S[1/p]$ implies $M_\infty(\sigma^0)[1/p]$ is locally free over $S[1/p]$.
 \end{proof}

The following remark shows that the assumption on $\sigma_i^{\rm smooth}$ is often satisfied.

\begin{rem}\label{rk:image-of-inertial-llc}
  If $\sigma$ is any irreducible smooth representation of $K$ over $E$ that is tame (i.e.\ the action of $K$ factors through $K\twoheadrightarrow \GL_2(k)$)
  and that is not a twist of the Steinberg representation of $\GL_2(k)$ (equivalently, is not of dimension $q_v$), then $\sigma\otimes_E\Qpbar$ lies in the image of
  the inertial local Langlands correspondence $\tau \mapsto \sigma(\tau)$, after extending scalars to $\Qpbar$. (To see this, first note that $\sigma$ is
  absolutely irreducible by (the proof of) \cite[Lemma 3.1.1]{EGS}. If $\sigma$ is one-dimensional, then it is clear that $\sigma$ lies in the image; otherwise, $\sigma$ is
  a principal series or cuspidal representation of $\GL_2(k)$, and the claim follows from the case $a = 1$ in \cite[Th.\ 2.1.1.4]{BM} or alternatively \cite[Prop.\ 2.4.1]{EGH}.)
\end{rem}
 
 For any Serre weight $\sigma_v$, recall that we have defined in \S\ref{sec:lattices} the two $\GL_2(k)$-representations $P_{\sigma_v}=\Proj_{\GL_2(k)}\sigma_v$ and $\widetilde P_{\sigma_v}$ over, respectively, $\F$ and ${\mathcal O}=W(\F)$. 

\begin{prop}\label{HT10}
If $\sigma_v\in W(\rbar_v^\vee)$, then $M_\infty(\widetilde P_{\sigma_v})$ is free of rank $r$ over $R_\infty/\cap_{\tau}{\mathfrak p}_\tau$, where $\tau$ runs over the tame inertial types such that $\sigma_v\in \JH(\ovl{\sigma(\tau)})$ and ${\mathfrak p}_\tau$ is the prime ideal $\ker(R_\infty\twoheadrightarrow R_\infty^{(1,0),{\tau}})$.
\end{prop}
\begin{proof}
(i) We first prove that the $R_\infty$-module $M_\infty(\widetilde P_{\sigma_v})$ can be generated by $r$ elements. By Nakayama's lemma, it is enough to prove the same statement with $M_\infty(P_{\sigma_v})$, or, equivalently, that $\dim_{\F}(M_\infty(P_{\sigma_v})/{\mathfrak m}_\infty)\le r$. By \eqref{eq:Minfty-pi} it is enough to prove \[\dim_{\F} \Hom_{K}(P_{\sigma_v},\pi)\!=\dim_{\F} \Hom_{\GL_2(k)}(P_{\sigma_v},W)\!=r,\] where $\pi$ is the admissible smooth representation of $\GL_2(F_v)$ defined in \eqref{eq:pi-indef} or \eqref{eq:pi-def} and $W\defeq\pi^{K_1}$. By Proposition \ref{starringr}\ref{it:star1} we have $\dim_{\F} \Hom_{\GL_2(k)}({\sigma_v},W)=r$. Let $D_0(\rbar_v^\vee)$ be the representation of $\GL_2(k)$ over $\F$ defined in \cite[\S 13]{BP} (see also Lemma \ref{lem:connected}) and recall that by construction
$$\Hom_{\GL_2(k)}\big(P_{\sigma_v}, D_0(\rbar_v^\vee)/\soc_{\GL_2(k)}D_0(\rbar_v^\vee)\big)=0.$$
Hence it is enough to prove that there is a $\GL_2(k)$-equivariant injection
\begin{equation*}
W\hookrightarrow D_0(\rbar_v^\vee)^{\oplus r}
\end{equation*}
(which is necessarily an isomorphism on $\soc_{\GL_2(k)}W=(\soc_{\GL_2(k)}D_0(\rbar_v^\vee))^{\oplus r}$), or equivalently a $\GL_2(k)$-equivariant surjection $(D_0(\rbar_v^\vee)^\vee)^{\oplus r}\twoheadrightarrow W^\vee$. But this follows exactly as in the proofs of \cite[Lemma\ 4.5]{LMS} and \cite[Prop.\ 4.6]{LMS} (plus Proposition~\ref{starringr}). More precisely, one replaces the integer $1$ by the integer $r$ in the statements of {\it loc.~cit.}, and the proofs are basically the same, replacing the surjection $\bigoplus_\kappa P_\kappa\twoheadrightarrow D_0^\vee$ by a surjection $\bigoplus_\kappa P_\kappa^{\oplus r}\twoheadrightarrow D_0^\vee$ (for \cite[Lemma\ 4.5]{LMS}, one gets at the end of the proof $\dim_{\F}\Hom_K(D_0^\vee,\overline\sigma^0(\tau))>r$ instead of $\dim_{\F}\Hom_K(D_0^\vee,\overline\sigma^0(\tau))>1$).

(ii) We now prove the proposition. Let $S=R_\infty/{\rm Ann}_{R_\infty}(M_\infty(\widetilde P_{\sigma_v}))$ be the scheme-theoretic support of $M_\infty(\widetilde P_{\sigma_v})$. The representation $\widetilde P_{\sigma_v}[1/p]$ over $E$ is the direct sum of the (tame smooth) representations $\sigma(\tau)$ for all the tame inertial types $\tau$ such that $\sigma_v\in \JH(\ovl{\sigma(\tau)})$, and each such $\sigma(\tau)$ occurs only once. It follows from (\ref{generic}) (with all $n_i=1$), (\ref{chinese}) and Proposition \ref{starringr}\ref{it:star2} that $M_\infty(\widetilde P_{\sigma_v})[1/p]$ is free of rank $r$ over $S[1/p]$. By (i), we have a surjection $S^r\twoheadrightarrow M_\infty(\widetilde P_{\sigma_v})$ which is thus an isomorphism after inverting $p$ (\cite[Thm.\ 2.4]{Ma}), hence is also injective. Finally we obtain $S=R_\infty/\cap_{\tau}{\mathfrak p}_\tau$ from (\ref{generic}), from $M_\infty(\widetilde P_{\sigma_v})\hookrightarrow M_\infty(\widetilde P_{\sigma_v})[1/p]$ and from the fact the rings $R_\infty^{(1,0),{\tau}}$ are all domains (Proposition \ref{starringr}\ref{it:star2}), as then the support of each $M_\infty(\sigma(\tau)^0)[1/p]$ in (\ref{generic}) is exactly $R_\infty/{\mathfrak p}_\tau$. 
\end{proof}
 
\subsection{Freeness for projective covers}
\label{sec:freeness}
 
We prove that $M_\infty(R)$ is free over its scheme-theoretic support, where $R$ is the lattice defined in \eqref{eq:R} of  \S\ref{sec:constructionlattice}.

We keep all the notation of \S\ref{tobefree1} and we fix a Serre weight $\sigma_v\in W(\rbar^\vee_v)$. We start with the following lemma. 

\begin{lem} \label{lem:Loewy3}  
If $Q$ is a quotient of $\Proj_{K/Z_1}\sigma_v/\mathfrak{m}_{K_1}^2(\Proj_{K/Z_1}\sigma_v)$ satisfying the following conditions
\begin{itemize}
\item[$\mathrm{(a)}$]  $\JH(\soc_K(Q))\subseteq \JH(\mathrm{Proj}_{K/K_1}\sigma_v)$ up to multiplicity, 
\item[$\mathrm{(b)}$] $[Q/\soc_K(Q):\sigma_v]=1$,
\end{itemize}
then both $\rad_K(Q)$ and $Q/S$ are fixed by $K_1$,  where $S$ denotes the largest submodule of $\soc_K(Q)$ which is $\sigma_v$-isotypic.   %
If furthermore $Q$ satisfies
\begin{itemize}
\item[$\mathrm{(c)}$] $\JH(\soc_K(Q))\subseteq W(\rbar^\vee_v)$ up to multiplicity,  
\item[$\mathrm{(d)}$]  $\JH(\rad_K(Q)/\soc_K(Q))\cap W(\overline{r}_v^{\vee})=\emptyset$, 
\end{itemize} 
then  $Q$ has Loewy length $\leq 3$. 
\end{lem}

\begin{proof}
Note that $\mathrm{cosoc}_K(Q)\cong \sigma_v$. Fix a  decomposition of $\soc_K(Q)$ as $\bigoplus_{i=1}^n \sigma_i$, with $\sigma_i$ irreducible (with $\sigma_i \cong \sigma_j$ allowed). For each $i$, $Q$ admits a quotient, say $Q_{\sigma_i}$, with socle $\sigma_i$ (via $\sigma_i \into Q \onto Q_{\sigma_i}$). Then the natural morphism $Q\ra \bigoplus_{i=1}^nQ_{\sigma_i}$ is injective and 
\begin{equation}\label{eq:radsum}
\rad_K(Q)\subset \rad_K(\bigoplus_{i=1}^nQ_{\sigma_i})=\bigoplus_{i=1}^n\rad_K(Q_{\sigma_i}). \end{equation}
Moreover, since taking radical preserves surjective morphisms, see \cite[\S1, Prop.~4]{Alperin} (applied to a suitable finite-dimensional quotient of the ring $\F[K]$), we have an induced surjection \begin{equation}\label{eq:rad-soc}\rad_K(Q)/\soc_K(Q)\twoheadrightarrow \rad_K(Q_{\sigma_i})/\soc_K(Q_{\sigma_i}). \end{equation}

Assume first that $Q$ satisfies conditions (a), (b). To prove that $\rad_K(Q)$ is fixed by $K_1$, using \eqref{eq:radsum} 
we may assume $\soc_K(Q)$ is irreducible (replace $Q$ by some $Q_{\sigma_i}$).  We have two cases.
\begin{itemize}
\item[--] If $\soc_K(Q)\not\cong \sigma_v$, then $[Q:\sigma_v]=1$ by (b).  By \cite[Thm.~2.30]{HuWang2} $Q$ is isomorphic to $I(\soc_K(Q),\sigma_v)$, and   $Q$ is itself fixed by $K_1$ by (a).  
\item[--] If $\soc_K(Q)\cong \sigma_v$, then $[Q/\sigma_v:\sigma_v]=1$ and  $Q/\sigma_v$ is multiplicity free  by \cite[Cor.~2.26]{HuWang2}. Then $Q$ fits in an exact sequence $0\ra \sigma_v\ra Q\ra Q/\sigma_v\ra0$ (analogous to \cite[(4.9)]{HuWang2}), and   the end of the proof of \cite[Prop.~4.18]{HuWang2} shows that  $\rad_K(Q)/\soc_K(Q)=\rad_K(Q)/\sigma_v$ is semisimple and embeds in   $ \bigoplus \sigma_v'$, where the sum is taken over all Serre weights   $\sigma_v'$   such that $\Ext^1_{K/K_1}(\sigma_v',\sigma_v)\neq0$.  Hence, $\rad_K(Q)\subset Q^{K_1}$ by (the dual version of)  \cite[Cor.~2.31]{HuWang2}. We also deduce that $Q$ has Loewy length $3$ in this case.
\end{itemize}

 We prove that $Q/S$ is  fixed by $K_1$. Using the exact sequence $0\ra S\ra Q\ra Q/S\ra0$, we deduce that if $\Hom_{K}(\sigma,Q/S)\neq0$ for some Serre weight $\sigma$, then either $\sigma\in \soc_K(Q)$, or $\Ext^1_{K/Z_1}(\sigma,\sigma_v)\neq0$. In either case, we have $\sigma\in \JH(\Proj_{K/K_1}(\sigma_v))$ (use \cite[Lemma~2.10(ii)]{HuWang2} in the second case). Noting that $[Q/S:\sigma_v]=1$ by the construction of $S$,  the conclusion follows from \cite[Cor.~2.31]{HuWang2}. 

Assume now that $Q$ also satisfies conditions (c) and (d). Again using \eqref{eq:radsum} and  \eqref{eq:rad-soc}, we may assume $\soc_K(Q)$ is irreducible. The case $\soc_K(Q)\cong \sigma_v$ is treated above. Assume $\soc_K(Q)\not\cong \sigma_v$. As seen above, $Q=I(\soc_K(Q),\sigma_v)$. Since $\soc_K(Q)\in W(\overline{r}_v^{\vee})$ by (c), it follows from \cite[Prop.~2.24]{HuWang} that any Jordan--H\"older factor of $Q$ lies in $W(\overline{r}_v^{\vee})$. Hence, we must have  $\rad_K(Q)/\soc_K(Q)=0$ by (d),  and $Q$ has Loewy length $2$. This finishes the proof.   
\end{proof}

For $j\in \{0,\dots,f-1\}$ let $V(\alpha_j)\defeq V((1,-1))^{(j)}_{/W(\F)}\cong ({\rm Sym}^2(W(\F)^2)\otimes {\det}^{-1})^{(j)}$ be the algebraic representation of $K$ over $W(\F)$ as defined in \S\ref{sec:GT:prel}. As in \S\ref{sec:constructionlattice} we define the locally algebraic representation $R_{2,j}\defeq V(\alpha_j)\otimes_{W(\F)}\widetilde P_{\sigma_v}$ of $K$ over $W(\F)$ (so $R_2=\bigoplus_j R_{2,j}$). We set
$$R'_{2,j}\defeq \{x\in R_{2,j} : (x\bmod pR_{2,j})\in\ P_{\sigma_v}\}$$
using the fixed embedding $\iota_j:P_{\sigma_v} \hookrightarrow R_{2,j}/pR_{2,j}$ from \S\ref{sec:constructionlattice}. This is a $K$-invariant $W(\F)$-lattice in $R_{2,j}[1/p]$ such that $pR_{2,j}\subset R'_{2,j} \subset R_{2,j}$ and $R'_{2,j}/pR_{2,j}\simto P_{\sigma_v}$.  Comparing the constructions of $R'_{2,j}$ and of $R_{2}'$ (in \S\ref{sec:constructionlattice}), it is direct to see that  the natural map $R_2'\ra R'_{2,j}$ (induced by the projection $R_2\cong\bigoplus_i R_{2,i}\twoheadrightarrow R_{2,j}$)  is surjective, hence  $R'_2/pR'_2\rightarrow R'_{2,j}/pR'_{2,j}$ is also surjective. By Proposition \ref{prop:alg-lattice}, we deduce $(R'_{2,j}/pR'_{2,j})_{K_1}=P_{\sigma_v}$ (hence $\cosoc_K(R'_{2,j}/pR'_{2,j})=\sigma_v$) and (using (\ref{decomp})) a $K$-equivariant short exact sequence 
\begin{equation}\label{eq:R'2j}0 \longrightarrow P_{\sigma_{1,j}}\oplus P_{\sigma_{2,j}} \longrightarrow R'_{2,j}/pR'_{2,j} \longrightarrow P_{\sigma_v} \longrightarrow 0.\end{equation}

\begin{lem}\label{easyglueingi}
For all $j\in \{0,\dots,f-1\}$ the $R_\infty$-module $M_\infty(R'_{2,j})$ is minimally generated by $r$ elements.
\end{lem}
\begin{proof}

We prove by induction on the length of $Q$ (as a representation of $K$) that if $Q$ is a nonzero quotient of $R_{2,j}'/pR_{2,j}'$, then $M_\infty(Q)$ is minimally generated by $r$ elements. 
If $\lg(Q)=1$, then $Q=\sigma_v$ (as $(R_{2,j}'/pR_{2,j}')_{K_1}=P_{\sigma_v}$) and $M_\infty(\sigma_v)$ is minimally generated by $r$ elements by Proposition \ref{starringr}\ref{it:star1}.  Now assume that the result is proved for all quotients of $R_{2,j}'/pR_{2,j}'$ of length $\leq n$. Let $Q$ be a quotient of $R_{2,j}'/pR_{2,j}'$ of length $n+1$. If the socle of $Q$ contains a Serre weight $\sigma$ which is not in $W(\rbar_v^\vee)$, then $M_\infty(Q)=M_\infty(Q/\sigma)$ and $M_\infty(Q/\sigma)$ is minimally generated by $r$ elements by induction. Hence we can assume that all the Serre weights in the socle of $Q$ are in $W(\rbar_v^\vee)$.

 Assume first that $[\rad_K(Q)/\soc_K(Q):\sigma_v]\neq 0$ (in particular $[Q:\sigma_v]\geq 2$ as $\cosoc_K(Q)=\sigma_v$). Then we may find a submodule $Q'\subsetneq Q$ such that $\cosoc_K(Q')\cong \sigma_v$ and $Q'$ is not contained in $\soc_KQ$. By  Proposition \ref{prop:projectivity}, $Q'$ is isomorphic to a (proper) quotient of $Q$, so  $M_{\infty}(Q')$ is minimally generated by $r$ elements by induction. On the other hand, let $\sigma'$ be  a Serre weight in $\soc_K(Q')$. Then $M_{\infty}(Q'/\sigma')$ and $M_{\infty}(Q/\sigma')$ are also  minimally generated by $r$ elements by induction. The conclusion follows from \cite[Lemma~4.5]{DanWild} with $M$, $M'$ and $M''$ taken to be $M_{\infty}(Q)$, $ M_{\infty}(Q')$ and $M_{\infty}(\sigma')$ respectively.  

Assume now that $[\rad_K(Q)/\soc_K(Q):\sigma_v]=0$, so that $[Q/\soc_K(Q):\sigma_v]=1$.  Moreover, if $S$ denotes the largest submodule of $\soc_K(Q)$ which is $\sigma_v$-isotypic, then $Q/S$ is  a quotient of $P_{\sigma}=(R'_{2,j}/pR'_{2,j})_{K_1}$  by the first part of Lemma \ref{lem:Loewy3}; note that condition (a) of that lemma holds by our assumption on $\soc_K(Q)$ above and the fact $W(\overline{r}_v^{\vee})\subseteq \JH(\Proj_{K/K_1}\sigma_v)$   (see \cite[\S11]{BP}). Using \eqref{eq:R'2j}, this means that the composite morphism 
$P_{\sigma_{1,j}}\oplus P_{\sigma_{2,j}}\ra R'_{2,j}/pR'_{2,j}\twoheadrightarrow Q$ has image contained in $S$. Since $S$ is $\sigma_{v}$-isotypic (and $\sigma_v\ncong\sigma_{1,j},\sigma_{2,j}$), this image must be zero and $Q$ is a quotient of $P_{\sigma_v}$.   
 As $M_\infty(P_{\sigma_v})$ is generated by $r$ elements by Proposition \ref{HT10}, it follows that $M_\infty(Q)$ is also generated by $r$ elements. As $M_\infty(\sigma_v)$ is minimally generated by $r$ elements by Proposition \ref{starringr}\ref{it:star1} and $M_\infty(\sigma_v)$ is a quotient of $M_\infty(Q)$, we finally have that $M_\infty(Q)$ is minimally generated by $r$ elements.

We conclude that the $R_\infty$-module $M_\infty(R'_{2,j}/pR'_{2,j})$ is generated by $r$ elements, from which the result follows by Nakayama's lemma. 
\end{proof}

\begin{prop}\label{prop:freeness-loc-alg-type}
  Suppose that $\tau^0$ is a representation of $K/K_1 = \GL_2(k)$ on a finite free $W(\F)$-module such that $\tau^0[1/p]$ is (absolutely) irreducible
  and $\cosoc_K \tau^0 \cong \sigma_v$.
  Fix $j \in \{0,\dots,f-1\}$. 
  If $L$ is any $K$-stable lattice  in $(V(\alpha_j) \otimes_{W(\F)} \tau^0)[1/p]$ such that $\cosoc_K L \cong \sigma_v$,
  then $M_\infty(L)$ is free of rank $r$ over its schematic support, which is a domain.
\end{prop}
 
The proof shows that such a lattice $L$ exists and is unique up to homothety.

\begin{proof}
We write $\sigma_v \cong F(\lambda)$ for some $\lambda\in X_1(\un{T})$ that is 8-deep in $\un{C}_0$ by Proposition~\ref{prop:SW:graph} and our genericity assumption. 
  By Remark~\ref{rk:image-of-inertial-llc} and Proposition \ref{prop:JH:graph} (and Lemma~\ref{lm:change-origin}) there exist $\mu \in X_1(\un{T})$ and signs $\un{\eps}\in\set{\pm1}^\cJ$ such that $\JH(\overline{\tau^0}) = \JH(D_{\mu,\un{\eps}})$ (with multiplicities!), where $F(\t_\mu(\sum \eps_i \ovl{\eta}_i)) \cong \sigma_v$ and $D_{\mu,\un{\eps}}$ is defined in \eqref{eq:D-lambda-eps}. 
We deduce that $F(\alpha_j) \otimes_{\F} \overline{\tau^0}$ is multiplicity-free, as $F(\alpha_j) \otimes_{\F} D_{\mu,\un{\eps}}$ is multiplicity-free by Lemma~\ref{lemmJH}.
  
As $\cosoc_K \tau^0 \cong \sigma_v$ we have a surjection $\pr : \widetilde P_{\sigma_v} \onto \tau^0$, and we let $L_0$
denote the image of the composition \[R_{2,j}' \into V(\alpha_j) \otimes_{W(\F)} \widetilde P_{\sigma_v} \onto
  V(\alpha_j) \otimes_{W(\F)} \tau^0\] where the second morphism is $\id \otimes \pr$. By the paragraph before
Lemma~\ref{easyglueingi} we see that $L_0$ is a lattice as in the statement of the proposition, and that it moreover
contains $p(V(\alpha_j) \otimes_{W(\F)} \tau^0)$. 
Since $(V(\alpha_j) \otimes_{W(\F)} \tau^0)[1/p]$ is irreducible and residually 
  multiplicity-free by the preceding paragraph,  up to homothety $L_0$ is the unique lattice in $(V(\alpha_j) \otimes_{W(\F)} \tau^0)[1/p]$ such that $\cosoc_{K}(L_0)\cong \sigma_v$ by \cite[Lemma 4.1.1]{EGS}.
Therefore, we may assume that $L = L_0$.

  We first show that $M_\infty(L/pL)$
  is free of rank $r$ over its schematic support. We have a short exact sequence  \[ 0 \to W_1 \xrightarrow{\,p\,} L/pL \to W_2 \to 0, \]
  where $W_1 \defeq (V(\alpha_j) \otimes_{W(\F)} \tau^0)/L$ and $W_2 \defeq L/p(V(\alpha_j) \otimes_{W(\F)} \tau^0)$.  We
  will show that (i) $M_\infty(W_1) = 0$ and (ii) there exists a lattice $\tau'^0$ in a tame type with cosocle $\sigma_v$ and
  a surjection $W_2 \onto \overline{\tau'^0}$ such that $M_\infty(W_2) = M_\infty(\overline{\tau'^0})$. We then conclude by
  Proposition~\ref{starringr} that $M_\infty(L/pL) = M_\infty(\overline{\tau'^0})$ can be generated by $r$ elements.

  We will use repeatedly in this proof that if $V_1$ and $V_2$ are multiplicity-free representations of $\GL_2(k)$ over $\F$ 
  having cosocle $\sigma_v$, then $\JH(V_1) \subset \JH(V_2)$ implies that $V_1$ is a quotient of $V_2$. (The reason is that
  the $V_i$ are quotients of $P_{\sigma_v}$, hence factor through the largest quotient of $P_{\sigma_v}$ that is multiplicity-free
  \cite[Prop.\ 3.6, Thm.\ 4.7]{BP}.) Using notation as in \S\ref{sec:preliminaries} locally at the place $v$ we will also use that 
  if a weight $\lambda$ is $7$-deep in $\un{C}_0$ and $\un{\eps}\in\set{\pm1}^\cJ$, then the submodule structure of 
  $D_{\lambda,\un{\eps}}$ is known by Theorem~\ref{thm:elimination}, parts \ref{it:elimination-1} and \ref{it:elimination-5}
  (where the integers $a_i$ are now restricted by $0 \le a_i \le 1$). (It
  is also known by \cite[Thm.\ 4.7]{BP}, but using different notation.)

  As $\ovl{\tau^0}$ and $D_{\mu,\un{\eps}}$ are multiplicity-free with cosocle $\sigma_v$ (see the beginning of the proof for $D_{\mu,\un{\eps}}$), we deduce by the previous paragraph that $\overline{\tau^0} \cong D_{\mu,\un{\eps}}$.
  As $\lambda$ is 8-deep in $\un{C}_0$ we know by Remark~\ref{rk:t_lambda}\ref{it:t_lambda:4} that $\mu$ and $\mu+\eps_j\alpha_j$ are 7-deep in $\un{C}_0$.
  By the beginning of the proof $F(\alpha_j) \otimes_{\F} D_{\mu,\un{\eps}}$ is multiplicity-free, hence 
  $W_2$ is (by its definition) the unique subrepresentation of $F(\alpha_j) \otimes_{\F} D_{\mu,\un{\eps}}$ with cosocle $\sigma_v \cong F(\t_\mu(\sum \eps_i \ovl{\eta}_i))$.
  Therefore, the irreducible constituent of $W_2$ are given by the set $\{ F(\t_\mu(\sum \eps_i a_i \ovl{\eta}_i)) : 0 \le a_i \le 1$ for all $i \ne j$, $0 \le a_j \le 2\}$.
  (This is a consequence of Proposition \ref{prop:main}, which combined with Lemma~\ref{lm:ext1}, completely determines the submodule structure of $F(\alpha_j) \otimes_{\F} D_{\mu,\un{\eps}}$.
  We leave the details as a pleasant exercise to the reader.)

By Proposition \ref{prop:main} and since we know the constituents of $W_2$, the constituents of $W_1$ have the form $F(\t_\mu(\sum \eps_i a_i \ovl{\eta}_i))$, where $0 \le a_i \le 1$ for all
  $i \ne j$ and $a_j \in \{-2,-1,3\}$.  As $\sigma_v \cong F(\t_\mu(\sum \eps_i \ovl{\eta}_i))$ is modular, we see by Proposition \ref{prop:SW:graph} 
  that any other modular Serre weight is of the form $F(\t_\mu(\sum \eps_i b_i \ovl{\eta}_i))$ with $0 \le b_i \le 2$ for all $i\in\cJ$. We conclude that $M_\infty(W_1) = 0$.

  For short let $\nu \defeq \sum \eps_i \ovl{\eta}_i$.  Using again Proposition \ref{prop:SW:graph} we write the modular Serre weights as
  $F(\t_\mu(\nu + \sum \eps_i' b_i \ovl{\eta}_i))$ for some signs $\un{\eps}'\in\set{\pm1}^\cJ$ and integers $0 \le b_i \le 1$  (with $\un b = 0$
  corresponding to $\sigma_v$) and note that the constituents of $W_2$ are given by $F(\t_\mu(\nu - \sum c_i \ovl{\eta}_i))$,
  $0 \le c_i \le 1$ for $i \ne j$ and $-1 \le c_j \le 1$.

  By Proposition \ref{prop:JH:graph} (and Lemma \ref{lm:change-origin}) we can find a representation $\tau'^0$ of $\GL_2(k)$ on a finite free $W(\F)$-module such that $\tau'^0[1/p]$ is irreducible
  and such that
  \[ \JH(\ovl{\tau'^0}) = \Big\{ F(\t_\mu(\nu - \sum_{i \ne j} d_i \ovl{\eta}_i + \eps'_j d_j \ovl{\eta}_j)) : \text{$0 \le d_i \le 1$ for all $i$}\Big\}.\]
  We may
  assume that $\tau'^0$ has cosocle $\sigma_v$. As $\ovl{\tau'^0}$ and $W_2$ are multiplicity-free, have cosocle $\sigma_v$,
  and $\JH(\ovl{\tau'^0}) \subset \JH(W_2)$ we see that $\ovl{\tau'^0}$ is a quotient of $W_2$.
  By the above definition of $\varepsilon'_j$, the Serre weights in the complement $\JH(W_2)\setminus \JH(\ovl{\tau'^0})$ are not
  modular, so $M_\infty(W_2) = M_\infty(\overline{\tau'^0})$, as desired.

  We have shown that $M_\infty(L/pL)$ is free of rank $r$ over its schematic support. To deduce that $M_\infty(L)$ is
  free of rank $r$ over its schematic support $S$, we first observe that $S = R_\infty^{(2,-1)_j,\tau}$, 
  where $(2,-1)_j$ is $(2,-1)$ in the embedding $\sigma_j:F_v\hookrightarrow E$ and $(1,0)$ elsewhere,
  as $R_\infty^{(2,-1)_j,\tau}$ is a domain (apply Proposition \ref{prop:def:ring} and \cite[Rk.\ 5.2.2]{gee-kisin} to $\rhobar=\rbar_v^\vee$ after a suitable twist).
  Therefore, $M_\infty(L/pL)$ is an $S/pS = \ovl{R_\infty^{(2,-1)_j,\tau}}$-module that is 
  (set-theoretically) supported on all of $\Spec(S/pS)$. By Corollary~\ref{cor:special-fibre-def-ring}, $S/pS$ is reduced. Hence $S/pS$ is the schematic support
  of $M_\infty(L/pL)$. By the argument in the last paragraph of the proof of Proposition~\ref{starringr} we deduce
  that $M_\infty(L)$ is free of rank $r$ over its schematic support $S$.
\end{proof}

\begin{thm}\label{HT2-1}
Let $j\in \{0,\dots,f-1\}$. Then $M_\infty(R'_{2,j})$ is free of rank $r$ over $R_\infty/\cap_{\tau}{\mathfrak p}_\tau$, where $\tau$ runs over the tame inertial types such that $\sigma_v\in \JH(\ovl{\sigma(\tau)})$ and ${\mathfrak p}_\tau$ is the prime ideal $\ker(R_\infty\twoheadrightarrow R_\infty^{(2,-1)_j,{\tau}})$, where $(2,-1)_j$ is $(2,-1)$ in the embedding $\sigma_j:F_v\hookrightarrow E$ and $(1,0)$ elsewhere.
\end{thm}
\begin{proof}
By Lemma \ref{easyglueingi} the $R_\infty$-module $M_\infty(R'_{2,j})$ is generated by $r$ elements, i.e.\ there is a surjection $f:S^r\twoheadrightarrow M_\infty(R'_{2,j})$, where $S\defeq R_\infty/{\rm Ann}_{R_\infty}(M_\infty(R'_{2,j}))$.

Note that $R'_{2,j}[1/p]=R_{2,j}[1/p]$ over $E$ is the direct sum of the representations $V(\alpha_j)_E \otimes_{E} \sigma(\tau)$ for all the tame inertial types $\tau$ such that $\sigma_v\in \JH(\ovl{\sigma(\tau)})$, and each such $\sigma(\tau)$ occurs only once. In particular, it is as in Corollary \ref{penible}, where for all $i$ we have $n_i=1$.
Arguing as in the last sentence of the proof of Proposition \ref{HT10}, it follows from (\ref{generic}) and the fact that all the rings $R_\infty^{(2,-1)_j,{\tau}}$ for $\tau$ such that $\sigma_v\in \JH(\ovl{\sigma(\tau)})$ are domains (apply Proposition \ref{prop:def:ring} and \cite[Rk.\ 5.2.2]{gee-kisin} to $\rhobar=\rbar_v^\vee$ after a suitable twist) that $S= R_\infty/\cap_{\tau}{\mathfrak p}_\tau$ for ${\mathfrak p}_\tau$ as in the statement. 

By Proposition~\ref{prop:freeness-loc-alg-type}, for each type $\tau$ as in the previous paragraph the module $M_\infty(V(\alpha_j) \otimes_{W(\F)} \sigma(\tau)^0)[1/p]$ is free of rank $r$ over $R_\infty^{(2,-1)_j,\tau}[1/p]$. Thus by \eqref{generic} and \eqref{chinese} the $S[1/p]$-module $M_\infty(R'_{2,j})[1/p]$ is locally free of rank $r$, i.e.\ the localization of $M_\infty(R'_{2,j})[1/p]$ at each prime ideal of $S[1/p]$ is free of rank $r$. Hence (using again \cite[Thm.\ 2.4]{Ma}), we see that $(\ker(f)[1/p])_{\mathfrak p}=0$ for all prime ideals $\mathfrak p$ of $S[1/p]$, which implies $\ker(f)[1/p]=0$, and hence $\ker(f)=0$ since $S$ has no $p$-torsion. This finishes the proof.
\end{proof}

Set $L_{-1}\defeq \widetilde P_{\sigma_v}$ and for $j\in \{0,\dots, f-1\}$ define a $K$-stable $W(\F)$-lattice $L_j$ in
$$\widetilde P_{\sigma_v}[1/p]\oplus \big(\bigoplus_{j'=0}^jV(\alpha_j)\otimes_{W(\F)}\widetilde P_{\sigma_v}\big)[1/p]=\widetilde P_{\sigma_v}[1/p]\oplus \big(\bigoplus_{j'=0}^jR_{2,j'}\big)[1/p]$$
by induction by
\begin{equation}\label{produitfibre}
L_j \defeq L_{j-1}\times_{P_{\sigma_v}} R'_{2,j},
\end{equation}
or equivalently
\begin{multline*}
L_j = \{(x_1,(x_{2,j'})_{0\leq j'\leq j})\in \widetilde P_{\sigma_v} \oplus \big(\bigoplus_{j'=0}^{j} R_{2,j'}\big) : (x_{2,j'}\bmod pR_{2,j'})=(x_1\bmod p\widetilde P_{\sigma_v})\\
{\rm in}\ P_{\sigma_v}\hookrightarrow R_{2,j'}/pR_{2,j'}\ \forall\ j'\in \{0,\dots,j\}\}.
\end{multline*}
Note that $L_{f-1}=R$ (see \S\ref{sec:constructionlattice}).

Let $\tau$ be a tame inertial type such that $\sigma_v \in \JH(\overline{\sigma(\tau)})$.  Then $\sigma(\tau)$ is a quotient of $\widetilde{P}_{\sigma_v}[1/p]$, and the image of $\widetilde{P}_{\sigma_v}$ is a $W(\F)$-lattice in $\sigma(\tau)$ with cosocle $\sigma_v$ which we denote by $\sigma(\tau)^0$. Let \[T_{2,j}\defeq V(\alpha_j)\otimes_{W(\F)} \sigma(\tau)^0\] and let $T_{2,j}'\subset T_{2,j}$ be the sublattice constructed in  the second paragraph of the proof of Proposition \ref{prop:freeness-loc-alg-type}, which satisfies $\cosoc_KT_{2,j}'\cong\sigma_v$. 
Then by the proof of \emph{loc.\ cit.}, $T_{2,j}'$ is identified with the image of 
the composite morphism
\[R_{2,j}'\hookrightarrow R_{2,j}\twoheadrightarrow T_{2,j}.\]
In particular, we have $pT_{2,j}\subset T_{2,j}'$ (as $pR_{2,j}\subset R_{2,j}'$).  Set $Y_j\defeq T_{2,j}'/pT_{2,j}$, so $Y_j$ is a quotient of $P_{\sigma_v}$ and hence of $L_{j-1}$. 
For $0\leq j\leq f-1$, define
\[N_j\defeq L_{j-1}\times_{Y_j}T'_{2,j}.\]

\begin{lem}\label{lem:WinYj}
With the above notation, the surjection $T'_{2,j}/pT'_{2,j} \onto Y_j$ induces an isomorphism $M_{\infty}(T'_{2,j}/pT'_{2,j}) \cong M_{\infty}(Y_j)$.
\end{lem}
\begin{proof}
Note that the representations $T'_{2,j}/pT'_{2,j}$ and $Y_j$ are exactly the representations denoted by $L/pL$ and $W_2$ respectively in the proof of Proposition \ref{prop:freeness-loc-alg-type},  and that $M_{\infty}(L/pL) = M_{\infty}(W_2)$ follows
from $M_{\infty}(W_1)=0$, see the second paragraph of this proof. 
\end{proof}

For a smooth $K$-representation $V$ over $\F$ of finite dimension, we denote by $(\rad_K^i(V))_{i\geq 0}$ its radical filtration: $\rad_K^0(V)=V$ and inductively $\rad^i_K(V)=\rad_K(\rad_K^{i-1}(V))$ for $i\geq 1$.   
As remarked in the proof of Lemma \ref{lem:Loewy3}, taking $\rad_K^i(-)$ preserves surjective morphisms (see \cite[\S1, Prop.~4]{Alperin}).

\begin{lem}\label{lem:cosoc3}
The surjection $R_{2,j}'\twoheadrightarrow T_{2,j}'$ induces a surjection $L_j\twoheadrightarrow N_j$, which induces an isomorphism \[(L_j/pL_j)/\rad_K^3(L_j/pL_j)\cong (N_j/pN_j)/\rad_K^3(N_j/pN_j).\]
\end{lem} 

\begin{proof}
As seen above, we have $\ker(T_{2,j}'\twoheadrightarrow Y_j)=pT_{2,j}$,  which implies a short exact sequence 
\[0\ra T_{2,j}\overset{\times p}{\lra} N_j\ra L_{j-1}\ra0\]
and consequently
\[0\ra T_{2,j}/pT_{2,j}\overset{\times p}{\lra} N_j/pN_j\ra L_{j-1}/pL_{j-1}\ra0.\]
Since $\ker(R_{2,j}'\twoheadrightarrow P_{\sigma_v})=pR_{2,j}$, we have a similar exact sequence for $L_j$ which fits in the following commutative diagram
\begin{equation}\label{eq:LjNj}
  \begin{gathered}
    \xymatrix{0\ar[r]&R_{2,j}/pR_{2,j}\ar^{\times p}[r]\ar^{\gamma}[d]& L_{j}/pL_j  \ar^{\beta}[d]\ar[r]&L_{j-1}/pL_{j-1}\ar[r]\ar@{=}[d]&0\\
      0\ar[r]& T_{2,j}/pT_{2,j}\ar^{\times p}[r]&N_{j}/pN_j\ar[r]& L_{j-1}/pL_{j-1}\ar[r]&0.}
  \end{gathered}
\end{equation}
It is direct to check that the morphism $\gamma$ is identified with
\[F(\alpha_j)\otimes_{\F} (\widetilde P_{\sigma_v}/p\widetilde P_{\sigma_v}) \ra F(\alpha_j)\otimes_{\F} (\sigma(\tau)^0/p\sigma(\tau)^0)\]
and is induced from the quotient morphism $\widetilde{P}_{\sigma_v}\twoheadrightarrow \sigma(\tau)^0$. In particular, $\gamma$ is surjective, hence so is $\beta$ from which the first claim follows. 

To prove the second claim, %
it is enough to show $\ker(\beta)\subset \rad_K^3(L_j/pL_j)$. Observe that if $M$ is a quotient of $(\Proj_{K/Z_1}\sigma_v)/\mathfrak{m}_{K_1}^2(\Proj_{K/Z_1}\sigma_v)$ which admits $P_{\sigma_v}$ as a quotient, then  the induced morphism
\[M/\rad^i_K(M)\onto P_{\sigma_v}/\rad_K^i(P_{\sigma_v})\]
is an isomorphism for $i=1,2$. Indeed, this is clear for $i=1$, and  can be deduced using \cite[Lemma~2.10(ii)]{HuWang2} for $i=2$. Thus, noting that both $L_{j}/pL_{j}$ and $N_j/pN_j$ are quotients of $(\Proj_{K/Z_1}\sigma_v)/\mathfrak{m}_{K_1}^2(\Proj_{K/Z_1}\sigma_v)$ (using Corollary \ref{rpr}) and both admit $P_{\sigma_v}$ as a quotient, we get $\ker(\beta)=\ker(\rad_K^2(\beta))$, and hence it is enough to prove 
\[\ker(\rad_K^2(\beta))\subseteq \rad_K(\rad_K^2(L_j/pL_j)).\]
Since $L_{j-1}/pL_{j-1}$ also admits $P_{\sigma}$ as a quotient, we   again obtain from the observation  above a commutative diagram as in \eqref{eq:LjNj}, but with $L_j/pL_j$, $N_j/pN_j$ and $L_{j-1}/pL_{j-1}$ replaced by their $\rad_K^2(-)$ and the left column in (\ref{eq:LjNj}) unchanged, 
from which we obtain  $\ker(\rad_K^2(\beta))=\ker(\gamma)$.
Hence  it is enough to prove  $\ker(\gamma)\subseteq \rad_K(R_{2,j}/pR_{2,j})$,  equivalently   $\gamma$ induces an isomorphism on cosocles. 
But this follows from the proof of Lemma \ref{socle} (taking duals there).
\end{proof}

The reason for introducing $N_j$ is as follows.

\begin{prop}\label{prop:Lj-Nj-equivalence-HW}
For $j\in\{0,\dots,f-1\}$, the following statements are equivalent:
 
\begin{enumerate}
\item  $M_{\infty}(L_j)$ can be generated by $r$ elements over $R_{\infty}$;
\item  $M_{\infty}(N_j)$ can be generated by $r$ elements over $R_{\infty}$.  
\end{enumerate}
\end{prop}

\begin{proof}
Let $\pi$ be the  admissible smooth representation of $\GL_2(F_v)$ over $\F$ defined in \eqref{eq:pi-indef} or \eqref{eq:pi-def}. Then by~\eqref{eq:Minfty-pi} we see that (i) (resp.~(ii)) is equivalent  to saying that $\dim_{\F}\Hom_K(L_j,\pi)=r$ (resp.~$\dim_{\F}\Hom_K(N_j,\pi)=r$).   Moreover, since $N_j$ is a quotient of $L_j$, we clearly have (i)$\Rightarrow$(ii).  
 
 The proof of (ii)$\Rightarrow$(i) is motivated by that of \cite[Prop.~4.18]{HuWang2}.  Assume $\dim_{\F}\Hom_K(L_j,\pi)>r$. Then, since $\dim_{\F}\Hom_{K}(\sigma_v,\pi)=f$ by Proposition \ref{starringr}, there exists a nonzero morphism $h:L_j\ra \pi$ which does not factor through $\cosoc_KL_j=\sigma_v$. We choose $h$ such that $[\mathrm{Im}(h):\sigma_v]$ is minimal; denote by $Q$ the image of $h$. We must have $[Q/\soc_K(Q):\sigma_v]=1$, otherwise $Q$ contains a submodule $Q'$ with cosocle $\sigma_v$ and such that $[Q'/\soc_K(Q'):\sigma_v]=1$, so there exists a morphism $L_j\ra \pi $ with image $Q'$ by Proposition \ref{prop:projectivity} (applied with $Q=L_j$),   which contradicts the choice of $h$.  By the proof of Proposition \ref{HT10} we have   $\pi^{K_1}=D_0(\overline{r}_v^{\vee})^{\oplus r}$. It in particular implies   \[\soc_K(Q)\subseteq \soc_K(D_0(\overline{r}_v^{\vee})^{\oplus r})=\bigoplus_{\sigma\in W(\overline{r}_v^{\vee})}\sigma^{\oplus r}.\]
Note that $W(\overline{r}_v^{\vee})\subseteq \JH(\Proj_{K/K_1}\sigma_v)$ (see \cite[\S11]{BP}), so  $Q$ satisfies conditions (a), (b) in Lemma \ref{lem:Loewy3}. Thus,  by the (first) part of \emph{loc.~cit.} we have $\rad_K(Q)\subseteq Q^{K_1}\subseteq \pi^{K_1}$. Since $\soc_K(Q)=Q\cap \soc_K(\pi)$ (in particular $\JH(\soc_K(Q))\subseteq W(\overline{r}_v^{\vee})$ up to multiplicities), we also have $\rad_K(Q)/\soc_K(Q)\hookrightarrow \pi^{K_1}/\soc_K(\pi)$ which implies that $Q$ satisfies conditions (c) and (d) of Lemma \ref{lem:Loewy3}, and  
hence $Q$ has Loewy length $\leq 3$.   Lemma \ref{lem:cosoc3} then shows that  $h:L_j\ra Q$ factors through $N_j$, hence gives a contradiction to (ii).
\end{proof}

\begin{lem}\label{lem:hard-glueing}
  Suppose that $R$ is a commutative noetherian local ring.
  Suppose that $M_1$, $M_2$, $M$ are nonzero $R$-modules that are free of rank $r$ over their respective schematic support and that we are given surjections $M_i \onto M$ for $i = 1, 2$.
Then $\Ann_R(M_1 \times_M M_2) = \Ann_R(M_1) \cap \Ann_R(M_2)$.
Moreover, the following are equivalent:
  \begin{enumerate}
  \item 
  \label{it:hard-glueing:1}
  $M_1 \times_M M_2$ is free of rank $r$ over its schematic support;
  \item 
  \label{it:hard-glueing:2}
  $\Ann_R(M) = \Ann_R(M_1) + \Ann_R(M_2)$;
  \item 
  \label{it:hard-glueing:3}
  $\Ann_R(M) \subset \Ann_R(M_1) + \Ann_R(M_2)$.
  \end{enumerate}
\end{lem}

\begin{proof}
The first assertion is clear, since the $M_i$ surject onto $M$.
We now prove the equivalence between \ref{it:hard-glueing:1}, \ref{it:hard-glueing:2} and \ref{it:hard-glueing:3}.
By assumption we can write $M_i = (R/I_i)^{\oplus r}$ and $M = (R/I)^{\oplus r}$ for (proper) ideals $I_i \subset I$.
  Without loss of generality we may assume that the given surjections are the natural maps $(R/I_i)^{\oplus r} \onto (R/I)^{\oplus r}$.
  Then $M_1 \times_M M_2 \cong (R/I_1 \times_{R/I} R/I_2)^{\oplus r}$ and by Nakayama we are reduced to the case $r = 1$, which is \cite[Lemma 8.11]{HuWang2}.
\end{proof}

From now on, we choose the tame inertial type $\tau$ in the discussion above such that $W(\overline{r}_v^{\vee})\subset \JH(\overline{\sigma(\tau)})$; this is always possible by \cite[Prop.~3.5.2]{EGS}. Since $\overline{r}_v^{\vee}$ is assumed to be semisimple,  this forces $W(\overline{r}_v^{\vee})= \JH(\overline{\sigma(\tau)})$ and $\tau$ is uniquely determined.
We will denote it by $\tau_0$ in what follows.  

\begin{thm}\label{HT102-1}
Let $j\in \{-1,\dots,f-1\}$. Then $M_\infty(L_j)$ is free of rank $r$ over $R_\infty/\cap_{\lambda,\tau}{\mathfrak p}_{\lambda,\tau}$, where ${\mathfrak p}_{\lambda,\tau}$ is the prime ideal $\ker(R_\infty\twoheadrightarrow R_\infty^{\lambda,{\tau}})$ with $\tau$ running over the tame inertial types such that $\sigma_v\in \JH(\ovl{\sigma(\tau)})$ and $\lambda = (\lambda_{j'})_{0\leq j'\leq f-1}$ running over the Hodge--Tate weights such that $\lambda_{j'}\in \{(1,0),(2,-1)\}$ if $0\leq j'\leq j$ and $\lambda_{j'} = (1,0)$ if $j+1\leq j'\leq f-1$.
\end{thm}
\begin{proof}
Twisting all the Galois deformations by $\varepsilon$, we can replace $\rbar_v^\vee$ by $\rbar_v^\vee(1)$, $\{(1,0),(2,-1)\}$ by $\{(2,1),(3,0)\}$ and $\sigma_v\in \JH(\ovl{\sigma(\tau)})$ by $\sigma_v\otimes (N_{k/\Fp}\circ \det^{-1})\in \JH(\ovl{\sigma(\tau)})$ (all the deformations now have determinant $\varepsilon^3\psi_v^{-1}$). Note first that all the rings $R_\infty^{\lambda,{\tau}}$ are domains by Proposition \ref{prop:def:ring} (and \cite[Rk.\ 5.2.2]{gee-kisin}) applied to a suitable twist of $\rbar_v^\vee$ to get $\rhobar=\rbar_v^\vee$ as in \S\ref{sec:setup}. The proof is by induction on $j\geq -1$. If $j=-1$, this is Proposition \ref{HT10}. Assume the statement is true for $M_\infty(L_{j-1})$ and let us prove it for $M_\infty(L_j)$. 

We first prove that the $R_\infty$-module $M_\infty(N_j)$ can be generated by $r$ elements.
From the exactness of $M_\infty$ and (\ref{produitfibre}) we deduce
$$M_\infty(N_j)=M_\infty(L_{j-1})\times_{M_\infty(Y_j)}M_\infty(T'_{2,j}).$$
Note that the maps $M_\infty(L_{j-1})\rightarrow M_\infty(Y_j)$, $M_\infty(T'_{2,j})\rightarrow M_\infty(Y_j)$ are surjective.
These three modules are free of rank $r$ over their schematic supports by induction hypothesis, Proposition~\ref{prop:freeness-loc-alg-type}, 
and Lemma~\ref{lem:WinYj}.
By Lemma~\ref{lem:hard-glueing} it is enough to prove
\begin{equation}\label{sumcontainsp}
{\rm Ann}_{R_\infty}(M_\infty(Y_j))\subseteq {\rm Ann}_{R_\infty}(M_\infty(L_{j-1})) + {\rm Ann}_{R_\infty}(M_\infty(T'_{2,j})).
\end{equation}
By Lemma \ref{lem:WinYj} we have $M_{\infty}(Y_j)=M_{\infty}(T_{2,j}'/pT_{2,j}')$, so
\[{\rm Ann}_{R_{\infty}}(M_{\infty}(Y_j))={\rm Ann}_{R_\infty}(M_\infty(T'_{2,j}/pT_{2,j}'))=(p)+{\rm Ann}_{R_\infty}(M_\infty(T'_{2,j})),\]
where the second equality holds because  $M_{\infty}(T_{2,j}')$ is free of rank $r$ over its schematic support. Hence, to prove  \eqref{sumcontainsp}   it is enough to prove
\begin{equation}\label{sumcontainsp2} p\in {\rm Ann}_{R_{\infty}}(M_{\infty}(L_{j-1}))+{\rm Ann}_{R_{\infty}}(M_{\infty}(T_{2,j}')).\end{equation} 

Consider the ring
$$R_\infty^{\leq (3,0),\sigma_v}\defeq R_\infty \otimes_{R_{\rbar_v^\vee}} R_{\rbar_v^\vee}^{\leq (3,0),\sigma_v}\cong R_\infty/\cap_{\lambda,\tau}{\mathfrak p}_{\lambda,\tau},$$
where $R_{\rbar_v^\vee}^{\leq (3,0),\sigma_v}$ is as in Proposition \ref{prop:multitype-def-ring} and where ${\mathfrak p}_{\lambda,\tau}=\ker(R_\infty\twoheadrightarrow R_\infty^{\lambda,{\tau}})$ with $\tau$ running over the tame inertial types such that $\sigma_v\otimes (N_{k/\Fp}\circ \det^{-1})\in \JH(\ovl{\sigma(\tau)})$ and $\lambda = (\lambda_{j'})_{0\leq j'\leq f-1}$ running over $\{(2,1),(3,0)\}^f$. By Proposition \ref{prop:multitype-def-ring} and (\ref{rinfini}), and increasing $q$ if necessary, we have for some integer $h\geq 1$ and a certain explicit ring $\oS = (\widehat{\bigotimes}_{\cO, 0\leq j\leq f-1}S^{(j)})/J$ that
$$R_\infty^{\leq (3,0),\sigma_v}\cong \oS\bbra{X_1,\dots,X_h}$$
(using again \cite[Rk.\ 5.2.2]{gee-kisin} and Lemma \ref{lm:hamann}, as we have conditions on the determinant here). 
For each $\lambda \in \{(2,1),(3,0)\}^f$ and $\un k \in \{1,2\}^f$ in Proposition \ref{prop:multitype-def-ring} an ``explicit'' prime ideal of $\oS$ is defined 
that we denote here simply by $\fp^\lambda_{\un k} = \sum_j \fp^{(j),\lambda_j}_{\un k}$ and that we consider as an ideal of $R_\infty^{\leq (3,0),\sigma_v}$ via $\oS \hookrightarrow R_\infty^{\leq (3,0),\sigma_v}$.
In other words, the ideals $\fp^{(j),\lambda_j}_{\tld{w}}$ from Proposition \ref{prop:multitype-def-ring} are relabeled as $\fp^{(j),\lambda_j}_{i(\tld{w})}$ and any value of $i(\tld w)_j$ equal to 3 is changed to 1 here, to simplify notation (see the beginning of \S\ref{sec:deformation-rings2} for the notation).
Moreover there is a bijection
\begin{equation*}
\iota: \{\tau : \sigma_v\otimes (N_{k/\Fp}\circ {\det})\in \JH(\ovl{\sigma(\tau)})\}\buildrel\sim\over\longrightarrow \{1,2\}^f
\end{equation*}
such that ${\mathfrak p}_{\lambda,\tau}= \fp^\lambda_{\iota(\tau)}$. 
From Lemma~\ref{lem:lowest-hodge-type} we also have
prime ideals of $S^{(j)}$ that we relabel here as $\fq_1^{(j),(2,1)}$, $\fq_2^{(j),(2,1)}$ such that $\fq_{k_j}^{(j),(2,1)} \subset \fp^\lambda_{\un k}$
whenever $\lambda_j = (2,1)$ and such that $\sum_j \fq_{k_j}^{(j),(2,1)} = \fp^{\un {(2,1)}}_{\un k}$ for all $\un k\in \{1,2\}^f$.

We note that by Lemma~\ref{lem:inter} we have $\tau_0 \cong \tau_{\tld w}$, where $\tld w_j = \mathfrak{w}t_{(2,1)}$ for each $0\leq j\leq f-1$,
so ${\mathfrak p}_{\lambda,\tau_0} = \fp^\lambda_{\un 2}$.
Then by Proposition~\ref{prop:freeness-loc-alg-type} and Theorem \ref{HT2-1} we deduce
\begin{eqnarray*}
{\rm Ann}_{R_\infty}(M_\infty(T'_{2,j})) &=& \fp^{\lambda(j)}_{\un 2},\\
{\rm Ann}_{R_\infty}(M_\infty(R'_{2,j'})) &=& \bigcap_{\un k} \fp^{\lambda(j')}_{\un k},
\end{eqnarray*}
where $\lambda(j')_{j''} \defeq (2,1)$ if $j'' \ne j'$ and $\lambda(j')_{j'} \defeq (3,0)$.
From the definition of $L_{j-1}$ as an iterated fiber product we have using the first part of Lemma~\ref{lem:hard-glueing} and Proposition~\ref{HT10} that
\begin{align*}
{\rm Ann}_{R_\infty}(M_\infty(L_{j-1}))&= {\rm Ann}_{R_\infty}(M_\infty(\widetilde P_{\sigma_v})) \cap \bigg(\bigcap_{0\leq j'\leq j-1}{\rm Ann}_{R_\infty}(M_\infty(R'_{2,j'}))\bigg)\\
&= \bigcap_{\un k} \fp^{\un {(2,1)}}_{\un k} \cap \bigcap_{0\leq j'\leq j-1} \bigcap_{\un k} \fp^{\lambda(j')}_{\un k}.
\end{align*}
By above %
we get that $\fq^{(j),(2,1)}_1 \cap \fq^{(j),(2,1)}_2 \subset {\rm Ann}_{R_\infty}(M_\infty(L_{j-1}))$ (note that $\lambda(j')_j=(2,1)$ for $0\leq j'\leq j-1$) and $\fp_{\un 2}^{(j),(3,0)} \subset {\rm Ann}_{R_\infty}(M_\infty(T'_{2,j}))$.
Hence to prove~\eqref{sumcontainsp2} it is enough to prove that $p\in \fq^{(j),(2,1)}_1 \cap \fq^{(j),(2,1)}_2 + \fp_{\un 2}^{(j),(3,0)}$, which is a special case of Proposition \ref{prop:p:in:inter}.

We have shown that $M_\infty(N_j)$ can be generated by $r$ elements, so the same is true for $M_\infty(L_j)$ by Proposition~\ref{prop:Lj-Nj-equivalence-HW}.
Let $S = R_\infty/{\rm Ann}_{R_\infty}(M_\infty(L_{j}))$.
Now we can argue just as in part (ii) of the proof of Proposition~\ref{HT10} to see first that $M_\infty(L_j)[1/p]$ is free of rank $r$ over $S[1/p]$ and then deduce
that any surjection $S^r \onto M_\infty(L_j)$ has to be an isomorphism.
This completes the proof.
\end{proof}

\begin{cor}\label{finally!}
The module $M_\infty(R)$ is free of rank $r$ over $R_\infty/\cap_{\lambda,\tau}{\mathfrak p}_{\lambda,\tau}$, where ${\mathfrak p}_{\lambda,\tau}$ is the prime ideal $\ker(R_\infty\twoheadrightarrow R_\infty^{\lambda,{\tau}})$ with $\tau$ running over the tame inertial types such that $\sigma_v\in \JH(\ovl{\sigma(\tau)})$ and $\lambda = (\lambda_{j})_{0\leq j\leq f-1}$ running over the Hodge--Tate weights such that $\lambda_{j}\in \{(1,0),(2,-1)\}$ for all $j$. In particular, $\dim_{\F}(M_\infty(R)/{\mathfrak m}_\infty) =r$.
\end{cor}

Recall that we have defined the $K$-representation $(\Proj_{K/Z_1}\sigma_v)/\mathfrak{m}_{K_1}^2$ with cosocle $\sigma_v$ (see e.g.\ \S\ref{sec:constructionlattice}). From Corollary \ref{finally!}, Proposition \ref{starringr}\ref{it:star1} and the isomorphism $R/pR\cong \! (\Proj_{K/Z_1}\!\sigma_v)/\mathfrak{m}_{K_1}^2$ of Corollary \ref{rpr}, we deduce the following result.

\begin{thm}\label{mainpatching}
The surjection
$$(\Proj_{K/Z_1}\sigma_v)/\mathfrak{m}_{K_1}^2\twoheadrightarrow \sigma_v$$
induces an isomorphism of \emph{(}nonzero finite-dimensional\emph{)} $\F$-vector spaces
$$M_\infty\big((\Proj_{K/Z_1}\sigma_v)/\mathfrak{m}_{K_1}^2\big)/{\mathfrak m}_\infty \buildrel\sim\over\longrightarrow M_\infty(\sigma_v)/{\mathfrak m}_\infty.$$
\end{thm}

\begin{rem}
The exactness of the functor $M_\infty$ shows that the isomorphism in Theorem \ref{mainpatching} is of course totally wrong without quotienting by ${\mathfrak m}_\infty$.
\end{rem}

\subsection{Gelfand--Kirillov dimensions}\label{sec:GKsection}

We prove our main global results.

We keep all our previous notation. We recall our assumptions: $F$ is a totally real number field unramified at $p$, $D$ is a quaternion algebra of center $F$ split above $p$ and at not more than one infinite place, $v$ is a fixed place of $F$ above $p$ and $\rbar:G_F\ra \GL_2(\F)$ is a continuous representation satisfying the following conditions: $\rbar |_{G_{F(\!\sqrt[p]{1})}}$ is absolutely irreducible, $\rbar_w$ is generic in the sense of \cite[Def.\ 11.7]{BP} if $w|p$, $w\ne v$, $\rbar_v$ is semisimple generic in the sense of \S\ref{patching} {(the latter implies $p>23$)} and $R_{\rbar_w}$ is formally smooth over $W(\F)$ if $w\in (S_D \cup S_{\rbar})\backslash S_p$.

We choose $w_1$, $S$ and $U=\prod U_w$ as in \S\ref{patching}, and consider   the  admissible smooth representation $\pi$ of $\GL_2(F_v)$  defined in \eqref{eq:pi-indef} or \eqref{eq:pi-def}. 
Recall we defined the Gelfand--Kirillov dimension $\dim_{\GL_2(F_v)}(\pi)$ in \S\ref{sec:kirillov}.

\begin{thm}\label{mainpatching2}
We have $\dim_{\GL_2(F_v)}(\pi) = [F_v:\Qp]$.
\end{thm}
\begin{proof}
(i) By \cite[\S 5.5]{gee-kisin} $\pi$ satisfies assumption (i) in Theorem \ref{thm:GKdim-criterion} (for $\rhobar=\rbar_v^\vee$). It follows from \eqref{eq:Minfty-pi} and Theorem \ref{mainpatching} (choosing $M_\infty=M_\infty^{\sigma_p^v}$ as in \S\ref{patching} for $\sigma_p^v$ as in (\ref{horsv}) with $\JH(\ovl{\sigma(\tau_w^\vee)})\cap W(\rbar_w^\vee)=\{\sigma_w\}$) that for all $\sigma_v\in W(\rbar_v^\vee)$ we have
$$[\pi[\mathfrak{m}_{K_1}^2] : \sigma_v]=[\soc_K(\pi) : \sigma_v],$$
so that $\pi$ satisfies also assumption (ii) in Theorem \ref{thm:GKdim-criterion}. 
Finally, we prove that $\JH(\pi^{I_1})=\JH(D_1(\rbar_v^\vee))$ (up to multiplicity), and so by Lemma \ref{lem:connected} $\pi$ satisfies assumption (iii) in Theorem \ref{thm:GKdim-criterion}. We only give the proof in the definite case, the indefinite case can be treated similarly (see e.g.\ (\ref{injindefinite}) below). The $K$-equivariant embedding $\bigoplus_{\sigma_v\in W(\rbar_v^\vee)}\sigma_v^{m_{\sigma_v}}\hookrightarrow \pi$, where $m_{\sigma_v}=[\soc_K(\pi):\sigma_v]$, induces a $K\times (U^v/V^v)$-equivariant morphism 
\[\Big(\bigoplus_{\sigma_v\in W(\rbar_v^\vee)}\sigma_v^{m_{\sigma_v}}\Big)\otimes_{\F}\Big(\bigotimes_{w\in S_p\backslash\{v\}}\sigma_w\Big)\ra \varinjlim_{V_v}S(V^vV_v,\F)[{\mathfrak m}],\]
which is injective because $\bigotimes_{w\in S_p\backslash\{v\}}\sigma_w$ is irreducible. 
By \cite[Lemma 9.2]{breuil-buzzati}, the last embedding extends to an embedding 
\[\Big(\bigoplus_{\sigma_v\in W(\rbar_v^\vee)}D_{0,\sigma_v}(\brho)^{m_{\sigma_v}}\Big)\otimes_{\F} \Big(\bigotimes_{w\in S_p\backslash\{v\}}\sigma_w\Big)\hookrightarrow \varinjlim_{V_v}S(V^vV_v,\F)[{\mathfrak m}]\]
and gives in turn an embedding
\[\bigoplus_{\sigma_v\in W(\rbar_v^\vee)}D_{0,\sigma_v}(\brho)^{m_{\sigma_v}}\hookrightarrow \pi.\]
In particular, we have $\JH(D_1(\rbar_v^\vee))\subseteq \JH(\pi^{I_1})$. But using \cite[Lemma 14.1]{BP}, we actually have $\JH(D_1(\rbar_v^\vee))=\JH(\pi^{I_1})$ (up to multiplicity), and so $\pi$ satisfies assumption (iii) in Theorem \ref{thm:GKdim-criterion}. 
We can thus apply Theorem \ref{thm:GKdim-criterion} which gives $\dim_{\GL_2(F_v)}(\pi) \leq [F_v:\Qp]$.

(ii) By the arguments of \cite[\S 6]{DoLe}, replacing $K^v$ in
\cite[\S 6.1]{DoLe} by $U^v$, the representation $V=\bigotimes_{w\in S,
  w\ne v}V_w$ of $K^v$ in {\it loc.~cit.} by the representation
$\sigma_p^v$ of $U^v$ in (\ref{horsv}) and forgetting the Hecke
operators $T_w$ at places $w\in S'$ (since we do not care about
multiplicity $1$), the same patching process as in \cite[\S 6.2]{DoLe}
(which is a variant/special case of the main construction of
\cite{CEGGPS} and \cite[\S9]{ScholzeLT}) produces a ``big'' patched module ${\mathbb M}_\infty$ over $R_\infty\bbra{\GL_2({\mathcal O}_{F_v})}$ (with a compatible action of $\GL_2(F_v)$) which is finitely generated free over the local ring $S_\infty\bbra{K_1/Z_1}$, where $S_\infty\defeq W(\F)\bbra{x_1,\dots, x_{4|S|+q-1}}$ (see (\ref{rinfini}) for $q$). Moreover we have ${\mathbb M}_\infty/{\mathfrak m}_\infty \cong \pi^\vee$ and for any continuous representation $\sigma_v$ of $\GL_2({\mathcal O}_{F_{v}})$ over a finite type $W(\F)$-module with central character $\psi|_{I_{F_v}}^{-1}$ we have $M_\infty(\sigma_v)=\Hom_{W(\F)\bbra{\GL_2({\mathcal O}_{F_v})}}^{\rm cont}({\mathbb M}_\infty,\sigma_v^\vee)^\vee$, where $(-)^\vee\defeq \Hom_{W(\F)}^{\rm cont}(-,E/W(\F))$ and ${\mathbb M}_\infty$ is endowed with its natural profinite topology. It follows from \cite[Lemma\ A.16]{GN}, Lemma \ref{lem:GvsG/Z} and (\ref{rinfini}) that we have (where the grade $j_{A}$ is as in \S\ref{sec:kirillov})
\begin{multline}\label{uppergrade}
j_{R_\infty\bbra{K_1/Z_1}}({\mathbb M}_\infty)\geq j_{\F\bbra{K_1/Z_1}}({\mathbb M}_\infty/{\mathfrak m}_\infty)=\dim(K_1/Z_1)-\dim_{\GL_2(F_v)}(\pi)\\
=3[F_v:\Qp]-\dim_{\GL_2(F_v)}(\pi).
\end{multline}
Since ${\mathbb M}_\infty$ is free of finite type over $S_\infty\bbra{K_1/Z_1}$, we have $j_{S_\infty\bbra{K_1/Z_1}}({\mathbb M}_\infty)=0$. It then follows from \cite[Lemma\ A.19]{GN} (together with \cite[Def.\ A.2]{GN} and \cite[Prop.\ A.4(1)]{GN}) that
\begin{equation}\label{rsinfini}
j_{R_\infty\bbra{K_1/Z_1}}({\mathbb M}_\infty)=\big(\dim(R_\infty) + \dim(K_1/Z_1)\big)-\big(\dim(S_\infty) + \dim(K_1/Z_1)\big)
=2[F_v:\Qp],\end{equation}
where the last equality follows from (\ref{rinfini}) and the definition of $S_\infty$. Combining (\ref{uppergrade}) and (\ref{rsinfini}), we deduce $2[F_v:\Qp]\geq 3[F_v:\Qp]-\dim_{\GL_2(F_v)}(\pi)$, i.e.\ $\dim_{\GL_2(F_v)}(\pi)\geq [F_v:\Qp]$, which finishes the proof. 
\end{proof}

Recall that for any Serre weight $\sigma_v$ we have defined in \S\ref{sec:smooth:rep} the injective envelope $\Inj_{K/Z_1}\sigma_v$ with socle $\sigma_v$.

\begin{thm}\label{largest}
There is an integer $r\geq 1$ such that $\pi[\mathfrak{m}_{K_1}^2]\cong \big(\bigoplus_{\sigma_v\in W(\rbar_v^\vee)}\widetilde D_{\sigma_v}\big)^{\oplus r}$, where $\widetilde D_{\sigma_v}$ is the largest subrepresentation of $(\Inj_{K/Z_1}\sigma_v)[\mathfrak{m}_{K_1}^2]$ containing $\sigma_v$ with multiplicity $1$ \emph{(}= its socle\emph{)} and no other Serre weights of $W(\rbar_v^\vee)$. In particular, each irreducible constituent of $\pi[\mathfrak{m}_{K_1}^2]$ has multiplicity $r$.
\end{thm}
\begin{proof}
The existence of $\widetilde D_{\sigma_v}$ is proven in Corollary
\ref{cor:J-fil}\ref{it:J-fil-1}. It follows from its construction in
\cite[\S 6.2]{DoLe} and \cite{CEGGPS} that ${\mathbb M}_\infty$ (see part (ii) of the proof of Theorem \ref{mainpatching2}) is projective of finite type over $S_\infty\bbra{K}_Z$, where $S_\infty\bbra{K}_Z$ is the largest quotient of $S_\infty\bbra{K}$ on which the center of $K=\GL_2({\mathcal O}_{F_v})$ acts by $\psi|_{I_{F_v}}$. In particular, ${\mathbb M}_\infty/(p,x_1,\dots,x_{4|S|+q-1})$ is finite projective over $\F\bbra{K}_Z$. Noting that \[\Hom_{W(\F)\bbra{K}}^{\rm cont}({\mathbb M}_\infty/(p,x_1,\dots,x_{4|S|+q-1}),\sigma_v^\vee)^\vee\cong M_\infty(\sigma_v)/(p,x_1,\dots,x_{4|S|+q-1}) \]
which is nonzero if and only if $\sigma_v\in W(\rbar_v^\vee)$, we deduce
$${\mathbb M}_\infty/(p,x_1,\dots,x_{4|S|+q-1})\cong \bigoplus_{\sigma_v\in W(\rbar_v^\vee)}(\Proj_{K/Z_1}\sigma_v^\vee)^{\oplus m_{\sigma_v}}$$
for some integers $m_{\sigma_v}\geq 1$ (in fact $m_{\sigma_v}\geq r$, where $r\geq 1$ is as in Proposition \ref{starringr}\ref{it:star1}). This implies by the definition of $\widetilde D_{\sigma_v}$
$$\Hom_{\F\bbra{K}}^{\rm cont}\Big({\mathbb M}_\infty/(p,x_1,\dots,x_{4|S|+q-1}),\widetilde D_{\sigma_v}^\vee\Big)\buildrel\sim\over\longrightarrow \Hom_{\F\bbra{K}}^{\rm cont}\Big({\mathbb M}_\infty/(p,x_1,\dots,x_{4|S|+q-1}),\sigma_v^\vee\Big)$$
and hence taking on both sides the subspaces where ${\mathfrak m}_\infty$ acts by $0$ (${\mathfrak m}_\infty$ acts through the action of $R_\infty$ on ${\mathbb M}_\infty/(p,x_1,\dots,x_{4|S|+q-1})$) we get
$$\Hom_{\F\bbra{K}}^{\rm cont}\Big({\mathbb M}_\infty/{\mathfrak m}_\infty,\widetilde D_{\sigma_v}^\vee\Big)\buildrel\sim\over\longrightarrow \Hom_{\F\bbra{K}}^{\rm cont}\Big({\mathbb M}_\infty/{\mathfrak m}_\infty,\sigma_v^\vee\Big).$$
Using ${\mathbb M}_\infty/{\mathfrak m}_\infty\cong \pi^\vee$ this last isomorphism can be rewritten 
$$\Hom_K(\widetilde D_{\sigma_v},\pi)=\Hom_K(\widetilde D_{\sigma_v},\pi[\mathfrak{m}_{K_1}^2])\buildrel\sim\over\longrightarrow\Hom_K(\sigma_v,\pi)=\Hom_K(\sigma_v,\soc_K\pi).$$
Since $\soc_K\pi=(\bigoplus_{\sigma_v\in W(\rbar_v^\vee)}\sigma_v)^{\oplus r}$ by Proposition \ref{starringr}\ref{it:star1}, we deduce an inclusion
\begin{equation}\label{injindefinite}
\big(\bigoplus_{\sigma_v\in W(\rbar_v^\vee)}\widetilde D_{\sigma_v}\big)^{\oplus r}\subseteq \pi[\mathfrak{m}_{K_1}^2].
\end{equation}
But it follows from Corollary \ref{cor:J-fil}\ref{it:J-fil-1} (using $\pi[\mathfrak{m}_{K_1}^2]\subseteq (\Inj_{K/Z_1} \tau_v)[\mathfrak{m}_{K_1}^2]$ for $\tau_v\defeq \bigoplus_{\sigma_v\in W(\rbar_v^\vee)}\sigma_v^{\oplus r}$) and Theorem \ref{mainpatching} that $\pi[\mathfrak{m}_{K_1}^2]$ cannot be (strictly) larger, whence the isomorphism of the statement. The last sentence in the statement then follows from Corollary \ref{cor:J-fil}\ref{it:J-fil-2} and \ref{it:J-fil-3}.
\end{proof}

\begin{thm}\label{thm:flat_infty}
  The $R_\infty$-module ${\mathbb M}_\infty$ is faithfully flat.  
\end{thm}
\begin{proof}
  Since ${\mathbb M}_\infty$ is free of finite type over $S_\infty\bbra{K_1/Z_1}$, it follows from \cite[Cor.\ A.29]{GN} applied to $M={\mathbb M}_\infty$, $A=S_\infty\bbra{K_1/Z_1}$ and $B=R_\infty\bbra{K_1/Z_1}$ (using (\ref{rinfini})) that ${\mathbb M}_\infty$ is a Cohen--Macaulay $R_\infty\bbra{K_1/Z_1}$-module. By Theorem \ref{mainpatching2}, (\ref{uppergrade}), and (\ref{rsinfini}) we have
$$j_{R_\infty\bbra{K_1/Z_1}}({\mathbb M}_\infty)=j_{\F\bbra{K_1/Z_1}}({\mathbb M}_\infty/{\mathfrak m}_\infty)=2[F_v:\Qp],$$
and it then follows from \cite[Cor.\ A.30]{GN} (``Miracle Flatness'')
that ${\mathbb M}_\infty$ is flat over $R_\infty$. As $R_\infty$ is
a local ring and ${\mathbb M}_\infty/{\mathfrak m}_\infty\ne 0$, it follows that
$\mathbb{M}_\infty$ is faithfully flat over $R_\infty$.
\end{proof}

\begin{cor}\label{padiclanglands}
Let $x:R_\infty\rightarrow {\mathcal O}'$ be any homomorphism of local $W(\F)$-algebras, where ${\mathcal O}'$ is the ring of integers of a finite extension $E'$ of $E$, and set
$$V(x)\defeq \Hom_{{\mathcal O}'}^{\rm cont}\big({\mathbb M}_\infty\otimes_{R_\infty,x}{\mathcal O}',E'\big).$$
Then $V(x)$ is a {\rm nonzero} admissible unitary Banach representation of $\GL_2(F_v)$ over $E'$ with a $\GL_2(F_v)$-invariant unit ball \emph{(}given by $\Hom_{{\mathcal O}'}^{\rm cont}\big({\mathbb M}_\infty\otimes_{R_\infty,x}{\mathcal O}',{\mathcal O}'\big)$\emph{)} lifting $\pi\otimes_{\F}\F'$, where $\F'$ is the residue field of $\mathcal O'$.
\end{cor}
\begin{proof}
The fact that $V(x)$ is an admissible unitary Banach representation of
$\GL_2(F_v)$ follows from \cite[Prop.\ 2.13]{CEGGPS}. We need to prove
$V(x)\ne 0$ (note that we know ${\mathbb
  M}_\infty\otimes_{R_\infty,x}{\mathcal O}'\ne 0$, as ${\mathbb
  M}_\infty/{\mathfrak m}_\infty\ne 0$, but it could be $p$-power
torsion).  By Theorem \ref{thm:flat_infty}, the $R_\infty$-module
${\mathbb M}_\infty$ is flat, hence ${\mathbb M}_\infty\otimes_{R_\infty,x}{\mathcal O}'$ is flat over ${\mathcal O}'$ by base change, and the result easily follows by \cite[Thm.\ 1.2]{schneider-teitelbaum-IL}.
\end{proof}

\begin{rem}\label{localfactor}
Under slightly more general hypotheses on $\rbar$, one can prove Theorem \ref{mainpatching2}, Theorem \ref{largest} and Corollary \ref{padiclanglands} with $\pi$ replaced by the ``minimal local factor'' of \cite[\S 3.3]{BD} and \cite[\S 6.5]{EGS}. The strategy is completely similar using Theorem \ref{thm:GKdim-criterion}, the patching functor $M_\infty^{\rm min}$ of \cite[\S 6.5]{EGS} (and the ``big'' minimal patched module of \cite[\S 6]{DoLe}), and the variant of Corollary \ref{finally!} with $M_\infty^{\rm min}$ instead of $M_\infty=M_\infty^{\sigma_p^v}$, where we now have $r=1$. Details are left to the reader.
\end{rem}

\begin{cor}\label{mainmain}
For any compact open subgroup
$$V^v=\prod_{w\notin S_D\cup S_{\rbar}}\!\!\!({\mathcal O}_D)_{w}^\times\!\!\prod_{w\in (S_D\cup S_{\rbar})\backslash\{v\}}\!\!\!\!\!V_w\ \ \ \subseteq \ \ \ \prod_{w\ne v}({\mathcal O}_D)_{w}^\times$$
such that $V_w$ is a subgroup of $1+p\M_2({\mathcal O}_{F_w})$ for $w\in S_p\backslash \{v\}$ and such that $\pi\ne 0$, where
\begin{eqnarray*}
&&\pi\defeq \varinjlim_{V_v}\Hom_{G_F}\!\big(\rbar,H^1_{{\rm \acute et}}(X_{V^vV_v} \times_F \overline F, \F)\big)\ {\rm in\ the\ indefinite\ case},\\
&&\pi\defeq\varinjlim_{V_v}S(V^vV_v,\F)[{\mathfrak m}]\ {\rm in\ the\ definite\ case},
\end{eqnarray*}
we have $\dim_{\GL_2(F_v)}(\pi) = [F_v:\Qp]$.
\end{cor}
\begin{proof}
Note that the ideal $\mathfrak m$ in the definite case is as in Remark \ref{mEGS}\ref{it:EGS2} for $S$ big enough (the resulting eigenspace does not depend on $S$ by \cite[Lemma\ 4.6(a)]{BDJ}). We prove the indefinite case only, the definite case being similar. We can and do choose a place $w_1$ as in \S\ref{patching}.

(i) We first prove $\dim_{\GL_2(F_v)}(\pi) \leq [F_v:\Qp]$. Since the Gelfand--Kirillov dimension of a subspace is at most as big as the one of the space, it is enough to prove this upper bound for a smaller $V^v$. In particular, we can assume that $V_{w_1}$ is a subgroup of the group of matrices that are upper-triangular unipotent mod $w_1$ and that $V_w$ is a subgroup of $1+p\M_2({\mathcal O}_{F_w})$ which is normal in $\GL_2({\mathcal O}_{F_w})$ for $w\in S_p\backslash \{v\}$. Let $S\defeq S_D\cup S_{\rbar}$ and $U=\prod_wU_w$ with $U_w\defeq V_w$ if $w\notin S_p$ and $U_w\defeq ({\mathcal O}_D)_{w}^\times\cong \GL_2({\mathcal O}_{F_w})$ if $w\in S_p$, then $S$ and $U$ satisfy all the conditions in \S\ref{patching} and we have
\begin{equation}\label{ind}
\pi\cong \varinjlim_{V_v}\Hom_{U^v/V^v}\!\Big(\!\bigotimes _{w\in S_p\backslash\{v\}} \big(\Ind_{V_w}^{\GL_2({\mathcal O}_{F_w})}1\big)_Z, \Hom_{G_F}\!\big(\rbar, H^1_{{\rm \acute et}}(X_{V^vV_v} \times_F \overline F, \F)\big)\Big),
\end{equation}
where $(\Ind_{V_w}^{\GL_2({\mathcal O}_{F_w})}1)_{Z}$ is the maximal quotient of $\Ind_{V_w}^{\GL_2({\mathcal O}_{F_w})}1$ on which the center of $\GL_2({\mathcal O}_{F_w})$ acts by $\ovl\psi^{-1}|_{I_{F_w}}$. Writing each $(\Ind_{V_w}^{\GL_2({\mathcal O}_{F_w})}1)_Z$ as a successive extension of Serre weights for $\GL_2({\mathcal O}_{F_w})$, an obvious d\'evissage shows that it is enough to prove that for all Serre weights $(\sigma_w)_{w\in S_p\backslash \{v\}}$, we have
$$\dim_{\GL_2(F_v)}\bigg(\varinjlim_{V_v}\Hom_{U^v/V^v}\!\Big(\!\bigotimes _{w\in S_p\backslash\{v\}} \sigma_w, \Hom_{G_F}\!\big(\rbar, H^1_{{\rm \acute et}}(X_{V^vV_v} \times_F \overline F, \F)\big)\Big)\bigg)\leq [F_v:\Qp].$$
But this follows from (\ref{h1nonzero}) and Theorem \ref{mainpatching2}. In fact, using
$$\Hom_{U^v/V^v}\!\big(-, \Hom_{G_F}\!\big(\rbar, H^1_{{\rm \acute et}}(X_{V^vV_v} \times_F \overline F, \F)\big)\big)\cong \Hom_{G_F}\!\big(\rbar,\Hom_{U^v/V^v}\!\big(-, H^1_{{\rm \acute et}}(X_{V^vV_v} \times_F \overline F, \F)\big)\big)$$
together with
$$\Hom_{G_F}\!\big(\rbar,\Hom_{U^v/V^v}\!\big(-, H^1_{{\rm \acute et}}(X_{V^vV_v} \times_F \overline F, \F)\big)\big)\cong \Hom_{G_F}\!\big(\rbar,\Hom_{U^v/V^v}\!\big(-, H^1_{{\rm \acute et}}(X_{V^vV_v} \times_F \overline F, \F)_{\mathfrak m}\big)\big)$$
(for $\mathfrak m$ as in Remark \ref{mEGS}\ref{it:EGS2}) and the fact that $H^1_{{\rm \acute et}}(X_{V^vV_v} \times_F \overline F, \F)_{\mathfrak m}$ is an injective representation of $U^v/V^v$ over $\F$ (since $\mathfrak m$ is non-Eisenstein), we easily deduce that, in the above d\'evissage, $\pi$ as in (\ref{ind}) contains
$$\varinjlim_{V_v}\Hom_{U^v/V^v}\!\Big(\!\bigotimes _{w\in S_p\backslash\{v\}} \sigma_w, \Hom_{G_F}\!\big(\rbar, H^1_{{\rm \acute et}}(X_{V^vV_v} \times_F \overline F, \F)\big)\Big)$$
for at least one tuple $(\sigma_w)_{w\in S_p\backslash \{v\}}$ with $\sigma_w\in W(\rbar_w^\vee)$ for all $w\in S_p\backslash \{v\}$ (since $\pi\ne 0$). 
(We also use that $\Hom_{U^v/V^v}(\bigotimes_{w\in S_p\backslash \{v\}} \sigma_w, H^1_{{\rm \acute et}}(X_{V^vV_v} \times_F \overline F, \F)_{\mathfrak m}) \ne 0$ if and only if $\sigma_w \in W(\rbar_w^\vee$) for all $w$, by \cite[Lemma 4.10]{BDJ}.)
This implies $\dim_{\GL_2(F_v)}(\pi) = [F_v:\Qp]$ by Theorem \ref{mainpatching2} (for $\pi$ as in (\ref{ind})).

(ii) We now prove $\dim_{\GL_2(F_v)}(\pi) = [F_v:\Qp]$ for $\pi$ as in the statement. Set $V'^v=\prod_{w\ne v}V'_w$ with $V'_w=V_w$ if $w\ne w_1$ and $V'_{w_1}=$ subgroup of $({\mathcal O}_D)_{w_1}^\times$ of matrices that are upper-triangular unipotent mod $w_1$. By Ihara's Lemma at the place $w_1$, which is easy here thanks to all the assumptions on $w_1$, we have for sufficiently small $V_v$ that
$$\Hom_{G_F}\!\big(\rbar,H^1_{{\rm \acute et}}(X_{V^vV_v} \times_F \overline F, \F)\big)^{\oplus 2}\cong \Hom_{G_F}\!\big(\rbar,H^1_{{\rm \acute et}}(X_{V'^vV_v} \times_F \overline F, \F)\big)$$
and hence a $\GL_2(F_v)$-equivariant isomorphism
$$\pi^{\oplus 2}\cong \pi'\defeq \varinjlim_{V_v}\Hom_{G_F}\!\big(\rbar,H^1_{{\rm \acute et}}(X_{V'^vV_v} \times_F \overline F, \F)\big).$$
In particular, $\dim_{\GL_2(F_v)}(\pi)=\dim_{\GL_2(F_v)}(\pi')$. Replacing $V$ by $V'$, we can thus assume that $V_{w_1}$ is the subgroup of $({\mathcal O}_D)_{w_1}^\times$ of matrices that are upper-triangular unipotent mod $w_1$. It is enough to prove $\dim_{\GL_2(F_v)}(\pi) = [F_v:\Qp]$ when $V_w = 1+p\M_2({\mathcal O}_{F_w})$ for $w\in S_p\backslash \{v\}$ (as $\dim_{\GL_2(F_v)}(\pi)$ for the subgroup $V_w$ of $1+p\M_2({\mathcal O}_{F_w})$ can only grow, but is anyway bounded by $[F_v:\Qp]$ by (i)). But this follows from the last assertion in part (i) above.
\end{proof}

\begin{rem}\label{abitfurther}
If $V^v=\prod_{w\notin S}({\mathcal O}_D)_{w}^\times\prod_{w\in S\backslash\{v\}}\!V_w$ for some finite set $S$ containing $S_D\cup S_{\rbar}$ such that $R_{\rbar_w}$ is formally smooth for $w\in S\backslash S_p$, the same proof gives $\dim_{\GL_2(F_v)}(\pi) = [F_v:\Qp]$. Without assuming $V_w\subset 1+p\M_2({\mathcal O}_{F_w})$ for $w\in S_p\backslash \{v\}$, the above proof still gives the bound $\dim_{\GL_2(F_v)}(\pi) \leq [F_v:\Qp]$. 
\end{rem}

\subsection{Flatness for the dual of completed cohomology}\label{sec:platitude_Hecke} We give an
application to the flatness of the dual of completed cohomology.

In this section we assume moreover that $p$ is inert in $F$, so that $S_p=\{v\}$. Let $V^v$ be as in the final part of \S\ref{patching}, %
i.e.\ $V^v=\prod_{w\ne v}V_w=\prod_{w\ne v}U_w$ (as $p$ is inert, $V_w = U_w$ for all $w \ne v$).

For each compact open subgroup $V_v\subset 1+p\M_2(\mathcal{O}_{F_v})$
and for each $n\geq1$ we define the
$\psi^{-1}$-isotypic subspaces 
\[
  H^1_{\textrm{\'et}}(X_{V^vV_v}\times_F\overline{F},W(\F)/p^n)^{\psi^{-1}}
  \quad (\text{resp.~}S(V^vV_v,W(\F)/p^n)^{\psi^{-1}} \ \text{ in
    the definite case})\] 
for the action of the center
$(\mathbb{A}_F^\infty)^\times$ of $(D\otimes_F\mathbb{A}_F^\infty)^\times$, where $\psi^{-1}$ is viewed as a character
of $(\mathbb{A}_F^\infty)^{\times}$ via the global Artin map (sending uniformizers
to geometric Frobenius elements). Let
$\mathbb{T}(V^vV_v,W(\F)/p^n)^{\psi^{-1}}$ be the $W(\F)$-subalgebra of
$\End_{W(\F)}(H^1_{\textrm{\'et}}(X_{V^vV_v}\times_F\overline{F},W(\F)/p^n)^{\psi^{-1}})$
(respectively $\End_{W(\F)}(S(V^vV_v,W(\F)/p^n)^{\psi^{-1}})$) generated by
the endomorphisms $T_w$ and $S_w$ for $w\notin S\cup\set{w_1}$ and
$\mathbb{T}(V^vV_v,W(\F)/p^n)^{\psi^{-1}}_{\rbar}$ its localization at
the maximal ideal $\mathfrak{m}$ generated by the elements
$T_w-S_w\tr(\rbar(\mathrm{Frob}_w))$,
$\mathrm{Norm}(w)-S_w\det(\rbar(\mathrm{Frob}_w))$ for
$w\notin S\cup\set{w_1}$ (see Remark \ref{mEGS}\ref{it:EGS2}). Let
$\widehat{\mathbb{T}}(V^v)_{\rbar}^{\psi^{-1}}$ be the ``big
Hecke algebra''
\[
  \widehat{\mathbb{T}}(V^v)_{\rbar}^{\psi^{-1}}\defeq\varprojlim_{n,V_v}\mathbb{T}(V^vV_v,W(\F)/p^n)_{\rbar}^{\psi^{-1}}.\]
We define respectively
\begin{align*}
  \widehat{H}^1(V^v)^{\psi^{-1}}_{\rbar}&\defeq
  \text{\small{$\varprojlim_n\varinjlim_{V_v}\bigg(H^1(X_{V^vV_v}\times_F\overline{F},W(\F)/p^n)^{\psi^{-1}}\otimes_{\mathbb{T}(V^vV_v,W(\F)/p^n)^{\psi^{-1}}}\mathbb{T}(V^vV_v,W(\F)/p^n)_{\rbar}^{\psi^{-1}}\bigg)$}}
  \\ 
  \widehat{S}(V^v)^{\psi^{-1}}_{\rbar}&\defeq\text{\small{$\varprojlim_n\varinjlim_{V_v}\bigg(S(V^vV_v,W(\F)/p^n)^{\psi^{-1}}\otimes_{\mathbb{T}(V^vV_v,W(\F)/p^n)^{\psi^{-1}}}\mathbb{T}(V^vV_v,W(\F)/p^n)_{\rbar}^{\psi^{-1}}\bigg)$}}
\end{align*}
so that $\widehat{H}^1(V^v)^{\psi^{-1}}_{\rbar}$ and its dual
$\Hom_{W(\F)}(\widehat{H}^1(V^v)^{\psi^{-1}}_{\rbar},W(\F))$ (resp.~$\widehat{S}(V^v)_{\rbar}^{\psi^{-1}}$ and its dual) are
naturally
$\widehat{\mathbb{T}}(V^v)^{\psi^{-1}}_{\rbar}$-modules.

Let $R_{\rbar,S\cup\set{w_1}}^{\psi}$ be the universal deformation
$W(\F)$-algebra of $\rbar$ parametrizing (unframed) deformations $r$ of $\rbar$ such that
$r$ is unramified outside of $S\cup\set{w_1}$ and $\psi=\eps\det(r)$. It
follows from the construction of $\mathbb{M}_\infty$ in
\cite[\S2]{CEGGPS}, \cite[\S9]{ScholzeLT} and \cite[\S6]{DoLe} that we
have a sequence of surjective morphisms of local rings
\begin{equation}\label{eq:bigalgebrasmaps}
R_\infty\otimes_{S_\infty}W(\F)=R_\infty/(x_1,\dots, x_{4|S|+q-1})\twoheadrightarrow
  R_{\rbar,S\cup\set{w_1}}^\psi\twoheadrightarrow\widehat{\mathbb{T}}(V^{v})_{\rbar}^{\psi^{-1}}
\end{equation}
and a compatible isomorphism
\[ \mathbb{M}_\infty\otimes_{S_\infty}W(\F)=\mathbb{M}_\infty/(x_1,\dots, x_{4|S|+q-1}) \simeq
  \Hom_{W(\F)}(\widehat{H}^1(V^v)_{\rbar}^{\psi^{-1}},W(\F))\]
(resp.\ $\mathbb{M}_\infty\otimes_{S_\infty}W(\F)\cong
\Hom_{W(\F)}(\widehat{S}(V^v)_{\rbar}^{\psi^{-1}},W(\F))$ in the
definite case). Note that among all choices involved in this
construction, there is a choice of basis of the universal deformation
over $R_{\rbar,S\cup\set{w_1}}^{\psi}$.
\begin{cor}\label{cor:platitude_Hecke}
  All the maps in \eqref{eq:bigalgebrasmaps} are
  isomorphisms. Moreover $\Hom_{W(\F)}(\widehat{H}^1(V^v)^{\psi^{-1}}_{\rbar},W(\F))$
  \emph{(}resp.~$\Hom_{W(\F)}(\widehat{S}(V^v)_{\rbar}^{\psi^{-1}},W(\F))$\emph{)} is
  a faithfully flat
  $\widehat{\mathbb{T}}(V^v)_{\rbar}^{\psi^{-1}}$-module and
  $\widehat{\mathbb{T}}(V^v)_{\rbar}^{\psi^{-1}}$ is a complete
  intersection.
\end{cor}
\begin{proof}
  This is \cite[Prop.\ 4.3.1]{GN}. However, since our setup is
  slightly different, we reproduce the proof in our case. We only prove
  the case of Shimura curves, the definite case being identical.

  We first notice that
  $\Hom_{W(\F)}(\widehat{H}^1(V^v)_{\rbar}^{\psi^{-1}},W(\F))\simeq
  \mathbb{M}_\infty/(x_1,\dots, x_{4|S|+q-1})$ is a faithfully flat
  $R_\infty/(x_1,\dots, x_{4|S|+q-1})$-module, since $\mathbb{M}_\infty$ is a faithfully flat $R_\infty$-module by Theorem \ref{thm:flat_infty}. 
  As a consequence, the
  composite of the maps
  \[ R_\infty/(x_1,\dots, x_{4|S|+q-1})\twoheadrightarrow
    R_{\rbar,S\cup\set{w_1}}^\psi\twoheadrightarrow\widehat{\mathbb{T}}(V^v)_{\rbar}^{\psi^{-1}}\rightarrow\End_{W(\F)}\Big(\Hom_{W(\F)}(\widehat{H}^1(V^v)_{\rbar}^{\psi^{-1}},W(\F))\Big)\]
  is injective. This proves the first claim and the faithful flatness
  of $\Hom_{W(\F)}(\widehat{H}^1(V^v)^{\psi^{-1}}_{\rbar},W(\F))$ as a
  $\widehat{\mathbb{T}}(V^{v})_{\rbar}^{\psi^{-1}}$-module. As
  $\mathbb{M}_\infty$ is a faithfully flat $R_\infty$- and
  $S_\infty$-module, $R_\infty$ is a faithfully flat
  $S_\infty$-module. As $(x_1,\dots, x_{4|S|+q-1})$ is a regular
  sequence in $S_\infty$, it is $R_\infty$-regular and therefore
  $R_\infty/(x_1,\dots, x_{4|S|+q-1})\simeq
  R_{\rbar,S\cup\set{w_1}}^{\psi}\simeq\widehat{\mathbb{T}}(V^{v})_{\rbar}^{\psi^{-1}}$
  is a complete intersection.
\end{proof}

\begin{rem}
We expect the statement of Corollary \ref{cor:platitude_Hecke} to hold without assuming that $p$ is inert in $F$: one should extend the results of \S \ref{sec:GKsection} to include \emph{all} places above $p$, or use a non-constant coefficient system at all places $w\in S_p\backslash \{v\}$.
This is somewhat beyond the purpose of this work, and we decided not to pursue it here.
\end{rem}

\newpage

\bibliography{Biblio}
\bibliographystyle{amsalpha}

\end{document}